\documentclass[11pt,a4paper]{report}
\usepackage{etex}
\usepackage{inputenc}
\usepackage[T1]{fontenc}
\usepackage{lmodern}					
\usepackage{textcomp}				
\usepackage{verbatim}				
\usepackage[english]{babel}
\usepackage[toc,page]{appendix} 		
\usepackage{amsmath,					
			amsthm,
			amsrefs, 
			amsfonts,
			amssymb}
\usepackage{mathtools}
\usepackage[mathscr]{euscript}

\usepackage[all]{xy}
\usepackage{tikz-cd} 
\usetikzlibrary{tikzmark}

\usepackage{graphicx}				
\usepackage{tabularx}				

\usepackage{authblk}   

\definecolor{pur}{RGB}{255,102,102}

\usepackage[hmargin=2.2cm,vmargin=2.5cm]{geometry}		
\numberwithin{equation}{section}

\usepackage{hyperref}			

\newcommand{\bigslant}[2]{{\left.\raisebox{.2em}{$#1$}\middle/\raisebox{-.2em}{$#2$}\right.}}
\newcommand{\F}[1]{\mathfrak{#1}}
\newcommand{\C}[1]{\mathcal{#1}}
\newcommand{\R}[1]{\mathrm{#1}}

\newcommand{\cinf}{\mathcal{C}^{\infty}}

\newcommand{\dd}{\mathrm{d}}

\newcommand{\Funct}{\cinf( M)}

\newcommand{\funct}{\mathscr{E}}
\newcommand{\functt}{\mathscr{F}}

\newcommand{\Vect}{\F{X}}

\newcommand{\diff}{\mathrm d}

\newtheorem{definition}{Definition}[section]
\newtheorem*{definition*}{Definition}
\newtheorem{theoreme}[definition]{Theorem}
\newtheorem*{theoreme*}{Theorem}
\newtheorem{proposition}[definition]{Proposition}

\newtheorem{lemme}[definition]{Lemma}
\newtheorem*{lemme*}{Lemma}

\newtheorem{corollaire}[definition]{Corollary}

\theoremstyle{remark}

\newtheorem*{remarque}{Remark}
\newtheorem{example}{Example}

\newtheorem*{prooftheo*}{Proof of the Theorem}

\begin{document}

\thispagestyle{empty} 

\newgeometry{lmargin=2.2cm,rmargin=2.2cm,tmargin=2.5cm,bmargin=2.5cm}
\thispagestyle{empty} 
\begin{center}
\begin{minipage}{0.8\textwidth}
\begin{center}
\vskip4cm
\textbf{\LARGE 
\scalebox{1}[1.4]{Lie $\infty$-alg\'ebro\"ides et Feuilletages Singuliers}
\vskip3pt
\scalebox{1}[1.4]{}
}
\vskip6pt
\textbf{\Large\scalebox{1}[1.4]{}}
\end{center}
\end{minipage}
\end{center}

\vfil

\begin{figure}[!h]										 
\begin{center}											 
\includegraphics[width=10cm]{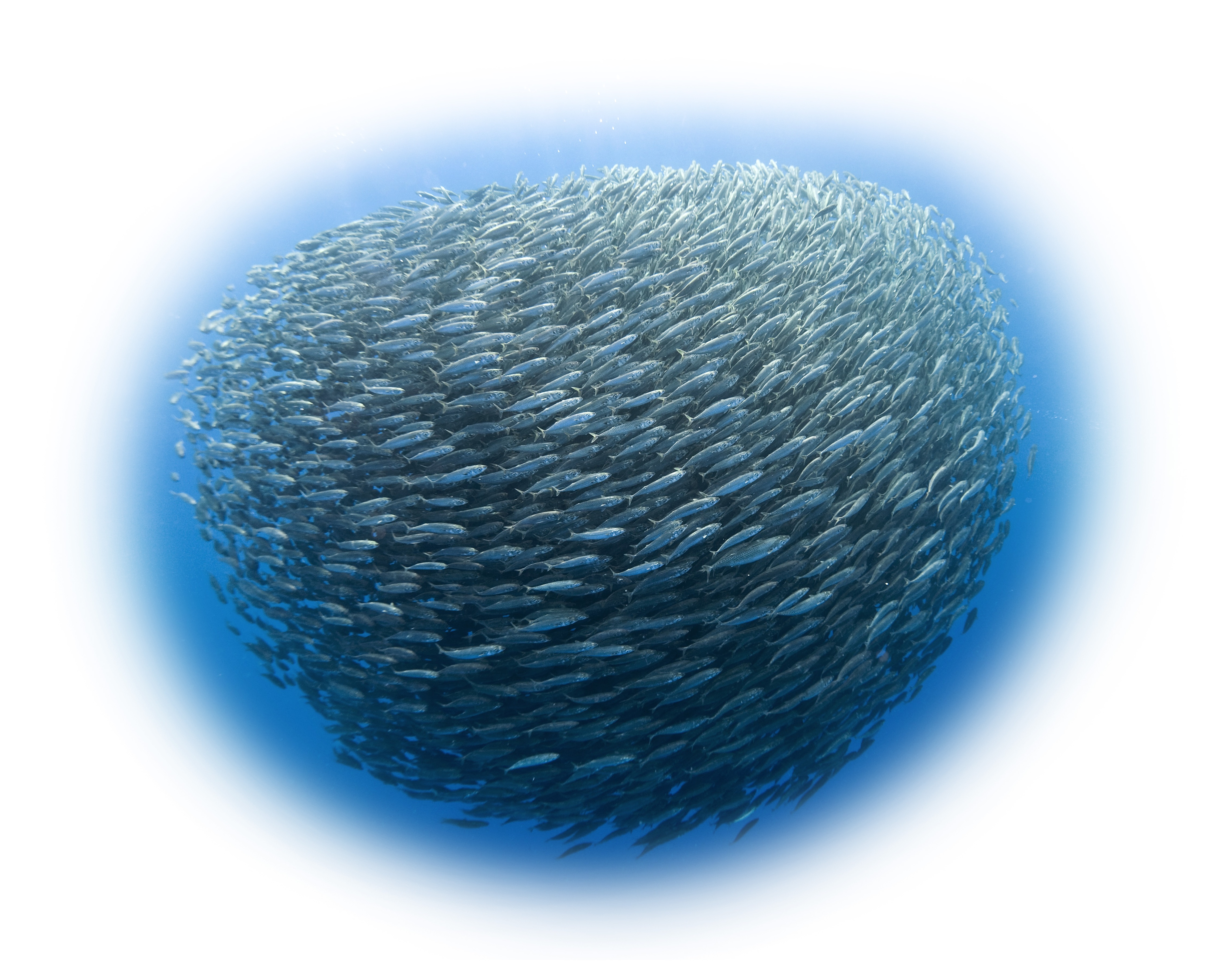}{\centering}  	 		 
\end{center}											 
\end{figure}

\vfil

\begin{center}
\textbf{\LARGE Sylvain Lavau}
\vskip0.4cm
{\LARGE Th\`ese de doctorat}
\end{center}

\vfil

\chapter*{}

\thispagestyle{empty} 




\newgeometry{lmargin=2.2cm,rmargin=2.2cm,tmargin=2.5cm,bmargin=2.5cm}

\thispagestyle{empty} 

\vfil

\begin{center}
\begin{Large}
Universit\'e Claude Bernard Lyon 1\\  
\vskip3pt
\'Ecole doctorale {\bf InfoMath}, ED~512\\ 
\vskip3pt
Sp\'ecialit\'e~: \textbf{Math\'ematiques}\\
\vskip4pt
N.~d'ordre 2016LYSE1215
\end{Large}
\end{center}

\vfil
\vfil

\begin{center}
\textbf{\LARGE
{\scalebox{1}[1.4]{Lie $\infty$-alg\'ebroides et Feuilletages Singuliers}\\
\vskip8pt
\scalebox{1}[1.4]{}}\\
\vskip8pt
\textbf{\Large
{\scalebox{1}[1.4]{}}}
}
\end{center}

\vfil
\vfil

\begin{center}
\begin{LARGE}
\textbf{Th\`ese de doctorat}
\end{LARGE}
\end{center}

\vfil

\begin{center}
{\Large
Soutenue publiquement le 4 Novembre 2016 par}
\end{center}

\vfil

\begin{center}
  \begin{Large}
\textbf{Sylvain Lavau}
  \end{Large}
\end{center}

\vfil

\begin{center}
{\Large devant le Jury compos\'e de:}

\vfil
\begin{large}
\begin{tabular}{lll}
Mme. Yvette Kosmann- &\quad&\\[3pt]
\qquad \qquad Schwarzbach &\quad \'Ecole Polytechnique &
\\[3pt]
M. Anton Alekseev &\quad Universit\'e de Gen\`eve &\qquad {\small Rapporteur}\\[3pt] 
M. Marco Zambon &\quad Katholieke Universiteit Leuven&\qquad {\small Rapporteur}\\[3pt]
Mme. Claire Debord &\quad Universit\'e Blaise Pascal (Clermont-Ferrand) &\\[3pt]
M. Henning Samtleben &\quad Ecole Normale Sup\'erieure de Lyon &\qquad  {\small Codirecteur de th\`ese} \\[3pt]	
M. Thomas Strobl &\quad Universit\'e Claude Bernard (Lyon 1) &\qquad {\small Directeur de th\`ese}\\[3pt]
\end{tabular}
\end{large}
\end{center}


%
%

\chapter*{}
\thispagestyle{empty}

\chapter*{}
\thispagestyle{empty}
\vspace{3cm}

\begin{flushright}
\emph{To Elyssa Beck Schaeffer}\hspace{-1cm}

\emph{}
\end{flushright}

\chapter*{}

\thispagestyle{empty}

\chapter*{}
\thispagestyle{empty}
\vspace{7cm}

\begin{flushright}
\emph{\`A mon fr\`ere}\hspace{-1cm}

\emph{}
\end{flushright}

\chapter*{}

\thispagestyle{empty}




\tableofcontents

\chapter*{}
\thispagestyle{empty}

\chapter*{Introduction}
\addcontentsline{toc}{chapter}{Introduction}

In a broad range of fields in Mathematics and Physics, one often encounters situations where \emph{foliations} $-$ possibly singular $-$ arise, from the geometric representation of the solutions of differential equations and subsequent applications in dynamical systems, to intrinsically topological curiosities such as the Reeb foliation which is a partition of the 3-sphere by tori. A foliation can be seen as a generalization to arbitrary dimensions of the pages of a book. The leaves carry different meanings, depending on the context, they can symbolize the trajectory of a particle (in the case of one-dimensional leaves), or they can be the set of points in the phase space which characterize the same state of a physical system, or even the Universe at a given time. While this notion is very important in physical applications, the theory of foliations has become an independent field of research since the 1940's following the work of C. Ehresmann and G. Reeb.

A smooth manifold $M$ is said to be \emph{foliated} if it admits a partition into immersed and connected submanifolds, called \emph{leaves}. If the dimension of the leaves is constant over $M$ then the foliation is called \emph{regular}, and if the dimension `jumps' from one leaf to another the foliation is called \emph{singular}. This gap, if it exists, can only locally be made toward higher dimensions, since the function evaluating the dimension of the leaf at a point is lower semi-continuous. 
We call regular leaves the leaves of highest dimension, whereas the others are called singular leaves. If a point is on a regular leaf then every point in a small neighborhood is also on a regular leaf. On the contrary, any point on the transversal of a singular leaf belongs either to another adjoining singular leaf, or to a regular leaf.

Since any leaf $S$ of a given foliation is a submanifold of $M$, at every point $x\in S$ it has a canonical tangent bundle $T_{x}S$, which is a subspace of the tangent space $T_{x}M$. For a smooth foliation this subspace should vary smoothly, and moreover stay invariant under the Lie bracket for vector fields. In other words, over the leaf $S$, it forms a subbundle of the tangent bundle restricted to $S$. Since other leaves can have various dimensions, this subspace may not remain of constant rank when we travel on the manifold. In particular on a singular leaf it will have a smaller dimension than on a regular leaf. We are thus naturally led to define the generalization of a vector bundle, what is called a \emph{distribution}: the assignment $\mathcal{D}_{x}$ of a subspace of $T_{x}M$ for each point $x\in M$. Since the tangent spaces of the leaves of a foliation form a distribution, a legitimate problem is to ask what are the conditions for a distribution to induce a (possibly singular) foliation whose tangent spaces form the given distribution or, in other words, to ask whether the distribution is \emph{integrable}.

The answer to this question has been sought for years, from the first advances (in different terms though) in the nineteenth century, to its enlightening outcome in the seventies. As usual in Mathematics, the process took time and was not deemed of interest at some periods, but was finally solved in a totally unexpected way. At first, after earlier results, the claim that a smooth regular distribution is integrable if and only if it is involutive was rigorously proven in the modern sense by F. Frobenius (1877). Unfortunately, it is not straightforward to generalize this elegant result to the singular case, and it took nearly one hundred years to find the correct formulation in that context. Indeed it was only at the beginning of the sixties that R. Hermann noticed that foliations were very efficient to reinterpret problems in Control Theory $-$ which up to then was only meant to study servomecanisms. Then, a quick sequence of moves were performed to successively refine the condition that a smooth singular distribution be integrable. Finally in 1973, P. Stefan and H. Sussmann simultaneously formulated the correct condition and, inspired by techniques from Control Theory, they proved the theorem now bearing their names, hence solving a long standing open question.

However, in most cases, the smooth distributions that one encounters in everyday life are much less specific than the one satisfying the conditions of Stefan and 
Sussmann. Lately, Mathematicians have focused their attention to algebraic objects: locally finitely generated involutive sub-sheaves of the sheaf of vector fields $\mathfrak{X}$. They define singular distributions that satisfy Hermann's integrability condition, and they are the most common and the most convenient objects that we happen to manipulate in various situations (such as in the examples given above). For this reason, we have chosen to call these algebraic objects \emph{Hermann foliations}. This shift from a purely analytico-geometric point of view to an algebraic one raises some questions. First, the fact that a foliation can be described as a submodule of $\mathfrak{X}(M)$ entails that it admits a resolution by $\C{C}^{\infty}(M)$-modules (in the smooth case). How is such a resolution related to the foliation? For example is it unique? Closure of the distribution under the Lie bracket suggests that we may be able to lift this algebraic operation into something more involved on the resolution. In that case, are there many ways to lift it, and how are they related to one another? And finally, what does this algebraic structure say about the foliation?

The results presented in this thesis give precise answers to these questions, when the resolution consists of graded vector bundles over $M$. They involve structure coming from higher algebra, namely, \emph{Lie $\infty$-algebroids} and \emph{Lie ${\infty}$-morphisms up to homotopy}. Lie $\infty$-algebroids are a natural mixing of Lie algebroids and $L_{\infty}$-algebra structures. The notion of Lie algebroid was first introduced by J. Pradines (1967) as the infinitesimal structure associated to Ehresmann's Lie groupoids. A Lie algebroid consists of a vector bundle $A$ over a manifold $M$, whose space of sections is equipped with a Lie bracket (hence the name), and a vector bundle morphism $\rho:A\to TM$ compatible with the bracket. Lie algebroids generalize both Lie algebras and tangent bundles of manifolds. On the other hand, $L_{\infty}$-algebras have their origin in mathematical physics, and were introduced at the beginning of the 90's by T. Lada and J. Stasheff as a natural extension of differential graded Lie algebras. They consist in a graded vector space equipped with a family of multilinear $n$-ary brackets that satisfy generalized Jacobi identities. These structures are also called strongly homotopy Lie algebras because originally the Jacobi identity was weakened and only satisfied \emph{up to homotopy}.

Hence, Lie $\infty$-algebroids consist of a family of (positively) graded vector bundles whose space of sections is a $L_{\infty}$-algebra, together with a vector bundle morphism from the bundle of highest degree to $TM$ satisfying some compatibility condition with the 2-bracket, inspired by the one encountered in the Lie algebroid case. The natural notion of morphism between two Lie $\infty$-algebroids is only defined up to homotopy (in some sense), which is appropriate to define the notion of isomorphism in the category of Lie $\infty$-algebroids. Unfortunately, Lie $\infty$-algebroids are not very convenient to work with in this form, so we have chosen in this thesis to use their characterization in terms of \emph{differential graded manifolds}. These manifolds are graded manifolds equipped with a degree $+1$ homological vector field $Q$. Many geometric objects, such as Lie algebroids and Poisson manifolds can be encoded in this language, and moreover most proofs in this manuscript are based on the property that $[Q,Q]=0$, hence justifying the use of this notion.

The aim of the present thesis is to show how one can build a Lie $\infty$-algebroid structure on a resolution of a Hermann foliation, and to study its uniqueness, and its consequences. In this context, a resolution is defined as a family of graded vector bundles $E=\bigoplus_{i\geq1}E_{-i}$ over $M$ such that the differential is exact at the level of sections. The first result of this thesis is that one can lift the Lie bracket of vector fields to a Lie $\infty$-algebroid structure on the resolution. The construction is canonical and formally consists in building the $n$-ary brackets from the data of all other brackets of lower arities. Moreover we show that this Lie $\infty$-algebroid structure is unique in the following sense: for any other choice of a Lie $\infty$-algebroid structure resolving the distribution, the two structures are isomorphic up to homotopy. Hence we call this Lie $\infty$-algebroid (or any other Lie $\infty$-algebroid structure resolving the Hermann foliation) the \emph{universal Lie $\infty$-algebroid} associated to the distribution. This structure carries information on the foliation which has a deep geometrical meaning, as we shall see. The existence of such a structure is guaranteed in the real analytic and holomorphic cases, at least locally.

The first natural application of this structure is that one can consider the $Q$-cohomology, that we call the \emph{universal foliated cohomology}. The name is justified by the fact that two Lie $\infty$-algebroid structures associated to a given Hermann foliation have isomorphic cohomologies. We show that this cohomology is related in some sense to the foliated de Rham cohomology of the manifold. The next results are obtained when one restricts the study at a point $x\in M$. The universal Lie $\infty$-algebroid then reduces to a mere Lie $\infty$-algebra on the fiber $E_{x}=\bigoplus_{i\geq1}E_{-i\ x}$ of the graded vector bundle over $x$. Since the exactness of the resolution is not necessarily preserved, we may have to consider the cohomology. Hence the universal Lie $\infty$-algebroid structure reduces to a graded Lie algebra structure on the cohomology over $x$, that we call the \emph{holonomy graded Lie algebra of the foliation at $x$}. Several points have to be emphasized: by uniqueness, any other choice of resolution gives a graded Lie algebra over $x$ which is canonically isomorphic to the former holonomy graded Lie algebra; moreover, the various holonomy graded Lie algebras defined at each point of a leaf are isomorphic. Hence one can define a holonomy graded Lie algebroid over each leaf, which restricts $-$ when one only considers the first level of the tower of graded vector bundles $M$ $-$ to the holonomy Lie algebroid defined by I. Androulidakis and G. Skandalis.

Here are possible applications and investigations of our results, one idea would be to answer Androulidakis-Zambon conjecture whish says that not all Hermann foliations come from Lie algebroids (as the image of the anchor map), not even locally. Another idea (due to I. Androulidakis and M. Zambon) may be to investigate to what extent the length of the resolution characterizes the singularity of the foliation. In the same spirit, one could apply this construction to the newly defined singular subalgebroids, which consist of locally finitely generated involutive submodule of the space of sections of some algebroid. If this algebroid is integrable, then it defines a singular foliation on the associated groupoid. Another promising field of application is the Batalin-Vilkoviski formalism and its well-known infinite hierarchy of ghost fields. It appears that the classical action defines a foliation on the space-time manifold, so there may be a link between the ghost hierarchy and the universal Lie $\infty$-algebroid associated to this foliation. And finally, one can apply all these results to Leibniz algebroids. Recently some advances have been made in the link between Leibniz algebras and the Tensor Hierarchy (as defined in supergravity theories) by T. Strobl and H. Samtleben. We observed that a Leibniz algebra can lead to the construction of a Lie $\infty$-algebra structure on an infinite tower of spaces. Passing to the -oid case would introduce a Lie $\infty$-algebroid, and hence from the only data of a Leibniz algebroid whose image by the anchor map is the Hermann foliation, we may be able to define a Lie $\infty$-algebroid structure associated to it over some graded vector bundle. That would be related to one of the final results of this thesis: that the universal Lie $\infty$-algebroid of a Hermann foliation induces a Leibniz algebroid over $M$ whose image by the anchor map is precisely the distribution.


\bigskip

The first part of this manuscript is a review of various results in graded geometry. We recall in particular the notions of graded manifolds and of Lie $\infty$-algebroids and the equivalence of Lie $\infty$-algebroids with $NQ$-manifolds in Section \ref{Qmanifolds}. We then define the concept of homotopy between Lie $\infty$-morphisms, which provides an important result for the second part. In Section \ref{foliations}, we present earlier results in foliation theory. We recall the notion of singular foliations and distributions, and explain how the problem of integrating smooth distributions to (possibly singular) foliations has a long history. We give the proofs of the theorems that can be seen as landmarks in the resolution of this long-standing open problem, solved independently by H. Sussmann and P. Stefan in the seventies.

In the second part, we give the main results of the thesis. We first explore the concept of \emph{arity}, which is at the core of the discussions and the proofs of the thesis. Then we define the notion of the universal Lie $\infty$-algebroid of a Hermann foliation, and we show that this structure is uniquely defined, up to homotopy. The end of the second part is devoted to applications and examples, in particular to the study of the restriction of the universal Lie $\infty$-algebroid structure to a leaf of the foliation. I would like to thank Camille Laurent-Gengoux and Thomas Strobl for their invaluable help during the years of my PhD leading to these results.

\nocite{Yvette}
\nocite{gunningrossi}
\nocite{AndrouZamb4}
\nocite{LadaStasheff}
 \nocite{LS}
 \nocite{piat}

\chapter{Higher geometry and integrability of singular foliations}

\section{Differential Graded manifolds and Lie \texorpdfstring{$\infty$}{infinity}-algebroids}\label{Qmanifolds}

Graded algebra and graded differential geometry are natural generalizations of linear algebra and differential geometry that allow vectors and variables to carry degrees. They emerged at the beginning of the seventies, when advances in theoretical physics $-$ such as supersymmetric models $-$ led to the need for some mathematization of particle physics. In the beginning though, the grading was restricted to $\mathbb{Z}/2\mathbb{Z}$, giving the so-called \emph{supermanifolds}, and in this context mathematicians prefer to use the word parity instead of grading. Some famous names were involved in the developpment and the improvement of these notions, such as F. Berezin and D. Leites \cite{Leites} who are considered as some sort of fathers of the field at the beginning of the seventies, and M. Batchelor who proved a famous theorem identifying graded manifolds and graded vector bundles. The particularity of graded geometry is that it does not put the points of the manifold at the core of the theory, but it rather focuses on the sheaf of functions associated with the manifold. The so called \emph{differential graded manifolds} then appear as a natural generalization of differential graded algebras. They were introduced at the beginning of the eighties in the Russian literature under the influence of V. Shander \cites{Shander1, Shander2}, and found applications in many fields in the nineties, examplified by the groundbreaking advances in topological field theories such as the AKSZ formalism. Essentially, a differential graded manifold encapsulates the data under only one flag: that of a degree $+1$ vector field whose selfcommutator vanishes. The advantage of the differential graded picture is that many different structures (such as Lie algebras, Poisson manifolds, etc.) can be analyzed in this framework, which simplifies and unifies the results.

At the same time, advances in theoretical physics led to the discovery of a new notion, as a generalization of differential graded Lie algebras: the so called \emph{$L_{\infty}$-algebras}. The original definition dates back to 1992 by J. Stasheff and T. Lada in their seminar \cite{LadaStasheff}, following considerations by B. Zwiebach and others physicists. The idea is to weaken the condition imposing that the Lie bracket satisfies the Jacobi identity. One asks that it be satisfied \emph{up to homotopy}, hence providing the necessity of adding a 3-bracket so that the Jacobiator becomes a coboundary with respect to the differential. Going further and further, one has to introduce a (possibly infinite) set of $n$-ary brackets compatible with one another. This notion shed a new light on various fields of mathematical physics, in particular in the Batalin-Vilkoviski quantization scheme. Surprisingly, one of the most striking properties is that any $L_{\infty}$-algebra can be encoded as a differential graded vector space. All the brackets are reunited into only one object: the homological vector field $Q$. Finally, in the same spirit, one could generalize the not so recent notion of Lie algebroid (a vector bundle equipped with a bracket and a compatible anchor map, a notion introduced by J. Pradines in 1967) by following a similar path. That was indeed formulated by T. Voronov \cite{voronov} who showed that one could identify $L_{\infty}$-algebroids with differential graded manifolds. That is precisely the point of view that we want to use in the present thesis.

\subsection{Graded vector spaces and higher structures}\label{gradvectspace}

The present manuscript presents notions and results most of which are valid not only in the smooth case, but also in the real-analytic and the holomorphic case. For  clarity however we will not recall at every step that the results are valid in the real analytic or holomorphic context. Then we claim and prove some results in the smooth case, indicating as a remark the validity for the other cases. By convention we use the letter $\mathscr{O}$ to symbolize the sheaf of smooth/real analytic/holomorphic functions. Then either $\mathscr{O}=\cinf(M)$ in the smooth case, or $\mathscr{O}=\mathcal{C}^{\omega}$ in the real analytic case, and $\mathscr{O}=\mathscr{H}$ in the holomorphic case.The only section that makes a crucial distinction between the smooth and the real analytic context is the historical section \ref{foliations} on singular foliations.

A differential graded manifold is a natural generalization of a smooth manifold which can be seen intuitively as a topological space locally homeomorphic to a graded vector space. A \emph{graded vector space} is a family of finite-dimensional vector spaces $E=\bigoplus_{i\in\mathbb{Z}}E_{i}$ indexed by $\mathbb{Z}$. An element $u$ belonging to some $E_{i}$ is said to be \emph{homogeneous of degree $i$}, and we write $|u|$ for the degree of $u$. Elements of the dual space $E_{i}^{\ast}$ send a degree $i$ element to a real number, and since $\mathbb{R}$ is considered to have degree 0, we deduce that the elements of $E_{i}^{\ast}$ have degree $-i$. The set of all linear forms on $E$ is a graded vector space $E^{\ast}$ consisting of the direct sum of the dual spaces of all degrees: $E^{\ast}=\bigoplus_{i\in \mathbb{Z}}E_{i}^{\ast}$.

The set of smooth functions on $E$ is not defined as for classical vector spaces where smoothness is the property of being infinitely differentiable. In the graded case, the convention is to define the set of functions on $E$ to be the graded symmetric algebra (with respect to the degree) of $E^{\ast}$: $S(E^{\ast})$. The degree has a profound influence here, given that for any two homogeneous elements $\xi_{1},\xi_{2}\in E^\ast$, the symmetric product of these elements is:
\begin{equation}
\xi_{1}\odot\xi_{2}=(-1)^{|\xi_{1}||\xi_{2}|}\xi_{2}\odot\xi_{1}
\end{equation}
The set of functions on $E$ equipped with the (graded) symmetric product becomes a graded algebra, in the sense that if $f,g\in S(E^{\ast})$ are homogeneous functions of respective degree $|f|$ and $|g|$, then their product is $f\odot g$ and it is a function of degree $|f|+|g|$. A \emph{morphism} between two graded vector spaces $E$ and $F$ is considered dually as a morphism $\Phi:S(F^{\ast})\to S(E^{\ast})$ of the corresponding graded algebras of functions. Morphisms of graded vector spaces can carry a degree, and those of degree zero are called \emph{strict morphisms}.

\begin{example} Let $V$ be a vector space of dimension $n$, and let $E=\bigwedge^{\bullet} V$ be its exterior algebra. By convention, we say that a $k$-vector $\alpha_{1}\wedge\ldots\wedge\alpha_{k}$ is of degree $-k$, making $E$ a graded vector space. It has an additional structure, given that the wedge product is graded skew-commutative, turning $E$ into a \emph{graded algebra}, that is: a graded vector space equipped with a product $\cdot$ (or here the wedge product) such that $x\cdot y\in E_{i+j}$ for every $x\in E_{i}$ and $y\in E_{j}$.
\end{example}

As in the previous example, graded vector spaces may have more intricate graded structures obtained by adding inner operations such as brackets or differentials. By introducing brackets which satisfy a graded generalization of the Jacobi identity, we now define \emph{graded Lie algebras}:

\begin{definition}
A \emph{graded Lie algebra} is a graded vector space $E$ together with a bilinear operation $[\ .\ ,\,.\ ]$ called the \emph{bracket} satisfying the following conditions:
\begin{itemize}
\item skewsymmetry \hfill \makebox[0pt][r]{%
            \begin{minipage}[b]{\textwidth}
              \begin{equation}
                 \hspace{0.7cm}[x,y]=-(-1)^{|x||y|}[y,x]
              \end{equation}
          \end{minipage}}
\item Jacobi identity \hfill \makebox[0pt][r]{%
            \begin{minipage}[b]{\textwidth}
              \begin{equation}
                 \hspace{2.32cm}\big[x,[y,z]\big]=\big[[x,y],z\big]+(-1)^{|x||y|}\big[y,[x,z]\big]
              \end{equation}
          \end{minipage}}
\end{itemize}
for all homogeneous elements $x,y,z\in E$.
\end{definition}

\begin{example}
A natural example of a graded Lie algebra is based on a well-known fact: that any associative algebra $A$ induces a Lie algebra structure on the space of derivations of $A$. Here we consider a graded algebra $(E=\bigoplus_{i\in\mathbb{Z}}E_{i},\cdot)$. A \emph{derivation of degree $k$} on $E$ is a linear application $\delta$ which satisfies the following two conditions:
\begin{itemize}
\item $\delta:E_{i}\to E_{i+k}$
\item $\delta(x\cdot y)=\delta(x)\cdot y+(-1)^{k|x|}x\cdot\delta(y)$
\end{itemize}
for every homogeneous elements $x,y\in E$. The set $\mathrm{Der}(E)$ of all derivations on $E$ is a graded vector space, graded by the degree of the derivations: $\mathrm{Der}(E)=\bigoplus_{l\in\mathbb{Z}}\mathrm{Der}_{l}(E)$. Then one defines a graded Lie bracket on $\mathrm{Der}(E)$ using the commutator:
\begin{equation}
[\delta_{1},\delta_{2}]=\delta_{1}\circ\delta_{2}-(-1)^{kl}\delta_{2}\circ\delta_{1}
\end{equation}
for any derivations $\delta_{1}\in \mathrm{Der}_{k}(E)$ and $\delta_{2}\in \mathrm{Der}_{l}(E)$. This turns $\big(\mathrm{Der}(E), [\ .\ ,\, .\ ]\big)$ into a graded Lie algebra.
\end{example}

A graded vector space $E=\bigoplus_{i\in\mathbb{Z}}E_{i}$ may be equipped with a differential $\dd:E_{i}\to E_{i+1}$ which turns $E$ into a chain complex. If, moreover, $E$ is a graded Lie algebra and the differential is compatible with the bracket in the sense that:
\begin{equation}\label{eq:dgla}
\dd\big([x,y]\big)=[\dd x,y]+(-1)^{|x|}[x,\dd y]
\end{equation}
for every $x,y\in E$, then we speak of a \emph{differential graded Lie algebra} $-$ or dgLa for short. Morphisms in the category of chain complexes are degree 0 linear maps that are called \emph{chain maps}, whose main property is to commute with the respective differentials. A chain map $f:E\to F$ between two chain complexes $E$ and $F$ induces a strict morphism of the respective graded vector spaces by taking its dual and extending it as a morphism of algebras on all of $S(F^{\ast})$. However, the fundamental notion arising in such a context is the notion of \emph{homotopy}, which in that case crucially comes from the fact that the differential squares to zero. More precisely, we say that two morphisms of chain complexes $f,g:E\to F$ are \emph{homotopic} if there exists a  map $h:E\to F$ of degree $-1$ such that
\begin{equation}
f-g=\dd' \circ h+h\circ\dd
\end{equation}
A diagram may be helpful to visualize the situation (see \cite{Weibel} for details):
\begin{center}
\begin{tikzcd}[column sep=1.5cm,row sep=0.4cm]
\ldots\ar[r]&\ar[dd,shift left =0.5ex]\ar[dd,shift right =0.5ex]E_{i-1}\ar[r,]&\ar[ddl,dashed,"h" above]\ar[dd,shift left =0.5ex,"f"]\ar[dd,shift right =0.5ex,"g" left]E_{i}\ar[r,"\dd"]&\ar[ddl,dashed,"h" above]\ar[dd,shift left =0.5ex]\ar[dd,shift right =0.5ex]E_{i+1}\ar[r]&\ldots\\
&&&&\\
\ldots\ar[r]&F_{i-1}\ar[r,"\dd'"]&F_{i}\ar[r]&F_{i+1}\ar[r]&\ldots\\
\end{tikzcd}
\end{center}

In the nineties, in order to solve problems coming from theoretical physics, J.Stasheff and T.Lada developed a generalization of dgLa \cite{LadaStasheff}, in which the Jacobi identity is no longer satisfied. It is only satisfied \emph{up to homotopy}, in the sense that the (graded) Jacobiator is a $[\dd,\,.\ ]$-coboundary, i.e. there exists a graded skew-symmetric map $l_{3}:\bigwedge^{3}E\to E$ which satisfies:
\begin{equation}
(-1)^{|x||z|}\big[x,[y,z]\big]+(-1)^{|y||x|}\big[y,[z,x]\big]+(-1)^{|z||y|}\big[z,[x,y]\big]=\dd\circ l_{3}(x,y,z) +l_{3}\circ\dd(x\wedge y\wedge z)
\end{equation}
where $\dd$ acts as a derivation on $\bigwedge^{3}E$. This equation can be rewritten as:
\begin{equation}
l_{2}\circ l_{2}=[\dd,l_{3}]
\end{equation}
where $l_{2}$ stands for the 2-bracket $[\ .\ ,\,.\ ]$, and where the bracket on the right-hand side is a notation for the commutator of $\dd$ and $l_{3}$. For degree reasons, the 3-bracket is a degree $-1$ operation. For consistency of the structure, the 3-bracket has to satisfy some sort of Jacobi identity, which can be symbolically written as:
\begin{equation}
l_{2}\circ l_{3}+l_{3}\circ l_{2}=0
\end{equation}
In full generality, we can assume that this equation is satisfied only \emph{up to homotopy}, that is: there is a graded skew-symmetric map $l_{4}:\bigwedge^{4}E\to E$ such that
\begin{equation}
l_{2}\circ l_{3}+l_{3}\circ l_{2}=[\dd,l_{4}]
\end{equation}
Again, the 4-bracket $l_{4}$ has to satisfy some sort of Jacobi identity as above, but if it is only satisfied up to homotopy, a 5-bracket has to be added to the content of the structure. Going further and further, we arrive at the concept of what was historically called  \emph{strongly homotopy Lie algebra} \cites{Marklada,LadaStasheff} (or \emph{sh-Lie algebra}).


In the literature, strongly homotopy Lie algebras have been known under several conventions $-$ all equivalent. The classical formulation found in the literature relies on graded skew-symmetric brackets satisfying so called \emph{higher Jacobi identities}. It has the advantage that it contains the Lie algebra case, but it may not be the more efficient to work with. In the present work we will use a different formulation which expresses the notion of sh-Lie algebras in terms of a codifferential on a graded manifold. The two notions are related by the fact that the skew-symmetric definition of a sh-Lie algebra structure on a graded vector space $E$ has an equivalent formulation in terms of symmetric brackets on the \emph{suspended version of $E$}, which is denoted by $E[1]$. The codifferential is then the dual operator corresponding to this family of brackets, and the fact that it squares to zero is equivalent to saying that the brackets satisfy the higher Jacobi identities.

Given a graded vector space $E=\bigoplus_{i\in \mathbb{Z}}E_{i}$, the suspended space $E[1]$ is a graded vector space whose subspace of degree $i$ is the vector space $E_{i+1}$, that is: $E[1]_{i}=E_{i+1}$.  Since the degrees are shifted by one, the degree and the symmetries of  morphisms are modified as well. Moreover, there is an isomorphism between the following spaces, for each $i\in\mathbb{Z}$:
\begin{equation*}
\mathrm{Hom}^{i}\Big(\bigwedge^{n}E,E\Big)\simeq \mathrm{Hom}^{i+n-1}\Big(S^{n}E[1],E[1]\Big)
\end{equation*}
so that one can turn to either one definition or the other, depending on one's preference (for more informations on this isomorphism, see \eqref{eq:formula1}-\eqref{eq:formula2} and \cite{Fiorenza}). We now  turn our attention to the symmetric formulation of what we will call in the present text \emph{$L_{\infty}$-algebras} (following Voronov's convention \cite{voronov2}).

\begin{definition} \label{def:Linftyalgebra}
 An \emph{$L_{\infty}$-algebra} is a graded vector space $E$ equipped with a family of graded symmetric $k$-multilinear maps $ \big(\{ \ldots \}_k\big)_{k \geq 1} $ of  degree $+1$, for all $k \geq 1$, such that they satisfy the \emph{generalized Jacobi identities}.
That is, for every homogeneous elements $x_{1},\ldots,x_{n}\in E$:
\begin{equation}\label{superjacobi}
\sum_{i=1}^{n}\ \sum_{\sigma\in Un(i,n-i)}\epsilon(\sigma)\,\big\{\{x_{\sigma(1)},\ldots,x_{\sigma(i)}\}_{i},x_{\sigma(i+1)},\ldots,x_{\sigma(n)}\big\}_{n-i+1}=0
\end{equation}
where $Un(i,n-i)$ is the set of $(i,n-i)$-\emph{unshuffles}, i.e. the permutations $\sigma$ of $n$ elements which preserve the order of the first $i$ elements and the last $n-i$ elements:
\begin{equation*}
\sigma(1)<\ldots<\sigma(i)\hspace{2cm}\sigma(i+1)<\ldots<\sigma(n)
\end{equation*}
and $\epsilon(\sigma)$ is the sign induced by the permutation of elements in the symmetric algebra of $E$.
\end{definition}

Writing $\dd$ for $\{\, . \,\}_{1}$ and $m_{n}$ for the $n$-bracket $\{\ \ldots \}_{n}$, the first higher Jacobi identities are:
\begin{itemize}
\item $\dd\circ\dd=0$ and $[\dd,m_{2}]=0$ imply that the differential $\dd$ is chain map and a derivation for the 2-bracket
\item $m_{2}\circ m_{2}+[\dd,m_{3}]=0$ means that the classical Jacobi identity for $m_2$ is satisfied \emph{up to homotopy}
\item the relation $\sum_{i=1}^{k}m_{i}\circ m_{k-i}=0$, valid for every $k\geq4$ are the 'new' Jacobi identities
\end{itemize}
These higher Jacobi identities can be seen as consistency conditions for the brackets involved in the $L_{\infty}$-algebra structure.

In the literature \cites{Mehta, Schatz}, $L_{\infty}$-algebras as defined above are usually called $L_{\infty}[1]$-algebras, in order not to confuse them with the original definition given by J. Stasheff, T. Lada and M. Markl in \cites{Marklada,LadaStasheff}. It was indeed problematic at some point because mathematicians had to navigate between the original definition with graded skew-symmetric brackets, and the above one (due to Voronov) which uses graded symmetric brackets. Both definitions are equivalent, up to shifting the degree and changing some signs. The advantage of the original convention is that it encapsulates automatically Lie algebras as a special case, whereas Voronov's convention does not. In some cases though, it is more efficient to use the symmetric version of the definition.

In this thesis we have chosen Voronov's convention and used graded symmetric brackets and we call $L_{\infty}$-algebra what is usually called an $L_{\infty}[1]$-algebra because it brings fluency and clarity to the text. Since the two definitions (symmetric and skew-symmetric) are related by an isomorphism, there is some confusion about names. The main point is that when a sh-Lie algebra structure is defined on a graded vector space $E$, there is a canonical $L_{\infty}$-algebra structure (in our sense) equipping $E[1]$, and conversely. For example take $n$ homogeneous elements $x_{1},\ldots,x_{n}$ of $E$, and denote by $y_{1},\ldots,y_{n}$ their respective representatives in $E[1]$. Then the isomorphism between the graded skew-symmetric bracket on $E$ and the graded symmetric bracket on $E[1]$ is given by:
\begin{equation}\label{eq:formula1}
\big[x_{1},\ldots,x_{n}\big]_{n}=(-1)^{n(2-n)+\sum_{i=1}^{n}(n-i)(|y_{i}|+1)}\,\big\{y_{1},\ldots,y_{n}\big\}_{n}
\end{equation}
where $|y_{i}|=|x_{i}|-1$ for every $k=1,\ldots,n$. The graded skew-symmetry on the left-hand side indeed translates as a graded symmetry on the right-hand side, because all degrees have been shifted by one.  The reverse formula can be obtained by:
\begin{equation}\label{eq:formula2}
\big\{y_{1},\ldots,y_{n}\big\}_{n}=(-1)^{n(2-n)+\sum_{i=1}^{n}(n-i)|x_{i}|}\,\big[x_{1},\ldots,x_{n}\big]_{n}
\end{equation}
See \cite{Fiorenza} for further details.

\begin{example}
In this context a graded Lie algebra is nothing but a graded vector space $E$, equipped with a graded symmetric bracket of degree $+1$ satisfying the Jacobi identity:
\begin{equation}\label{eq:jacobisym}
\big\{x,\{y,z\}\big\}=\big\{\{x,y\},z\big\}+(-1)^{|x||y|}\big\{y,\{x,z\}\big\}
\end{equation}
for every homogeneous elements $x,y,z\in E$. The relation with the former definition is given by shifting the degree of all elements by $1$, and then using a graded skew-symmetric bracket of degree 0 given by:
\begin{equation}
[\widetilde{x},\widetilde{y}]=(-1)^{|x|}\{x,y\}
\end{equation}
where $\widetilde{x}$ is the element $x$ whose degree has been shifted by $+1$.
\end{example}

\begin{example}
Using this isomorphism, a differential graded Lie algebra can be seen as a graded vector space $E$ equipped with a symmetric bracket satisfying the symmetric Jacobi identity \eqref{eq:jacobisym} and with a differential $\dd$ acting as:
\begin{equation}
\dd\{x,y\}=-\{\dd x,y\}-(-1)^{|x|}\{x,\dd y\}
\end{equation}
which is the Leibniz identity for the differential in the symmetric formulation (see formula \eqref{eq:formula1} to check that it gives back Equation \eqref{eq:dgla}). The above identity is equivalent to writing that $[\dd,m_{2}]=0$, hence a dgLa is a $L_{\infty}$-algebra with all $n$-brackets vanishing for $n\geq3$.
\end{example}

The degrees of the vector space underlying an $L_{\infty}$-algebra are not constrained, but in most cases of interest (coming mostly from physics), only the non positive integers are considered. That is why we call a \emph{Lie $n$-algebra} (for $n\in\mathbb{N}^{\ast}\cup\infty$) an $L_{\infty}$-algebra structure on some graded vector space $E=\bigoplus_{1\leq i\leq n} E_{-i}$ restricted to non positive integers (recall that we work with Voronov's convention, hence the counter starts at $-1$). Moreover, if $n<\infty$ then the $k$-ary brackets vanish at level $n+2$ and higher. In the present text, we will mostly deal with Lie $n$-algebras and more particularly their natural generalization: the so called Lie $n$-algebroids or more generally Lie $\infty$-algebroids. Moreover a morphism of $L_{\infty}$-algebras is a family of $k$-ary applications $(f_{k})_{k\geq0}$ of degree 0 commuting with the brackets, as in the Lie algebra case. However, this compatibility condition is such a complicated expression that it is more useful to write it in terms of morphisms of coalgebras on the space of functions on $E$ (see Definition \ref{def:morphism}), which justifies the use of differential graded manifolds in this context.

\subsection{Differential Graded Manifolds}\label{section:dgmanifold}

The natural geometric generalization of graded vector spaces are \emph{graded manifolds}. In differential geometry, an $n$-dimensional smooth manifold is defined as a topological space which is locally homeomorphic to $\mathbb{R}^{n}$, and such that two coordinate charts are smoothly compatible. In graded geometry, there is a similar idea (developed by B. DeWitt \cite{DeWitt}): a graded manifold is in some sense locally homeomorphic to some product $\mathbb{R}^{n}\times E$, where $E$ is a graded vector space.

However in the literature, graded manifolds are usually and were originally defined from the dual point of view: from the \emph{sheaf of functions} rather than from a set of coordinate charts, even if the two definitions are equivalent. This dual conception goes back to Berezin and Leites \cite{Leites} for supermanifolds, and has naturally been extended to graded manifolds by Kostant \cite{Kostant}. The two notions happen to be equivalent, as was shown by M. Batchelor \cite{Batchelor}.

The space of functions on a smooth manifold and its restrictions to open sets can be seen as a sheaf $\cinf$, that is an application from the topology of $M$ taking values in the category of commutative algebras, and satisfying some compatibility conditions over open sets (see \cite{maclane} for further details). A topological space $X$ together with a sheaf of rings $\mathscr{O}$ is called a \emph{locally ringed space}, thus any smooth manifold $M$ in the usual sense is a locally ringed space, with structure sheaf $\cinf$.
\begin{definition}
A \emph{graded manifold} is a locally ringed space $\mathcal{M}=(M,\mathscr{O}_{\mathcal{M}})$, where $M$ is a smooth/real analytic/holomorphic manifold called the \emph{base manifold} (or \emph{body}), such that the structure sheaf $\mathscr{O}_{\mathcal{M}}$ is locally of the form $\mathscr{O}(U)\otimes S(E^{\ast})$ for some open set $U\subset M$ and some graded vector space $E$.
\end{definition}
In other words, the space of functions on $\mathcal{M}$ is locally isomorphic to the space of functions on some open set of $M$ tensored with the functions on the graded vector space $E$, as defined in the former section. Hence $\mathscr{O}_{\mathcal{M}}$ is a sheaf of $\mathbb{Z}$-graded algebras. When $U=M$, we find that the structure sheaf $\mathscr{O}_{\mathcal{M}}(M)$ would be isomorphic to the space of smooth/real analytic/holomorphic functions taking values in $S(E^{\ast})$. It is thus tempting to identify $E$ with a graded vector bundle over $M$, that is: a vector bundle, whose fiber is a graded vector space. That would be helpful because it would enable to work in local coordinates or to consider only graded vector bundles. And indeed this result is a famous theorem of M. Batchelor \cite{Batchelor2} which ensures that we can realize (non canonically) a positively graded manifold as a graded vector bundle over a smooth manifold:

\begin{theoreme}
\textbf{Batchelor (1979)} Let $(M,\mathscr{O}_{\mathcal{M}})$ be a positively graded manifold, then there exists a graded vector bundle $E\to M$ such that the structure sheaf $\mathscr{O}_{\mathcal{M}}$ is isomorphic to the sheaf of sections $\Gamma\big(S(E^{\ast})\big)$.
\end{theoreme}

\noindent This important result allows us to talk of graded manifolds in terms of graded vector bundles, which is simpler and more systematic. Using this identification, we will now rather use the notation $E\to M$ or just $E$, instead of $\mathcal{M}$, to refer to a graded manifold. In this context the structure sheaf $\mathscr{O}_{\mathcal{M}}$ will be denoted by $\mathscr{E}$, and it is isomorphic to the sheaf of sections of the graded vector bundle $S(E^{\ast})$.

\begin{definition}
A \emph{morphism of graded manifolds} from $E$ to $F$ (with respective base manifolds $M$ and $N$) is the data of a smooth/real analytic/holomorphic map $\phi:M\to N$ that we call the \emph{base map} together with a morphism of sheaves $\Phi:\mathscr{F}\to\mathscr{E}$ over $\phi^{\ast}$:
\begin{equation}
\Phi(f G)=\phi^{\ast}(f)\Phi(G)
\end{equation}
for every $f\in\mathscr{O}(V)$ and $G\in\mathscr{F}(V)$, for any open set $V\subset N$. We say that the morphism $\Phi$ \emph{covers} the base map $\phi$.
\end{definition}
\noindent For convenience and clarity, we will write $\Phi$ for a morphism of graded manifold, without further mention of the base map.

\begin{example}
A famous example of graded manifold is the tangent space $TM$ whose fiber degree has been shifted by one: it is denoted by $T[1]M$. Coordinate functions on $T[1]M$ are 1-forms on $M$ and since the degrees of the fibers have been shifted, any polynomial of such coordinate functions is a differential form, hence the sheaf of functions is $\mathscr{O}_{T[1]M}=\Omega^{\bullet}$.
\end{example}

The notion of graded manifold allows us to define a generalization of a Lie $\infty$-algebra in this context:

\begin{definition} \label{def:Linftyoids}
 Let  $M$ be a smooth/real analytic/complex manifold whose sheaf of functions we denote by ${\mathscr O}$. Let $E$ be a sequence $E=(E_{-i})_{1\leq i\leq \infty}$ of vector bundles over $M$, then a \emph{Lie $\infty$-algebroid structure on $E$}, is defined by:
\begin{enumerate}
\item a degree 1 vector bundle morphism $\rho: E_{-1} \to TM$ called the \emph{anchor} of the Lie $\infty$-algebroid,
\item and a family, for all $k \geq 1$, of graded symmetric $k$-multilinear maps $ \big(\{ \ldots \}_k\big)_{k \geq 1} $ of  degree $+1$  on the sheaf of graded vector spaces $\Gamma(E)$,
\end{enumerate}
satisfying the following constraints. The first ones are the \emph{Leibniz conditions}:
\begin{enumerate}
\item the unary bracket $\dd:=\{\, . \,\}_1: \Gamma(E) \to \Gamma(E)$ is $\mathscr{O}$-linear,
 i.e. forms a family $\dd^{(i)}:E_{-i} \to E_{-i+1}$ of vector bundle morphisms, where we assume that $\dd^{(1)}=0$;
  \item the binary bracket obeys different rules, depending on its arguments: for all $x \in \Gamma(E_{-1})$ and $y \in \Gamma(E)$ it satisfies
 \begin{equation}\label{robinson}
 \{x,fy\}_2=f\{x,y\}_2+ \rho(x)[f] \, y 
 \end{equation}
whereas $ \{x,fy\}_2=f\{x,y\}_2$ for all $x \in \Gamma(E_{-i})$ with $i \geq 2$;
\item for all  $n  \geq 3$, each of the maps $\{\ldots\}_n $ is $\mathscr{O} $-linear.
\end{enumerate}
and the second ones are the \emph{higher Jacobi identities}:
\begin{enumerate}
\item $\rho \circ \dd^{(2)}=0$
\item $ \dd^{(i-1)} \circ {\dd}^{(i)}=0$\quad for all $ i \geq 3$
\item $ \rho\big(\{x,y\}\big)=\big[\rho(x),\rho(y)\big]$\quad for all $x,y \in \Gamma(E_{-1})$,
\item for all $n\geq2$, and for every homogeneous elements $x_{1},\ldots,x_{n}\in \Gamma(E)$:
\begin{equation}\label{superjacobi}
\sum_{i=1}^{n}\sum_{\sigma\in Un(i,n-i)}\epsilon(\sigma)\big\{\{x_{\sigma(1)},\ldots,x_{\sigma(i)}\}_{i},x_{\sigma(i+1)},\ldots,x_{n}\big\}_{n-i+1}=0
\end{equation}
\end{enumerate}
A Lie $\infty$-algebroid structure over $M$ is said to be 
a \emph{Lie $n$-algebroid} when $E_{-i}=0$ for all $i \geq n+1$.
\end{definition}
We observe that for degree reasons, if $n<\infty$ there is no bracket of arity higher than or equal to $n+2$. Moreover, we notice that the fact that the differential $\dd$ is $\mathscr{O}$-linear implies that the graded vector bundles $(E_{-i})_{1\leq i\leq n}$ form a chain complex of vector bundles:
%
%
%
For every Lie $\infty$-algebroid over $M$, it follows from items 1) and 2) in the higher Jacobi identities that the following is a complex of vector bundles
that we call its \emph{linear part}:
 \begin{center}
\begin{tikzcd}[column sep=0.9cm,row sep=0.6cm]
\ldots\ar[r,"\dd"]& E_{-3} \ar[r,"\dd"]& E_{-2}\ar[r,"\dd"]&  E_{-1}\ar[r,"\rho"]& TM
\end{tikzcd}
\end{center}
The first examples of Lie $n$-algebroids are for $n=1$, the so-called \emph{Lie algebroids}. In the literature, the definition is given as follows:

\begin{definition}
A \emph{Lie algebroid} over ${M}$ is a vector bundle $A \to  M$, equipped with a vector bundle morphism $\rho: A \to T M$ called the \emph{anchor map}, and a Lie bracket $[\ .\ ,\, .\ ]_A $ on $\Gamma (A) $, satisfying the \emph{Leibniz identity}:
\begin{align}\label{algebroid}
\forall\ x,y\in\Gamma(A),\ f\in\cinf({ M})\hspace{1cm}&[x,fy]_{A}=f[x,y]_{A}+\rho(x)[f] \, y\ ,
\end{align}
as well as the \emph{Lie algebra homomorphism condition}:
\begin{equation}\label{algebroid2}
\forall\  x,y\in\Gamma(A)\hspace{1cm}\rho\big([x,y]_{A}\big)=[\rho(x),\rho(y)]\ .
\end{equation}
\end{definition}
\begin{remarque}
The second axiom is already implied by the Leibniz identity and the Jacobi identity satisfied by the Lie bracket on $\Gamma(A)$, but we want to keep it for clarity.
\end{remarque}

The tangent bundle of a smooth manifold is a Lie algebroid with anchor map the identity, whereas any Lie algebra can be seen as a Lie algebroid over a point. Thus the notion of Lie algebroid offers a natural generalization of both Lie algebras and tangent bundles. In section \ref{gradvectspace} we saw that in the language of $L_{\infty}$-algebras, a Lie algebra is a Lie $1$-algebra with all $k$-brackets vanishing except for $k=2$, or more precisely: a vector space of degree $-1$ equipped with a symmetric bracket satisfying the Jacobi identity. We obtain the same result here for a Lie algebroid: it is a Lie $1$-algebroid with vanishing $k$-brackets except for $k=2$, and then the anchor map becomes a degree 1 application.

The algebraic structure of a Lie $n$-algebroid is more intricate than that of a mere Lie algebroid, and the rest of this thesis is devoted to finding a Lie $\infty$-algebroid structure on a resolution of a Hermann distribution. For this purpose, we will mostly work with the dual language, on the space of functions on graded manifolds, in which the brackets and the morphisms take a very simple and straightforward meaning.

The functions over a graded manifold $E$ $-$ when it is seen as a graded vector bundle $E\to M$ over a smooth manifold $M$ $-$ are the global sections of $S(E^{\ast})$, the graded symmetric algebra of the dual bundle. The sheaf of sections of $S(E^{\ast})$ is denoted by $\funct$. We define vector fields on $E$ to be the derivations of this sheaf of algebras. Since they are defined on a graded algebra, they are also naturally graded. We say that a vector field is \emph{vertical} if it is $\mathscr{O}$-linear; in other words, if it is tangent to the fibers. The space of vector fields can be equipped with a (graded) Lie bracket, as in the smooth ungraded case. Then we can define \emph{differential graded manifolds}:

\begin{definition}
A \emph{differential graded manifold} (or \emph{$Q$-manifold}) is a graded manifold equipped with a degree $+1$ vector field $Q$ which commutes with itself: $[Q,Q]=0$.
\end{definition}
\begin{remarque}
Note that for an odd vector field, it is a priori not automatic that the self-commutator vanish because $[Q,Q]=2Q^{2}$. Vector fields which have this property are said to be \emph{homological}.
\end{remarque}

\begin{example}
The first example is the Lie algebroid, whose reformulation as a $Q$-manifold is due to A. Vaintrob \cite{Vaintrob}. It eventually led T. Voronov to  generalize this notion and define the possibly higher Lie algebroids to be $Q$-manifolds \cite{voronov}, leading to the precise idea of Lie $\infty$-algebroid. Using formula \eqref{eq:formula2}, Lie algebroids can be defined with symmetric brackets on $\Gamma(A[1])$ instead of skew-symmetric ones on $\Gamma(A)$. On $A[1]$, sections have degree $-1$ whereas they have degree 0 when seen as sections of $A$. Thus, given a Lie algebroid $A$ over $M$, we define on the sections of the suspended vector bundle $A[1]$ the following symmetric bracket:
\begin{equation}
\{x,y\}=[\widetilde{x},\widetilde{y}]
\end{equation}
for any sections $x,y$ of $A[1]$, and where $\widetilde{x}$ is the representative of the section $x$ in $\Gamma(A)$ (i.e. whose degree has been shifted by $+1$).

The space of functions on $A[1]$ is isomorphic to $\Gamma(S(A[1]^{\ast}))$, then it is sufficient to define the vector field $Q$ on the smooth functions on $M$ and on the sections of $A[1]^{\ast}$, and then extend to all of $\Gamma(S(A[1]^{\ast}))$ by derivation. However since $Q$ is of degree one, then it sends smooth functions to sections of $A[1]^{\ast}$ and sections of $A[1]^{\ast}$ to sections of $S^{2}(A[1]^{\ast})$. Hence we define:
\begin{align}
\big\langle Q[f],x\big\rangle&=\rho(x)[f]\\
\big\langle Q[\alpha],x\odot y\big\rangle&=\rho(x)\langle\alpha,y\rangle-\rho(y)\langle\alpha,x\rangle-\big\langle\alpha,\{x,y\}\big\rangle
\end{align}
for every $f\in\cinf(M)$, $\alpha\in\Gamma(A[1]^{\ast})$, and for any $x,y\in\Gamma(A[1])$. We extend $Q$ to all of $\Gamma(S(A[1]^{\ast}))$ by derivation, so that the homological property comes from the morphism property and from the Jacobi identity:
\begin{align}
\big\langle Q^{2}[f],x\odot y\big\rangle
&=\Big(\big[\rho(x),\rho(y)\big]-\rho\big(\{x,y\}\big)\Big)[f]\\
\big\langle Q^{2}[\alpha],x\odot y\odot z\big\rangle&=\Big\langle\alpha,\big\{\{x,y\},z\big\}+\big\{\{y,z\},x\big\}+\big\{\{z,x\},y\big\}\Big\rangle
\end{align}
for every $f\in\cinf(M)$, $x,y,z\in\Gamma(A[1])$ and $\alpha\in\Gamma(A[1]^{\ast})$. This proves that the vector field $Q$ squares to zero, if and only if Equations \eqref{algebroid2} and \eqref{superjacobi} are satisfied. This explains the one-to-one correspondence between Lie algebroids and $Q$-manifold structures of degree 1 on $A[1]$.
\end{example}

\begin{example}\label{examplePoisson}
Poisson manifolds are examples also. What we call a \emph{Poisson manifold} is a smooth manifold $M$ equipped with a Lie bracket $\{\ .\ ,\,.\ \}$ on the algebra of functions $\cinf(M)$ satisfying the derivation property $\{f,gh\}=\{f,g\}h+g\{f,h\}$ for every smooth functions $f,g,h\in\cinf(M)$. There is an equivalent formulation in terms of a self-commuting bivector $\Pi\in\Gamma(\bigwedge^{2}TM)$ since the correspondence can be made explicit:
\begin{equation}
\{f,g\}=\Pi\big(\dd f,\dd g\big)
\end{equation}
for any $f,g\in\cinf(M)$. The Jacobi identity is translated as the vanishing  of the self-commutator $[\Pi,\Pi]=0$, where the bracket is the Schouten-Nijenhuis bracket on multi-vector fields. Then it is known that $T^{\ast}M$ is a Lie algebroid: the anchor map is the bundle map $\Pi^{\#}:T^{\ast}M\to TM$ defined by:
\begin{equation}
\omega\big(\Pi^{\#}(\alpha)\big)=\Pi(\omega,\alpha)
\end{equation}
for every $\alpha,\omega\in\Omega^{1}(M)$. The Lie bracket $[\ .\ ,\, .\ ]_{T^{\ast}M}$ on sections of $T^{\ast}M$ is defined by the following equation:
\begin{equation}
[\omega,\alpha]_{T^{\ast}M}(u)=\omega\big([\Pi,\alpha(u)]\big)-\alpha\big([\Pi,\omega(u)]\big)-[\Pi,u](\omega,\alpha)
\end{equation}
where $u\in \mathfrak{X}(M)$ and where the bracket on the right-hand side is the Schouten-Nijenhuis bracket. One readily checks that the consistency condition \eqref{algebroid} comes from the definition of $\Pi^{\#}$ and some of the properties of the Schouten-Nijenhuis bracket, and that \eqref{algebroid2}) as well as the Jacobi identity come from the fact that $[\Pi,\Pi]=0$. Then if $M$ is a Poisson manifold $T^{\ast}M$ is a Lie algebroid, and this implies that we can associate a $Q$-manifold $T^{\ast}M[1]$ to every Poisson manifold. The story goes even further, since it has been shown that there is a one-to-one correspondence between Poisson manifolds and symplectic $Q$-manifolds of degree 1. In some sort of generalization D. Roytenberg has shown that symplectic $NQ$-manifolds of degree 2 are in one-to-one correspondence with Courant algebroids \cites{Roytenberg2, Roytenberg}.
\end{example}

We speak of an \emph{$NQ$-manifold} (or \emph{positively graded dg manifold}) when the underlying graded manifold $-$ seen as a graded vector bundle $-$ involves only coordinate functions of positive degrees, i.e. when $E_{i}=0$ for all $i\geq0$. Coordinates on the base manifold have degree 0, whereas the fibers admit coordinate functions which are sections of their respective dual spaces, which are then supposed to be of degree greater than or equal to 1. We thus have a concept which would encode the aforementioned Lie $\infty$-algebroids, from a dual perspective however. We have more: they are actually equivalent notions, but to show this an additional object has to be introduced. It is a fundamental concept at the core of the present thesis: the concept of \emph{arity}. It describes how a map behaves as a graded operator, with respect to the symmetric powers of $E^{\ast}$ and not with respect to the classical grading that we have discussed until now.

\begin{definition}
A  function $f\in\funct$ is of \emph{degree $k$} and of \emph{arity $n$} if it belongs to $\Gamma(S^{n}(E^{\ast})_{k})$, that is, if it is a section of $\sum_{i_1+ \dots + i_n = k } E^*_{-i_1} \odot \dots \odot E_{-i_n}^*$, where $\odot$ is the symmetric product. A vector field $X$ is said to be of \emph{arity  $m\in\mathbb{Z}$}  when, for every function $f\in\funct$ of arity $n$, the function $X[f]$ has arity $n+m$.
\end{definition}

We have the following immediate result which is of interest to us now:

\begin{proposition}\label{prop:arity}
Let $E \to  M $ be a positively graded manifold. The allowed values of arity for a function go from $0$ to $+\infty$, and that of a vector field go from $-1$ to $ +\infty$.
\end{proposition}

\begin{proof}Since a function is formally identified with a section of $S(E^{\ast})$, it can only accept positive values of arity. Natural examples of vector fields of arity $-1$ are the vector fields on $E$ defined as the contractions by sections of $E_{-1}$, see Equation \eqref{innerder} below. These are the only operators which can lower the arity of a function. By definition, since vector fields are derivations of the sheaf of functions $\funct$, they can locally be formally written as a (possibly infinite) sum of components which carry such partial derivatives times a function on $E$. In other words, vector fields cannot be of arity stricly less than $-1$.
\end{proof}

Given these results, we can proceed to the formulation and the proof of the identification between $NQ$-manifolds and Lie $\infty$-algebroids, originally given by T. Voronov \cite{voronov}:

\begin{theoreme}\label{theo:equivoQue}
Let $E=(E_{-i})_{i\geq 1}$ be a sequence of positively graded vector bundles over a manifold $M$. Then there is a one-to-one correspondance between $NQ$-manifold structures on $E$ and Lie $\infty$-algebroid structures over $E$.
\end{theoreme}

\begin{proof}
We will use the identification between sections of $E$ and their canonically associated vector fields. That is to say, given any $x\in \Gamma(E)$, we define the constant vector field $\partial_{x}$ over $E$ to be:
\begin{equation}\label{innerder}
\partial_{x}[\xi]=\big<\xi,x\big>
\end{equation}
for every $\xi\in\Gamma(E^{\ast})$ and extend it as a $\mathscr{O}$-derivation on all of $\funct$. In other words, the coordinates of the constant vector field $\partial_{x}$ in $TE$ are the same as those of $x$ in $E$, so that they can be identified. For example, let $E=\bigoplus_{p\geq1}E_{-p}$ be a graded manifold that admits fiberwise coordinates $-$ say $\{q^{\alpha_{p}}\}$, of degree $p$ $-$ where $\alpha_{p}$ is an index ranging from 1 to $\mathrm{dim}(E_{-p})$. There exist respective homogeneous dual elements $q_{\alpha_{p}}$ in $E_{-p}^{\ast}$, so that the inner derivations $\partial_{q_{\alpha_{p}}}$ are given by the derivative with respect to the coordinate $q^{\alpha_{p}}$, that is to say:
\begin{equation}\partial_{q_{\alpha_{p}}}=\frac{\partial}{\partial q^{\alpha_{p}}}\end{equation}
Vector fields on $E$ that are constant on each fiber form an abelian subalgebra of the infinite dimensional Lie algebra of vector fields on $E$. Over a small open set $U$, they can locally be formally expressed as:
\begin{equation}
X=\sum_{i=1}^{n}\ v^{i}\frac{\partial}{\partial x^{i}}\quad+\quad\sum_{p=1}^{\infty}\quad\sum_{\alpha_{p}=1}^{\mathrm{dim}(E_{-p})}f^{\alpha_{p}}(x)\partial_{q_{\alpha_{p}}}
\end{equation}
for some functions $v^{i},f^{\alpha_{p}}\in\mathscr{O}(U)$. Let $\Pi$ be the operator that sends any vector field $X$ to the vector field $X_{0}$ on $E$ constant on each fiber that coincides with $X$ on the zero locus of the fibers. Therefore, $\Pi$ isolates the components of $X$ which do not depend on the fiber coordinates.

Given a $NQ$-structure on $E$, we can define the anchor map as:
\begin{equation}
\big<Q[f],x\big>=\rho(x)[f]
\end{equation}
for any $x\in E_{-1}$. The commutator $[Q,\partial_{x}]$ gives the following equalities, when acting on functions on $M$:
\begin{equation}
\big[Q,\partial_{x}\big][f]=\partial_{x}\circ Q[f]=\big<Q[f],x\big>=\rho(x)[f]
\end{equation}
In other words, the constant vector field $\Pi\big([Q,\partial_{x}]\big)$ gives the anchor map. In the same spirit, if one takes a homogeneous element $y\in E$, one can define the differential $\dd=\{\, . \,\}_{1}$ using the identification between $\dd y$ and $\partial _{\dd y}$, so that we can write symbolically:
\begin{equation}\label{rienavoir}
\dd y = \Pi\big( [Q,\partial_{y}]\big)
\end{equation}
The component of arity zero in the relation $[Q,Q]=0$ gives $[Q^{(0)},Q^{(0)}]=0$,
 which, in turn, proves that $\dd\circ \dd =0$. In short:
 \begin{lemme}\label{dualQ0}
  Let $E =(E_{-i})_{i\geq 1}$ be a positively graded manifold over a manifold $M$. There is a one-to-one correspondence between homological vertical vector fields of arity $0$ 
  and collections of maps  $ \dd=\big({\mathrm d}^{(i)} : E_{-i} \to E_{-i+1}\big)_{i \geq 2} $ making $E$ a complex.
 \end{lemme}
We define the $n$-ary brackets by Voronov's formula \cite{voronov2}:
\begin{equation}\label{eq:correspondence}
\big\{x_{1},\ldots,x_{n}\big\}_{n}=\Pi\Big(\big[\ldots[[Q,\partial_{x_{1}}],\partial_{x_{2}}],\ldots,\partial_{x_{n}}\big]\Big)
\end{equation}
where on the left-hand side we have used the identification between elements of $E$ and their associated inner derivations, see Equation \eqref{innerder}. For $n=1$ we recover Equation \eqref{rienavoir}. The above brackets are $\cinf(M)$ linear for $n\neq3$, and for $n=2$ they naturally satisfy the Leibniz identities. The Jacobi identities are satisfied due to the homological property of the vector field $Q$, see \cite{voronov2}. Hence we obtain a Lie $\infty$-algebroid structure on $E$.

Conversely, given such a Lie $\infty$-algebroid structure on the graded vector space $E$, we construct a homological vector field $Q$ turning $E$ into an $NQ$-manifold. The homological vector field $Q$ has degree $1$, hence it does not even have any component of arity $-1$. Moreover the component of arity 0 (that we call the \emph{linear part of $Q$}) is $\Funct$-linear for degree reasons. Hence we can formally decompose $Q$ as an infinite sum of components of homogeneous arities:
\begin{equation}
 Q= \sum_{i=0}^\infty \ Q^{(i)}
 \end{equation}
 First, the dual of the differential $\dd$ gives the component of arity 0 by the formula:
 \begin{equation}
\big\langle Q^{(0)}[\alpha],x\big\rangle=(-1)^{i-1}\big\langle \alpha,\dd^{(i)}(x)\big\rangle
\end{equation}
for any $x\in \Gamma(E_{-i})$ and $\alpha\in\Gamma(E^{\ast}_{-i+1})$.

The component of arity 1 is the only component which is not $\cinf(M)$-linear because it involves the anchor map. Given the 2-bracket $\{\ .\ ,\, .\ \}_{2}$ and the anchor map $\rho$, we define the component of arity $1$ by the following two identities:
\begin{align}
\big\langle Q^{(1)}[f],x\big\rangle&=\rho(x)[f]\label{anchor1}\\
\big\langle Q^{(1)}[\alpha],x\odot y\big\rangle&=\rho(x)\big[\langle\alpha,y\rangle\big]-\rho(y)\big[\langle\alpha,x\rangle\big]-\big\langle \alpha, \{x,y\}_{2}\big\rangle\label{anchor2}
\end{align}
for every $f\in\cinf(M)$, $\alpha\in\Gamma(E^{\ast})$ and for every homogeneous elements $x,y\in E$, where it is understood that the anchor map vanishes for elements of degree lower than or equal to $-2$. We then extend $Q^{(1)}$ to all the functions of $\funct$ by derivation using the above two definitions.

\begin{remarque}
In other words, the component of arity 1 acts on smooth functions through the following composition:
\begin{equation*}\label{anchorproperty}
\rho^{\ast}\circ\dd_{\mathrm{dR}}:\cinf(M)\longmapsto\Gamma(E_{-1}^{\ast})
\end{equation*}
since the dual of the anchor map sends differential forms on $M$ to sections of $E_{-1}^{\ast}$. The conditions $\rho\circ\dd=0$ and $[Q^{(0)},Q^{(0)}]=0$ imply that the linear part $Q^{(0)}$ of the homological vector field $Q$ defines a  differential complex:
\begin{center}
\begin{tikzcd}[column sep=0.9cm,row sep=0.6cm]
\ldots&\ar[l,"Q^{(0)}" above]\Gamma(E^{\ast}_{-2})&\ar[l,"Q^{(0)}" above]\Gamma(E^{\ast}_{-1})&\ar[l,"\rho^{\ast}\circ\dd_{\mathrm{dR}}" above]\cinf(M)
\end{tikzcd}
\end{center}
\end{remarque}

Now turn to the higher arities: the $n$-ary brackets $\big\{\ldots\big\}_{n}$ are $\cinf(M)$-linear hence they can be directly dualized:
\begin{equation*}
\big\{\ldots\big\}^{\ast}_{n}:E^{\ast}\longmapsto S^{n}(E^{\ast})
\end{equation*}
and can be extended to all of $\funct$ as $\cinf(M)$ linear derivations. We define the component of arity $n$ to be the dual of the $n$-bracket:
\begin{equation}
Q^{(n)}=\big\{\ldots\big\}^{\ast}_{n}
\end{equation}
It naturally has degree $+1$ because the original brackets have degree $+1$. Then define the vector field $Q$ by:
\begin{equation}
Q=\sum_{i=0}^{\infty}\ Q^{(i)}
\end{equation}
This degree $+1$ vector field squares to zero because the brackets and the anchor map satisfy the Jacobi identities, hence it defines an $NQ$-manifold structure on $E$.
\end{proof}

All along the present thesis, we will use this identification, and the proofs will rely exclusively on the dual picture ($NQ$-manifolds) so that it seems that it is not only a practical concept, but it is rather necessary for the whole process. In the following then, we will use the terms Lie $\infty$-algebroids or $NQ$-manifolds indifferently, but we shall generally describe them as a pair $(E,Q)$, with $Q$ the homological vector field of the associated $Q$-manifold while keeping in mind the one-to-one correspondence between the two notions.


\subsection{Lie \texorpdfstring{$\infty$}{infinity}-morphisms and homotopies}

When dealing with morphisms, $Q$-manifolds are much more practical than Lie $\infty$-algebroids. By a \emph{morphism} between two $N$-manifolds 
$E \to { M}$ and  $F \to { N}$, we mean a degree $0$ morphism $\Phi$ of sheaves of graded commutative algebras from the graded commutative algebra ${\mathscr F} $ 
of functions on $F \to N$ to the graded commutative algebra ${\mathscr E} $ of functions on $E \to M$.  

\begin{definition}
\label{def:morphism}
A \emph{Lie $\infty$-algebroid morphism} (or \emph{Lie $\infty$-morphism}) 
from a Lie $\infty$-algebroid $ ( E,Q_{E})$ to a Lie $\infty$-algebroid $ (F,Q_F)$ 
with sheaves of functions ${\mathscr E}$ and $ {\mathscr F}$ respectively is an algebra morphism $\Phi$ of degree $0 $
from ${\mathscr F}$ to $ {\mathscr E}$  which intertwines $Q_{E}$ and $Q_{F}$:
\begin{equation}\label{Qmorphism}
\Phi\circ Q_{F}=Q_{E}\circ\Phi
\end{equation}
When $\Phi$ is an algebra isomorphism, we shall speak of a \emph{strict isomorphism}.  
\end{definition}

Every Lie $\infty$-algebroid morphism $\Phi$ induces a smooth map $\phi$ from $ M$ to $ N$ that we call the \emph{base morphism}, and, 
for each $i \geq 1$, vector bundle morphisms $\phi_i:E_{-i} \to F_{-i}$ over $\phi$. 
 When $M$  = $N$, we say that a morphism $\Phi$ is \emph{over $M$} if $\Phi$ is ${\mathscr O}$-linear,
i.e. if the base morphism is the identity map.

A Lie ${\infty}$-algebroid morphism $\Phi$ from $ (E,Q_E)$ to $(F,Q_{F})$ is said to be of \emph{arity $k$} if it sends functions of arity $l$ in $\mathscr{F}$ to functions of arity $l+k$ in $\mathscr{E}$. The arity of such a morphism is necessarily positive, for it has to send smooth functions (of arity zero) to elements of $\Gamma(S(E^{\ast}))$ (of positive arity).
Any Lie $\infty $-algebroid morphism $\Phi $ over $M$ from  $ (E,Q_{E})$ to $(F,Q_{F})$ can be decomposed into components according to their arity:
since  $\Phi$ is ${\mathscr O}$-linear, and since it is determined by its restriction to functions of arity $1$,
i.e. sections of $E^*$, the component of arity $k$, namely $ \Phi^{(k)}$, can be identified with an element in $ \Gamma(S^{k+1}(E^*) \otimes F) $ that we denote by $\widehat{\Phi}^{(k)}$.
This allows to consider $\Phi$ as a formal sum:
\begin{equation}
\Phi=\sum_{k \geq 1}\ \Phi^{(k)}
\end{equation}
Taking arity into account, the morphism condition $\Phi(fg)=\Phi(f)\Phi(g)$ for any $f,g\in\functt$ becomes:
\begin{align}
\Phi^{(0)}(fg)&=\widehat{\Phi}^{(0)}(f)\widehat{\Phi}^{(0)}(g)\label{baboum}\\
\Phi^{(1)}(fg)&=\widehat{\Phi}^{(1)}(f)\widehat{\Phi}^{(0)}(g)+\widehat{\Phi}^{(0)}(f)\widehat{\Phi}^{(1)}(g)\label{boum}\\
\forall\ k\geq2\hspace{1cm}\Phi^{(k)}(fg)&=\underset{i+j=k}{\sum_{0\leq i,j\leq k}}\widehat{\Phi}^{(i)}(f)\widehat{\Phi}^{(j)}(g)\label{biboum}
\end{align}
The component $\Phi^{(0)}$ of arity $0$ is thus a graded algebra morphism that we call the \emph{linear part} of $\Phi$. It sends sections of $F^{\ast}_{-i}$ to sections of $E_{-i}^{\ast}$, for every $i\geq1$. We emphasize that $\Phi^{(1)}$ is not an algebra morphism but rather a $\Phi^{(0)}$-derivation. Being of arity 1, it sends  elements of $\Gamma(F^{\ast}_{-i})$ to the space $\sum_{1\leq k\leq i-1}\Gamma(E_{-i+k}^{\ast}\odot E_{-k}^{\ast})$ for any $i\geq1$ $-$ and so on for components of higher arities. In a similar fashion, any map $ \delta:\mathscr{F}\to\mathscr{E}$ can be decomposed into components of homogeneous arity, and then be considered as a family  $\big(\widehat{\delta}^{(k)}\big)_{k \geq 0}  $ with $ \widehat{\delta}^{(k)} \in \Gamma(S^{k+1}(E^*) \otimes F)$.

\begin{remarque}
 \label{rmk:linearparts}
Equation (\ref{Qmorphism}), 
restricted to terms of arity $0$, implies that the induced graded vector bundle morphism is a chain map between the respective linear parts of $(E,Q_{E})$ and $(F,Q_{F})$:
 \begin{center}
\begin{tikzcd}[column sep=0.9cm,row sep=0.6cm]
\ldots\ar[r,"\dd"] & E_{-3} \ar[r,"\dd"]\ar[dd,"\phi_3"]& E_{-2} \ar[dd,"\phi_2"]\ar[r,"\dd"]&  E_{-1} \ar[dd,"\phi_1"]\ar[r,"\rho"]& TM \ar[dd,"\phi_{\ast}"]\\
&\\
\ldots\ar[r,"\dd'"]& F_{-3} \ar[r,"\dd'"]& F_{-2}\ar[r,"\dd'"]&  F_{-1}\ar[r,"\rho'"]& TN
\end{tikzcd}
\end{center}
where $\dd$ and $\dd'$ are the linear parts of $Q_{E}$ and $Q_{F}$, respectively. We call this chain map the \emph{linear part of  $\Phi$}. The dual application is $\widehat{\Phi}^{(0)}$, so it defines a chain map as well, but on the cochain complex:
\begin{center}
\begin{tikzcd}[column sep=0.9cm,row sep=0.6cm]
\ldots&\ar[l,"Q_{E}^{(0)}"]\Gamma(E^{\ast}_{-3})&\ar[l,"Q_{E}^{(0)}"]\Gamma(E^{\ast}_{-2})&\ar[l,"Q_{E}^{(0)}"]\Gamma(E^{\ast}_{-1})&\ar[l,"\rho^{\ast}"]\Omega^{1}(M)\\
&\\
\ldots&\ar[l,"Q_{F}^{(0)}"]\Gamma(F^{\ast}_{-3})\ar[uu,"\widehat{\Phi}^{(0)}"]&\ar[l,"Q_{F}^{(0)}"]\Gamma(F^{\ast}_{-2})\ar[uu,"\widehat{\Phi}^{(0)}"]&\ar[l,"Q_{F}^{(0)}"]\Gamma(F^{\ast}_{-1})\ar[uu,"\widehat{\Phi}^{(0)}"]&\ar[l,"\rho'^{\ast}"]\Omega^{1}(N)\ar[uu,"\phi^{\ast}"]\\
\end{tikzcd}
\end{center}
\end{remarque}

Beyond strict isomorphisms (see Definition \ref{def:morphism}) between Lie $\infty$-algebroids, there is a larger notion of morphisms `invertible up to homotopy'.
That is the correct notion of invertible morphisms in the category of Lie $\infty$-algebroids. To begin with, it is not easy to define
what is a homotopy of Lie $\infty$-algebroid morphisms. There have been several attempts to define them \cite{Baez}.
We claim that the one that we propose now is both new and relevant.

We proceed step by step, starting by justifying the concept. From now on until the end of this section, we will assume that we work with a smooth manifold $M$, over which several Lie $\infty$-algebroid structures may coexist. It means that for any two graded manifolds $E$ and $F$, it is implied that in both cases the base manifold is $M$. Moreover we will assume $-$ unless otherwise stated $-$ that the base map associated to any Lie $\infty$-morphism is the identity $\mathrm{id}_{M}$. Now assume that we are given two Lie $\infty$-algebroids $(E,Q_{E})$ and $(F,Q_{F})$ over $M$, and two Lie $\infty$-morphisms $\Phi:\functt\to\funct$ and $\Psi:\funct\to\functt$. We are interested in investigating the possibility that the two Lie $\infty$-algebroid structures  be 'isomorphic', in some sense.  Assume for example that $\Phi$ and $\Psi$ are inverse to one another as morphisms of algebras:
\begin{equation*}
\Phi\circ\Psi=\mathrm{id}_{\funct}\hspace{2cm}\text{and}\hspace{2cm}\Psi\circ\Phi=\mathrm{id}_{\functt}
\end{equation*}
One can decompose both the left-hand side and the right-hand side in terms of components of homogeneous arity. Since the identity map is of arity 0, we find that
\begin{align*}
\Phi^{(0)}\circ\Psi^{(0)}&=\mathrm{id}_{\funct},&\phantom{and}\hspace{2cm}\Psi^{(0)}\circ\Phi^{(0)}&=\mathrm{id}_{\functt},\\
(\Phi\circ\Psi)^{(k)}&=0&\text{and}\hspace{2.7cm}(\Psi\circ\Phi)^{(k)}&=0
\end{align*}
for all $k\geq1$. But it is very inconvenient and way too stringent to define isomorphisms of Lie $\infty$-algebroids as Lie $\infty$-morphisms of algebras of arity 0 which are invertible. That is why we prefer to weaken the condition and treat equivalence of Lie $\infty$-algebroid structures only \emph{up to homotopy}, in the sense that the only condition that both left-hand sides above have to satisfy is that they are only homotopic (in some sense defined below) to the right-hand sides, that is: 
\begin{equation*}
\Phi\circ\Psi\sim\mathrm{id}_{\funct}\hspace{2cm}\text{and}\hspace{2cm}\Psi\circ\Phi\sim\mathrm{id}_{\functt}
\end{equation*}
A \emph{homotopy} between two Lie $\infty$-morphisms $\Phi$ and $\Psi$ will be a piecewise-$C^{1}$ path in the space of Lie $\infty$-algebroid morphisms, satisfying some consistency conditions.

If $(E,Q_{E})$ and $(F,Q_{F})$ are Lie $\infty$-algebroids, we define an operator $[Q,\,.\ ]$ on the space of maps $\mathrm{Map}(\mathscr{F},\mathscr{E})$ from $\functt$ to $\funct$ (not necessarily algebra morphisms) by:
\begin{align}\label{morphismdiff}
[Q,\,.\ ]:\hspace{0.3cm}\mathrm{Map}(\mathscr{F},\mathscr{E})\hspace{0.3cm}&\xrightarrow{\hspace*{2cm}} \hspace{1cm}\mathrm{Map}(\mathscr{F},\mathscr{E})\\
\alpha\hspace{1.1cm}&\xmapsto{\hspace*{2cm}}Q_{E}\circ\alpha-(-1)^{|\alpha|}\alpha\circ Q_{F}\nonumber
\end{align}
for every map of graded manifolds $\alpha:\mathscr{F}\to\mathscr{E}$ of homogeneous degree, and we extend it by derivation. It squares to zero because both vector fields are homological. Then it defines a differential on the space of maps between the graded manifolds $E$ and $F$. In this context, Lie $\infty$-morphisms are degree 0 cocycles of $[Q,\,.\ ]$. For $\Phi$ a Lie $\infty$-morphism from $(E,Q_E)$ to $(F,Q_F)$, let us call \emph{$\Phi $-derivations} ${\mathscr{O}}$-linear homogenous maps $\delta$ from $ {\mathscr F}$ to $ {\mathscr E}$ satisfying:
\begin{equation}\label{derivation}
\delta^{(n)}(fg) = \sum_{i=0}^{n}\,\widehat{\delta}^{(i)}(f) \widehat{\Phi}^{(n-i)}(g) + (-1)^{|\delta||f|} \widehat{\Phi}^{(n-i)}(f) \widehat{\delta}^{(i)}(g)
\end{equation}
for all functions $f,g \in {\mathscr F}$, and all arities $n\geq 0$, or, more symbolically:
\begin{equation}
\delta(fg)=\delta(f)\Phi(g)+(-1)^{|\delta||f|}\Phi(f)\delta(g)
\end{equation}
where $|\delta|$ is the degree of $\delta$ as an element of $\mathrm{Map}(\functt,\funct)$. It is easy to check that for every $\Phi$-derivation $\delta$, the quantity $[Q,\delta]$ is a ${\Phi}$-derivation again, of degree $|\delta|+1$. Of course, the relation $ \big[Q,[Q,\delta]\big]=0$ holds true, so that:

\begin{lemme}
\label{lem:complexPhideriv}
 For every Lie $\infty$-algebroid morphism $\Phi$ over $M$ from $(E,Q_{E})$ to $(F,Q_{F})$, $\Phi$-derivations form a complex when equipped with the differential 
 $\delta \mapsto [Q,\delta] $ defined in \eqref{morphismdiff}
\end{lemme}

Now, let us define what we mean by piecewise-$C^1$ paths valued in Lie $\infty$-morphisms from $(E,Q_E)$ to $(F,Q_F)$. Recall that \emph{a piecewise-$C^1$ path valued in a manifold $N$} is a continuous map from a compact interval $I$ of ${\mathbb R}$ to the manifold $N$ such that there exists a subdivision $a=x_0 < \dots < x_k=b$ of $I=[a,b]$ such that the path is $C^1$ on each $]x_{i},x_{i+1}[$. An important feature of these functions is that their derivatives are piecewise continuous and satisfy $ f(b)-f(a)  = \int_a^b f'(t)dt$.

\begin{definition}
\label{def:pickyAboutPaths}
Let $(E,Q_E)$ and $(F,Q_F)$ be Lie $\infty $-algebroids over $M$. By a \emph{piecewise-$C^1$ path valued in Lie ${\infty}$-morphisms from $E$ to $F$}, we mean a path $t\mapsto  \Phi_t$ valued in this set such that for all $k \in {\mathbb N}$, the  path $ t \mapsto \widehat{\Phi}_t^{(k)}$ of arity $k$ is piecewise-$C^1$ in the usual sense.
\end{definition}

\begin{remarque}
We assume that for every  $t$, the base map associated to $\Phi_{t}$ is the identity morphism $\mathrm{id}_{M}$. The subtle point in this definition is that the subdivision of $ I$ for which $\widehat{\Phi}^{(k)}_{t}$ is $ C^1$ may depend on $k$. Moreover, without loss of generality we can always set $I=[0,1]$.
\end{remarque}

Notice that for a fixed arity $k$, the derivative $\frac{d}{dt} \widehat{\Phi}^{(k)}_t $ is piecewise continuous in the usual sense because the number of points of the subdivision of $[0,1]$ is finite.  For each arity $l\leq k$, the number of singularities of the piecewise-$C^1$ path $\widehat{\Phi}^{(l)}_{t}$ is finite. Thus, there exists a subdivision $0=x_{0}<\ldots<x_{n}=1$ of the segment $[0,1]$ on which the piecewise-$C^{1}$ paths $\widehat{\Phi}_{t}^{(l)}$ are simultaneously well-defined for all arities $0\leq l\leq k$. Using Equation \ref{derivation}, one can define a $\Phi_{t}$-derivation:
\begin{equation}
\dot{\Phi}^{(k)}_{t}= \frac{\dd }{\dd t} \Phi^{(k)}_t
\end{equation}
at each point $t$ which does not belong to the set of points $x_{0}, \ldots, x_{n}$. Hence formally summing all the components of various arities we observe that the derivative $\frac{d}{dt} \Phi_t $ is a well-defined derivation for all $ t \in I$  which is not in the countable set of the points of the subdivisions. It is immediate to conclude that wherever it is defined, $\dot{\Phi}_{t}$ is a $\Phi_t$-derivation of degree $0$. Hence a piecewise-$C^{1}$ path $t\mapsto\Phi_{t}$ valued in Lie ${\infty}$-morphisms induces to piecewise continuous path valued in $\Phi_{t}$-derivations or, more precisely:

\begin{definition}
For a piecewise-$C^1$ path $t \mapsto  \Phi_t$ valued on Lie $\infty $-algebroids morphisms from $(E,Q_{E})$ to $(F,Q_{F})$,
we call \emph{ piecewise continuous path valued in $ \Phi_t$-derivations} a path $t \mapsto \delta_t $, with $\delta_t$ a $\Phi_t$-derivation,
such that the path $ t \mapsto \widehat{\delta}_t^{(k)}$, valued in $ \Gamma(S^{(k+1)}(E^*) \otimes F)$ for all $k  \in {\mathbb N}$,
is piecewise-continuous in the usual sense.
\end{definition}

One can also check that for any arity $k$, $ \dot{\Phi}_t$ satisfies $[Q,\dot{\Phi}_t]^{(k)}=0$ for each value of $t$ where $\dot{\Phi}_{t}^{(k)}$ is defined. Then the $\Phi_{t}$-derivation $\dot{\Phi}_{t}$ is a cocycle in the complex of Lemma \ref{lem:complexPhideriv} for every $t$ for which it makes sense.
 Thus any piecewise-$C^{1}$ path $t\mapsto\Phi_{t}$ valued in Lie $\infty$-morphisms induces a piecewise continuous path $t\mapsto\delta_{t}$ valued in $\Phi_{t}$-derivations which coincides with $\dot{\Phi}_{t}$ wherever it is defined, and which is a $[Q,\, .\ ]$-cocycle. But of course, it is not necessarily a $[Q,\, .\ ]$-coboundary. A \emph{homotopy} between two Lie $\infty$-morphisms is precisely a piecewise-$C^{1}$ path between the two morphisms, such that its derivative is a $[Q,\, .\ ]$-coboundary. More precisely:

\begin{definition}\label{def:homotopy}
Let $\Phi$ and $\Psi$ be two Lie $\infty$-morphisms from $(E,Q_{E})$ to $(F,Q_{F})$ covering the identity morphism. We say that $\Phi$ and $\Psi$ are \emph{homotopic} if there exist
\begin{enumerate}
\item a piecewise-$C^{1}$ path $t\mapsto\Phi_t$ valued in Lie $\infty$-morphisms between $E$ and $F$ such that:
\begin{equation*}
\Phi_0=\Phi\hspace{2cm}\text{and}\hspace{2cm}\Phi_1=\Psi
\end{equation*}
\item a piecewise continuous path $t\mapsto H_t $ valued in $\Phi_t$-derivations of degree $-1$, such that the following equation:
\begin{equation}\label{Hortense}
\frac{\dd\Phi^{(k)}_t}{\dd t}=[Q, H_{t}]^{(k)}
\end{equation}
holds for every $k$ and every $t\in\,]0,1[$ where it is defined.
\end{enumerate}
The previous data is called a \emph{homotopy between $\Phi$ and $\Psi$}, and shall be denoted by $(\Phi_{t},H_{t})$.
\end{definition}

Beyond the natural definition, we can deduce this immediate but very important result:
\begin{proposition}
Homotopy is an equivalence relation denoted by $\sim$ on Lie $\infty$-morphisms, which is compatible with composition.
\end{proposition}
\begin{proof}
Let us show that homotopy defines an equivalence relation $\sim$ between Lie $\infty$-morphisms:
\begin{itemize}
\item \emph{reflexivity}: $\Phi\sim\Phi$, as can be seen by choosing $\Phi_{t}=\Phi$ and $H_{t}=0$ for every $t\in[0,1]$.
\item \emph{symmetry}: $\Phi\sim\Psi$ implies that $\Psi\sim\Phi$ by reversing the flow of time, i.e. by considering the homotopy $(\Phi_{1-t},H_{1-t})$.
\item \emph{transitivity}: if $\Phi\sim\Psi$ and $\Psi\sim\chi$ then there exists a homotopy $(\Theta_{1\, t},H_{1\, t})$ between $\Phi$ and $\Psi$ and a homotopy $(\Theta_{2\, t},H_{2\, t})$ joining $\Psi$ and $\chi$. It is then sufficient to glue $\Theta_{1}$ and $\Theta_{2}$ and rescale the time variables, so that the new time variable takes values in the closed interval $[0,1]$. The resulting map will be continuous at the junction, but not differentiable in general at that point.
\end{itemize}
Now assume that $\Phi,\Psi:\mathscr{F}\to\mathscr{E}$ are homotopic Lie $\infty$-morphisms between $E$ and $F$, and that $\alpha,\beta:\mathscr{G}\to\mathscr{F}$ are homotopic Lie $\infty$-morphisms between $F$ and $H$. Let us denote by $(\Phi_{t},H_{t})$ the homotopy between $\Phi$ and $\Psi$, and $(\alpha_{s},\Xi_{s})$ the homotopy between $\alpha$ and $\beta$. Then $\Phi\circ\alpha$ and $\Psi\circ\beta$ are homotopic via $(\Phi_{t}\circ\alpha_{t},\Phi_{t}\circ\Xi_{t}+H_{t}\circ\alpha_{t})$.
\end{proof}

We can obviously use this new notion of homotopy to improve and make sense of the notion of isomorphism between Lie $\infty$-algebroid structures over $M$. This enables us to give an adequate formulation of equivalence of Lie $\infty$-algebroid structures:
\begin{definition}
Let $(E,Q_{E})$ and $(F,Q_{F})$ be two Lie $\infty$-algebroids over $M$, and let $\Phi:\mathscr{F}\to\mathscr{E}$ and $\Psi:\mathscr{E}\to\mathscr{F}$ be Lie $\infty$-morphisms between $E$ and $F$. We say that $\Phi$ and $\Psi$ are \emph{isomorphisms up to homotopy} if:
\begin{equation*}
\Phi\circ\Psi\sim\mathrm{id}_{\funct}\hspace{2cm}\text{and}\hspace{2cm}\Psi\circ\Phi\sim\mathrm{id}_{\functt}
\end{equation*}
In that case, we say that the Lie $\infty$-algebroids $(E,Q_{E})$ and $(F,Q_{F})$ are \emph{isomorphic up to homotopy}.
\end{definition}
The main improvement coming from this definition is that now $\Phi$ and $\Psi$ can carry components of arity stricly higher than $0$. We thus have built an equivalence relation for Lie $\infty$-algebroid structures over $M$. The importance of such a definition relies on the following result, which says that two homotopic Lie $\infty$-morphisms are related by a $[Q,\,.\ ]$-exact term:
\begin{proposition}
\label{prop:HomotopyMeansHomotopy}
Let $(E,Q_E)$ and $(F,Q_F)$ be Lie $\infty$-algebroids over $M$.
For any two homotopic  $L_\infty $-morphisms $\Phi,\Psi$ from $E$ to $F$, there exists a $\cinf(M)$-linear map
$ H : {\mathscr F } \to {\mathscr E}$ of degree $-1 $  such that:
 \begin{equation}\label{tolley} \Psi - \Phi = [Q, H] \end{equation}
\end{proposition}
\begin{proof}
From $  \frac{\dd }{\dd t} \Phi_t = [Q,H_{t}] $ and the fact that the path is piecewise-$C^{1}$, we  obtain:
\begin{eqnarray}  \Psi - \Phi &=& \int_0^1[Q,H_{t}] dt \\
&=& \int_0^1 (Q_E \circ H_t + H_t \circ Q_F ) dt\nonumber \\ 
&=&  Q_E \circ \left( \int_0^1 H_t dt \right)  + \left(\int_0^1 H_t dt \right)\circ Q_F \nonumber  \end{eqnarray}
Hence $H= \int_0^1 H_t dt $ satisfies Condition \eqref{tolley}. It is $\cinf(M)$-linear since so is $H_{t}$ for all $t\in]0,1[$.
\end{proof}

\begin{remarque}
The homotopy $H$ is neither an algebra morphism, nor a derivation of any sort (except if the left-hand sides vanish \cite{Lada}). At least it is a $\cinf(M)$-linear operation. If we isolate the components of arity $0$ in Equation \eqref{tolley}, we find that the linear parts $\phi,\psi$ induced by the homotopic Lie $\infty$-algebroid morphisms $\Phi,\Psi$ from $(E,Q_{E})$ to $(F,Q_{F})$ are homotopic in the usual sense, relatively to the chain complexes:
 \begin{center}
\begin{tikzcd}[column sep=1.3cm,row sep=0.7cm]
\ldots\ar[r,"\dd"] & E_{-3} \ar[r,"\dd"]\ar[dd,shift left =0.5ex,"\phi"]\ar[dd,shift right =0.5ex,"\psi" left]& E_{-2} \ar[ddl,dashed, "h\ " above] \ar[dd,shift left =0.5ex,"\phi"]\ar[dd,shift right =0.5ex,"\psi" left]\ar[r,"\dd"]& E_{-1} \ar[ddl,dashed,"h\ " above] \ar[dd,shift left =0.5ex,"\phi"]\ar[dd,shift right =0.5ex,"\psi" left]\\
&\\
\ldots\ar[r,"\dd'"]& F_{-3} \ar[r,"\dd'"]& F_{-2}\ar[r,"\dd'"]&  F_{-1}
\end{tikzcd}
\end{center}
where $h$ is the dual map of the component of arity $0$ of $H$.
\end{remarque}


\section{Singular Foliations}\label{foliations}

At the turn of the fifties, numerous investigations in the field of control theory $-$ which studies the behaviour of servomechanisms $-$ appeared and have developed in the following years. The pionneer work of R. Kalman aimed at translating the language of Laplace transform into the language of dynamical systems, that is: the study of the solvability of first-order differential equations under the influence of one or more external parameters. The dynamics can be modelled by the flow of a family of vector fields defined on the phase space $P=M\times E$ where $M$ is the configuration manifold and $E$ is the space of external parameters. To each value of the external parameters corresponds a set of first-order differential equations which are encoded by vector fields tangent to $M$. Thus, solving the equations of the flow, one finds the set of points of the configuration manifold which are \emph{accessible} in a finite time. The problem is then to understand what does the set of all accessible points looks like for a fixed value of the parameter, and under which assumptions on the differential equations can any point of $M$ be accessible? Moreover one can ask whether some kind of stability condition is met: that is, given an initial point $x_{0}$ and an accessible point $x_{1}$ joined by a path $\gamma$, can any point in the neighborhood of $x_{1}$ be linked to a point in the neighborhood of $x_{0}$ by a path close to $\gamma$?

Many mathematicians, for instance geometers \cite{Jurdjevic}, worked in the field of control theory and its developments in dynamical systems. In this context, R. Hermann was the first to relate dynamical systems to the theory of foliations \cite{Hermann1963}, drawing attention to the then unsolved problem of integrating a family of vector fields into a singular foliation. Mathematicians have naturally considered a generalization of control theory by turning the phase space into a phase manifold, in which the external parameters are coordinates. Mathematicians observed that when a family of vector fields is given, under natural assumptions, the set of accessible points from any fixed initial point with fixed external parameter is an immersed submanifold of the phase manifold. Moreover they realized that changing the external parameters would give another submanifold in some sense \emph{parallel} to the first one, hence providing a partition of the phase manifold into disjoint submanifolds: that is, a (possibly singular) \emph{foliation}. Merging these ideas from control theory together with the well-known result of F. Frobenius on the integration of regular distributions on a manifold, mathematicians tried to understand what are the conditions for a singular distribution on the space of vector fields to be integrable into a singular foliation.

The first result of this kind was conjectured by R. Hermann \cite{Hermann1963}, and later proved by T. Nagano \cite{Nagano}: an analytic distribution is integrable if and only if it is involutive. In the smooth case, it was observed long ago that an integrable distribution was necessarily involutive, but that the converse was not true. However Hermann found that if the distribution is also locally finitely generated then it is integrable \cite{Hermann1962}, but that assumption might be too strong. Then Lobry weakened the condition \cite{Lobry70}, and at the beginning of the seventies, P. Stefan and H. Sussmann independently published groundbreaking results \cites{Stefanofficiel, Sussmannofficiel}, using the same methods, inspired by the techniques used in control theory. They gave the necessary and sufficient conditions under which a singular smooth distribution is integrable.

\subsection{Foliations and distributions}\label{foliationsanddistributions}

In the sixties and early seventies, the concept of foliations was well established, but their denomination did not make any consensus yet. We have thus chosen the most obvious notations and denominations which are close enough to the historical definitions. Given a topological manifold $M$, we define a \emph{foliation} $\mathcal{F}=\{\mathcal{F}_{\alpha}\}$ to be a partition of $M$ into a disjoint union of immersed connected submanifolds $\mathcal{F}_{\alpha}$. Elements of the foliation are called \emph{leaves} and we distinguish between \emph{regular foliations} $-$ in which the leaves have the same dimensions $-$ and \emph{singular foliations}, for which their dimensions can drop from one point to another.

Regular foliations are characterized by \emph{foliated atlases}, that is: for any given $x\in M$, there is a chart $(U,\phi)$ containing $x$ and a neighborhood $V\times W$ of $(0,0)$ in $\mathbb{R}^{p}\times\mathbb{R}^{n-p}$ (where $p$ is the dimension of the leaves) such that:
\begin{itemize}
\item$\phi(U)=V\times W$
\item $\phi(x)=(0,0)$
\item transition functions from $V\times W$ to $V'\times W'$ read as
\begin{equation*}
\psi(v,w)=\big(\theta(v,w),h(w)\big)
\end{equation*}
for every $v\in V$ and $w\in W$, and some smooth functions $\theta:V\times W\to V'$ and $h:W\to W'$
\item for any leaf $\mathcal{F}_{\alpha}$ intersecting $U$ and for any connected component $\mathcal{G}\subset\mathcal{F}_{\alpha}\,\cap\, U$, there exists $y\in W$ such that $\mathcal{G}=\phi^{-1}(V\times\{y\})$
\end{itemize}
The last condition means that any chart of the foliated atlas is a saturated set: it is the union of disjoint connected submanifolds of the form $\phi^{-1}(V\times \{y\})$ that we call \emph{plaques}. In the singular case, there is no such thing as a foliated atlas: transition functions cannot be defined as in the regular case. We only observe that $\phi$ is a submersion from $U$ to a foliated open set $W\subset\mathbb{R}^{n-p}$ such that for every $y\in W$ there exists a leaf $\mathcal{F}_{\alpha}$ satisfying the inclusion: $\phi^{-1}(y)\subset\mathcal{F}_{\alpha}$.

As a notation we define $\mathcal{F}_{x}$ to be the leaf through $x$, hence $M=\bigcup_{x\in M} \mathcal{F}_{x}$. Since the map $x\mapsto\mathrm{dim}(\mathcal{F}_{x})$ which associates to any point $x$ the dimension of its leaf is lower semi-continuous, the leaves of the foliation near $\mathcal{F}_{x}$ have a dimension higher of equal to $\mathrm{dim}(\mathcal{F}_{x})$. 
 A point $x\in M$ is called a \emph{regular point} if the dimension of the leaves is constant in some neighborhood of $x$. In such a case, the leaf $\mathcal{F}_x$ is said \emph{regular}, and it is said \emph{singular} otherwise. The set of regular points is open and dense in $M$. 
From now on, $M$ is taken to be a smooth or analytic manifold. A foliation is said to be \emph{smooth} (resp. \emph{analytic}) if for every $x\in M$, any tangent vector to the leaf $\mathcal{F}_{x}$ at $x$ can be extended into a smooth (resp. analytic) vector field which is tangent to $\mathcal{F}_{y}$ for all $y\in M$. From now on until the end of this part, foliations will be considered smooth or analytic, depending on the context. Here are a few examples:

\begin{example}
The set of integral curves of a given non vanishing vector field on $M$ forms a regular foliation of dimension 1 of the underlying manifold.
\end{example}

\begin{example}
The action of $\mathfrak{so}(2)$ on $\mathbb{R}^{2}$ gives a singular foliation consisting in concentric circles, except at the origin where the circle collapses to a mere point.
\end{example}

At a point $x\in M$, the tangent space of the leaf $\mathcal{F}_{x}$ is denoted by $T_{x}\mathcal{F}_{x}$. The union $\bigcup_{x\in M}T_{x}\mathcal{F}_{x}$ of all the tangent spaces of a regular foliation is a subbundle of the global tangent bundle. However, if the foliation is singular, the dimension of the leaves can jump from one point to another, and so does the dimension of their tangent spaces. Hence there is no hope to get a subbundle of the tangent bundle, and we rather have what is called a distribution. More precisely, a \emph{distribution} $\mathcal{D}$ on a smooth manifold $M$ is the assignment, for each $x\in M$, of a subspace of $T_{x}M$. Even though the tangent space of a regular foliation is a particular case of a distribution, in general the dimension of the subspace can vary from one point to another. The smooth sections $\Gamma(\mathcal{D})$ of the distribution $\mathcal{D}$ form a sub-sheaf of the sheaf of vector fields $\mathfrak{X}(M)$: for every open set $U$, the $\cinf(U)$-module $\Gamma_{U}(\mathcal{D})$ consists in elements of $\mathfrak{X}(U)$ that take values in $\mathcal{D}$.

A distribution is \emph{smooth} if for every point $x$, any tangent vector $X(x)\in\mathcal{D}_{x}$ can be extended in some neighborhood $U$ of $x$ to some vector field $X$ which takes values in $\mathcal{D}$ for every point $y\in U$. An equivalent formulation is that  there exists a (possibly infinite) family of vector fields $\{X_{a}\}$ such that $\mathcal{D}_{y}=\mathrm{span}\big(\{X_{a}(y)\}\big)$ for all points $y\in U$. Notice that this does not necessarily mean that this family of vector fields generates the $\cinf(U)$-module $\Gamma_{U}(\mathcal{D})$ \cite{Texas}, nor that $X_{a}(y)$ belongs to $\mathcal{D}_{y}$ for any other $y\in M$ outside this neighborhood. Then if the distribution is smooth, by definition, there exists a (possibly infinite) family of vector fields that span the distribution at each point. It has been shown that this family can be taken to be finite \cite{Texas}, even if the space of sections of a distribution is not necessarily finitely generated as a module (see Example \ref{ex:smoothdist}). The distribution is \emph{locally finitely generated} if for every point $x$, there exists a neighborhood $U$ such that the $\cinf(U)$-module $\Gamma_{U}(\mathcal{D})$ is finitely generated. Finally, we say that a smooth distribution is \emph{involutive} if it is stable under the Lie bracket: for every two sections $X,Y\in \Gamma(\mathcal{D})$, the bracket $[X,Y]$ is also a section of $\mathcal{D}$.

We saw that a distribution can be seen as a generalization of the tangent bundle of a regular foliation to the singular case. It is then legitimate to wonder if a given distribution on $M$ arises from a foliation. That is, if there exists a (possibly singular) foliation $\mathcal{F}$ such that $\mathcal{D}_{x}=T_{x}\mathcal{F}_{x}$ for every $x\in M$. If this is the case, we say that the distribution is \emph{integrable}. In the same spirit we define an \emph{integral manifold} of the distribution $\mathcal{D}$ as an immersed submanifold $N$ such that $T_{x}N=\mathcal{D}_{x}$ for every $x\in N$. When the distribution is regular, the answer to the integration problem is given by the well-known theorem of F. Frobenius \cite{Frobenius}. He was the first to give a rigorous proof (with respect to the standards of the ninetieth century) of the statement, though it was originally proven by F. Deahna \cite{Deahna} and made more explicit by A. Clebsch \cite{Clebsch}:
\begin{theoreme}\emph{\textbf{Frobenius (1877)}}\label{theo:frobenius}
A smooth distribution $\mathcal{D}$ of constant rank on a smooth manifold $M$ is integrable into a regular foliation if and only if it is involutive.
\end{theoreme}

In geometric control theory, a set of parametrized first-order differential equations is described with the help of a distribution on some phase manifold $M$. Knowing if this distribution is integrable is crucial to solve the problem of controlling the solutions of the equations. In particular, a leaf corresponds to a set of solutions for a fixed parameter. Changing the parameter will give $-$ smoothly $-$ another leaf. Thus at the beginning of the sixties mathematicians wanted to understand under which condition a distribution would be integrable.

\subsection{Analytic distributions}\label{Nagano}

One of the first mathematicians who provided important results on the problem of integrating singular distributions was R. Hermann. In a sequence of seminal articles \cites{Hermann1962, Hermann1963}, he introduced the main ideas that one would need to approach the problem. In particular in \cite{Hermann1963} in 1963, he conjectured that an \emph{analytic distribution} would be integrable. Here, by analytic distribution, we mean a distribution $\mathcal{D}$ on a real analytic manifold $M$ such that there exist local generators that are analytic. Even though the number of local sections can be infinite in principle, the fact that the ring of germs of analytic functions is Noetherian implies that there is actually only a finite number of generators. Then in 1966, T. Nagano proved the following result \cite{Nagano}:


\begin{theoreme}\label{theo:Nagano66}
\emph{\textbf{Nagano (1966)}} Let $\mathcal{D}$ be an analytic distribution on a real analytic manifold $M$. Then $\mathcal{D}$ is integrable if and only if it is involutive.
\end{theoreme}
\begin{proof}
Assuming that the distribution $\mathcal{D}$ is integrable, there exists a (possibly singular) foliation $\mathcal{F}$ such that for every $x$, $T_{x}\mathcal{F}_{x}=\mathcal{D}_{x}$ where $\mathcal{F}_{x}$ is the leaf through $x$. The sections of $\mathcal{D}$ are tangent at any point to the leaves, and then so does their Lie bracket, hence the distribution $\mathcal{D}$ is involutive.

In the reverse direction, the sketch of the proof is the following: first, we show that there exists an integral manifold of $\mathcal{D}$ passing through every point $x$, and then using this result, we show that $M$ can be foliated by integral manifolds of $\mathcal{D}$. Thus we first want to prove the following lemma:
\begin{lemme}
Let $\mathfrak{N}_{x}$ be the set of all integral manifolds of $\mathcal{D}$ which contain the point $x$ of $M$. Then $\mathfrak{N}_{x}$ is not empty.
\end{lemme}
\begin{proof}
Let $x\in M$ and $d$ be the dimension of $\mathcal{D}_{x}$. Since $\mathcal{D}$ is locally finitely generated by analyticity, there exist a neighborhood $U$ of $x$, and a family $X_{1},\ldots,X_{d}$ of sections of $\mathcal{D}$ on $U$ such that $X_{1}(x),\ldots,X_{d}(x)$ span $\mathcal{D}_{x}$. At the price of shrinking $U$, we can assume that the family of vector fields $X_{1},\ldots, X_{d}$ is free on $U$. Thus there exists a system of coordinates associated to the leaf, that is: the first $d$ coordinates $x^{i}$ are defined by the flow of the vector fields $X_{i}$, and we add $n-d$ coordinates $x^{\lambda}$ which label the coordinates adapted to a transversal of the leaf. Then we can write the vector fields $X_{i}$ as:
\begin{equation}
X_{i}=\frac{\partial}{\partial x^{i}}+\sum_{\lambda=1}^{n-d} u_{i}^{\lambda}\frac{\partial}{\partial x^{\lambda}}
\end{equation}
where the functions $u_{i}^{\lambda}$ are real analytic functions of the coordinates $x^{i}, x^{\lambda}$ and vanish at the point $x$.

We shall split the distribution $\mathcal{D}$ over $U$ into two parts. Since the family of vector fields $X_{1},\ldots,X_{d}$ is locally free on the open set $U$, and at the price of shrinking $U$ once more, we define $\mathcal{J}$ to be the trivial vector subbundle of $TM$ over $U$ spanned by the vector fields $X_{1},\ldots,X_{d}$. Let $\mathcal{K}$ be the distribution defined on $U$ as the subdistribution of $\mathcal{D}$ generated by the vector fields $\frac{\partial}{\partial x^{\lambda}}$:
\begin{equation*}
\mathcal{K}_{y}=\left\{Y\in\mathcal{D}_{y}\hspace{0.1cm}\Big|\hspace{0.1cm}\exists\,\,X\in\Gamma_{U}(\mathcal{D})\quad\text{such that}\quad\begin{cases}X=\sum_{\lambda=1}^{n-d}f^{\lambda}\frac{\partial}{\partial x^{\lambda}}\,\,\text{for some}\,\, f^{\lambda}\in\mathcal{C}^{\omega}(U)\\
X(y)=Y
\end{cases} \right\}
\end{equation*}
for every $y\in U$ or, in others words:
\begin{equation*}
\mathcal{K}_{y}\simeq\bigslant{\mathcal{D}_{y}}{\mathcal{J}_{y}}
\end{equation*}
Letting $\mathcal{L}=\mathcal{J}+\mathcal{K}$ (the sum being done pointwise in $TM$), we observe that the space $\mathcal{L}$ coincides with the distribution $\mathcal{D}$ over $U$. Since $\mathcal{D}$ is involutive, we deduce that $[\mathcal{L}_{y},\mathcal{L}_{y}]\subset\mathcal{L}_{y}$ for all $y\in U$. However, the splitting $\mathcal{L}=\mathcal{J}+\mathcal{K}$ enables us to check that
\begin{equation}\label{eq:Nagano0}
[\mathcal{L}_{y},\mathcal{L}_{y}]\subset\mathcal{K}_{y}
\end{equation}
Since $\mathcal{L}_{x}$ coincides with $\mathcal{D}_{x}$ at the point $x$, the Lemma is proven when we show the following claim:
\begin{center}
\emph{$\mathcal{L}$ has an integral manifold $N$ which contains $x$}
\end{center}

The idea of this proof is as follows: first we build a submanifold $N\subset U$ of $M$ which contains $x$, and then we show that $\mathcal{J}_{y}=\mathcal{L}_{y}$ for every $y\in N$. And finally, it is sufficient to show that $T_{y}N=\mathcal{J}_{y}$ by showing that any element of $\mathcal{J}_{x}$ can be transported to $T_{y}N$ by staying in the tangent bundle $TN$, and then for dimensional reasons, they should coincide. We denote by $t\mapsto\phi_{X}^{t}$ the flow of the vector field $X\in\mathcal{J}$. The map $\phi_{X}^{1}$ sends any element $x\in M$ to a point in its neighborhood, in a one-to-one fashion. Let $\mathcal{U}$ be a neighborhood of $0$ in $\mathcal{J}_{x}$, then the map $X\mapsto \phi_{X}^{1}(x)$ (where $X$ is a constant vector field) is an embedding of $\mathcal{U}$ into $M$, whose image we call $N$. Is it an embedded submanifold of dimension $d$.

We now want to show that $\mathcal{L}_{y}=\mathcal{J}_{y}$ for every $y\in N$ or, in other words that:
\begin{center}
\emph{The vector fields in $\mathcal{K}$ vanish at all points of $N$}
\end{center}
We know that any vector field $Y\in\mathcal{K}$ vanishes at $x$, thus it is sufficient to show that it vanishes on the curve $t\mapsto\phi_{X}^{t}(x)$ for any constant vector field $X\in\mathcal{J}$, because of the isomorphism between $N$ and $\mathcal{U}$. To do this we will use the analyticity of $Y$ and show that every coefficient of its Taylor series is zero, hence $Y=0$. The derivative of $Y$ with respect to time $t$ is obtained by taking the bracket of $X$ and $Y$:
\begin{equation}
\frac{\dd}{\dd t}Y=[X,Y]
\end{equation}
Thus, we obtain the successive derivatives as $\frac{\dd^{k} Y}{\dd t^{k}}=[\mathrm{ad}(X)^{(k)},Y]$, for every $k\geq1$. But this bracket, taking values in $\mathcal{K}$ by Condition \eqref{eq:Nagano0}, vanishes at the point $x$, and since all the derivatives $\frac{\dd^{k} Y}{\dd t^{k}}$ vanish at $x$ as well, the bracket vanishes at every point on the curve. The vector field $Y$ is then zero on the curve, hence we have proven that $\mathcal{L}_{y}=\mathcal{J}_{y}$ on the submanifold $N$.

We will now show that $\mathcal{J}_{y}=T_{y}N$, that is: $T_{y}N$ is spanned by the tangent vectors $X_{i}(y)$ for every $y\in N$. The only thing we know is that any constant vector field $X\in\mathcal{U}$ lies in $T_{\phi^{1}_{X}(x)}N$, but we do not have any information on the other directions at $\phi^{1}_{X}(x)$. However we claim that:
\begin{center}
\emph{$\mathcal{J}_{y}$ is included in $T_{y}N$ for every $y\in N$}
\end{center}
The idea is to show that every constant tangent vector $Z\in\mathcal{U}$ belongs to $T_{\phi^{1}_{X}(x)}N$ for every $X\in\mathcal{U}$ so that, by the equality of dimensions between the former and the latter, we finally obtain that $T_{y}N=\mathcal{J}_{y}$. To do this we take a non zero tangent vector $X\in\mathcal{U}$ and we observe that the path
\begin{equation}
x(s,t)=\phi^{1}_{t(sZ+X)}(x)
\end{equation}
takes values in $N$ whenever this makes sense for $(s,t)\in\mathbb{R}^{2}$. Fixing the $t$ variable and moving along the $s$ curve, we can differentiate with respect to $s$ to obtain the tangent vector to that curve. We claim that:
\begin{equation}\label{eq:Nagano1}
\left.\frac{\partial}{\partial s}x(s,t)\right|_{s=0}=t\, Z\big(x(s,t)\big)
\end{equation}
The left-hand side is a vector field $Z'(t)$ on the curve $t\mapsto x(0,t)$ and vanishes at $t=0$. We show that it satisfies a first-order linear differential equation, which is also satisfied by the vector field $t\, Z$ with the same initial value condition, so they coincide. We indeed have:
\begin{equation}
\frac{\partial^{2}}{\partial s\,\partial t}\big(x(s,t)\big)=\frac{\partial}{\partial s}(sZ+X)=Z+\sum_{i=1}^{d}\frac{\partial (sZ+X)}{\partial x^{i}}\frac{\partial x^{i}(s,t)}{\partial s}
\end{equation}
where the last equality comes from the fact that for a fixed time $t$, differentiating along $s$ is equivalent to taking the directional derivative along the $s$-curve $s\mapsto x(s,t)$. Taking this equation at $s=0$, we obtain a linear differential equation on $Z'$:
\begin{equation}\label{eq:Nagano2}
\frac{\dd}{\dd t}Z'=Z+Z'\circ X
\end{equation}
On the other hand, the vector field $t\, Z$ satisfies at $s=0$:
\begin{align}\label{eq:Nagano2bis}
\frac{\dd}{\dd t}(t\,Z)
&=Z+\sum_{i=1}^{d}\frac{\partial (t\, Z)}{\partial x^{i}}\frac{\partial x^{i}(s,t)}{\partial t}\nonumber\\
&=Z+t\,\sum_{i=1}^{d}\frac{\partial Z}{\partial x^{i}}X^{i}\nonumber\\
&=Z+t\,[X,Z]+(t\,Z)\circ X
\end{align}
The only difference with Equation \eqref{eq:Nagano2} is the Lie bracket $[X,Z]$ which belongs to $\mathcal{K}$ by \eqref{eq:Nagano0}. But it should vanish on $N$, then Equations \eqref{eq:Nagano2} and \eqref{eq:Nagano2bis} are thus identical. In other words, the vector fields $Z'(t)$ and $tZ$ are both solutions of the following first-order linear differential equation:
\begin{equation}
\begin{cases}\frac{\dd u}{\dd t}(t)=Z+u(t)\circ X\\
u(0)=0
\end{cases}
\end{equation}
Hence the result: $Z'=t\,Z$, that is, Equation \eqref{eq:Nagano1} is satisfied, which means that the vector field $Z$ is in $T_{x(s,t)}N$. Being randomly chosen in $\mathcal{U}$, we conclude that $\mathcal{J}_{x(s,t)}\subset T_{x(s,t)}N$, and by reasons of dimensions, we finally have the equality. To summarize, we have proven that $\mathcal{J}_{y}=T_{y}N$ for every $y\in N$. Thus we actually have $\mathcal{L}_{y}=T_{y}N$ for every $y\in N$ or, by definition of $\mathcal{L}$, we obtain $T_{y}N=\mathcal{D}_{y}$, hence proving the Lemma.\end{proof}

Let us conclude the proof of Theorem \ref{theo:Nagano66}. Set
\begin{equation*}
\mathcal{F}_{x}=\bigcup_{N\in\,\mathfrak{N}_{x}}N
\end{equation*}
then it is a submanifold of $M$. Indeed, the topology and the analytic structure on $\mathcal{F}_x$ are given by the charts of all integral 
manifolds $N$ in $\mathfrak{N}_{x}$, and the fact that for every $y\in M$, the intersection of two integral manifolds through $y$, is an embedded integral manifold in some neighborhood of $y$, see \cite{Nagano}. Let us define $\mathcal{F}=\{\mathcal{F}_{x}\,|\,x\in M\}$ and we show that it is a partition of $M$. Obviously, the union of all the integral manifolds $\mathcal{F}_{x}$ is all of $M$. Moreover to show that the leaves are disjoint, we have to show that for any $x,y\in M$, either $\mathcal{F}_{x}=\mathcal{F}_{y}$ or $\mathcal{F}_{x}\cap\mathcal{F}_{y}=\varnothing$. Suppose that there exists $z\in\mathcal{F}_{x}\cap\mathcal{F}_{y}$, then both $\mathcal{F}_{x}$ and $\mathcal{F}_{y}$ are contained in $\mathfrak{N}_{z}$, and then in $ \mathcal{F}_{z}$ as well, which implies conversely that both $x$ and $y$ are contained in $\mathcal{F}_{z}$. Then this integral manifold belongs to both $\mathfrak{N}_{x}$ and $\mathfrak{N}_{y}$, and thus to $\mathcal{F}_{x}$ and $\mathcal{F}_{y}$ as well. Hence the result: $\mathcal{F}_{z}=\mathcal{F}_{x}=\mathcal{F}_{y}$. We thus have shown that an involutive analytic distribution is integrable.\end{proof}

The proof has twice used the analycity of the objects: the first time to deduce that the distribution is locally finitely generated, and the second time to prove that the vector fields of $\mathcal{K}$ vanish on $N$. We have shown that an analytic distribution which is involutive is integrable. Here is now a well-known counter example of an involutive smooth distribution which is not integrable:

\begin{example}\label{example1}
Let $\mathcal{D}$ be the smooth distribution on $\mathbb{R}^{2}$ defined by:
\begin{equation*}
\mathcal{D}_{(x,y)}=
\begin{cases}T_{(x,y)}\mathbb{R}^{2} \quad \text{for}\quad 0> x\\
\langle\frac{\partial}{\partial x}\rangle \hspace{1cm} \text{for}\quad x\leq 0
\end{cases}
\end{equation*}
where we understand $\langle\frac{\partial}{\partial x}\rangle$ as the subspace of $T_{(x,y)}\mathbb{R}^{2}$ spanned by the tangent vector $\frac{\partial}{\partial x}$. Sections of this distribution consists of sums of horizontal vector fields and vertical vectors fields which vanish for $x\leq0$. The bracket will preserve this property and therefore the distribution is involutive.

We now show that though this smooth distribution is involutive, it cannot be integrated into a singular foliation. On the right half-plane (for $x>0$), the leaf associated to this distribution is all of the open half-plane. On the contrary, on the left half-plane (for $x<0$) the vertical vector field vanishes hence the leaves are horizontal (since at each point the vector field $\frac{\partial}{\partial x}$ generates the tangent space to the leaf). This is still true for $x=0$, thus the horizontal leaves can be extended to the vertical axis. But each one of these half-lines $-$ when considered as a subset of $\mathbb{R}$ $-$ does not form a submanifold because it is not open on its right end, hence these half-lines are not submanifolds of $\mathbb{R}^{2}$. The distribution is not integrable.
\end{example}

This counter-example illustrates the fact that in the smooth case there exist smooth functions which can have singularities on an open set, although they are not identically zero $-$ which is not the case for analytic functions.

\subsection{Hermann's Theorem}\label{Hermann}

We saw that Theorem \ref{theo:Nagano66} is the analog of the Theorem of Frobenius, in the context of analytic manifolds. A first step toward a reformulation of these theorems in the smooth case would be to drop the analycity condition, though keeping the assumption of involutivity. The analycity implied that the distribution was locally finitely generated, and this appeared as the starting point of the proof of Theorem \ref{theo:Nagano66}. That would lead to the idea of keeping this assumption in the smooth case, which leads to the famous result of R. Hermann \cite{Hermann1962}, which was proved before Nagano's result:

\begin{theoreme}
\emph{\textbf{Hermann (1962)}}\label{theo:Hermann1962} Let $\mathcal{D}$ be a locally finitely generated smooth distribution on a smooth manifold $M$. Then $\mathcal{D}$ is integrable if and only if it is involutive.
\end{theoreme}

\begin{remarque}
Actually, Hermann proved a more general result. That any submodule of the module of vector fields that is locally finitely generated defines an integrable distribution if and only if it is involutive.\end{remarque}

From now on until the end of this section, the proofs will heavily rely on local arguments. The key concept is that of \emph{integral path} of the distribution. Indeed one realizes that the leaf of the foliation passing through some point $x\in M$ corresponds in fact to the submanifold consisting of all points reachable from $x$ by integral paths. We define an \emph{integral path of the distribution $\mathcal{D}$} to be a piecewise smooth path $\gamma:[a,b]\to M$ such that for every open interval $I\subset[a,b]$ where $\gamma$ is differentiable, there exists a vector field $X\in\Gamma(\mathcal{D})$ such that:
\begin{equation}
\frac{\dd}{\dd t}\gamma (t)=X\big(\gamma(t)\big)\label{integralpath}
\end{equation}
The set of all points reachable from a point $x$ by an integral path of $\mathcal{D}$ forms a subset $\mathcal{L}_{x}$ of $M$ that we call a \emph{preleaf of $\mathcal{D}$}. The aim of the proof is to show that $\mathcal{L}_{x}$ can be endowed with a smooth manifold structure such that it appears as an immersed submanifold of $M$, that is $\mathcal{L}_{x}$ is an integral submanifold of $\mathcal{D}$ passing through $x$.

\begin{proof}
The proof relies on the idea that given $x\in M$, the distribution has constant dimension over the preleaf $\mathcal{L}_{x}$. From this fact, we will build smoothly compatible coordinate charts, equipping $\mathcal{L}_{x}$ with an atlas. Since every point on $\mathcal{L}_{x}$ can be reached by an integral path, it is sufficient to show that the rank of $\mathcal{D}$ is constant on an integral curve passing through $x$.

Thus, let $x\in M$, let $X\in\Gamma(\mathcal{D})$ be a vector field which does not vanish at the point $x$ and let $t\mapsto\phi_{X}^{t}$ be its flow. Since $X$ is not zero at $x$, there exists a neighborhood $U$ of $x$ with local coordinates $\{x_{1},\ldots,x_{n}\}$ such that $X(x)=\frac{\partial}{\partial x_{1}}$. The distribution being locally finitely generated, we can moreover assume that the $p$ vector fields that locally span $\mathcal{D}$ around $x$ are defined over $U$ so that we can write for every $1\leq i\leq p$:
\begin{equation}
X_{i}(y)=\sum_{k=1}^{n}a_{i,k}(y)\frac{\partial}{\partial x^k}
\end{equation}
for some smooth functions $y\mapsto a_{i,k}(y)$, $y\in U$. Then the dimension of $\mathcal{D}_{y}$ is nothing but the rank of the matrix $A(y)=\big(a_{i,k}(y)\big)$.

We would like to determine if the rank of the matrix changes when one runs over the path $t\mapsto\phi_{X}^{t}(x)$. This can be performed by first applying $X$ to the functions $a_{i,k}$:
\begin{equation}
X(a_{i,k})=\frac{\partial a_{i,k}}{\partial x^{1}}
\end{equation}
More generally, we find these derivatives in the bracket $[X,X_{i}]=\sum_{1\,\leq\, k\,\leq\, n}\frac{\partial a_{i,k}}{\partial x^{1}}\frac{\partial}{\partial x^{k}}$. But the distribution is involutive and locally finitely generated, so there exist smooth functions $f_{i,j}$ on $U$ such that $[X,X_{i}]=\sum_{1\,\leq\, j\,\leq\, p}f_{i,j}X_{j}$. Thus we obtain the following system of linear differential equations on the matrix coefficients $\{a_{i,k}\}$:
\begin{equation}
\frac{\partial a_{i,k}}{\partial x^{1}}=\sum_{1\,\leq\, j\,\leq\, p}f_{i,j}a_{j,k}
\end{equation}
Letting all other coordinates fixed and allowing only $x^{1}$ to vary, and gathering all functions $f_{i,j}$ in a $p\times p$ matrix, it is equivalent to write the above equation as a first-order linear differential equations on matrices:
\begin{equation}
\frac{\dd}{\dd t}A(t)=F(t)\times A(t)
\end{equation}
which is solved by $A(t)=\exp(\mathfrak{F}(t))\times A(0)$, where $\mathfrak{F}(t)$ is a primitive of $F(t)$. Then the rank of $A(t)$ is the same as the rank of $A(0)$, which is the dimension of $\mathcal{D}_{x}$. The rank of $\mathcal{D}$ is thus constant on the integral curve of $X$ in the open set $U$.

Now let fix $t_0>0$ and let us show that $\mathcal{D}_{\phi^X_{t_0}(x)}$ has same dimension as $\mathcal{D}_x$.
We proceed along the same lines as above for any point on the integral curve of $X$ between $0$ and $t_0$, so that we obtain a covering of $[0,t_{0}]$ by open intervals such that the distribution $\mathcal{D}$ has constant rank on the interval curve of $X$ parametrized by these open intervals. By compactness, we can choose a finite number of such open sets so that they necessarily overlap. Since the rank of $\mathcal{D}$ is locally constant on the integral curve of $X$, we conclude that it is constant all over the path $t\mapsto\phi^X_t(x)$ between 0 and $t_0$. Hence, we obtain that $\mathrm{dim}\big(\mathcal{D}_{\phi^X_{t_0}(x)}\big)=\mathrm{dim}\big(\mathcal{D}_x\big)$. By concatenation of integral curves of sections of $\mathcal{D}$, it implies that the rank of $\mathcal{D}$ is constant over any integral path, hence over the preleaf $\mathcal{L}_{x}$.

Now let us define a coordinate chart for $\mathcal{L}_{x}$ in a neighborhood $U$ of some point $y$. Locally, the tangent bundle of $M$ is trivializable, and more precisely here: $T_{U}M\simeq U\times \mathbb{R}^{n}$. Let $V$ be a small neighborhood of $0$ in $\mathbb{R}^{p}\simeq\mathcal{D}_{y}$ and let $Y\in V$. We can associate to $Y$ a point in $\mathcal{L}_{x}$ defined by $w=\phi_{Y}^{1}(y)$, that is: the point of the integral curve of $Y$ that is reached at time $t=1$. It is an homeomorphism between $V$ and $\mathcal{L}_{x}\cap U$, hence the inverse map $-$ that we call $\psi_{y}$ $-$ is a very convenient coordinate chart for $\mathcal{L}_{x}$ in a neighborhood of $y$.
Two such charts are smoothly compatible, hence
the set of such coordinate charts defines a smooth atlas for $\mathcal{L}_{x}$, turning it into an immersed submanifold of $M$. By definition, the tangent space at each point of the preleaf $\mathcal{L}_{x}$ coincides with the distribution evaluated at the same point, then $\mathcal{L}_{x}$ is an integral submanifold of $\mathcal{D}$ through $x$. The family $\mathcal{L}=\{\mathcal{L}_{x}\,|\,x\in M\}$ forms a partition of $M$, since for any two $x,y\in M$, either $\mathcal{L}_{x}=\mathcal{L}_{y}$, or $\mathcal{L}_{x}\cap\mathcal{L}_{y}=\varnothing$. We have shown that the distribution $\mathcal{D}$ can be integrated to a (singular) foliation of $M$.\end{proof}

Let us illustrate the Theorem with a concrete example:

\begin{example}
As an example (taken from \cite{Miranda}), take the 2-sphere $\mathbb{S}^{2}$ embedded in $\mathbb{R}^{3}$ with the cylindrical coordinates $(h,\theta)$, and define the following Poisson bivector $\Pi=h\frac{\partial}{\partial \widetilde{h}}\wedge\frac{\partial}{\partial \theta}$, where $\frac{\partial}{\partial \widetilde{h}}$ is the projection of $\frac{\partial}{\partial h}$ to the tangent space of the 2-sphere. The bivector $\Pi$ is tangent to the sphere and vanishes on the equator and at both poles. Poisson manifolds have been defined in Example \ref{examplePoisson}, and we know that the cotangent bundle $T^{\ast}\mathbb{S}^{2}$ is a Lie algebroid with anchor the  linear map $\Pi^{\sharp}:T^{\ast}\mathbb{S}^{2}\to T\mathbb{S}^{2}$. Then the image of $\Pi^{\sharp}$ forms a smooth distribution. It is involutive since the anchor map is compatible with the bracket, and it is locally finitely generated because the fiber of the cotangent bundle is finitely generated. This distribution is thus integrable into a singular foliation. In the present case, the bivector defines two kinds of symplectic leaves: 0-dimensional points on the equator and at both poles on the one hand, and  the punctunred 2-dimensional hemispheres on the other hand. 
More generally, it has been shown by A. Weinstein \cite{Weinstein1983} that every Poisson manifold admits singular foliation by symplectic leaves. 
\end{example}

Obviously the aim of the mathematicians would be to drop the assumption that the distribution is locally finitely generated. Since Hermann's condition stating that the sections of the distribution $\mathcal{D}$ are locally finitely generated is a condition involving the module of sections $\Gamma(\mathcal{D})$, and not on the distribution itself, it may be weakened. The successive developments of C. Lobry, P. Stefan and H. Sussmann have provided answers to this problem. 

Notice that in most cases, today, and most notably in the work of I.~Androulidakis, G.~Skandalis and M.~Zambon \cites{AndrouSkandal,AndrouZamb}, the kind of distributions which are considered are those coming from locally finitely generated involutive $\C{C}^{\infty}(M)$-submodules of compactly supported vector fields $\mathfrak{X}_c(M)$. They name it \emph{Stefan-Sussmann foliation}, in honor of P. Stefan and H. Sussmann. We chose to use a slightly different definition, which is to us more convenient: 


\begin{definition}\label{hermannfoliation}
A \emph{Hermann foliation} is a sub-sheaf ${\mathcal D}: U \mapsto  {\mathcal D}(U)$ of ${\mathfrak X}$, which is locally finitely generated as a ${\mathscr \cinf}$-submodule and closed under the Lie bracket of vector fields.
\end{definition}


In this precise context $\mathcal{D}$ symbolizes a sheaf, and not a distribution as is the norm in this section. It will not lead to any confusion since we can restrict the sheaf $\mathcal{D}$ to a point $x$ by taking the sub-space of $T_{x}M$ consisting of all the evaluations at $x$ of the sections of $\mathcal{D}$ on open sets containing $x$. Then a Hermann foliation defines an honest smooth distribution. 

\begin{remarque}
A Hermann foliation does not necessarily defines a distribution satisfying Hermann's conditions, but the proof of Hermann's theorem applies to the Hermann foliation by using the sheaf $\mathcal{D}$ instead of the section of the distribution. Then the induced distribution is integrable. This was the true idea behind Hermann Theorem in his original paper \cite{Hermann1962}. For example, on $M=\mathbb{R}$, take the submodule of $\mathfrak{X}(\mathbb{R})$ generated by $\psi(x)\frac{\partial}{\partial x}$, where $\psi\in\mathcal{C}^{\infty}(\mathbb{R})$ is a smooth function vanishing on $\mathbb{R}_{-}=\{x\in\mathbb{R}\ |\ x\leq0\}$. It is involutive and finitely generated, and it defines an integrable distribution: the leaves are points on $\mathbb{R}_{-}$ and the open line $\mathbb{R}_{+}^*$. However, as Example \ref{ex:smoothdist} shows, it is not locally finitely generated at $x=0$.
\end{remarque}


Note that involutivity for modules is slightly more stringent than for distributions: it implies for example that if some free module is given with a set of generators, the bracket of any two elements of the module should be expressed in terms of a sum of products of smooth functions (living in $\cinf(M)$) with those generators. To clarify this and illustrate Hermann's theorem, we now give an example of a smooth distribution which is not locally finitely generated, though involutive. We naturally show that it is not integrable, showing that in the smooth case, involutivity is not indeed a sufficient condition for integrability.

\begin{example}\label{ex:smoothdist}
On $\mathbb{R}^{2}$, take the submodule $\mathfrak{V}\subset\mathfrak{X}(M)$ generated by $\frac{\partial}{\partial x}$ and $\phi(x) \frac{\partial}{\partial y}$, where $\phi$ is a smooth function of $x$ such that $\phi$ and all its derivative identically vanish on the half-line $]-\infty,0]$. For example one could take
\begin{equation*}
\phi=
\begin{cases}e^{-\frac{1}{x}} \quad \text{for}\quad 0< x\\
0 \hspace{1cm} \text{for}\quad x\leq 0
\end{cases}
\end{equation*}
Then the bracket of the two vector fields $\frac{\partial}{\partial x}$ and $\phi(x)\frac{\partial}{\partial y}$ is:
\begin{equation}
\big[\frac{\partial}{\partial x},\phi(x)\frac{\partial}{\partial y}\big]=\frac{1}{x^{2}}\phi(x)\frac{\partial}{\partial y}
\end{equation}
which is everywhere defined on $\mathbb{R}^{2}$. However, since the function $x\mapsto\frac{1}{x^{2}}$ is not smooth on $\mathbb{R}$, the module $\mathfrak{V}$ is not involutive. On the other hand, the underlying distribution is the one defined in Example \ref{example1}, which is involutive in the sense of the first section.

Should this distribution be locally finitely generated, Hermann's Theorem would apply and it would be integrable. However, it has been shown that it is not integrable, hence it cannot be locally finitely generated: in some neighborhood of $(0,0)$ for example, one cannot find any finite set of vector fields which generates the sections of $\mathcal{D}$. The proof is mainly taken from \cite{Texas}: assume that there exists vector fields $X_{1},\ldots, X_{p}\in\Gamma(\mathcal{D})$ generating the module of sections of $\mathcal{D}$ on some neighborhood $U$ of $(0,0)$. Then any section $Y$ of $\mathcal{D}$ can be written as:
\begin{equation}
Y=\sum_{i=1}^{p}\, f_{i}X_{i}
\end{equation}
for some functions $f_{i}\in\cinf(U)$.

Let us call $U_{>0}$ the intersection of $U$ with the open right half-space: $U_{>0}=\big\{(x,y)\in U\, |\, x>0\big\}$. We can legitimately assume that each of the vector fields $X_{i}$ can be decomposed on $U$ as:
\begin{equation}
X_{i}=g_{i}\frac{\partial}{\partial x}+h_{i}\frac{\partial}{\partial y}
\end{equation}
for some functions $g_{i},h_{i}\in \cinf(U)$ such that $h_{i}=0$ on $U\backslash U_{>0}$. Now there is no point $(x,y)\in U_{>0}$ such that all $h_{i}$ vanish, for otherwise any vertical section of $\mathcal{D}$ (parallel to $\frac{\partial}{\partial y}$) would vanish at that point, but the above vector field $e^{-\frac{1}{x}}\frac{\partial}{\partial y}$ does not vanish for $x>0$.

Let us define $h=\sum_{i=1}^{p}h_{i}^{2}$, then it is strictly positive on $U_{>0}$. The vector field $Y=h^{\frac{1}{4}}\frac{\partial}{\partial y}$ is a section of the distribution $\mathcal{D}$ vanishing on $U\backslash U_{>0}$, hence there are some functions $f_{i}\in\cinf(U)$ such that $Y=\sum_{i=1}^{p}f_{i}X_{i}$ or, in other words:
\begin{equation}
h^{\frac{1}{4}}=\sum_{i=1}^{p}\, f_{i}h_{i}
\end{equation}
Using the Cauchy-Schwarz identity we obtain:
\begin{equation}
h^{\frac{1}{4}}=\left|\sum_{i=1}^{p}\,f_{i}h_{i}\right|\leq\sqrt{\sum_{i=1}^{p}f_{i}^{2}}\sqrt{\sum_{i=1}^{p}h_{i}^{2}}=h^{\frac{1}{2}}\sqrt{\sum_{i=1}^{p}\,f_{i}^{2}}
\end{equation}
If we restrict to $U_{>0}$ we see that $h$ does not vanish then we can divide by $\sqrt{h}$ and we obtain:
\begin{equation}
\frac{1}{h^{\frac{1}{4}}}\leq\sqrt{\sum_{i=1}^{p}\, f_{i}^{2}}
\end{equation}
but the left-hand side diverges as we approach the vertical axis from the right, whereas the right-hand side is supposedly continuous on $U$ by hypothesis, hence it should rather converge. This contradiction shows that the distribution is not locally finitely generated.
\end{example}

To conclude, we have shown that involutivity alone is not sufficient to guarantee integrability. Although Hermann proved that this condition becomes sufficient if the smooth distribution is locally finitely generated, it appears that this property of the foliation can be weakened.

\subsection{Improvement by Lobry}\label{Lobry}

In 1970, C. Lobry challenged the view that local finiteness and involutivity were adequate sufficient conditions to guarantee that a smooth distribution is integrable \cite{Lobry70}. Instead he tried to introduce a notion which merges both properties into one assumption. Unfortunately, this turned out to be not sufficient, as remarked by Stefan in \cite{Stefanofficiel, stefan1980}. Thus we slightly modify Lobry's assumption to match both requirements: a generalization of Hermann's statement, as well as staying close to Lobry's original statement. Thus, we say that a distribution $\mathcal{D}$ is \emph{locally of finite type} if for every point $x\in M$ there exist vector fields $X_{1},\ldots,X_{p}\in\Gamma(\mathcal{D})$ defined on an open neighborhood $\Omega$ of $x$ such that:
\begin{itemize}
\item $X_{1}(y),\ldots,X_{p}(y)$ span $\mathcal{D}_{y}$ for every $y\in \Omega$
\item for every $X\in \Gamma(\mathcal{D})$ there exists a neighborhood $U$ of $x$, and functions $f_{i,j}\in\mathcal{C}^{\infty}(U)$ such that:
\begin{equation}\label{eq:locfintype}
[X,X_{i}](x')=\sum_{j=1}^{p}\,f_{i,j}(x')X_{j}(x')
\end{equation}
for every $1\leq i\leq p$ and $x'\in U$.
\end{itemize}
We see that the distribution is not necessarily locally finitely generated nor involutive, but it is `nearly' both. Originally, Lobry did not assume the first item, it was Stefan who showed that Lobry's conditions were not sufficient for integrability, hence one should consider that `locally of finite type' here means something different than in Lobry's original article. The result of Lobry is the following:
\begin{theoreme}\emph{\textbf{Lobry (1970)}}\label{theo:Lobry70}
Let $\mathcal{D}$ be as smooth distribution on a smooth manifold $M$. If $\mathcal{D}$ is locally of finite type, then the distribution $\mathcal{D}$ is integrable.
\end{theoreme}

\noindent Notice that the reverse assumption is not true: an integrable distribution is obviously not necessarily locally of finite type. 

\begin{proof}The proof of the Theorem follows the same reasoning as the one of Hermann's, but in a more subtle way since Lobry's assumption is weaker that Hermann's. The proof hence relies on one important lemma which contains the same ideas as the proof of Hermann, but the proof is more subtle because one cannot rely on the fact that the distribution is locally finitely generated:
%

Let $x\in M$, $X\in\Gamma(\mathcal{D})$ and $\phi_{X}^{t}$ be the flow of the vector field $X$. We want to show that the rank of $\mathcal{D}_{\phi_X^t}$ is constant for sufficiently small $t$.
To show that there is an isomorphism between $\mathcal{D}_{x}$ and $\mathcal{D}_{\phi_{X}^{t}(x)}$, we will work entirely in $\mathcal{D}_{x}$ by transporting the vector field $X_{1},\ldots,X_{p}$ back along the integral curve of $X$. We have to do this because the distribution might not be locally finitely generated. Then, we will ensure that the rank of $\mathcal{D}_{\phi_X^t(x)}$ is constant on this integral curve. 

The vector field is defined in a neighborhood of $x$ so there exists $\epsilon>0$ such that the flow is well defined for every $|t|<\epsilon$. The distribution is locally of finite type, thus there exists a family of vector fields $X_{1},\ldots,X_{p}$ which span $\mathcal{D}$ in a neighborhood $\Omega$ of $x$, and they satisfy Equation \eqref{eq:locfintype}. Assuming that $\epsilon$ is smaller if necessary, let $V_{i}(t)$ the element of $T_{x}M$ defined by:
\begin{equation}
V_{i}(t)=\phi_{X\ast}^{-t}\Big(X_{i}(\phi_{X}^{t}(x))\Big)
\end{equation}
for every $|t|<\epsilon$. Then the map $t\mapsto V_{i}(t)$ defines a path in $T_{x}M$ and, by definition of the Lie bracket:
\begin{equation}\label{eq:julia}
\frac{\dd}{\dd t}V_{i}(t)=\phi_{X\ast}^{-t}\Big([X,X_{i}](\phi_{X}^{t}(x))\Big)
\end{equation}
Let $f_{i,j}$ be the family of smooth functions which satisfy the local involutivity condition \eqref{eq:locfintype}. Then Equation \eqref{eq:julia} turns into:
\begin{equation}\label{eq:julia2}
\frac{\dd}{\dd t}V_{i}(t)=\sum_{j=1}^{p}\ f_{i,j}\big(\phi_{X}^{t}(x)\big)V_{j}(t)
\end{equation}
for every $i=1,\ldots, p$. This is a set of $p$ first-order differential equations, which can be translated into matrix notation:
\begin{equation}
\frac{\dd}{\dd t}V(t)=V(t)\times F(t)
\end{equation}
where $V(t)$ is the $n\times p$ matrix whose columns are the vectors $V_{i}(t)$, and $F(t)$ is the $p\times p$ matrix collecting the functions $f_{i,j}$ at the point $\phi_{X}^{t}(x)$. The initial condition is $V(0)=(X_{1},\ldots,X_{p})$ and it is a rank $m\leq p$ matrix. The general solution is of the type $V(t)=V(0)\exp(\mathfrak{F}(t))$, where $\mathfrak{F}$ is a primitive of $F$. Then, the rank of the matrix $V(t)$ is $m$, for every $|t|<\epsilon$. Since $\phi_{X}^{t}$ is a local diffeomorphism, its pushforward is injective, then the vector fields:
\begin{equation}
X_i(\phi_{X}^{t})=\phi_{X\ast}^{t}\big(V(t)\big)
\end{equation}
form a family of rank $m$, for every $|t|<\epsilon$. Since we know that $X_1,\ldots, X_p$ span $\mathcal{D}$ in a neighborhood $\Omega$ of $x$, we deduce that $\mathcal{D}_{\phi_X^t(x)}$ is of rank $m$ on the integral curve $\phi_X^t$. In particular $\phi_{X\ast}^t(\mathcal{D}_x)=\mathcal{D}_{\phi_X^t(x)}$
As in Hermann's proof, by a compactness argument we conclude that, for any $t_0$ for which the flow of $X$ is defined, the ranks of $\mathcal{D}_{x}$ and $\mathcal{D}_{\phi_{X}^{t_{0}}(x)}$ are equal and that the two vector spaces are isomorphic.

At this point one can go back to the proof of Hermann's theorem, where it is shown that if the rank of the distribution is constant over an integral path, then it is constant over the preleaf $\mathcal{L}_{x}$. On the other hand, Lobry presented an enhancement of Hermann's proof by providing an explicit description of the coordinate charts on the preleaf. Let $X_{1},\ldots,X_{p}$ be a family of vector fields of $\Gamma(\mathcal{D})$ forming a basis of $\mathcal{D}_{x}$. Let the map $\phi_{x}$ be defined by:
\begin{align}
\phi_{x}:\hspace{0.6cm}\mathbb{R}^{p}\hspace{0.7cm}&\xrightarrow{\hspace*{2cm}} \hspace{1.1cm} M\\
(t_{1},\ldots,t_{p})&\xmapsto{\hspace*{2cm}} \phi_{X_{p}}^{t_{p}}\circ\ldots\circ\phi_{X_{1}}^{t_{1}}(x)\nonumber
\end{align}
for every $p$-tuple such that the right-hand side makes sense. We propose that the inverse map is a candidate for a coordinate chart based at $x$. Notice that at any point $y\in\mathcal{L}_{x}$ a corresponding map $\phi_{y}$ can be defined (possibly with a different number of vector fields). Thus we will prove the following facts:
\begin{itemize}
\item for every $y\in\mathcal{L}_{x}$, $\mathrm{Im}(\phi_{y})$ is an integral manifold of $\mathcal{D}$  of dimension $p$ contained in $\mathcal{L}_{x}$
\item the set of coordinate charts $(\phi_{y}^{-1})_{y\in\mathcal{L}_{x}}$ defines a smooth atlas for $\mathcal{L}_{x}$
\end{itemize}
so that we can finally conclude that $\mathcal{L}_{x}$ is an immersed submanifold of $M$ of dimension $p$. The end of the proof will then follow.

\begin{center}
\emph{For every $y\in\mathcal{L}_{x}$, $\mathrm{Im}(\phi_{y})$ is an integral manifold of $\mathcal{D}$ of dimension $p$ contained in $\mathcal{L}_{x}$}
\end{center}

We will show that $\phi_{x}$ defines an integral manifold of $\mathcal{D}$ through $x$, and then apply the same reasoning for any $y\in\mathcal{L}_{x}$. The smoothness of $\phi_{x}$ is automatic. Showing that $\mathrm{Im}(\phi_{x})$ is an immersed submanifold of $M$ is equivalent to showing that $\phi_{x}$ is an injective immersion. The first point is obvious because every map $\phi_{X_{i}}^{t_{i}}$ being bijective, the map $\phi_{x}$ is naturally injective in a sufficiently small neighborhood of $x$ (in any case, an immersion is locally injective). Then we will prove that $\phi_{x \ast}$ is an injection taking values in $\mathcal{D}$. Since the rank of $\mathcal{D}$ is constant over $\mathrm{Im}(\phi_{x})$ and equal to the dimension of $\mathbb{R}^{p}$, we will conclude that it is an integral manifold of $\mathcal{D}$ passing through $x$.

The map $\phi_{x}$ is smooth and its push-forward at the origin of $\mathbb{R}^{p}$ is:
\begin{equation}
\phi_{x\ast}\big(\frac{\partial}{\partial t_{i}}\big)=\left.\frac{\partial}{\partial t_{i}}\phi_{x}\right|_{(0,\ldots,0)}=X_{i}(x)
\end{equation}
then $\phi_{x\ast}:T_{(0,\ldots,0)}\mathbb{R}^{p}\to \mathcal{D}_{x}$ is of rank $p$, hence it is bijective at the origin. Now for any small $t_{1}$, we know that the vector spaces $\mathcal{D}_{x}$ and $\mathcal{D}_{\phi_{X_{1}}^{t_{1}}(x)}$ are isomorphic. In particular it means that $\mathcal{D}_{\phi_{X_{1}}^{t_{1}}(x)}$ is of rank $p$ with basis $\phi_{X_{1}\ast}^{t_{1}}\big(X_{i}(x)\big)$. And since at the point $(t_{1},0,\ldots,0)\in\mathbb{R}^{p}$ we have:
\begin{equation}
\phi_{x\ast}\big(\frac{\partial}{\partial t_{i}}\big)=\left.\frac{\partial}{\partial t_{i}}\phi_{x}\right|_{(t_{1},0,\ldots,0)}=\phi_{X_{1}\ast}^{t_{1}}\big(X_{i}(x)\big)
\end{equation}
we conclude that $\phi_{x\ast}:T_{(t_{1},0,\ldots,0)}\mathbb{R}^{p}\to \mathcal{D}_{\phi_{X_{1}}^{t_{1}}(x)}$ is bijective. We could have done the same for the other vector fields $X_{i}$, as well as for the concatenation of two flows. Reproducing the same argument for any $p$-tuple $(t_{1},\ldots,t_{p})$, we obtain that $\phi_{x}$ is an immersion taking values in $\mathcal{D}$, but the dimension of $\mathbb{R}^{p}$ and the rank of $\mathcal{D}_{x}$ are the same, then it is a bijection, hence $\mathrm{Im}(\phi_{x})$ is an integral manifold of $\mathcal{D}$.

Then we notice that for any $p$-tuple $(t_{1},\ldots,t_{p})$, the path $\gamma$ defined by applying first $\phi_{X_{1}}$ on $x$ from 0 to $t_{1}$, then $\phi_{X_{2}}$  on $\phi_{X_{1}}^{t_{1}}(x)$ from 0 to $t_{2}$, and so on...is actually an integral path of $\mathcal{D}$. Moreover, being true for any $p$-tuple $(t_1,\ldots,t_p)$ in a neighborhood of $0$ in $\mathbb{R}^{p}$, the image of $\phi_{x}$ is included in the integral leaf $\mathcal{L}_{x}$. We apply the same reasoning for any $y\in\mathcal{L}_{x}$ and we indeed obtain the first item:
\begin{equation*}
\mathrm{Im}(\phi_{y})\subset \mathcal{L}_{x}
\end{equation*}

Now, let $y\in \mathcal{L}_{x}$, and let $p(y)=\mathrm{dim}(\mathcal{D}_{y})$, obviously the dimension $p(y)$ of the vector space may depend of $y$. Let $X_{1},\ldots,X_{p(y)}\in\Gamma(\mathcal{D})$ forming a basis of $\mathcal{D}_{y}$, then there exists a map $\phi_{y}$ defined as:
\begin{align*}
\phi_{y}:\hspace{0.6cm}\mathbb{R}^{p(y)}\hspace{0.7cm}&\xrightarrow{\hspace*{2cm}} \hspace{1.1cm} M\\
(t_{1},\ldots,t_{p(y)})&\xmapsto{\hspace*{2cm}} \phi_{X_{p(y)}}^{t_{p(y)}}\circ\ldots\circ\phi_{X_{1}}^{t_{1}}(y)
\end{align*}
By the above discussion, we deduce that $\mathrm{Im}(\phi_{y})$ is an integral manifold of $\mathcal{D}$ passing through $y$. Moreover, the dimension $p(y)$ is constant on any integral path of $\mathcal{D}$ starting from $y$. And since any two points of $\mathcal{L}_{x}$ can be linked by such a path,  compactness argument implies that $p(y)=p$ (the rank of $\mathcal{D}_{x}$) on all of $\mathcal{L}_{x}$. Now let us turn to the second item:

\begin{center}
\emph{The set of coordinate charts $(\phi_{y}^{-1})_{y\in\mathcal{L}_{x}}$ defines a smooth atlas for $\mathcal{L}_{x}$}
\end{center}

We expect that this preleaf may be equipped with a differentiable structure such that it is an integral submanifold of $\mathcal{D}$ of dimension $p$. One may naturally take the set of maps $(\phi_{y}^{-1})_{y\in\mathcal{L}_{x}}$ as an atlas for $\mathcal{L}_{x}$.
We have to verify the compatibility of the coordinates functions $\phi_{y}$ and $\phi_{x}$, for any point $y\in\mathcal{L}_{x}$ close to $x$. In other words, we have to show that $\phi_{y}^{-1}\circ\phi_{x}:\mathbb{R}^{p}\to\mathbb{R}^{p}$ is a diffeomorphism. It is  actually sufficient to show that for any  open set $U\in\mathbb{R}^{p}$ such that $y\in\phi_{x}(U)$, the set $\phi_{y}^{-1}\circ\phi_{x}(U)$ is a neighborhood of 0, because we can always smoothly map it to $U$ again. To show that $\phi_{y}^{-1}\circ\phi_{x}(U)$ is open, we merely have to show that for any point belonging to this preimage, a neighborhood of this point is contained in it as well.

Let $y\in\mathrm{Im}(\phi_{x})$ and let $U\subset\mathbb{R}^{p}$ be an open set such that $\phi_{y}^{-1}\circ\phi_{x}(U)$. The goal is to show that every point in  $\phi_{y}^{-1}\circ\phi_{x}(U)$ has a neighborhood contained in $\phi_{y}^{-1}\circ\phi_{x}(U)$. The preimage of $y$ is the origin by definition, then let $T=(t_{1},\ldots,t_{p})$ be a point in $\phi_{y}^{-1}\circ\phi_{x}(U)$. Every point in the neighborhood of $T$ can be reached by a concatenation of paths of the form:
\begin{equation}
\gamma_{i}|_{(s_{1},\ldots,s_{p})}:t\longmapsto(s_{1},\ldots,s_{i}+t,\ldots,s_{p})
\end{equation}
for any $s_{i}$ in the vicinity of $t_{i}$, thus it is sufficient to show that every such piecewise smooth path is contained in $\phi_{y}^{-1}\circ\phi_{x}(U)$ to obtain the result. In the following, to simplify the discussion, we impose to act first on the first coordinate $t_{1}$, and then on the second coordinate $t_{2}$, and so on, that is, for any $p$-tuple $S=(s_{1},\ldots,s_{p})$, we can define a path from $T$ to $S$ by the following algorithm: from $T$, one first goes to $(s_{1},t_{2},\ldots,t_{p})$ on the path $\gamma_{1}$. Then one uses the path $\gamma_{2}$ to go from $(s_{1},t_{2},\ldots,t_{p})$ to $(s_{1},s_{2},t_{3},\ldots, t_{p})$, etc. until the last path $\gamma_{p}$ links the point of coordinates $(s_{1},s_{2},\ldots,s_{p-1},t_{p})$ to the final point $S$.

Assuming that $S$ is in the domain of definition of $\phi_{y}$, the image of the path $\gamma_{i}|_{S}$ by $\phi_{y}$ is an integral curve of some vector field $X_{i}|_S\in\Gamma(\mathcal{D})$ because $\mathrm{Im}(\phi_{y})$ is an integral manifold of $\mathcal{D}$. For the same reason, it follows that for any $S$ near $T$, if the image by $\phi_{y}$ of $\gamma_{i}|_S(t_{0})$ is contained in $\phi_{x}(U)$ for some $t_{0}$, then so does $\phi_{y}(\gamma_{i}|_S(t))$ for all $t$ in a neighborhood of $t_{0}$, and so does $\phi_{y}(\gamma_{j}|_{\gamma_{i}|_{S}(t_{0})}(t))$ for $j\neq i$ and small $t$. The same argument is true for longer compositions of the paths $\gamma_{k}$. Since any point $S=(s_{1},\ldots,s_{p})$ sufficiently near $T$ can be reached by a concatenation of the $\gamma_{i}$'s,   the above arguments imply that the image $\phi_{y}(s_{1},\ldots,s_{p})$ belongs to $\phi_{x}(U)$. Hence $T$ has a neighborhood contained in $\phi_{y}^{-1}\circ\phi_{x}(U)$, which is then an open set of $\mathbb{R}^{p}$.

Hence it means that $\phi_{y}$ and $\phi_{x}$ are smoothly compatible and more generally so are two coordinate charts on $\mathcal{L}_{x}$. They form together a smooth atlas and turn $\mathcal{L}_{x}$ into an immersed submanifold of $M$ of dimension $p$. Following the discussion of the first item, its tangent bundle coincides with the distribution $\mathcal{D}$, that is, $\mathcal{L}_{x}$ is an integral manifold of $\mathcal{D}$ passing through $x$. We conclude the proof of the Theorem by noticing that the family $\mathcal{L}=\{\mathcal{L}_{x}\,|\,x\in M\}$ forms a partition of $M$, thus it is indeed a foliation of $M$.\end{proof}

\begin{remarque}
The vector fields $X_{1},\ldots,X_{p}$ span the distribution $\mathcal{D}$ pointwise in a neighborhood of $x$ in $M$, but that does not mean that they span the space of sections $\Gamma(\mathcal{D})$. It means that $\mathcal{D}$ is not necessarily locally finitely generated, as in Hermann's proposition. This is precisely here that the two approaches differ.
\end{remarque}


\subsection{The results of Stefan and Sussmann}\label{StefanSussmann}

Even if various results in foliation theory were obtained in the sixties, there was no precise theorem stating a necessary and sufficient condition for integrating smooth distributions into singular foliations, as was the theorem of Frobenius for regular ones. Step by step however, the assumptions have been weakened, until finally one is left with the minimal assumptions allowing a smooth singular distribution to be integrable. Interestingly, they do not involve neither involutivity nor finiteness, but are a refinement of both, and summon techniques which come directly from topics in control theory. The simultaneous work of P. Stefan in his PhD thesis \cites{Stefanphd, Stefanofficiel}, and of H. Sussmann \cites{Sussmanntrailer, Sussmannofficiel} provided a profound understanding of sufficient and necessary conditions for a smooth distribution to be integrable.

These conditions rely on a weakening of Lobry's hypothesis, in the sense that Condition \ref{eq:locfintype} of 'partial involutivity' which is valid on a neighborhood of a point  $x\in M$, is now taken to be valid only on integral curve passing through the point $x$. Indeed, we observe that the proof of Lobry's theorem does not completely rely on the fact that the distribution is locally of finite type: it rather requires the fact that this condition is satisfied only on the integral curves of chosen sections of $\mathcal{D}$. In other words, given a point $x$ and some vector fields $X_{1},\ldots,X_{p}\in\Gamma(\mathcal{D})$ spanning $\mathcal{D}$ on a neighborhood of $x$, Lobry's condition implied that the $\cinf(M)$-module generated by those vector fields is stable under the Lie bracket with any vector field $X\in\Gamma(\mathcal{D})$, in some neighborhood of the point $x$ depending on the section $X$. But it seems that the adequate condition is to require this module to be stable under the bracket operation on the integral curve of $X$: so instead of having stability on a ball around $x$, we only need stability on any integral path passing through $x$. A refinement of Lobry's conditions have been found by Stefan in \cite{Stefanofficiel}. The condition presented by Sussmann in his seminal paper \cite{Sussmannofficiel} is not sufficient, because it was inspired by Lobry's original idea which was flawed.

 More precisely we say that the distribution $\mathcal{D}$ satisfies the \emph{Stefan-Sussmann conditions} if for every point $x\in M$ and every $X\in \Gamma(\mathcal{D})$ which is defined in $x$, there exist vector fields $X_{1},\ldots,X_{p}\in\Gamma(\mathcal{D})$, some $\epsilon>0$, and functions $f_{i,j}\in\mathcal{C}^{\infty}\big(]-\epsilon,\epsilon[\big)$ such that for every $t\in ]-\epsilon,\epsilon[$,
 $X_{1}\big(\phi_X^t(x)\big),\ldots,X_{p}\big(\phi_X^t(x)\big)$ span $\mathcal{D}_{\big(\phi_X^t(x)\big)}$
and \begin{equation}\label{eq:locfintype2}
[X,X_{i}]\big(\phi^{t}_{X}(x)\big)=\sum_{j=1}^{p}\,f_{i,j}(t)X_{j}\big(\phi^{t}_{X}(x)\big)
\end{equation}
for every $1\leq i\leq p$.
These are the conditions that both P. Stefan and H. Sussmann found to be necessary and sufficient for a smooth distribution to be integrable. Their well-known theorem answers a century old question \cites{Stefanofficiel,Sussmannofficiel}:

\begin{theoreme}\emph{\textbf{Stefan-Sussmann (1973)}}\label{theo:stefsuss}
Let $\mathcal{D}$ be as smooth distribution on a smooth manifold M. If $\mathcal{D}$ satisfies the Stefan-Sussmann conditions, then $\mathcal{D}$ is integrable.
\end{theoreme}

\begin{proof}
This is essentially the proof by Lobry, restricted to integral curves. As always the idea is that the integral manifold of $\mathcal{D}$ passing through a point $x$ is the preleaf $\mathcal{L}_{x}$. 
\end{proof}


Originally, inspired by geometric control theory, Stefan and Sussmann did not rely on the sections of the distribution, but rather worked with a \emph{generating set of vector fields}: that is, a subset $\mathfrak{V}\subset\mathfrak{X}(M)$ (and not necessarily a submodule) which spans the distribution at each point, and in particular $\mathfrak{V}\subset\Gamma(\mathcal{D})$. In that case one cannot reverse the direction of the flows, as we would have been allowed to do in Hermann's or Lobry's formulation of the problem. As a side remark, notice that a recent result \cite{Texas} shows that even if the generating set $\mathfrak{V}$ has infinite cardinal, we can always choose a finite subset which still spans $\mathcal{D}_{x}$ at every point $x$.

To start with, we need a refinement of integral paths defined in Section \ref{Hermann}: we call an \emph{orbital path of the generating set $\mathfrak{V}$} any piecewise smooth path $\gamma:[a,b]\to M$ such that for every open interval $I\subset[a,b]$ where $\gamma$ is differentiable, there exists a vector field $X\in\mathfrak{V}$ satisfying:
\begin{equation}
\frac{\dd}{\dd t}\gamma (t)= X\big(\gamma(t)\big)
\end{equation}
Two points $x$ and $y$ which are linked by an orbital path are said to be on the same \emph{orbit} of $\mathfrak{V}$. We denote by $\mathcal{O}_{x}$ the orbit of $\mathfrak{V}$ passing through the point $x$. Orbits form equivalence classes of points, and as such they form a partition of $M$. Obviously when $\mathfrak{V}=\Gamma(\mathcal{D})$, orbital paths are  the integral paths of the distribution $\mathcal{D}$, so that the orbits of the subset $\Gamma(\mathcal{D})$ are the preleaves of the distribution $\mathcal{D}$. We want to show that the orbits of $\mathfrak{V}$ are the integral manifolds of the distribution $\mathcal{D}$.
As a historical note, P. Stefan used to call the orbits the \emph{accessible sets} of the subset of vector fields $\mathfrak{V}$.

Thus, let us proceed to some analysis in order to find the condition for a smooth distribution $\mathcal{D}$ to be integrable.
If one considers a generating set $\mathfrak{V}$ that spans the distribution at every point of $M$, we would like to find the conditions under which the tangent spaces of the orbits of $\mathfrak{V}$ coincide with the distribution at every point. If this is the case, for any $x\in M$, the orbit $\mathcal{O}_{x}$ of $\mathfrak{V}$ through $x$ would be an integral manifold of $\mathcal{D}$ of the same dimension as the preleaf $\mathcal{L}_{x}$, hence they would coincide. Let $x\in M$ and $X$ be an element of $\mathfrak{V}$, then the integral curve $t\mapsto\phi_{X}^{t}(x)$ defines a tangent vector at $x$ that belongs to $\mathcal{D}_{x}$. Now if we look for all the orbital paths through $x$, their tangent vectors at $x$ should define elements of $\mathcal{D}_{x}$. Given another element $Y$ of $\mathfrak{V}$, for small $t$ the path
\begin{equation}
t\longmapsto\phi_{X}^{t}\circ\phi_{Y}^{t}\circ\phi_{X}^{-t}\circ\phi_{Y}^{-t}(x)
\end{equation}
is an orbital curve, hence at $t=0$ the tangent vector $[X,Y](x)$ should belong to $\mathcal{D}_{x}$ for the orbit to be an integral manifold of $\mathcal{D}$. Any other higher bracket should also belong to $\mathcal{D}_{x}$, hence the intuition that the distribution is involutive at the point $x$ is natural.

It is not sufficient though because $\mathcal{D}_{x}$ may involve more directions that have to be taken into account to guarantee that the distribution is integrable. Let $y=\phi_{X}^{T}(x)$ be a point located on the integral curve of $X$, and let $Z$ be an element of $\mathfrak{V}$ defined on a neighborhood of $y$. Then the curve $t\mapsto\phi_{X}^{-T}\big(\phi_{Z}^{t}(y)\big)$ is an orbital path passing through $x$. The condition that the orbit $\mathcal{O}_{x}$ is a submanifold of $M$, and the fact that $X$ is tangent to it, imply that the pushforward $\phi_{X\ast}^{-T}$ sends $T_{y}\mathcal{O}_{x}$ to $T_{x}\mathcal{O}_{x}$. Hence if $\mathcal{D}$ is integrable and if the orbits are the integral submanifolds of $\mathcal{D}$, then the tangent vector $\phi_{X\ast}^{-T}\big(Z(y)\big)$ should necessarily belong to $\mathcal{D}_{x}$. That is to say, the distribution $\mathcal{D}$ should be invariant with respect to the action of the flows of elements of $\mathfrak{V}$. More precisely we say that a distribution is \emph{invariant under a group $G$ of local diffeomorphisms} if for any $g\in G$ the push forward $g_{\ast}$ sends $\mathcal{D}_{x}$ into $\mathcal{D}_{g(x)}$. In other words, to be integrable, the distribution at the point $x$ should not exclude directions that are pushed forward from any point in the neighborhood, as can be shown in the next example:

\begin{example}\label{example13}
We already have an example of a smooth distribution which is not integrable: the one given in example \ref{ex:smoothdist}. This happens because even though it is involutive, it is not invariant under the action of itself: the vector field $-\frac{\partial}{\partial x}$ is a (constant) section of the distribution, hence one can push forward the distribution $\mathcal{D}_{(x,y)}=T_{(x,y)}\mathbb{R}^{2}$ (for $x>0$) to negative $x$, but the image would still be isomorphic to $\mathbb{R}^{2}$, however the distribution for negative $x$ is only one dimensional. Hence it is not invariant under itself, hence it is not integrable.
\end{example}


At each point $x$ the submodule $\mathfrak{V}$ defines a subspace $\mathcal{D}_{x}$ of $T_{x}M$, and H. Sussmann has shown that the induced distribution $\mathcal{D}$ is integrable if $\mathcal{D}$ is stable under the action of the flows of elements of $\mathfrak{V}$ (and $-\mathfrak{V}$). 
Let us first fix the notations: given a smooth distribution $\mathcal{D}$ and a subset $\mathfrak{V}\subset\mathfrak{X}(M)$ (and not necessarily a $\cinf(M)$-submodule), we say that \emph{$\mathcal{D}$ is $\mathfrak{V}$-invariant} if it is invariant under the group of flows of both elements of $\mathfrak{V}$ and elements of $-\mathfrak{V}$. In other words it means that given a vector field $X\in\mathfrak{V}$, the distribution $\mathcal{D}$ should be invariant under $\phi_{X\ast}^{t}$ and $\phi_{-X\ast}^{t}$, for any $t$ for which the flow is defined. We now state Sussmann's result:

\begin{proposition}
Let $\mathcal{D}$ be a smooth distribution on a smooth manifold $M$ and let $\mathfrak{V}$ be a subset of $\mathfrak{X}(M)$ which spans $\mathcal{D}$ at every point. Then the distribution $\mathcal{D}$ is integrable if and only if it is $\mathfrak{V}$-invariant.
\end{proposition}

\begin{remarque}
Notice that the proposition is more general than just assuming that $\mathcal{D}$ is $\Gamma(\mathcal{D})$-invariant. 
\end{remarque}

\begin{proof}
First assume that $\mathcal{D}$ is integrable, then its sections satisfy the Stefan-Sussmann conditions. These conditions are sufficient to prove that the distribution $\mathcal{D}$ is invariant under the flow of any section $X\in\Gamma(\mathcal{D})$, that is: it is $\Gamma(\mathcal{D})$-invariant. Since $\mathfrak{V}\subset\Gamma(\mathcal{D})$, it is indeed $\mathfrak{V}$-invariant.

Conversely assume that the distribution $\mathcal{D}$ is $\mathfrak{V}$-invariant. We first show that the orbits of $\mathfrak{V}$ are submanifolds of $M$, and that these submanifolds are the integral submanifolds of $\mathcal{D}$. We can equip the orbits of $\mathfrak{V}$ with a natural topology: given an orbit $\mathcal{O}$ of $\mathfrak{V}$, for any given $\xi=(X_{1},\ldots,X_{n})\in\mathfrak{V}^{n}$ and  $x\in\mathcal{O}$, we define the map $\phi_{x\xi}$ by:
\begin{align}
\phi_{x,\xi}:\hspace{0.6cm}\mathbb{R}^{n}\hspace{0.6cm}&\xrightarrow{\hspace*{2cm}} \hspace{1.1cm} M\\
(t_{1},\ldots,t_{n})&\xmapsto{\hspace*{2cm}} \phi_{X_{p}}^{t_{p}}\circ\ldots\circ\phi_{X_{1}}^{t_{1}}(x)\nonumber
\end{align}
We introduced the same map in the proof of Lobry's Theorem, except that the vector fields here are restricted to be elements of $\mathfrak{V}$. The orbit $\mathcal{O}$ consists of the images of all such maps (with fixed $x$, for all $n$ and all $\xi\in\mathfrak{V}^{n}$). We equip the orbit $\mathcal{O}$ with the strongest topology such that these maps are continuous. One can show \cite{Sussmannofficiel} that the topology does not depend on the choice of a point $x\in\mathcal{O}$, and that the subspace topology is contained in this topology. Then the orbit $\mathcal{O}$ is a (connected) topological subspace of $M$. In our case, for every $x\in\mathcal{O}$ the maps $\{\phi_{x,\xi}\}_{\xi\in\mathfrak{V}^{n}}$ define the orbit $\mathcal{O}$ but they are not considered as charts for $\mathcal{O}$ because the dimension $n$ varies. We will build the coordinate charts and the smooth atlas at the end of the proof.

The goal of the proof is to show that the tangent space of the orbit $\mathcal{O}$ at each point $y$ is precisely $\mathcal{D}_{y}$. Recall that for an immersed submanifold $S\xhookrightarrow{\rho}M$ (where $\rho$ is an injective immersion), the tangent space of $S$ in $M$ at the point $\rho(x)$ is denoted by $T_{\rho(x)}S$ and by definition it is identified with the image of $T_{x}S$ by $\rho_{\ast}$:
\begin{equation*}
T_{\rho(x)}S=\rho_{\ast}(T_{x}S)
\end{equation*}
Since we do not have a proper definition of charts for the orbit $\mathcal{O}$, we will first deduce some result on the relationships between the maps $\phi_{x,\xi}$ and the distribution.


Let $x\in\mathcal{O}$, and let $\tau=(\tau_{1},\ldots,\tau_{n})$ be an $n$-tuple in the neighborhood of $0\in\mathbb{R}^{n}$. We define $V_{x;\xi,\tau}$ to be the subset:
\begin{equation*}
V_{x;\xi,\tau}=(\phi_{x,\xi})_{\ast}T_{\tau}\mathbb{R}^{n}\subset T_{\phi_{x,\xi}(\tau)}M
\end{equation*}
We prove the assertion that $V_{x;\xi,\tau}\subset\mathcal{D}_{\phi_{x,\xi}(\tau)}$ by induction: for $n=1$ we know that given a vector field $X\in\mathfrak{V}$, the tangent vector at $t=t_{0}$ of the integral curve $t\mapsto\phi_{X}^{t}(x)$ is $X(\phi_{X}^{t_{0}}(x))$ and is then contained in $\mathcal{D}_{\phi_{X}^{t_{0}}(x)}$ because $\mathfrak{V}\subset\Gamma(\mathcal{D})$. Now for any $n\geq2$, take 
$\xi\in\mathfrak{V}^{n}$ and $\tau\in\mathbb{R}^{n}$ such that the point ${\phi_{x,\xi}(\tau)}$ is well defined. Write $\xi=(\eta,X_{p})$ for some element $X_{p}\in\mathfrak{V}$ and $\eta\in\mathfrak{V}^{n-1}$. Moreover write $\tau=(\sigma,t_{p})$ for some time $t_{p}$ and $\sigma\in\mathbb{R}^{n-1}$. Assume that the property holds at the point $\phi_{x,\eta}(\sigma)$, that is: $V_{x;\eta,\sigma}\subset\mathcal{D}_{\phi_{x,\eta}(\sigma)}$. The space $V_{x;\xi,\tau}\subset T_{z}M$ at the point $z=\phi_{x,\xi}(\tau)$ is generated by $X_{p}(z)$ and by the push forward of $V_{x;\eta,\sigma}$ by the flow of $X_{p}$. But we know by hypothesis that the distribution is $\mathfrak{V}$-invariant, hence $(\phi_{X_{p}}^{t_{p}})_{\ast}V_{x;\eta,\sigma}\subset\mathcal{D}_{z}$. And since $X_{p}(z)\in\mathcal{D}_{z}$ as well, we obtain that:
\begin{equation*}\label{reverseinclusion}
V_{x;\xi,\tau}\subset\mathcal{D}_{\phi_{x,\xi}(\tau)}
\end{equation*}
We conclude that the tangent space of the orbit at every $x\in M$ is included into the distribution $\mathcal{D}_{x}$.

We now would like to prove the reverse claim, that is: for any $x\in\mathcal{O}$, there exist $z\in\mathcal{O}$, $\xi\in\mathfrak{V}^{n}$ (for some $n\geq 1$) and $\tau\in\mathbb{R}^{n}$ such that $x=\phi_{z,\xi}(\tau)$ and $\mathcal{D}_{x}\subset V_{z;\xi,\tau}$. Let $\mathcal{A}_{x}$ be the subset of $\mathcal{D}_{x}$ which consists of the image at $x$ of any vector field of $\mathfrak{V}$ through the group of local diffeomorphism generated by the flows of elements of $\mathfrak{V}$. In other words, for any vector $X\in\mathcal{A}_{x}$, then there exist $z\in\mathcal{O}$, $Z\in T_{z}\mathcal{O}$, and also $\xi\in\mathfrak{V}^{n}$(for some $n\geq 1$) and $\tau\in\mathbb{R}^{n}$ such that:
\begin{equation*}
x=\phi_{z,\xi}(\tau)\hspace{1.2cm}\text{and}\hspace{1.2cm}X=\big(\phi_{z,\xi}(\tau)\big)_{\ast}(Z)
\end{equation*}
The subset $\mathcal{A}_{x}$ necessarily spans $\mathcal{D}_{x}$, for it contains the elements of $\mathfrak{V}$ at $x$, which span $\mathcal{D}_{x}$ by hypothesis.
Next, we let $\eta=(Z,\xi)\in\mathfrak{V}^{n+1}$ and $\sigma=(0,\tau)\in\mathbb{R}^{n+1}$ so that necessarily:
\begin{equation}
x=\phi_{z,\xi}(\tau)=\phi_{X_{p}}^{t_{p}}\circ\ldots\circ\phi_{X_{1}}^{t_{1}}(z)=\phi_{X_{p}}^{t_{p}}\circ\ldots\circ\phi_{X_{1}}^{t_{1}}\circ\phi^{0}_{Z}(z)=\phi_{z,\eta}(\sigma)
\end{equation}
Moreover, we observe that $X=\big(\phi_{z,\xi}(\tau)\big)_{\ast}(Z)\in(\phi_{z,\eta})_{\ast}T_{\sigma}\mathbb{R}^{n+1}$. However this result is not sufficient to show that the whole subset $\mathcal{A}_{x}$ (and hence all of $\mathcal{D}_{x}$) is contained in some $V_{z';\eta',\tau'}$, because we have treated only one vector field.

Now assume that there exist two orbital paths ending at the point $x$, that is: $x=\phi_{z,\xi}(\tau)$ and $x=\phi_{z',\xi'}(\tau')$ for some points $z,z'$ and some tuples $\tau,\tau'$ and $\xi,\xi'$. We define $\tau^{-1}=(t_{n},\ldots,t_{1})$ and $\xi^{-1}=(X_{n},\ldots,X_{1})$, and we set $\eta=(\xi,\xi^{-1},\xi')$ and $\sigma=(\tau,-\tau^{-1},\tau')$. Then one can check that:
\begin{equation}
\phi_{z',\eta}(\sigma)=\phi_{\phi_{z',\xi'}(\tau'),(\xi,\xi^{-1})}(\tau,-\tau^{-1})=\phi_{z',\xi'}(\tau')=x
\end{equation}
because $x\mapsto\phi_{x,(\xi,\xi^{-1})}(\tau,-\tau^{-1})$ is the identity. Hence the image of $T_{\sigma}\mathbb{R}^{2n+n'}$ by the pushforward $(\phi_{z',\eta})_{\ast}$ naturally contains $V_{z';\xi',\tau'}$, and also $V_{z;\xi,\tau}$ by definition of $\sigma$ and $\eta$. Thus, if $X,X'\in\mathcal{A}_{x}$ then $X$ is contained in some $V_{z;\eta,\tau}$ and $X'$ is contained in some $V_{z';\eta',\tau'}$, such that $x=\phi_{z,\xi}(\tau)=\phi_{z',\xi'}(\tau')$. But we have just shown that it means that both $X$ and $X'$ are contained in some $V_{y;\eta,\sigma}$ (for $y=z$ or $y=z'$). Hence applying the same process to other vectors of $\mathcal{A}_{x}$, and since we only need a finite number of elements of $\mathcal{A}_{x}$ to span the distribution $\mathcal{D}_{x}$, we conclude that $\mathcal{D}_{x}$ is contained in some $V_{m;\theta,\pi}$. Since we have shown the reverse inclusion $V_{m;\theta,\pi}\subset\mathcal{D}_{\phi_{m,\theta}(\pi)}=\mathcal{D}_{x}$, we conclude that for every $x\in\mathcal{O}$, there exist $z\in\mathcal{O}$, $\xi\in\mathfrak{V}^{n}$ and $\tau\in\mathbb{R}^{n}$ (for some $n\geq 1$) such that $\mathcal{D}_{x}=V_{z;\xi,\tau}$.

To finish the proof, we have to define some smooth atlas for $\mathcal{O}$, and then show that this orbit is an integral manifold of $\mathcal{D}$, that is: to show that $T_{y}\mathcal{O}=\mathcal{D}_{y}$ for every $y\in\mathcal{O}$. Let $x\in\mathcal{O}$, then we know that there exist $z\in\mathcal{O}$, $\xi\in\mathfrak{V}^{n}$ and $\tau\in\mathbb{R}^{n}$ such that
\begin{equation*}
x=\phi_{z,\xi}(\tau)\hspace{1.2cm}\text{and}\hspace{1.2cm}\mathcal{D}_{x}=V_{z;\xi,\tau}
\end{equation*}
Let $k$ be the dimension of $\mathcal{D}_{x}$, and recall that $V_{x;\xi,\tau'}\subset\mathcal{D}_{\phi_{z,\xi}(\tau')}$ for every $n$-tuple $\tau'$ in the neighborhood of $\tau$. Acknowledging that the rank of the distribution $\mathcal{D}$ is constant on the orbit $\mathcal{O}$ because $\mathcal{D}$ is $\mathfrak{V}$-invariant, we deduce that the rank of the map $\phi_{z,\xi}$ cannot exceed $k$ in a neighborhood of $\tau$. Since the map $(\phi_{z,\xi})_{\ast}$ is upper semi-continuous, its rank cannot decrease, thus it means that locally around $\tau$, the rank of $\phi_{z,\xi}$ is constant equal to $k$. Then, by the Rank Theorem, we deduce that its image that is denoted by $N$ is an embedded submanifold of $M$, locally parametrized by $k$ coordinates. The tangent space of $N$ at any point $\phi_{z,\xi}(\tau')$ close to $x$ (hence for $\tau'$ in a neighborhood of $\tau$) is naturally $V_{z;\xi,\tau'}$. Since it is included in $\mathcal{D}_{\phi_{z,\xi}(\tau')}$ and since they have the same dimension, they coincide, and we deduce that $N$ is an integral manifold of $\mathcal{D}$ passing through $x$.

There remains to show that $N$ is in fact an open subset of the orbit $\mathcal{O}$, and that all submanifolds constructed in this way cover $\mathcal{O}$, so that we can conclude that they form a smooth atlas for the orbit $\mathcal{O}$.
Let $x\in N$, and let $X_{1},\ldots,X_{n}$ be elements of $\mathcal{A}_{x}$ defining a base of $\mathcal{D}_{x}$. Since these tangent vectors can be smoothly extended to vector fields that are locally free taking values in $\mathcal{D}$, and since $D_{y}=T_{y}N$ for every $y$ in a neighborhood of $x$ (intersected with the submanifold $N$), the map $\phi_{x,\xi}$ induces a diffeomorphism of $0\in\mathbb{R}^{n}$ onto this neighborhood of $x$, where $\xi=(X_{1},\ldots,X_{n})$. By construction of the map $\phi_{x,\xi}$, every point in the image belongs to the same orbit as $x$. This property being true for every point of the submanifold $N$ (and $N$ being connected), we deduce that $N$ is contained in a simple orbit of $\mathfrak{V}$. 

From the preceding result, if there is an orbit $\mathcal{O}$ intersecting $N$, then $N\subset \mathcal{O}$. Let us now show that $N$ is open in $\mathcal{O}$, with respect to the topology of $\mathcal{O}$. That is, given $z\in\mathcal{O}$ and $\xi\in\mathfrak{V}^{n}$, $(\phi_{z,\xi})^{-1}(N)$ should be an open set of $\mathbb{R}^{n}$. Let $\tau$ be an element of the preimage. We want to show that there is a neighborhood $U$ of this point which is contained in $(\phi_{z,\xi})^{-1}(N)$; this is equivalent to saying that $\phi_{z,\xi}(U)\subset N$. Let $\tau=(t_{1,\ldots,t_{n}})\in U$ and let $x=\phi_{z,\xi}(\tau)$. For every $1\leq i\leq n$, let $\gamma_{i}$ be the path defined as:
\begin{equation}
\gamma_{i}:s\longmapsto(t_{1},\ldots,t_{i}+s,\ldots,t_{n})
\end{equation}
It is sent by $\phi_{z,\xi}$ to a smooth path in $\mathcal{O}$ passing through $x$. We have to show that it is contained in $N$. Let $X_{s}$ be he tangent vector to the path $\gamma_{i}$ at the point $\gamma_{i}(s)$. It is sent by $(\phi_{z,\xi})_{\ast}$ to the tangent vector $Y_{s}$ of the path $s\mapsto \phi_{z,\xi}\big(\gamma_{i}(s)\big)$ at the corresponding point. Since $(\phi_{z,\xi})_{\ast}(T_{\gamma_{i}(s)}\mathbb{R}^{n})\subset \mathcal{D}_{\phi_{z,\xi}(\gamma_{i}(s))}$, the tangent vector $Y_{s}$ belongs to $T_{\phi_{z,\xi}(\gamma_{i}(s))}N$ for every $s$ in a neighborhood of $0\in\mathbb{R}$. It can be extended to a smooth vector field $Y$ on $N$. The path $s\mapsto \phi_{z,\xi}\big(\gamma_{i}(s)\big)$ is then an integral curve of the vector field $Y$ passing through
$x$. Then it belongs to $N$ for small enough $s$. This result being true for any other path $\gamma_{j}$ with $1\leq j\leq n$, we deduce that there exists a small neighborhood of $\tau$ in $\mathbb{R}^{n}$ which is sent to $N$ by $\phi_{z,\xi}$. Hence $N$ is an open subset of $\mathcal{O}$.

For any point $x\in\mathcal{O}$, the above discussion shows that we can construct an (embedded) integral manifold $N$ passing through $x$. Let $\mathfrak{N}$ be the set of all such manifolds. They are open in $\mathcal{O}$ and they all have the same dimension, because the rank of the distribution $\mathcal{D}$ is constant on $\mathcal{O}$. Thus $\mathfrak{N}$ forms an open cover of the orbit, and we want to promote it to be a smooth atlas for $\mathcal{O}$. But since any two elements $N_{1}$ and $N_{2}$ of $\mathfrak{N}$ are embedded submanifolds of $M$ such that $N_{1}\cap N_{2}\neq\varnothing$, their transition functions are induced by the transition functions between the charts on $M$, and as such they are diffeomorphisms. Hence $\mathfrak{N}$ is a smooth atlas for the orbit $\mathcal{O}$, which turns it into an immersed submanifold of $M$, whose tangent spaces coincide with the distribution $\mathcal{D}$ at every point. Thus we have proven that any orbit $\mathcal{O}$ is an integral manifold of $\mathcal{D}$, and since they form a partition of $M$, it means that the distribution $\mathcal{D}$ is integrable, hence concluding the proof.
\end{proof}

\begin{remarque}
The result of Sussmann is strong because the sufficient condition is that the distribution is invariant under the action of a generating set $\mathfrak{V}$, and we do not need to know more or test the invariance under any other section of $\mathcal{D}$. In practical cases though, this proposition is mainly used to show that a distribution is not integrable: one only has to find a generating family of vector fields, which does not leave the distribution invariant, such as in Example \ref{example13}.
\end{remarque}

\nocite{Bullo}
\nocite{Lee}
\nocite{Tolley}
\nocite{Hawkins}

\chapter{The universal Lie \texorpdfstring{$\infty$}{infinity}-algebroid of a singular foliation}

\section{Introduction, main definitions and main results}

\subsection{Purpose}

Regular foliations, \emph{i.e.} integrable distributions of constant rank, are familiar objects of differential geometry \cite{Hector}, 
and the role of Lie algebroids and groupoids in their study is now well-understood \cite{MoeMrc}. Hermann foliations are much 
less studied, but the pioneering works of Stefan and Sussmann \cites{Stefanofficiel,Sussmannofficiel} is now being 
continued by several authors, including Androulidakis, Debord, Skandalis and Zambon \cites{AndrouSkandal,AndrouZamb,Debord}, leading to a new interest in those. 
Also, Sinan Sert\"oz \cites{Sertoz1, Sertoz2}
in the holomorphic case gave several interesting results about their local structures that seems not to have been explored further.

Hermann foliations have been introduced and studied in Section \ref{Hermann}.
There are abundant examples of Hermann foliations: orbits of group actions, vector fields vanishing at given points with given orders, symplectic leaves of a Poisson manifold,
leaves of source-connected Lie groupoids, to cite a few, and also polynomial or analytic  infinitesimal symmetries of an object.

Hermann foliations were explored by several authors after the pioneering works of P. Stefan and H. Sussman \cites{Stefanofficiel,Sussmannofficiel}. C. Debord \cite{Debord} studied
in detail the case where a Hermann foliation arises from a Lie algebroid whose anchor is injective on a dense open subset. More recently, the geometric understanding of Hermann foliations was improved by I. Androulidakis, G. Skandalis and M. Zambon, who 
constructed a \emph{holonomy groupoid} that encodes some geometrical features of the foliation \cites{AndrouSkandal,AndrouZamb}.

At the same time, in several distinct communities of both mathematics and theoretical physics, an increasing interest for the so-called `higher structures' or derived objects'
has emerged. In particular, there has been an intense focus on an object that has been given different names by different communities: the Lie  $\infty$-algebroid, or $Q$-manifold (better known as differential graded manifold), see Section \ref{Qmanifolds}.

Of these, we feel that the language of  Lie $\infty$-algebroids, although quite  complicated, is conceptually more elementary, and easier to grasp. 
Since we wish to be understandable for the general public of differential geometers and theoretical physicists, we state our theorems in terms of Lie $\infty$-algebroids. But we think and prove results with in $Q$-manifold language, in particular when we deal with morphisms of Lie $\infty$-algebroids. All the Lie $\infty$-algebroids that will appear in the present thesis are negatively graded, i.e. are defined on a negatively graded vector space, and, similarly, all $Q$-manifolds are positively graded, i.e. are what is referred to 
as $NQ$-manifolds as explained in Section \ref{Qmanifolds} or in \cite{BQZ}.

\nocite{AKSZ}
 
The purpose of the present thesis is to give a precise meaning to the idea that \emph{a Hermann foliation on a manifold admits (often) a Lie $\infty$-algebroid structure $-$
unique up to homotopy $-$ resolving it.} 
In the previous sentence, `often' includes the (locally) real analytic and holomorphic Hermann foliations around a point, 
but some smooth foliations do not satisfy the requirements (see Example \ref{tu}). In addition to being unique up to homotopy, it is universal $-$ in the sense that it is a terminal object in the category of Lie $\infty$-algebroids associated to a given Hermann foliation $-$- see Theorem \ref{theo:onlyOne} and the following corollaries.  Moreover, we claim that the \emph{cohomologies of this Lie $\infty$-algebroid
structure have a  geometrical meaning}, 
for instance when dealing with holonomies along leaves. We intend in particular to relate our construction with the recent works of I. Adroulidakis, G. Skandalis and M. Zambon \cites{AndrouSkandal,AndrouSkandal2,AndrouZamb}, and to derive 
several cohomological objects that are relevant to the study of Hermann foliations, although their complete geometrical meaning is still to be understood. 

%

\nocite{Debord2}
\nocite{AndrouSkandal3}

Let $M$ be a manifold that may be smooth, real analytic or complex. 
Denote by $ {\mathscr O}( U)$, with ${ U} \subset M$ an open subset, the algebra of smooth, 
real analytic or holomorphic functions over $ U$ and by ${\mathfrak X}( U) $ the ${\mathscr O}( U) $-module
of vector fields over $ U$. The assignment $ {\mathscr O}:{ U} \mapsto  {\mathscr O}( U)$ is  a sheaf of algebras while 
the assignment ${\mathfrak X}:{ U} \mapsto  {\mathfrak X}( U)$ is a sheaf of Lie algebras, as well as a module over the algebra 
sheaf $ {\mathscr O}$.

Hermann foliations are generally defined in the smooth category. It is obvious, however, that real analytic or holomorphic Hermann foliations induce
smooth Hermann foliations, so that the results that will describe still hold true in their respective categories. Then we generalize Definition \ref{hermannfoliation} to include both the real analytic and the holomorphic cases, that is:

\begin{definition}
A \emph{Hermann foliation} is a sub-sheaf ${\mathcal D}: U \mapsto  {\mathcal D}(U)$ of ${\mathfrak X}$, which is locally finitely generated as an ${\mathscr O}$-submodule and closed under the Lie bracket of vector fields.
\end{definition}

By  `locally finitely generated' we mean that every $x \in M$ admits a neighborhood $ U$ and $k$ sections $X_1, \dots,X_k\in{\mathcal D}( U)$ such that, 
for every $ V \subset  U$, the ${\mathscr O}( V)$-module ${\mathcal D}( V)$ is generated by the restrictions to $ V$ of $X_1, \dots,X_k$.
We call such a neighborhood $ U$ a \emph{trivializing neighborhood for ${\mathcal D}$}. The restriction of $\mathcal{D}$ at a point $x$ is given by the evaluation of all the sections of $\mathcal{D}$ at the point $x$ and is denoted by $\mathcal{D}_{x}$.
A \emph{singular sub-foliation} of a Hermann foliation ${\mathcal D}$ is a Hermann foliation ${\mathcal D}'$ such that ${\mathcal D}'( U) \subset {\mathcal D}( U)$ for 
all open subsets ${ U} \subset M$. 



A Hermann foliation on a manifold $M$ will be said to be \emph{finitely generated} when there exist $k$ vector fields $X_1, \dots, X_k \in {\mathfrak X}(M)$,
globally defined on the whole of $M$, that generate ${\mathcal D}$, i.e., such that, for every open subset ${ U} \subset M$, ${\mathcal D}( U)$ is 
generated over $ {\mathscr O}( U)$ by the restrictions to $ U$ of $X_1, \dots, X_k$.

%
 
When there exists a (smooth, real analytic or holomorphic) vector bundle $A$ over $M$ and a vector bundle morphism $\rho: A \to TM$ over the identity of $M$ such 
that\footnote{i.e. the sheaf ${\mathcal D}$ is obtained by sheafifying the presheaf $ \rho(\Gamma(A))$, which  means in this context that every point $x\in M$ admits 
a neighborhood $ U$ such that $ {\mathcal D}( U) =  \rho(\Gamma_{ U}(A)) $.} ${\mathcal D} =  \rho(\Gamma(A))$, where ${\Gamma}(A):  U \mapsto 
\Gamma_{ U}(A)$ is the sheaf of sections of $A$, we say that ${\mathcal D}$ \emph{is covered by} $(A,{\rho})$.

Let $\mathcal{D}$ be a Hermann foliation on $M$. We call  \emph{leaf of the Hermann foliation ${\mathcal D}$} a connected submanifold of $M$ whose tangent space at all points $x$ of $M$, is ${\mathcal D}_x$, and which is maximal with respect to inclusion
among such submanifolds. The integration of Hermann foliations has been studied in Section \ref{Hermann} for the smooth case, and in Section \ref{Nagano} for the real analytic case:

\begin{proposition}
A Hermann foliation ${\mathcal D}$ on a manifold $M$ induces a partition of $M$ into leaves. 
\end{proposition}

\begin{remarque}
Unlike the case of regular foliations,  Hermann foliations are not characterized by their leaves, and two different Hermann foliations may have the same leaves but differ as sheaves 
of vector fields. For instance, as noticed in \cite{AndrouZamb}, for $M$ a real or complex vector space (supposed, in the real case,
to be of dimension greater than or equal to $2$) and for each integer $k \geq 1 $, consider ${\mathcal D}_k$ to be the module of all smooth,
real analytic, or holomorphic vector fields  vanishing at order $k$ at the origin. This is clearly a Hermann foliation for all $k$, and all such 
Hermann foliations have exactly the  same two leaves: the origin and the complement of the origin. They are not, however,  identical as sub-modules of the module of vector fields. 
\end{remarque}

We would like to convince the reader of the interest of the notion of Hermann foliations by giving an extended list of examples. 

\begin{example}\label{Ex:AlgebroidsAreExamples}
\normalfont
For $A $ a (smooth or holomorphic \cite{LSX1}) Lie algebroid over ${ M}$ with anchor $\rho:A \to T M$, the ${\mathscr O}$-module $\rho(\Gamma(A))
\subset {\mathcal D}$ is a Hermann foliation. It is a finitely generated foliation when $\Gamma(A)$ is a finitely generated 
${\mathscr O}$-module, which is always the case when the vector bundle $ A $ is trivial, or more generally, when there exists a vector bundle $B$ such 
that the direct sum $A \oplus B$ is a trivial vector bundle. In the smooth case, it is always the case when ${ M} $ is compact.
\end{example}

In particular, regular foliations, orbits of connected Lie group actions, orbits of Lie algebra or Lie algebroid actions, 
symplectic leaves of Poisson manifolds and foliations induced by Dirac structures, are Hermann foliations covered by a vector bundle.

\begin{example}\label{ex:symmetries}
\normalfont 
Consider, for $\mathbb{K}=\mathbb{R}$ or $\mathbb{C}$:
 \begin{enumerate}
 \item Let $P:=(P_1,\dots,P_k)$ be a $k$-tuple of polynomial functions in $d$ variables over $ {\mathbb K}= {\mathbb R}$ or $ {\mathbb C}$. The symmetries of $ P$,  i.e. all polynomial vector fields $X \in {\mathfrak X}({\mathbb K}^d)$ that satisfy $X[P_i] =0$ for all $i \in \{1, \dots,k\}$, form a Hermann foliation. An interesting application, that appears in the Batalin-Vilkovisky context, occurs when considering the symmetries of a polynomial function $S$, which represents the classical action.
 \item \emph{Symmetries of $W$} i.e. all polynomial vector fields $X$ such that $X[P] \in P $,
 $P$ being now assumed to be the ideal of polynomial functions vanishing on some affine variety $W \subset {\mathbb K}^d$.
 \item \emph{Vector fields vanishing on $W$}, a self-consistent name.
\end{enumerate} 
  All the previous spaces of polynomial vector fields are closed under the Lie bracket, and form a sub-module (over the ring of polynomial functions on ${\mathbb K}^d$) of
  the module of algebraic vector fields. Since the latter is finitely generated over the ring of polynomial functions on ${\mathbb K}^d$,
  and since the ring of polynomial functions is Noetherian, each of these spaces is a finitely generated module over the polynomial functions. The ${\mathscr O}({\mathbb R}^d)$-module
  generated by these polynomial vector fields (with ${\mathscr O}$ standing here for smooth, or real analytic or holomorphic functions) 
  is therefore also a Hermann foliation. 
 \end{example}
 
 \begin{example}
 \label{ex:symmetriesMore}
 \normalfont
  Note that  Example \ref{ex:symmetries} can be extended to other types of polynomial vector fields. Given $\varphi:Z \to W$ a resolution
  in the sense of Hironaka of a singular affine manifold $W \subset {\mathbb C}^d$, one can consider the Hermann foliation of 
  \emph{symmetries of $W$ compatible with $\varphi$}, generated by all symmetries of $W$ (i.e. vector fields tangent to $W$) whose pull-back 
  through $\varphi$ is a well-defined regular vector field on $Z$. 
\end{example}

 It has been conjectured (see \cite{AndrouZambis} for instance) that not every smooth Hermann foliation is, locally, of the type described in Example~\ref{Ex:AlgebroidsAreExamples},
 i.e. is the image under the anchor map of a Lie algebroid. As far as we know, the question remains open to this day. We show that we can drop the assumption that the bracket on $\Gamma(A)$ satisfies the Jacobi identity, by generalizing the notion of Lie algebroids:

\begin{definition}
An \emph{almost-Lie algebroid} over ${ M}$ is a vector bundle $A \to  M$, equipped with a vector bundle morphism $\rho: A \to T M$ called the \emph{anchor map}, and a skew-symmetric bracket $[\ .\ ,\, .\ ]_A $ on $\Gamma (A) $, 
satisfying the \emph{Leibniz identity}:
\begin{align}\label{algebroid}
\forall\ x,y\in\Gamma(A),\ f\in\cinf({ M})\hspace{1cm}&[x,fy]_{A}=f[x,y]_{A}+\rho(x)[f] \, y\ ,
\end{align}
together with the \emph{Lie algebra homomorphism condition}:
\begin{equation}\label{algebroid2}
\forall\  x,y\in\Gamma(A)\hspace{1cm}\rho\big([x,y]_{A}\big)=\big[\rho(x),\rho(y)\big]\ .
\end{equation}
\end{definition}

We do not require that the bracket $[\ .\ ,\, .\ ]_A $ be a Lie bracket: it may not satisfy the Jacobi identity.
However, the Jacobi identity being satisfied for vector fields, condition (\ref{algebroid2}) imposes that the Jacobiator takes values in 
the kernel of the anchor map. 
We the following proposition to M. Zambon:

\begin{proposition} 
\label{prop:Almost}
Let $M$ be a smooth, or real analytic or complex manifold, and let $(A,\rho)$ be a pair, where $A\to M$ is a vector bundle and $\rho:A \to TM$ is a vector bundle morphism called the anchor map.
\begin{enumerate}
\item For every almost-Lie algebroid structure on $A \to  M$, the image of the anchor map $\rho : \Gamma(A) \to {\mathfrak X}({ M})$ is a Hermann foliation.
 \item Every finitely generated foliation on $M$ is the image under the anchor map of an almost-Lie algebroid, defined on a trivial bundle.
 \item In the smooth case, if $(A,\rho)$ covers a Hermann foliation ${\mathcal D}$, the vector bundle $A$ can be equipped with an almost-Lie algebroid structure with anchor $\rho$.
 \end{enumerate}
\end{proposition}
\begin{proof}
The first item follows from Conditions \eqref{algebroid} and \eqref{algebroid2}. 
For the second  item, consider a finitely generated Hermann foliation ${\mathcal D} $ and let $X_1,\dots,X_r $ be generators of $\mathcal{D}$. Since ${\mathcal D}$ is closed under the Lie bracket of vector fields, there exist functions $c^k_{ij} \in {\mathscr O}(M)$ satisfying:
 \begin{equation} [X_i,X_j] = \sum_{k=1}^r \ c^k_{ij}\, X_k \end{equation}
for all indices $i,j \in \{1, \dots,r\}$. Upon replacing $ c^k_{ij}$ by
$  \frac{1}{2}(c^k_{ij}-c^k_{ji})$ if necessary, we can assume that the functions $c^k_{ij} \in {\mathscr O}(M)$ satisfy the skew-symmetry 
relations $ c^k_{ij}= -c^k_{ji}$ for all possible indices. Choose $A$ to be the trivial bundle $A = {\mathbb R}^r \times { M} \to { M}$. Denote its canonical global sections by $e_1, \dots,e_r$ and define:
\begin{enumerate}
\item an anchor map by $\rho(e_i) = X_i$ for all $i=1, \dots,r$,
\item a skew-symmetric bracket by $[e_i,e_j]_{A} = \sum_{k=1}^r c^k_{ij} e_k$ for all $i,j=1, \dots,r$,
\end{enumerate}
then extend these structures by, respectively, ${\mathscr O}$-linearity and the Leibniz property. These structures define by construction an almost-Lie algebroid structure on $A$ with anchor~$\rho$, 
that covers ${\mathcal D} $ by construction.

Let us now prove the third item of the proposition. Let $n$ be the rank of $A$ and ${ U} $ be an open set on which $A$ admits a trivialization. By the previous point, 
an almost-Lie algebroid bracket on $A$ with anchor $ \rho|_{ U}$ can be defined over $ U$. Unlike Lie algebroid brackets, almost-Lie algebroid brackets can be glued  using partitions of unity. 
More precisely, let $(\varphi_i)_{i \in I}$ be a partition of unity subordinate to a covering  $({ U}_i)_{i \in I} $ by open sets
trivializing the vector bundle $A$. By the proof of item 2., we can define a bracket $[\ \cdot\ ,\, \cdot\ ]_{ U_{i}} $ that satisfies (\ref{algebroid}-\ref{algebroid2}) on each space of local 
sections $\Gamma_{ U_{i}}(A)$. The bracket:
\begin{equation}
[\ .\ ,\, .\ ]_{A}=\sum_{i \in I}\ \varphi_i [\ .\ ,\,.\ ]_{ U_{i}}
\end{equation}
still satisfies \eqref{algebroid} and \eqref{algebroid2}, hence defines an almost-Lie algebroid structure on $A$ with anchor~$\rho$.
\end{proof}

Recall that for every Lie $\infty$-algebroid $E$ over $M$, the binary bracket restricts to a skew-symmetric bilinear bracket on $\Gamma(E_{-1})$. Together with the anchor map, 
it defines an almost-Lie algebroid structure on $E_{-1}$, therefore by using the first item of Proposition \ref{prop:Almost}, we obtain:
\begin{proposition} \label{prop:fromNQtoSingularFoliations2}
For every Lie $\infty$-algebroid $E$ over $M$ with anchor $\rho$, ${\mathcal D} = \rho\big(\Gamma(E_{-1})\big)$ is a Hermann foliation.
\end{proposition}
\noindent We call this Hermann foliation the \emph{Hermann foliation of the Lie $\infty$-algebroid structure on $E$}. Looking for Lie $\infty$-algebroid structures associated to Hermann foliations would be a first step toward giving an answer to the Androulidakis-Zambon conjecture.

%

Below, we list various Hermann foliations, which seem not to be of any of the previous types.

 \begin{example}
  \label{ex:sousvariete}
 \normalfont
Vector fields on a manifold $M$ which are tangent to a submanifold $L$ are an example of Hermann foliation. 
Of course, $L$ is the only singular leaf, while connected components of  $M\slash L$
 are the regular ones.
 \end{example}

\begin{example}\label{ex:orderk}
\normalfont
Let $k$ and $d$ be integers greater than or equal to $1$. Vector fields vanishing at order $k$ at the origin of ${\mathbb R}^d$ form a Hermann foliation.
For $k=1$, it is the Hermann foliation associated to the action of the group of invertible $d \times d$ -matrices on ${\mathbb R}^d$ and is the image through the anchor map 
of a transformation Lie algebroid. For other values of $k$, it
is not obvious from which Lie algebroid it could arise.
\end{example}

\begin{example}\label{ex:bivectorfield}
\normalfont
In \cite{Turki}, a bivector field $\pi \in \Gamma(\wedge^2 T M)$ on a manifold $M$ is said to be \emph{foliated} when the space of vector fields of the form $\pi^\# (\alpha)$ for 
some $\alpha =\Gamma(T^*M)$ is closed under the Lie bracket, hence defining a Hermann foliation. When $\pi$ is a Poisson bivector, or
at least a twisted Poisson bivector (sometimes also referred to as `Poisson with background' \cites{Park,Klimcik,Severaweinstein}), it
is known that $T^* M$ comes equipped with a Lie algebroid structure \cite{costeweinstein} with anchor $\pi^\#:T^*  M \to T  M$, but for `generic' 
foliated bivector fields, no such formula seems to exist, as discussed in \cite{Turki}.
\end{example}

\begin{example}\label{ex:Leibnizoids}
\normalfont
Leibniz algebroid are defined as Lie algebroids, but the space of sections is assumed to be a Leibniz algebra, 
see Section \ref{sec:leibnizoid}). The image of the anchor map of Leibniz algebroids is obviously also a Hermann foliation. 
Courant algebroids are examples of Leibniz algebroids. Also, for $S$ a function on $M$, a Leibniz algebroid structure on the bundle $\wedge^2 TM$ is 
given by the anchor $P \mapsto P_S := P^\# ({\diff } S)$ together with the Leibniz bracket:
 $$  (P,Q) \mapsto {\mathcal L}_{P_S} Q  $$
for $P$ and $Q$ two sections of $\wedge^2 TM$, i.e., bivector fields. Note 
 that for $M$ a vector space and $S$ a polynomial function, the associated Hermann foliation is a sub-foliation of 
 the foliation of symmetries of $S$ described in Example \ref{symetriesofS}, referred to as the \emph{Hermann foliation of trivial symmetries of $S$}.
\end{example}
 
Now we give an example of a sub-sheaf of the sheaf of vector fields, which is closed under the Lie bracket, but which is \emph{not} a Hermann foliation.
\begin{example}
\normalfont
On $M = {\mathbb R}$, vector fields vanishing on ${\mathbb R}_-$ are closed under the Lie bracket but are not locally finitely generated, hence are 
not a Hermann foliation.
On $M={\mathbb R}^2$ with variables $(x,y)$, set ${\mathcal D}$ to be the $\mathcal{C}^{\infty}(M)$-module generated by
the vector field $ \frac{\partial }{\partial x}$ and vector fields of the form $ \varphi \frac{\partial }{\partial y}$
where $ \varphi$ is a smooth function vanishing on the half-plane $x \leq 0$. Then ${\mathcal D}$ is stable under the Lie bracket but it is not locally finitely generated.
\end{example}

\subsection{Main results}


We now assume that the reader is familiar with the equivalence between Lie $\infty$-algebroids and $NQ$-manifolds, as well as with Hermann foliations, and state the main results of this thesis. We recall that we intend to state results that are true in the smooth, real analytic and holomorphic settings altogether.

\begin{definition}\label{def:resolution}
Let ${\mathcal D}  \subset \Vect ( M)$ be a Hermann foliation on a manifold $M$. A \emph{resolution $(E,{\mathrm d},\rho)$ of the foliation $\mathcal{D}$} is a triple consisting of:
\begin{enumerate}
\item a collection of vector bundles $E=({\bigoplus}E_{-i})_{i \geq 1}$ over  $ M$,
\item a collection ${\mathrm d} = ({\mathrm d}^{(i)})_{i \geq 2}$ of vector bundle morphisms
${\mathrm d}^{(i)}: E_{-i} \to E_{-i+1}$ over the identity of $M$,
\item a vector bundle morphism $\rho:E_{-1} \to T M$ over the identity of $ M$ called the \emph{anchor of the resolution},
\end{enumerate}
 such that the following sequence of sections of ${\mathscr O}$-modules is an exact sequence of sheaves:
\begin{center}
\begin{tikzcd}[column sep=0.7cm,row sep=0.4cm] \label{eq:resolutions}
\ldots\ar[r,"\dd^{(3)}"]&\Gamma(E_{-2})\ar[r,"\dd^{(2)}"]&\Gamma(E_{-1})\ar[r,"\rho"]&\mathcal{D}\ar[r]&0
\end{tikzcd}
\end{center}
A resolution is said to be \emph{of length $n$} if $E_{-i}=0$ for $i \geq n+1$, and \emph{finite} if it admits a finite length.
We shall speak of a \emph{resolution by trivial bundles} when all the vector bundles $(E_{-i})_{i \geq 1}$ are trivial. 
\end{definition}
 Recall that being an 'exact sequence of sheaves' means that in some neighborhood $ { U}$ of every point $ x \in M$, the following complex is exact:
\begin{center}
\begin{tikzcd}[column sep=0.7cm,row sep=0.4cm] \label{eq:resolutions}
\ldots\ar[r,"\dd^{(3)}"]&\Gamma_{ U}(E_{-2})\ar[r,"\dd^{(2)}"]&\Gamma_{ U}(E_{-1})\ar[r,"\rho"]&\mathcal{D}({ U})\ar[r]&0
\end{tikzcd}
\end{center}
%
where $\Gamma_{ U}$ stands for sections over $ { U}$ and $\mathcal{D}({ U})$ stands for the $\mathscr{O}({ U})$-module
of vector fields defining the Hermann foliation $ \mathcal{D}$ over ${ U}$. In the smooth case, the complex is exact at the level of global sections 
if and only if it is exact in a neighborhood of each point, since local exactness can be extended to any open set by using partitions of unity. In the holomorphic and real analytic cases, the definition is only valid locally in a neighborhood of each point.




  Since sections of the vector bundles over $M$ are projective $\Funct$-modules by the Serre-Swan theorem, resolutions of smooth Hermann foliations are in fact a 
  very classical object of algebraic topology: they are projective resolutions of ${\mathcal D}$ in the category of $\Funct$-modules.
  It is a classical that such resolutions always exist. But these projective modules may not correspond to vector bundles - they may not be locally finitely generated. 
  By the Serre-Swan theorem, however:
  
  \begin{lemme}
   Resolutions of a Hermann foliation $ {\mathcal D}$ are resolutions of ${\mathcal D}$ by locally finitely generated ${\mathscr O}$-modules. 
  \end{lemme}
  
  There are several contexts in which such resolutions always exist, at least locally, and are finite. For instance , for Hermann foliations generated by polynomial vector fields on $\mathbb{C}^{n}$, the existence is due to the fact that the ring of polynomial
  functions is Noetherian, and finiteness is due to the Hilbert syzygy theorem. Moreover,
  a real analytic or holomorphic resolution is also a smooth resolution: this is not an easy result, the proof of which uses theorems due to Malgrange and Tougeron \cite{Tougeron}. In short:

\begin{proposition}\label{prop:exres}The two following items hold:
\begin{enumerate}
\item Any holomorphic (resp. real analytic) Hermann foliation on a complex (resp. real analytic) manifold $M$ admits, in a neighborhood of a point, a resolution by trivial vector bundles 
whose length is less or equal to $n$ (i.e. $E_{-i}=0$ for all $i \geq n+2$).
\item Moreover, a real analytic resolution of a real analytic Hermann foliation $\mathcal{D}$ is also a smooth resolution of the smooth Hermann foliation generated by $\mathcal{D}$.
\end{enumerate}
\end{proposition}


%


\noindent We refer to Section \ref{existence}
  for a proof of this proposition. We now introduce the main object and the two main theorems of the present part. 

\begin{definition}
Let ${\mathcal D}$ be a Hermann foliation of a manifold $M$. We say that a Lie $\infty$-algebroid $(E,Q)$ over $M$ is an \emph{universal
Lie $\infty$-algebroid resolving ${\mathcal D}$} if:
\begin{enumerate}
\item $ {\mathcal D}$ is the foliation associated to $(E,Q)$, i.e. $\rho\big(\Gamma(E_{-1})\big) = {\mathcal D}$,  
\item the linear part of $(E,Q)$ is a resolution of  $ {\mathcal D}$.
\end{enumerate}
When $E_{-k}=0$ for all $k \geq n+1$, we speak of a \emph{universal Lie $n$-algebroid resolving ${\mathcal D}$}. 
\end{definition}

\noindent Obviously, the existence of such a structure depends on the existence of a resolution of the Hermann foliation.

\bigskip
\begin{theoreme}
\label{theo:existe}
Let ${\mathcal D}$ be a Hermann foliation of $M$. A universal Lie $\infty$-algebroid resolving ${\mathcal D}$ exists:
\begin{enumerate}
\item in the smooth case, when a resolution  $(E,{\mathrm d},\rho)$ of ${\mathcal D} $ exists,
\item in a neighborhood of every point in $M$ in the real analytic and holomorphic cases.
\end{enumerate}
\end{theoreme}
\bigskip

\begin{remarque}
In the last two cases, there exists, in a neighborhood ${ U} $ of every point in $M$, a universal Lie $n$-algebroid resolving ${\mathcal D}$, with $n$ the dimension of $M$, 
whose linear part is a resolution of $ {\mathcal D}$ on ${ U} $ by trivial vector bundles.
\end{remarque}

The goal of this thesis is precisely to show that given any resolution of a Hermann foliation, the existence and uniqueness 
(in some sense discussed below) of a Lie $\infty$-algebroid structure on the resolution is guaranteed.
The use of the word `universal' is justified by the second item in the next theorem:

\begin{theoreme}\label{theo:onlyOne}
Let $(E,Q)$ be  a universal Lie $\infty$-algebroid  resolving a Hermann foliation $ {\mathcal D}$. 
Then, 
\begin{enumerate}
 \item for any Lie $\infty$-algebroid $ (E',Q')$ defining a sub-Hermann foliation of $ {\mathcal D}$ (i.e. such that  $\rho'\big(\Gamma(E_{-1}')\big) \subset {\mathcal D}$),
there is a Lie $\infty$-algebroid morphism from $ (E',Q')$ to $ (E,Q)$ over the identity of $M$ and any two such Lie $\infty$-algebroid morphisms are homotopic.
 \item in particular, two  universal Lie $\infty$-algebroids  resolving the Hermann foliation $ {\mathcal D}$ are isomorphic up to homotopy and two such isomorphisms are homotopic.
\end{enumerate}
\end{theoreme}

\begin{remarque}
This theorem also holds for real analytic and holomorphic singular foliations in a neighborhood of a point.
\end{remarque}

In particular, for every Lie algebroid $A$ defining a Hermann foliation ${\mathcal D} $,
a Lie $\infty$-morphism from $A$ to any universal Lie $\infty$-algebroid $ (E,Q)$ resolving ${\mathcal D}$ exists and any two such morphisms are homotopic. 
By the second item in Theorem \ref{theo:onlyOne}, for a Hermann foliation that admits a resolution,
it makes sense to speak of the cohomology of $({\mathscr E},Q)$,  where ${\mathscr E}$ is the sheaf of
functions on the graded vector bundle $E$, since a different resolution will have a canonically isomorphic cohomology.
Moreover, restricting the resolution to an arbitrary point of $M$ or an arbitrary leaf $L$ yields cohomologies that come
equipped with graded Lie algebras structures. We are able to relate those to the holonomies of Androulidakis and Skandalis \cite{AndrouSkandal}:

\begin{proposition}
\label{prop:mainConsequences}
Let ${\mathcal D}$ be a Hermann foliation, and $(E,Q)$ a universal Lie $\infty$-algebroid with anchor $\rho$ resolving ${\mathcal D}$.
\begin{enumerate}
\item For each leaf $L$, the linear operators $\dd^{(i)}:E_{-i} \to E_{-i+1}$ of the linear part of $(E,Q)$ are of constant rank along $L$,
so that the cohomologies of this complex form a graded vector bundle $ \sum_{i \geq 1} H_{\mathcal D}^{-i}$ over $L$, which: 
 \begin{enumerate}
  \item does not depend on the choice of a resolution;
  \item has a component in  degree  $-1$ equipped with a natural Lie algebroid structure which coincides with the holonomy
  Lie algebroid of Androulidakis and Skandalis;
  \item admits a graded Lie algebra structure on its space of sections, on which the previous Lie algebroid acts.
  \end{enumerate}
\item The $1$-truncated groupoid of $(E,Q)$ is a cover of the connected component of the manifold of identities of the holonomy groupoid of Androulidakis and Skandalis. 
\end{enumerate}%
\end{proposition}

\noindent Notice that
the geometrical meaning of the graded Lie algebra  that appears in the third item of the previous proposition is still quite unclear to us. For clarity, we have dedicated different sections to the proofs of these various results. Theorem \ref{theo:existe} is proven in Section \ref{preuveexistence}, Theorem \ref{theo:onlyOne} is proven in Section \ref{morphismfoliations}, and Proposition \ref{prop:mainConsequences} is proven in Section \ref{moineau}.

We would like to address two objections. We are aware that the universal Lie $\infty$-algebroid structure resolving a Hermann foliation that we construct is, 
somehow, related to derived geometry. But we still claim that our construction is \emph{not} trivial from  the point of view of derived geometry. 
Let us explain this point. First, it is true that our construction follows in some sense the general idea presented in \cite{LS}, where it is shown 
that any resolution of a Lie algebra can be equipped with an $L_\infty$-algebra structure (then called a sh-Lie algebra). But the analogy should not be carried too far:
there is indeed a subtle but crucial difference between the case of vector spaces and the case of modules over spaces of functions. For vector space complexes, inexistence of homology 
is equivalent to saying that the identity is homotopic to the zero map. But for modules, this is not true. This makes the constructions we present here much more complicated,
but also richer in structure, since it means that the cohomology, does not contain all the relevant information and one has to look at the level of chains. 
In particular, \emph{we cannot apply classical tools like the transfer theorem  or the perturbation lemma \cite{perturbation} to derive our Lie $\infty$-algebroid structures, 
since for ${\mathscr O}$-modules, we do not have chain equivalence between a complex and its cohomology}. Of course, we do not deny that there might be general principles 
of derived geometry or higher structures behind our construction of the universal $L_\infty$-algebra resolving a foliation. We are convinced that there are such general ideas. 
We simply claim that the most classical constructions do not work, and that our construction has at least the merit of being understandable even by mathematicians or theoretical physicists who do not know anything about derived geometry or higher structures beside the definition of a $Q$-manifold, and that we do see the need of more sophisticated tools at this point.

Let us address a second objection. Since it has been conjectured by I. Androulidakis and M. Zambon \cite{AndrouZambis} that not every Hermann
foliation arises as the image under the anchor map of a Lie algebroid
in the neighborhood of a point, 
it is totally relevant to look for objects as universal Lie
$\infty$-algebroids. But we claim that even for foliations that arise in that way, the Lie $\infty$-algebroid 
resolving ${\mathcal  D}$ makes sense, and in fact makes more sense than the Lie algebroid in question. This comes from the simple observation that the algebroid, if any, 
covering a Hermann foliation is not unique. It is, in some sense, `too big' and may encode non-relevant materials. For instance, for $A$ a Lie algebroid defining
a Hermann foliation ${\mathcal D}$ and ${\mathfrak g}$ a Lie algebra, the direct product $A \times {\mathfrak g}$ is again a Lie algebroid that defines the same singular
foliation ${\mathcal D}$. The same would hold for semi-direct products for some $A$-action on ${\mathfrak g}$. More generally, the Lie alegbroid defining a foliation is not
a relevant object because sections which are in the kernel of the anchor map play no role. It is therefore hard to see which information encoded in the Lie algebroid is related to the Hermann foliation. On the contrary, any infomation extracted out of our negatively graded resolution-type Lie $\infty$-algebroid  is relevant provided that it 
be homotopy invariant: \emph{non-relevant information in a Lie algebroid defining  the Hermann foliation is  `killed' while considering the universal Lie $\infty$-algebroid}.

Let us say a few words about the proofs of the previous theorems. They are mainly based on step-by-step construction, and the vanishing of some obstruction classes. We warn the reader that too naive step-by-step constructions obtained by directly defining the homological vector field $Q$ on functions will not work. For Lie $\infty$-algebroids, unlike for
Lie $\infty$-algebras, the vanishing of the cohomology of their linear part does not trivialize the problem, because the cohomology of $Q$ is \emph{not} zero, and even the cohomology of its component of arity $0$ is not. But it happens that the cohomology of $P \mapsto [Q^{(0)},P]$, restricted to vertical vector fields, vanishes at least in the degrees that we are interested in. In fact Lemma \ref{lemmefondamental2} is crucial to understand the proof of all these theorems. 


\section{Proof of the main results}

\subsection{Existence of resolutions of a Hermann foliation}\label{existence}

%


The purpose of this subsection is to prove Proposition \ref{prop:exres}. 
\begin{proof}
The first item simply comes from Hilbert's syzygy theorem, which is valid for finitely generated ${\mathscr O}$-modules,
with ${\mathscr O}$ being the ring of holomorphic functions in a neighborhood of a point in ${\mathbb C}^n$. See  for references Theorem 4 page 137 in \cite{gunningrossi}. On the other hand, any real analytic manifold admits a complexification, such that the original manifold is the fixed point of the involution $z\mapsto \bar{z}$. The module of real analytic functions on this manifold consists of the real part of the module of holomorphic functions on its complexification. Then Hilbert's syzygy Theorem applies to finitely generated $\mathscr{O}$-modules $-$ where $\mathscr{O}$ is the sheaf of real analytic functions on $\mathbb{R}^{n}$. Hermann foliations being locally finitely generated submodules of holomorphic (resp. real analytic) vector fields, Hilbert's syzygy Theorem applies on a neighborhood of any point. Recall that this theorem claims that such finitely generated modules admit resolutions of length less than or equal to $n+1$,
where $n$ is the dimension of $M$. Elements of the resolution are free modules on holomorphic (resp. real analytic) functions defined in a neighborhood of this point - which we interpret as sections of trivial bundles.

Now, we have to prove the second item. According to Theorem 4 in \cite{Tougeron}, germs of smooth functions at a given point are a flat module over germs of real analytic functions at the same point. In particular, it means that given a complex $E_{-k-1} \to E_{-k} \to E_{-k+1}$ of vector bundles on the base manifold such that germs of real analytic functions 
have no cohomology at degree $-k$, the sheaf of germs of smooth sections also has no cohomology at degree $-k$. But since partitions of unity do exist for smooth sections,
it implies that sections on arbitrary open subsets do not have cohomology. Let $e \in \Gamma_U(E_{-k})$ be a local section of $E_{-k}$ defined on an open subset $U$ 
such that $e$ is an element of the kernel of $\dd^{(k)}:E_{-k} \to E_{-k+1}$. For every point $x \in U$, and for any neighborhood $U_x\subset U$ of $x$, there exists
a section $f_x \in \Gamma_{U_{x}}(E_{-k-1})$ such that $ \dd^{(k+1)} f_x = e$. A locally finite open cover $(U_{x_{i}})_{i \in I}$ indexed by $I$  admitting a partition of unity $
(\chi_i)_{i \in I}$
can be extracted from the family $(U_x)_{x \in U}$. Since $\dd^{(-k-1)}$ is ${\mathscr O}$-linear $f := \sum_{i \in I} \chi_i f_{i} $ is a section over $U$ of $E_{-k-1}$ satisfies $\dd^{(-k-1)} f=e$ by construction.
\end{proof}

Let us give several examples of resolutions of Hermann foliations.

\begin{example}
For a regular foliation ${\mathcal F}$, we define $E_{-1 x}=T_{x}{\mathcal F}$ the tangent space of the foliation, equipped with the bracket of vector fields tangent to the foliation. 
\end{example}

\begin{example}
Lie algebroids of quasi-graphoids (as defined by C. Debord \cite{Debord}) are also examples of resolutions with $E_{-i}=0$ for  all $i\geq 2$.
 These follows from Proposition 2 in \cite{Debord} which states that their anchor map is injective on an open subset, hence injective when seen as a map $\Gamma(E_{-1}) \to \Vect ( M)$. In general, a Lie algebroid $(A,[\ .\ ,\, .\ ],\rho)$ for which the anchor map $\rho:A \to TM$ is injective on an open dense subset give resolutions with $E_{-i}=0$ for  all $i\geq 2$ and $E_{-1}=A$.
\end{example}

\begin{example}\label{sl2}
The Lie algebra $\mathfrak{sl}_{2}$ acts on ${\mathbb R}^2$ (equipped with coordinates $x,y$) through the vector fields
 \begin{equation} \label{eq:sl2} 
  h = x \frac{\partial}{\partial x} - y  \frac{\partial}{\partial y}, \quad e = x \frac{\partial}{\partial y},\quad f= y \frac{\partial}{\partial x} .  
 \end{equation}
  Since $ [h,e]=2e, [h,f]=-2f $ and $[e,f]=h$,  the $\cinf ({\mathbb R}^2)$-module generated by $e,f,h$ is a Hermann foliation. The vector fields in
 (\ref{eq:sl2}) are not independent over $\cinf ({\mathbb R}^2)$, since:
   \begin{equation} \label{eq:relationsl2}    xy h  + y^2 e - x^2 f =0 . 
   \end{equation}
We use this equation to define a resolution by the following elements: 
\begin{enumerate} 
\item $E_{-1}$ is the trivial vector bundle of rank $3$ generated by $3$ sections that we denote by
$ \tilde{e},\tilde{f},\tilde{g}$, 
\item then we define an anchor by
 \begin{equation} \rho(\tilde{e})=e,\ \rho(\tilde{f})=f,\ \rho(\tilde{h})=h  \end{equation}
\item $E_{-2}$ shall be a trivial vector bundle of rank $1$, generated by a section that we call $r$,
\item  and we define a vector bundle morphism from $E_{-2}$ to $E_{-1}$ by:
  \begin{equation} {\mathrm d} (r) = xy \tilde{h}  + y^2 \tilde{e} - x^2 \tilde{f}, \end{equation}
\item  and we set $E_{-i}=0$ and $\dd=0$ for $i \geq 3$.
\end{enumerate}
  The triple $(E,{\mathrm d},\rho)$ is a resolution of the foliation given by the action of $\mathfrak{sl}_{2}$ on ${\mathbb R}^2$.
\end{example}

\begin{example}
 The adjoint action of  $\mathfrak{gl}_{n}$ on itself defines a Hermann foliation over $M := \mathfrak{gl}_n$.
To find a resolution of it, it suffices to consider $ E_{-1} $ to be the trivial bundle over  $M= \mathfrak{gl}_n $ with typical fiber $ \mathfrak{gl}_n$,
 and $E_{-2}$ to be the trivial bundle over $M$ with typical fiber $ {\mathbb R}^n$.
 The map $\rho : E_{-1} \to TM $ is, at a point $m \in M= \mathfrak{gl}_n$, obtained by mapping $a \in (E_{-1})_m \simeq \mathfrak{gl}_n $ to $[a,m] \in T_m M \simeq \mathfrak{gl}_n $,
 while $ \dd^{(2)}$ is the vector bundle morphism mapping, for all $m \in M $, an $n$-tuple
 $ (\lambda_1, \dots,\lambda_n) \in (E_{-2})_m $ to $ \sum_{i=1 }^n \lambda_i  m^i  \in (E_{-1})_m \simeq \mathfrak{gl}_n $.
 This comes from the fact that a smooth function $f:\mathfrak{gl}_n \to \mathfrak{gl}_n$ such that $[f(m),m]=0$ for all $m \in \mathfrak{gl}_n$
 has to be of the form $f(m)=\sum_{i=0}^{n} \lambda_i (m) m^i $, with $\lambda_0,\dots,\lambda_n$ smooth functions on $ \mathfrak{gl}_{n-1}$.
 
\end{example}

\nocite{komatsu}

Here is an example of a smooth Hermann foliation that does not admit smooth resolutions.

\begin{example}\label{tu}
Let $\chi$ be a smooth real-valued function on $M:={\mathbb R}$ vanishing identically on $]-\infty,0]$ and $[1,+\infty[$ and strictly positive on $]0,1[$. Consider the Hermann foliation generated by the vector field ${ V}$ defined by:
$${ V}_t := \chi(t)\frac{\dd}{\dd t}\hbox{ for all $t \in {\mathbb R}$}.$$
Let $E \to {\mathbb R}$ be a resolution. 
Let us replace $E$ by a resolution $E'$ defined by:
$E_{-i}'=E_{-i}$ for $i\geq 3$, $E_{-2}'=E_{-2}\oplus {\mathbb R}$ and  $E_{-1}'=E_{-1}\oplus {\mathbb R}$, equipped with the differential
 $ \dd^{(2)} \oplus id_{\mathbb R} $ in degree $-2$ while the additional generator in degree $-1$ is assumed to be in  the kernel of the anchor map.

In a neighborhood of each $ t \in {\mathbb R} $, $E_{-1}'$ admits a nowhere vanishing section $e_t$ such that $\rho(e_t|_{s}) = V_{s} $ for every $s$ in some neighborhood of $t$. Assume that its component in the additional generator is constant equal to $1$. 
Since $[0,1]$ is compact, we can find a finite family $e_{t_0} , \dots e_{t_k}$ of such local sections with $0=t_{0}<\ldots<t_{k}=1$ defined on intervals $I_0,\dots,I_k$. $I_{0}$ and $I_{k}$ can be chosen so that $0\in I_{0}$ and $1\in I_{k}$. For $\varphi_0, \dots,\varphi_k$
a partition of unity relative to these open subsets, the section $e|_{s}=\sum_{i=1}^k \varphi_i (s) e_{t_i}|_{s}$ is nowhere vanishing on a neighborhood of $[0,1]$ and satisfies  $ \rho(e|_{s})= V_{s}$ for every $s$ in this neighborhood. It can be extended to a nowhere vanishing section satisfying the same property on $M={\mathbb R}$.

Since ${\mathbb R}$ is a contractible manifold, each of the vector spaces $(E_{-i}')_{ i \in \mathbb{N}} $ must be trivial, and we denote by $n_{i}$ the rank of $E'_{-i}$. In particular, $\Gamma (E_{-1}')$ is generated by $n_1$ canonical generators.  Without any loss of generality, we can assume that $e_1=e$.
 Moreover, since the image of the anchor map is $V$, we have for every $1\leq k\leq n_{1}$: $\rho(e_k) = g_k { V} = g_k \rho(e_1)$ for some function $g_{k}\in\cinf(\mathbb{R})$. we have $\rho(e_k - g_k e_1)=0$, so that, upon replacing $e_k$ by $e_k - g_k e_1$ for $k=2, \dots,n_1$, we can assume that $\rho(e_k)=0$ for $k=2, \dots,n_1$ . Since $\mathrm{Im}(\dd^{(2)})=\mathrm{Ker}(\rho)$, let $f_2,\dots,f_{n_1}$
 be sections of $E_{-2}$ such that $\dd^{(2)}(f_i)=e_i$ for all $i=2, \dots,n_1$. 
 These sections are linearly independent at all points of $\mathbb{R}$ and it is clear that a second resolution of the foliation is obtained 
 by taking the quotient of $E_{-2}'$ by the sub-bundle generated by  $ f_2,\dots,f_{n_1}$ (denoted by $F_{-2}$) and by replacing $E_{-1}'$ by the trivial one-dimensional vector bundle generated by $e$, that we denote $F_{-1}$. All other vector bundles are left untouched, and we write $F_{-i}=E_{-i}$, for $i\geq3$. For this new resolution $F$, we now have that $F_{-1}$ is a trivial bundle of rank $1$ and that the anchor map $\rho$ projects to an anchor map $\rho_{F}:F_{-1}\to T\mathbb{R}$. When applied to the constant section equal to $1$, it gives ${ V}$.
 
The expression of the anchor map $\rho_{F}$ implies that its kernel in the sections of $F_{-1}$ is the ideal of all smooth real-valued functions on ${\mathbb R}$ vanishing identically on $[0,1]$. But this ideal is \emph{not} finitely generated as $\cinf ({\mathbb R})$-module, hence it cannot be the image of $\Gamma(E_2')\simeq \cinf ({\mathbb R})^k$ through a $\cinf ({\mathbb R})$-linear map. Hence no smooth resolution exists in this case.    
\end{example}

\subsection{A fundamental lemma on vertical vector fields}

The notion of arity introduced in Section \ref{section:dgmanifold} will be at the core of most of the proofs of the present thesis, and as such it deserves to be made precise. Let $E$ be a positively-graded manifold, that is a family of vector bundles $(E_{-i})_{i\geq1}$ over a base manifold $M$. Recall that, by a \emph{vector field}, we mean a derivation of the sheaf of functions $\mathscr{E}$ over $E$ and by a \emph{vertical vector field} we mean a $\mathscr{O}$-linear derivation of $\mathscr{E}$ (which geometrically means that the vector field is tangent to the fibers of $E \to M$). 

We say that a function $f\in\mathscr{E}$ is of \emph{arity $n$} and \emph{degree $k$} if 
$ f \in \Gamma (S^{n}(E^{\ast})_{k})$, \emph{i.e.} if it is a section of $ \sum_{i_1+ \dots + i_n = k } E^*_{-i_1} \odot \dots \odot E_{-i_n}^*  $, where $\odot$ denotes the symmetric product. 
A vector field is said to be of \emph{arity  $n$ and of degree $k \in {\mathbb Z}$} when, seen as a derivation of $\cinf(E)$, it increases the arity by $n$ and the degree by $k$.
The following proposition, states the main properties of the arity of a function and of a vector field:

\begin{proposition}\label{prop:arity}
Let $E \to  M $ be a positively graded manifold.
\begin{enumerate}
\item  The allowed values of arity of a function range from $0$ to $+\infty$, and that of a vector field range from $-1$ to $ +\infty$.
\item The arity of a function is less than or equal to its degree. 
\item Vector fields of arity $-1$ are vertical and of negative degree.
\item Vector fields of arity $0$ and of non zero degree are vertical.
\item Vector fields of arity $n \neq 1$ and of degree $+1$ are vertical.
\item The Lie  bracket of vector fields or arity $n$ and $n'$ is of arity $n+n'$.
\end{enumerate}
\end{proposition}

Now let us say a few words about vertical vector fields. Vertical vector fields form a graded Lie subalgebra of the graded  Lie algebra ${\mathfrak X}(E)$ of vector fields (the grading being given by the degree) that we denote by $ {\mathfrak U}(E) $. The following proposition is an important result that will be used for the proof of the main theorem:

\begin{proposition}\label{lemmefondamental}
There is a one-to-one correspondance between almost-Lie algebroids $A\to M$ and graded manifolds $A[1]\to{ M}$ equipped with a degree 1 vector field of arity 1 whose self-commutator is vertical.
\end{proposition}

\begin{remarque}
This result holds in the smooth, real analytic and holomorphic case.
\end{remarque}

\begin{proof}
The correspondence between Lie algebroids and $NQ$-manifolds of degree 1 is well-known. We can develop the same machinery here. For any almost-Lie algebroid $(A,[\ .\ ,\,.\ ]_{A},\rho_{A})$, there is a graded manifold $A[1]\to{ M}$ whose space of sections $\Gamma(A[1])$ is equipped with a graded symmetric bracket $\{\ .\ ,\,.\ \}$ of degree $+1$ (because local coordinates on the fibers of $A[1]$ have degree $+1$), and a degree $+1$ bundle map $\rho:A[1]\to TM$ compatible with the bracket.
%

Recall that functions on the graded manifold $A[1]$ form the graded algebra $\Gamma(\wedge^{\bullet} A^*)$.
Following Equations \eqref{anchor1} and \eqref{anchor2}, there is an unique vector field of degree $+1$
on the graded manifold $A[1] $ such that:
 $ \rho^* \dd_{\dd R} f = [X,f]  $
 for all functions $f$ on $M$ and $\{a,a'\}=\big[[Q,\partial_{a}],\partial_{a'}\big] $ 
for all sections  $a, a' \in \Gamma(A^*)$.
Above $\partial_a$ and $\partial_{a'}$ correspond to the vector fields of degree $-1$ on $A[1]$ given by contraction by $a$ and $a'$ respectively, as defined in Equation \eqref{innerder}.
An easy computation gives that Condition \eqref{algebroid2} is equivalent to requiring $[Q,Q]$ to be vertical, i.e. vanishes when applied to a function
on the base manifold.

Reciprocally, a degree 1 graded manifold $B\to{ M}$ with a degree 1 vector field $Y$ of arity 1 whose square is vertical provides us with the following elements: the horizontal part of $Y$ gives the action of the anchor map, whereas the vertical part gives the dual of the bracket (on constant sections), as in Equations \eqref{anchor1} and \eqref{anchor2}. The condition on the vertical self-commutator implies that Equation \eqref{algebroid2} is satisfied. Knowing that Equation \eqref{algebroid} is automatically satisfied by applying Equation \ref{eq:correspondence} to define $\{x,f y\}$, it turns $B[-1]$ into an almost-Lie algebroid.
\end{proof}


We will heavily use the concept of arity, in particular when decomposing the vector fields and the morphisms in components of homogeneous arities.
The space of vector fields on $E$ then carries two different gradings : the arity and the total degree.

Recall that a \emph{vertical vector field} on $E$ is a vector field which is $\mathscr{O}$-linear, which geometrically means that it is parallel to the fibers of the projection from $E$
onto its base $M$. Identifying the tangent space of $E$ at each point with the fiber $E$, we obtain that for every $n\geq1$, there is an isomorphism between the vector space of vertical vector fields of arity $n-1$ and elements of the direct sum:
\begin{equation*}
\F{U}^{(n-1)}=\bigoplus_{k=-\infty}^{+\infty}\underset{i,j\geq1}{\bigoplus_{i-j=k}} \Gamma\Big(S^{n}(E^*)_i  \otimes E_{-j}\Big)
\end{equation*}
Sections of $S^{n}(E^*)_i  \otimes E_{-j}$ are said to be of \emph{height} $i$ and \emph{depth} $j$. Since homogeneous elements of $E^{\ast}$ have at least degree one, the height is valued in $\{n, n+1,\ldots\}$, whereas the depth is valued in $\{1,2,\ldots\}$, so that vertical vector fields of arity $n-1$ and degree $k$ can be represented as infinite sums of elements in the anti-diagonals $i-j=k$ (`\emph{height $-$ depth = k}') in the sections of the bicomplex:
\bigskip


\begin{center}
\begin{tikzpicture}
\draw[xstep=3.5cm,ystep=1,color=gray] (1,0) grid (14,3.5);
    \node at (2,0.5)  {$\cdots$};
    \node at (5.3,0.5)  {$S^n(E^*)_{n} \otimes E_{-3}$};
    \node at (8.8,0.5)  {$S^n(E^*)_{n} \otimes E_{-2}$};
    \node at (12.3,0.5)  {$S^n(E^*)_{n} \otimes E_{-1}$};
    \node at (2,1.5)  {$\cdots$};
    \node at (5.3,1.5)  {$S^n(E^*)_{n+1} \otimes E_{-3}$};
    \node at (8.8,1.5)  {$S^n(E^*)_{n+1} \otimes E_{-2}$};
    \node at (12.3,1.5)  {$S^n(E^*)_{n+1} \otimes E_{-1}$};
    \node at (2,2.5)  {$\cdots$};
    \node at (5.3,2.5)  {$S^n(E^*)_{n+2} \otimes E_{-3}$};
    \node at (8.8,2.5)  {$S^n(E^*)_{n+2} \otimes E_{-2}$};
    \node at (12.3,2.5)  {$S^n(E^*)_{n+2} \otimes E_{-1}$};
    \node at (5.3,3.5)  {$\cdots$};
    \node at (8.8,3.5)  {$\cdots$};
    \node at (12.3,3.5)  {$\cdots$};
\end{tikzpicture}
\end{center}


\bigskip
The \emph{depth} (resp. \emph{height}) of a vertical vector field of a fixed degree is the minimum of the depths (resp. height) of all its non zero constituents. In the diagram above it would coincide with the depth and height of the lowest element of the anti diagonal which symbolizes the given vector field. The \emph{root} of a vertical vector field of degree $k$ is its component of depth $1$ $-$ the $root$ of $X$ is denoted by $rt(X)$. It can be zero, and in that case the depth of $X$ is stricly higher than 1. A \emph{root-free} element is a vertical vector field which does not have a root (in particular it means that its depth is stricly higher than $1$), otherwise it is said to be \emph{rooted}. For degree reason, any vertical vector field of arity $n$ and degree less than or equal to $n-1$ is root-free.


\begin{figure}[ht]
  \centering
  \begin{tikzpicture}[scale=0.50]
    \coordinate (Origin)   at (0,0);
    \coordinate (XAxisMin) at (10,-2);
    \coordinate (XAxisMax) at (0,-2);
    \coordinate (YAxisMin) at (10,1);
    \coordinate (YAxisMax) at (10,8);
    \coordinate (ZAxisMin) at (10,-2);
    \coordinate (ZAxisMax) at (10,-1);
    \coordinate (WAxisMin) at (10,-1);
    \coordinate (WAxisMax) at (10,1);
    \draw [ultra thick, black,-latex] (XAxisMin) -- (XAxisMax) node [left] {$depth$};
    \draw [ultra thick, black,-latex] (YAxisMin) -- (YAxisMax) node [above] {$height$};
    \draw [ultra thick, black] (ZAxisMin) -- (ZAxisMax);
    \draw [ultra thick, black,dashed] (WAxisMin) -- (WAxisMax);

    \clip (-3,-3) rectangle (12cm,9cm); 
    \coordinate (Bone) at (7,5);
    \coordinate (Btwo) at (3,9);
    \coordinate (B1) at (-1,9);
    \coordinate (B2) at (6,2);
    \coordinate (B3) at (4,4);
    \coordinate (B4) at (0,8);
    \draw[style=help lines,dashed] (-1,1) grid[step=2cm] (10,9);
    \draw (-1,-2) -- (10,-2)[dashed, gray];
    \draw (8,-2) -- (8,-1)[dashed, gray];
    \draw (6,-2) -- (6,-1)[dashed, gray];
    \draw (4,-2) -- (4,-1)[dashed, gray];
    \draw (2,-2) -- (2,-1)[dashed, gray];
    \draw (0,-2) -- (0,-1)[dashed, gray];
    \foreach \x in {0,1,...,4}{
      \foreach \y in {1,2,...,4}{
        \node[draw,circle,inner sep=1pt,fill] at (2*\x,2*\y) {};
        \node[draw,circle,inner sep=1pt,fill,red] at (6,6){};
        \node[draw,circle,inner sep=1pt,fill,red] at (4,8){};
         \node[draw,circle,inner sep=1pt,fill,blue] at (6,4){};
        \node[draw,circle,inner sep=1pt,fill,blue] at (4,6){};
        \node[draw,circle,inner sep=1pt,fill,blue] at (2,8){};
        \node[draw,circle,inner sep=1pt,fill,red] at (8,4){};
        \node[draw,circle,inner sep=1pt,fill,green] at (B2){};
        \node[draw,circle,inner sep=1pt,fill,green] at (B3){};
        \node[draw,circle,inner sep=1pt,fill,green] at (B4){};
      }
    }
    \node [below] at (8,-2)  {$1$};
    \node [below] at (6,-2)  {$2$};
    \node [below] at (4,-2)  {$3$};
    \node [right] at (10,2)  {$n$};
    \node [right] at (10,4)  {$n+1$};
    \node [right] at (10,6)  {$n+2$};
    \node [below right] at (10,-2)  {$0$};
    
    \draw [ultra thick,red] (Btwo)
        -- (8,4);
    \node [above right,red] at (5,7) {\large $ U$};
    \draw [ultra thick,green] (B1)
        -- (B4);
    \draw [ultra thick,green,dashed] (B4)
        -- (B3);
    \draw [ultra thick,green] (B3)
        -- (B2);
    \node [below left,green] at (3,5) {\large $ V$};
    \draw [ultra thick,blue] (1,9)
        -- (6,4);
    \draw [ultra thick,blue,dashed] (6,4)
        -- (8,2);
    \node [below ,blue] at (8,2) {\large $W$};
  \end{tikzpicture}
  \caption{\footnotesize This picture represents vertical vector field of degree $n$. Each box in the diagram above is represented here by a dot. Antidiagonals represent vertical vector fields of fixed degrees. Colored dots represents components of the vector fields which are not zero. Dashed lines make a link between components which are zero to any other adjacent component. The red antidiagonal symbolizes a rooted vector field $U$ of degree $n$ (of height $n+1$). The green antidiagonal represents a root-free vector field $ V$ of degree $n-2$ (of height $n$ and depth $2$). The blue antidiagonal represents a vector field $W$ of degree $n-1$ which is $-$ although rooted $-$ of depth $2$. Then necessarily $rt(W)=0$.}
  \label{figure1}
\end{figure}
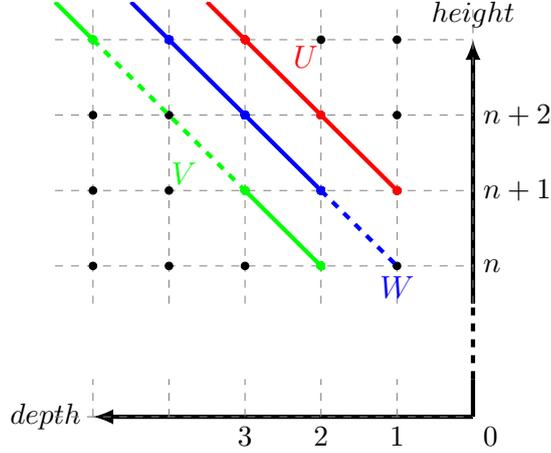

\bigskip
Now, it is clear that, when $(E,Q)$ is a Lie $\infty$-algebroid, and $Q^{(0)}$ the component of arity $0$ of $Q$
whose dual differential we denote by $\dd$ (as in Lemma \ref{dualQ0}),
then $X \mapsto [Q^{(0)},X]$ squares to zero and therefore makes vertical vector fields a complex.
This complex restricts to vertical vector fields of a given arity.
Moreover, upon decomposing vertical vector fields of a given arity as above, 
this operator is the total differential of a
bicomplex structure on ${\mathfrak{U}}^{(n-1)}$ with horizontal differential $\mathrm{id}\otimes\dd$ and a vertical differential $Q^{(0)}\otimes\mathrm{id}$:

\bigskip

\begin{center}
\begin{tikzpicture}
\draw[xstep=3.5cm,ystep=1,color=gray] (1,0) grid (14,3.5);
    \node at (2,0.5)  {$\cdots$};
    \node at (5.3,0.5)  {$S^n(E^*)_{n} \otimes E_{-3}$};
    \node at (8.8,0.5)  {$S^n(E^*)_{n} \otimes E_{-2}$};
    \node at (12.3,0.5)  {$S^n(E^*)_{n} \otimes E_{-1}$};
    \node at (2,1.5)  {$\cdots$};
    \node at (5.3,1.5)  {$S^n(E^*)_{n+1} \otimes E_{-3}$};
    \node at (8.8,1.5)  {$S^n(E^*)_{n+1} \otimes E_{-2}$};
    \node at (12.3,1.5)  {$S^n(E^*)_{n+1} \otimes E_{-1}$};
    \node at (2,2.5)  {$\cdots$};
    \node at (5.3,2.5)  {$S^n(E^*)_{n+2} \otimes E_{-3}$};
    \node at (8.8,2.5)  {$S^n(E^*)_{n+2} \otimes E_{-2}$};
    \node at (12.3,2.5)  {$S^n(E^*)_{n+2} \otimes E_{-1}$};
    \node at (5.3,3.5)  {$\cdots$};
    \node at (8.8,3.5)  {$\cdots$};
    \node at (12.3,3.5)  {$\cdots$};
    \node at (14.6,0.5)  {$0$};
    \node at (14.6,1.5)  {$0$};
    \node at (14.6,2.5)  {$0$};
    \node at (5.3,-0.5)  {$0$};
    \node at (8.8,-0.5)  {$0$};
    \node at (12.3,-0.5)  {$0$};
    
    \draw [-latex] (3,0.5) -- (3.9,0.5);
    \draw [-latex] (6.7,0.5) -- (7.4,0.5);
    \draw [-latex] (10.2,0.5) -- (10.9,0.5);
    \draw [-latex] (13.7,0.5) -- (14.4,0.5);
    \draw [-latex] (3,1.5) -- (3.8,1.5);
    \draw [-latex] (6.8,1.5) -- (7.3,1.5);
    \draw [-latex] (10.3,1.5) -- (10.8,1.5);
    \draw [-latex] (13.8,1.5) -- (14.4,1.5);
    \draw [-latex] (3,2.5) -- (3.8,2.5);
    \draw [-latex] (6.8,2.5) -- (7.3,2.5);
    \draw [-latex] (10.3,2.5) -- (10.8,2.5);
    \draw [-latex] (13.8,2.5) -- (14.4,2.5);
    \draw [-latex] (5.3,-0.2) -- (5.3,0.2);
    \draw [-latex] (5.3,0.8) -- (5.3,1.2);
    \draw [-latex] (5.3,1.8) -- (5.3,2.2);
    \draw [-latex] (5.3,2.8) -- (5.3,3.2);
    \draw [-latex] (8.8,-0.2) -- (8.8,0.2);
    \draw [-latex] (8.8,0.8) -- (8.8,1.2);
    \draw [-latex] (8.8,1.8) -- (8.8,2.2);
    \draw [-latex] (8.8,2.8) -- (8.8,3.2);
    \draw [-latex] (12.3,-0.2) -- (12.3,0.2);
    \draw [-latex] (12.3,0.8) -- (12.3,1.2);
    \draw [-latex] (12.3,1.8) -- (12.3,2.2);
    \draw [-latex] (12.3,2.8) -- (12.3,3.2);
\end{tikzpicture}
\end{center}

\noindent The following lemma then gives us the total differential on this bicomplex:

\begin{lemme}\label{lemmedifferentiel}
For every $n\geq0$, the space of vertical vector fields of arity $n$, equipped with the adjoint action $X\longmapsto[Q^{(0)},X]$, is, as a complex, isomorphic to the bicomplex $\mathfrak{ U}^{(n)}$, equipped with its total differential.
\end{lemme}

\begin{proof}
We have to show that the adjoint action $X\mapsto[Q^{(0)},X]$ is the sum of the horizontal map $\R{id} \otimes \R{d} $ and the vertical map $ Q^{(0)} \otimes \R{id}$. Let $i,j\geq1$ and consider an element $P\otimes \xi\in \Gamma(S^{n+1}(E^{\ast})_{i}\otimes E_{-j})$ and $u\in\Gamma(E^\ast)$. Recall the identity:
\begin{equation}\label{actiondual}
\big<u,\dd^{(j)}(\xi)\big>=(-1)^{j-1}\big<Q^{(0)}(u),\xi\big>
\end{equation}
which is non identically vanishing if and only if $u\in\Gamma(E_{-j+1}^{\ast})$. Recall as well as the identification between $\xi$ and the degree $-j$ derivation $\partial_{\xi}$ obtained in Equation \ref{innerder}. Thus we can indifferentely write $P\partial_{\xi}$ or $P\otimes\xi$ when thinking either from the vertical vector field point of view, or from the bicomplex point of view. Then we have for all $u\in\Gamma(E^{\ast})$:
\begin{align}
\big[Q^{(0)},P\partial_\xi\big][u]
&=Q^{(0)}(P\partial_\xi[u])-(-1)^{i-j}(P\partial_\xi) (Q^{(0)}(u))\\
&=Q^{(0)}(P\big<u,\xi\big>)-(-1)^{i-j}P \big<Q^{(0)}(u),\xi\big>\nonumber\\
&=Q^{(0)}(P)\big<u,\xi\big>+(-1)^{i}P \big<u,\dd(\xi)\big>\nonumber\\
&=Q^{(0)}(P)\partial_{\xi}[u]+(-1)^{i}P\partial_{\dd(\xi)}[u]\nonumber\\
&=\big((Q^{(0)}\otimes\mathrm{id}+\mathrm{id}\otimes\dd)\circ(P\otimes\xi)\big)(u)\nonumber
\end{align}
Thus, we indeed have that $[Q^{(0)},\ .\ ]$ can be identified with $Q^{(0)}\otimes\mathrm{id}+\mathrm{id}\otimes\dd$. The properties of the brackets of vertical vector fields implies that this action preserves the arity and increases the degree by one. Moreover:
\begin{equation}
(Q^{(0)}\otimes\mathrm{id}+\mathrm{id}\otimes\dd)^{2}=Q^{(0)}\otimes\dd+\mathrm{id}\otimes\dd\circ (Q^{(0)}\otimes\mathrm{id})=Q^{(0)}\otimes\dd-Q^{(0)}\otimes\dd=0
\end{equation}
which is equivalent to this identity on vector fields, for any $X\in\mathfrak{U}^{(n)}$:
\begin{equation}
\big[Q^{(0)},[Q^{(0)},X]\big]=\frac{1}{2}\big[\underbrace{[Q^{(0)},Q^{(0)}]}_{=\ 0},X\big]
\end{equation}
Thus it is a differential on $\mathfrak{U}^{(n)}$. 
\end{proof}


The vertical lines in the previous bicomplex may not be  exact, whereas the exactness of the sequence:
\begin{center}
\begin{tikzcd}[column sep=0.7cm,row sep=0.4cm]
\ldots\ar[r,"\dd"]&\Gamma(E_{-2})\ar[r,"\dd"]&\Gamma(E_{-1})\ar[r,"\rho"]&\mathcal{D}\ar[r]&0
\end{tikzcd}
\end{center}
implies that the horizontal lines are exact, except maybe at depth $1$. More precisely, by exactness of the short sequence  $\Gamma(E_{-2})\overset{\R{d}}{\longrightarrow} \Gamma(E_{-1})\overset{\rho}{\longrightarrow}\mathcal{D} $,  an element of height $i$ and depth $1$ is a coboundary if and only if its image under $\R{id} \otimes \rho$ is zero. By diagram chasing, this leads to the following fundamental lemma, which will be used thoroughly:

\begin{lemme}\label{lemmefondamental2}
Let  $n \geq 1$ be an integer, and consider the bicomplex $\big(\F{\mathfrak U}^{(n)},[Q^{(0)},\,.\ ]\big) $ of vertical vector fields of arity $n$. 
\begin{enumerate}
\item A root-free cocycle is a coboundary.
\item A cocycle whose root is in the kernel of $ \R{id} \otimes \rho$ is a coboundary.

\end{enumerate}
 
\end{lemme}

\begin{proof}
This is just a matter of classical arguments using the exactness of the lines of the augmented bicomplex.
Let us explain the first point. We let $p_{0}\geq2$ be the depth of $X$, that is the depth of its non zero component of lower height. Let $X$ be a cocycle of $\mathfrak{U}^{(n)}$ of degree $k$ lower than or equal to $n-1$. In that case its depth is stricly higher than 1. We can decompose this vertical vector field into its components of different depths (which are necessarily strictly greater than 1): $X=\sum_{p\geq p_{0}}X_{p}$. 
The vertical vector field $X$ being a cocycle, the commutator $[Q^{(0)},X]$ vanishes, and if we write it depthwise we obtain at the lower level (which is naturally at depth $p_{0}-1$):
\begin{equation}
0=[Q^{(0)},X]_{p_{0}-1}=(\mathrm{id}\otimes\dd)(X_{p_{0}})
\end{equation}

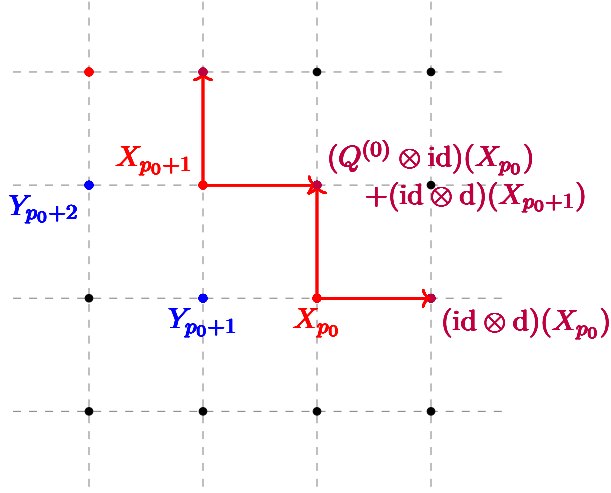
\begin{figure}[ht]
  \centering
  \begin{tikzpicture}[scale=0.50]
    \coordinate (Origin)   at (0,0);
    \coordinate (XAxisMin) at (10,-2);
    \coordinate (XAxisMax) at (0,-2);
    \coordinate (AAxisMin) at (9,-2);
    \coordinate (AAxisMax) at (7,-2);
    \coordinate (YAxisMin) at (10,1);
    \coordinate (YAxisMax) at (10,8);
    \coordinate (ZAxisMin) at (10,-2);
    \coordinate (ZAxisMax) at (10,-1);
    \coordinate (WAxisMin) at (10,-1);
    \coordinate (WAxisMax) at (10,1);

    \clip (-3,-3) rectangle (15cm,11cm); 
    \coordinate (Bone) at (4,2);
    \coordinate (Btwo) at (-6,12);
    \coordinate (B1) at (-7,11);
    \coordinate (B2) at (2,2);
    \coordinate (B3) at (0,4);
    \coordinate (B4) at (-4,8);
    \draw[style=help lines,dashed] (-2,-2) grid[step=3cm] (11,11);
    \foreach \x in {0,1,...,3}{
      \foreach \y in {0,1,...,3}{
        \node[draw,circle,inner sep=1pt,fill] at (3*\x,3*\y) {};
        \node[draw,circle,inner sep=1pt,fill,red] at (6,3){};
        \node [below,red] at (6,3) { $X_{p_{0}}$};
        \node[draw,circle,inner sep=1pt,fill,red] at (3,6){};
        \node [above left,red] at (3,6) { $X_{p_{0}+1}$};
         \node[draw,circle,inner sep=1pt,fill,red] at (0,9){};
        \node[draw,circle,inner sep=1pt,fill,blue] at (3,3){};
        \node [below,blue] at (3,3) { $Y_{p_{0}+1}$};
        \node[draw,circle,inner sep=1pt,fill,blue] at (0,6){};
        \node [below left,blue] at (0,6) { $Y_{p_{0}+2}$};
        \node[draw,circle,inner sep=1pt,fill,purple] at (9,3){};
        \node [below right,purple] at (9,3) { $(\mathrm{id}\otimes\dd)(X_{p_{0}})$};
        \draw [->,thick,red] (6,3) -- (9,3);
        \draw [->,thick,red] (6,3) -- (6,6);
        \draw [->,thick,red] (3,6) -- (6,6);
        \draw [->,thick,red] (3,6) -- (3,9);
        \node[draw,circle,inner sep=1pt,fill,purple] at (6,6){};
        \node [above right,purple] at (6,6) { $(Q^{(0)}\otimes\mathrm{id})(X_{p_{0}})$};
         \node [above right,purple] at (7,5) {$+(\mathrm{id}\otimes\dd)(X_{p_{0}+1})$};
         \node[draw,circle,inner sep=1pt,fill,purple] at (3,9){};
      }
    }

%
  \end{tikzpicture}
  \caption{\footnotesize Illustrated here is the process of building the first components of the vector field $Y$.}
  \label{figure12}
\end{figure}

Since $X$ is root-free, we know that $(\mathrm{id}\otimes\dd)(X_{p_{0}})$ has at least depth $p_{0}-1\geq1$. Since the differential $\dd$ is exact on the complex of sections of $E$, it implies that $X_{p_{0}}$ is a $(\mathrm{id}\otimes\dd)$-coboundary: there exists $Y_{p_{0}+1}\in\Gamma(S^{n+1}(E^{\ast})\otimes E_{-p_{0}-1})$ of depth $p_{0}+1$ such that
\begin{equation}
(\mathrm{id}\otimes\dd)(Y_{p_{0}+1})=X_{p_{0}}
\end{equation}
 Considering the term of depth $p_{0}$ we obtain
\begin{align}
0=[Q^{(0)},X]_{p_{0}}
&=(Q^{(0)}\otimes\mathrm{id})(X_{p_{0}})+(\mathrm{id}\otimes\dd)(X_{p_{0}+1})\\
&=(Q^{(0)}\otimes\mathrm{id})\circ(\mathrm{id}\otimes\dd)(Y_{p_{0}+1})+(\mathrm{id}\otimes\dd)(X_{p_{0}+1})\nonumber\\
&=(\mathrm{id}\otimes\dd)\Big(-(Q^{(0)}\otimes\mathrm{id})(Y_{p_{0}+1})+X_{p_{0}+1}\Big)\nonumber
\end{align}
The depth being higher or equal to 2, the exactness of the resolution implies that the term in parenthesis is an $(\mathrm{id}\otimes\dd)$-coboundary: there exists $Y_{p_{0}+2}\in\Gamma(S^{n+1}(E^{\ast})\otimes E_{-p_{0}-2})$ such that:
\begin{equation}
(\mathrm{id}\otimes\dd)(Y_{p_{0}+2})=-(Q^{(0)}\otimes\mathrm{id})(Y_{p_{0}+1})+X_{p_{0}+1}
\end{equation}
In particular, one can check that the commutator $[Q^{(0)},Y_{p_{0}+2}+Y_{p_{0}+1}]$ has components of depth $p_{0}$ and $p_{0}+1$ equal to that of $X$:
\begin{align}
[Q^{(0)},Y_{p_{0}+2}+Y_{p_{0}+1}]_{p_{0}}&=(\mathrm{id}\otimes\dd)(Y_{p_{0}+1})=X_{p_{0}}\\
[Q^{(0)},Y_{p_{0}+2}+Y_{p_{0}+1}]_{p_{0}+1}&=(\mathrm{id}\otimes\dd)(Y_{p_{0}+2})+(Q^{(0)}\otimes\mathrm{id})(Y_{p_{0}+1})=X_{p_{0}+1}
\end{align}
Using the same arguments at each step, we  build a vertical vector field $Y$ of depth $p_{0}+1$ such that $[Q^{(0)},Y]=X$. For the second item, the discussion is similar, except that we need that $(\mathrm{id}\otimes\rho)(X_{1})=0$ to construct the first element $Y_{2}$. The construction is then strictly identical. \end{proof}

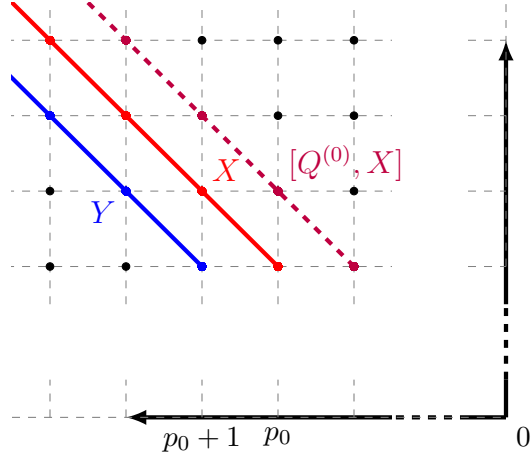
\begin{figure}[ht]
  \centering
  \begin{tikzpicture}[scale=0.50]
    \coordinate (Origin)   at (0,0);
    \coordinate (XAxisMin) at (10,-2);
    \coordinate (XAxisMax) at (0,-2);
    \coordinate (AAxisMin) at (9,-2);
    \coordinate (AAxisMax) at (7,-2);
    \coordinate (YAxisMin) at (10,1);
    \coordinate (YAxisMax) at (10,8);
    \coordinate (ZAxisMin) at (10,-2);
    \coordinate (ZAxisMax) at (10,-1);
    \coordinate (WAxisMin) at (10,-1);
    \coordinate (WAxisMax) at (10,1);
    \draw [ultra thick, black,-latex] (AAxisMax) -- (XAxisMax) node [left] {};
    \draw [ultra thick, black,-latex] (YAxisMin) -- (YAxisMax) node [above] {};
    \draw [ultra thick, black] (ZAxisMin) -- (ZAxisMax);
    \draw [ultra thick, black,dashed] (WAxisMin) -- (WAxisMax);
    \draw [ultra thick, black] (XAxisMin) -- (AAxisMin);
    \draw [ultra thick, black,dashed] (AAxisMin) -- (AAxisMax);

    \clip (-3,-3) rectangle (12cm,9cm); 
    \coordinate (Bone) at (4,2);
    \coordinate (Btwo) at (-6,12);
    \coordinate (B1) at (-7,11);
    \coordinate (B2) at (2,2);
    \coordinate (B3) at (0,4);
    \coordinate (B4) at (-4,8);
    \draw[style=help lines,dashed] (-4,1) grid[step=2cm] (7,9);
    \draw[style=help lines,dashed] (9,2) grid[step=2cm] (10,9);
    \draw (-3,-2) -- (10,-2)[dashed, gray];
    \draw (6,-2) -- (6,-1)[dashed, gray];
    \draw (4,-2) -- (4,-1)[dashed, gray];
    \draw (2,-2) -- (2,-1)[dashed, gray];
    \draw (0,-2) -- (0,-1)[dashed, gray];
    \draw (-2,-2) -- (-2,-1)[dashed, gray];
    \foreach \x in {-4,-3,...,3}{
      \foreach \y in {1,2,...,4}{
        \node[draw,circle,inner sep=1pt,fill] at (2*\x,2*\y) {};
        \node[draw,circle,inner sep=1pt,fill,red] at (Bone){};
        \node[draw,circle,inner sep=1pt,fill,red] at (0,6){};
        \node[draw,circle,inner sep=1pt,fill,red] at (2,4){};
        \node[draw,circle,inner sep=1pt,fill,red] at (-2,8){};
        \node[draw,circle,inner sep=1pt,fill,red] at (-4,10){};
        \node[draw,circle,inner sep=1pt,fill,blue] at (-6,10){};
        \node[draw,circle,inner sep=1pt,fill,blue] at (B2){};
        \node[draw,circle,inner sep=1pt,fill,blue] at (B3){};
        \node[draw,circle,inner sep=1pt,fill,blue] at (B4){};
        \node[draw,circle,inner sep=1pt,fill,blue] at (-2,6){};
        \node[draw,circle,inner sep=1pt,fill,purple] at (0,8){};
        \node[draw,circle,inner sep=1pt,fill,purple] at (2,6){};
        \node[draw,circle,inner sep=1pt,fill,purple] at (4,4){};
        \node[draw,circle,inner sep=1pt,fill,purple] at (6,2){};
      }
    }
    \node [below] at (4,-2)  {$p_{0}$};
    \node [below] at (2,-2)  {$p_{0}+1$};
    \node [right] at (10,2)  {};
    \node [right] at (10,4)  {};
    \node [right] at (10,6)  {};
    \node [below right] at (10,-2)  {$0$};
    
    \draw [ultra thick,red] (Btwo)
        -- (Bone);
    \node [above right,red] at (2,4) {\large $X$};
    \draw [ultra thick,blue] (B1)
        -- (B2);
    \node [below left,blue] at (0,4) {\large $Y$};
    
   \draw [ultra thick,purple, dashed] (-1,9) -- (6,2);
   \node [above right,purple] at (4,4) {\large $[Q^{(0)},X]$};
  \end{tikzpicture}
  \caption{\footnotesize $X$ being a root-free cocycle implies that it is actually a coboundary.}
  \label{figure11}
\end{figure}

\begin{remarque}
The first item implies that the cohomology of this bicomplex is zero in degree less than or equal to $ n - 1$.\end{remarque}

\subsection{The universal Lie \texorpdfstring{$\infty$}{infinity}-algebroid resolving a Hermann foliation}\label{preuveexistence}

We now intend to prove Theorem \ref{theo:existe}: that any resolution of a singular foliation can be equipped with a Lie $\infty$-algebroid structure. We chose to present a proof for smooth resolutions, but the same arguments will also work 
when working in the real analytic or holomorphic case in a neighborhood of a point.

According to Proposition \ref{prop:exres}, resolutions of a Hermann foliation always exist
for the holomorphic and real analytic cases in a neighborhood of a point, and these resolutions can be chosen to be finite
and by trivial vector bundles. In the smooth case, we have to assume that it exists. Thus, we have to prove that, given a resolution $(E,\dd,\rho)$ of $ {\mathcal D}$,
there is a homological degree $+1$ vector field $Q$ on the graded manifold $E$
whose linear part is the given resolution of ${\mathcal D}$. The proof of the theorem will focus on finding a homological degree $+1$ vector field $Q$ on $E\to M$ under the form:
\begin{equation}
Q=Q^{(0)}+Q^{(1)}+Q^{(2)}+\ldots
\end{equation}
where $Q^{(i)}$ is a vector field on $E\to M$ of degree 1 and arity $i$. Given the results in Proposition \ref{prop:arity}, the element $Q^{(i)}$ is vertical as soon as $i\neq1$ whereas $Q^{(1)}$ can carry both a vertical and a horizontal part. For degree reasons, there cannot be other kind of terms.

Le us start by defining $Q^{(0)}$: the differential $\dd$ can be dualized to give a differential $\R{d}^{\ast}$ on the dual of the resolution. As a morphism of vector bundles, it is $\Funct$-linear and can be extended to all of $\Gamma(S(E^{\ast}))$ by derivation. This extension is a vertical homological degree 1 vector field of arity 0 on the graded vector bundle $E$. Then we will define $Q^{(0)}$ as in \eqref{actiondual}:
\begin{equation}
Q^{(0)}=\R{d}^{\ast}
\end{equation}
Its action on elements of the dual is explicitely presented in Equation \eqref{actiondual}. The assumption that $ \R{d} $ squares to zero implies dually that $Q^{(0)}$ is a cohomological vector field, i.e. $[Q^{(0)},Q^{(0)}]=0$. 

The homological condition $[Q,Q]=0$ gives the following new equations:
\begin{align}
[Q^{(0)},Q^{(1)}]&=0\label{jacques0}\\
\forall\ n\geq2\hspace{1.3cm}[Q^{(0)},Q^{(n)}]&=-\frac{1}{2}\underset{i+j=n}{\sum_{1\leq i, j\leq n-1}} [Q^{(i)},Q^{(j)}]\label{jacques}
\end{align}
Thus, except for the first element $Q^{(1)}$, knowing all $Q^{(i)}$ up to level $n-1$ and showing that the right-hand side of \eqref{jacques} is a $[Q^{(0)},\cdot\,]$ cocycle would be sufficient to define $Q^{(n)}$, if this cocycle happens to be a coboundary. That will happen when the $[Q^{(0)},\cdot\,]$ cohomology is trivial. This enables us to define at each step a degree 1 vector fields $Q_{n}$ as
\begin{equation}
Q_n=\underset{0\leq i\leq n}{\sum}Q^{(i)}=Q_{n-1}+Q^{(n)}
\end{equation}
which has the following property: the commutator $[Q_{n},Q_{n}]$ is a sum of vector fields of respective arity higher than or equal to $n+1$. Eventually, in principle, defining formally $Q=\underset{n\rightarrow\infty}{\mathrm{lim}}Q_{n}$, we obtain at infinity: $[Q,Q]=0$. That is, there is a unique degree $+1$ homological vector field $Q$ such that its component of arity $i$ is $Q^{(i)}$.



\bigskip

Now we have everything in hand to go along the proof of the main theorem. Since we already have defined $Q^{(0)}$, the proof will then develop in three steps:
\begin{itemize}
\item Find $Q^{(1)}$
\item Find $Q^{(2)}$
\item Find $Q^{(n)}$ for every $n\geq3$
\end{itemize}

According to Proposition \ref{prop:Almost}, there exists an almost-Lie algebroid structure on $E_{-1}$ whose anchor is $\rho$.
According to Proposition \ref{lemmefondamental}, this almost Lie almost algebroid structure corresponds to a vector field $X$ on the graded manifold $ E_{-1}$,
which can be considered as a vector field of arity $1$ on $E$. It is tempting to set $Q^{(1)}=X$, but it is obviously not sufficient because $[Q^{(0)},X]$ is not zero as required. The vector field $[Q^{(0)},X]$ is a coboundary with respect to $Q^{(0)}$ in the space of vector fields. We will show that it is actually a coboundary in the space of \emph{vertical vector fields}, \emph{i.e.} there exists a vertical vector field $Y\in\mathfrak{U}^{(1)}$ such that
\begin{equation}
[Q^{(0)},X]=-[Q^{(0)},Y]
\end{equation}
which had been a priori not granted if $X$ had been a random vector field on $E$.
The property that $E_{-1}$ is an almost-Lie algebroid implies that $[X,X]$ is vertical (see Proposition \ref{lemmefondamental}).
Hence, $\big[Q^{(0)},[X,X]\big]$ is vertical as well, so for every $f\in\cinf({ M})$ we have:
\begin{equation}
0
=\frac{1}{2}\big[Q^{(0)},[X,X]\big](f)=\big[[Q^{(0)},X],X\big](f)=[Q^{(0)},X]\circ\rho^{\ast}(\dd_{\mathrm{dR}}f)\label{jenaiplusdidee}
\end{equation}
The operator $\rho^{\ast}\circ\dd_{\text{dR}}$ is a horizontal vector field, sending functions on $M$ to sections of $E_{-1}^{\ast}$. It fully encodes the action of the anchor map:
\begin{equation}\label{anchormapsuper}
\forall\ x\in E_{-1},\,f\in\Funct\hspace{1cm}\rho^{\ast}(\dd_{\mathrm{dR}}f)(x)=\dd_{\mathrm{dR}}f(\rho(x))=\rho(x)[f]
\end{equation}
Hence, Equation \eqref{jenaiplusdidee} is equivalent to the following condition in the bicomplex:
\begin{equation}\label{conditionQ1}
(\R{id}\otimes\rho)\circ\big(rt[Q^{(0)},X]\big)=0
\end{equation}
Thus the vertical vector field $[Q^{(0)},X]$ satisfies the hypothesis of Lemma \ref{lemmefondamental2}. Then there exists a vertical vector field $Y$ (see Figure \ref{figure2}) of arity 1 and depth greater than or equal to 2 such that:
\begin{equation}
[Q^{(0)},X]=-[Q^{(0)},Y]
\end{equation}
Letting $Q^{(1)}=X+Y$, we finally get Equation \eqref{jacques} for $n=2$:
\begin{equation}\label{jacques2}
[Q^{(0)},Q^{(1)}]=0
\end{equation}
The vertical vector field $Y$ generates the 2-brackets between two elements of the higher sections of $E$, whereas $X$ was limited to $E_{-1}$. 

\bigskip

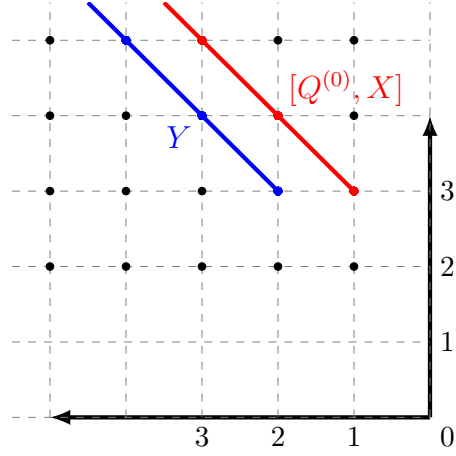
\begin{figure}[ht]
  \centering
  \begin{tikzpicture}[scale=0.50]
    \coordinate (Origin)   at (0,0);
    \coordinate (XAxisMin) at (10,0);
    \coordinate (XAxisMax) at (0,0);
    \coordinate (YAxisMin) at (10,0);
    \coordinate (YAxisMax) at (10,8);
    \coordinate (ZAxisMin) at (10,-2);
    \coordinate (ZAxisMax) at (10,-1);
    \coordinate (WAxisMin) at (10,-1);
    \coordinate (WAxisMax) at (10,1);
    \draw [ultra thick, black,-latex] (XAxisMin) -- (XAxisMax) node [left] {};
    \draw [ultra thick, black,-latex] (YAxisMin) -- (YAxisMax) node [above] {};

    \clip (-3,-1) rectangle (12cm,11cm); 
    \coordinate (Bone) at (8,6);
    \coordinate (Btwo) at (3,11);
    \coordinate (B1) at (1,11);
    \coordinate (B2) at (6,6);
    \draw[style=help lines,dashed] (-1,0) grid[step=2cm] (10,11);
    \foreach \x in {0,1,...,4}{
      \foreach \y in {2,3,...,5}{
        \node[draw,circle,inner sep=1pt,fill] at (2*\x,2*\y) {};
        \node[draw,circle,inner sep=1pt,fill,red] at (Bone){};
        \node[draw,circle,inner sep=1pt,fill,red] at (4,10){};
        \node[draw,circle,inner sep=1pt,fill,red] at (6,8){};
        \node[draw,circle,inner sep=1pt,fill,blue] at (6,6){};
        \node[draw,circle,inner sep=1pt,fill,blue] at (4,8){};
        \node[draw,circle,inner sep=1pt,fill,blue] at (2,10){};
      }
    }
    \node [below] at (8,0)  {$1$};
    \node [below] at (6,0)  {$2$};
    \node [below] at (4,0)  {$3$};
    \node [right] at (10,2)  {$1$};
    \node [right] at (10,4)  {$2$};
    \node [right] at (10,6)  {$3$};
    \node [below right] at (10,0)  {$0$};
    
    \draw [ultra thick,red] (Btwo)
        -- (Bone);
    \node [above right,red] at (6,8) {\large $[Q^{(0)},X]$};
    \draw [ultra thick,blue] (B1)
        -- (B2);
    \node [below left,blue] at (4,8) {\large $Y$};
  \end{tikzpicture}
  \caption{\footnotesize The commutator $[Q^{(0)},X]$ is obviously a coboundary in the space of all vector fields, but it is also a coboundary in the space of vertical vector fields, and $[Q^{(0)},X]=-[Q^{(0)},Y]$, with $Y$ vertical.}
  \label{figure2}
\end{figure}

\bigskip

\begin{remarque}
Condition \eqref{conditionQ1} is equivalent to the fact that the kernel of $\rho$ is stable under the adjoint action:
\begin{equation}
\forall\ x\in \Gamma(E_{-2}),\ y\in \Gamma(E_{-1})\hspace{1.5cm}\rho\big(\{\R{d}(x),y\}_{E_{-1}}\big)=0
\end{equation}
\end{remarque}

\bigskip

From now on we write $Q_{1}=Q^{(0)}+Q^{(1)}$ and we are interested in the commutator $[Q_{1},Q_{1}]$. By Equation \eqref{jacques2}, it is naturally equal to $[Q^{(1)},Q^{(1)}]$. In the decomposition $Q^{(1)}=X+Y$, the vector field $Y$ is vertical and has depth greater than or equal to 2, thus cannot act on $X$ in the commutator $[Q^{(1)},Q^{(1)}]$ because $X$ does not carry any element of degree $2$: this implies that $[X,Y]$ is vertical.
Since $[X,X]$ is vertical, this implies that $[Q^{(1)},Q^{(1)}]$ is.

Now, the vector field $[Q^{(1)},Q^{(1)}]$ is a $Q^{(0)}$-cocycle:
\begin{equation}
\big[Q^{(0)},\frac{1}{2}[Q^{(1)},Q^{(1)}]\big]=\big[[Q^{(0)},Q^{(1)}],Q^{(1)}\big]=0
\end{equation}
and we will show that it is in fact a coboundary in the bicomplex ${\mathfrak U}^{(2)}$. Recall that for odd vector fields, the relation $\big[[U,U],U\big]=0$ holds. For $U=Q^{(1)}$ and for every function $f\in\cinf({ M})$, we then have:
\begin{align}
0=\big[[Q^{(1)},Q^{(1)}],Q^{(1)}\big](f)=[Q^{(1)},Q^{(1)}]\circ\rho^{\ast}(\dd_{\mathrm{dR}}f)
\end{align}
because $[Q^{(1)},Q^{(1)}]$ is vertical. Using Equation \eqref{anchormapsuper}, we understand that the above equation is equivalent to the following one in the bicomplex ${\mathfrak U}^{(2)}$:
\begin{equation}\label{blabla}
(\R{id}\otimes\rho)\circ\big(rt[Q^{(1)},Q^{(1)}]\big)=0
\end{equation}
With the fact that $[Q^{(1)},Q^{(1)}]$ is a cocycle, item 2 in Lemma \ref{lemmefondamental2} implies that it is in fact a coboundary, i.e there exists a degree $1$ element $Q^{(2)}\in\F{\mathfrak  U}^{(2)}$ such that:
\begin{equation}
\frac{1}{2}[Q^{(1)},Q^{(1)}]=-[Q^{(0)},Q^{(2)}]
\end{equation}
that is, Equation \eqref{jacques} for $n=2.$

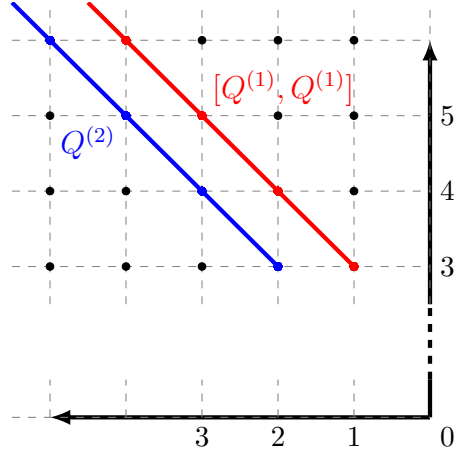
\begin{figure}[ht]
  \centering
  \begin{tikzpicture}[scale=0.50]
    \coordinate (Origin)   at (0,0);
    \coordinate (XAxisMin) at (10,-2);
    \coordinate (XAxisMax) at (0,-2);
    \coordinate (YAxisMin) at (10,1);
    \coordinate (YAxisMax) at (10,8);
    \coordinate (ZAxisMin) at (10,-2);
    \coordinate (ZAxisMax) at (10,-1);
    \coordinate (WAxisMin) at (10,-1);
    \coordinate (WAxisMax) at (10,1);
    \draw [ultra thick, black,-latex] (XAxisMin) -- (XAxisMax) node [left] {};
    \draw [ultra thick, black,-latex] (YAxisMin) -- (YAxisMax) node [above] {};
    \draw [ultra thick, black] (ZAxisMin) -- (ZAxisMax);
    \draw [ultra thick, black,dashed] (WAxisMin) -- (WAxisMax);

    \clip (-3,-3) rectangle (12cm,9cm); 
    \coordinate (Bone) at (8,2);
    \coordinate (Btwo) at (1,9);
    \coordinate (B1) at (-1,9);
    \coordinate (B2) at (6,2);
    \coordinate (B3) at (4,4);
    \coordinate (B4) at (0,8);
    \draw[style=help lines,dashed] (-1,1) grid[step=2cm] (10,9);
    \draw (-1,-2) -- (10,-2)[dashed, gray];
    \draw (8,-2) -- (8,-1)[dashed, gray];
    \draw (6,-2) -- (6,-1)[dashed, gray];
    \draw (4,-2) -- (4,-1)[dashed, gray];
    \draw (2,-2) -- (2,-1)[dashed, gray];
    \draw (0,-2) -- (0,-1)[dashed, gray];
    \foreach \x in {0,1,...,4}{
      \foreach \y in {1,2,...,4}{
        \node[draw,circle,inner sep=1pt,fill] at (2*\x,2*\y) {};
        \node[draw,circle,inner sep=1pt,fill,red] at (Bone){};
        \node[draw,circle,inner sep=1pt,fill,red] at (4,6){};
        \node[draw,circle,inner sep=1pt,fill,red] at (6,4){};
        \node[draw,circle,inner sep=1pt,fill,red] at (2,8){};
        \node[draw,circle,inner sep=1pt,fill,blue] at (B2){};
        \node[draw,circle,inner sep=1pt,fill,blue] at (B3){};
        \node[draw,circle,inner sep=1pt,fill,blue] at (B4){};
        \node[draw,circle,inner sep=1pt,fill,blue] at (2,6){};
      }
    }
    \node [below] at (8,-2)  {$1$};
    \node [below] at (6,-2)  {$2$};
    \node [below] at (4,-2)  {$3$};
    \node [right] at (10,2)  {$3$};
    \node [right] at (10,4)  {$4$};
    \node [right] at (10,6)  {$5$};
    \node [below right] at (10,-2)  {$0$};
    
    \draw [ultra thick,red] (Btwo)
        -- (Bone);
    \node [above right,red] at (4,6) {\large $[Q^{(1)},Q^{(1)}]$};
    \draw [ultra thick,blue] (B1)
        -- (B2);
    \node [below left,blue] at (2,6) {\large $Q^{(2)}$};
  \end{tikzpicture}
  \caption{\footnotesize The commutator $[Q^{(1)},Q^{(1)}]$ is a cocycle in the space of vertical vector fields, whose root lies in the kernel of the anchor map. As such it is a coboundary, and $\frac{1}{2}[Q^{(1)},Q^{(1)}]=-[Q^{(0)},Q^{(2)}]$}
  \label{figure3}
\end{figure}

\bigskip

\begin{remarque}
The commutator $[Q^{(1)},Q^{(1)}]$ corresponds by duality to the Jacobiator of the 2-bracket. In particular it may not identically vanish (this is why $Q^{(2)}$ is needed). However the Jacobi identity of the bracket $[\ .\ ,\, .\ ]_{E_{-1}}$ on the sections of $E_{-1}$ should lie in the kernel of the anchor map because the Jacobi identity is satisfied on the tangent space:
\begin{equation}
\forall\ x,y,z\in\Gamma(E_{-1})\hspace{1.3cm}\rho\big(Jac(x,y,z)\big)=0
\end{equation}
where $Jac(x,y,z)$ is the Jacobiator of the bracket on $E_{-1}$. This result is precisely the dual of Equation \eqref{blabla}.
\end{remarque}

\bigskip

 We will quickly handle the case $n=3$ to simplify the discussion of the induction for $n\geq4$.  Let us write $Q_{2}$ for the sum $Q^{(0)}+Q^{(1)}+Q^{(2)}$. Then
\begin{equation}
[Q_{2},Q_{2}]=2[Q^{(1)},Q^{(2)}]+[Q^{(2)},Q^{(2)}]
\end{equation}
by the above identities. By construction, $Q^{(2)}$ has depth greater than or equal to $2$, then the commutator $[Q^{(2)},Q^{(2)}]$ is vertical, living in ${\mathfrak U}^{(4)}$. For the same reason, the commutator $[Q^{(2)},X]$ is vertical as well, and so is $[Q^{(1)},Q^{(2)}]$. This term is a cocycle in the bicomplex ${\mathfrak U}^{(3)}$:
\begin{equation}\label{alerterouge}
\big[Q^{(0)},[Q^{(1)},Q^{(2)}]\big]=\big[\underbrace{[Q^{(0)},Q^{(1)}]}_{=\, 0},Q^{(2)}\big]+\frac{1}{2}\big[Q^{(1)},[Q^{(1)},Q^{(1)}]\big]
\end{equation}
The triple commutator of a vector field vanishes so that both sides of \eqref{alerterouge} are equal to zero. Since this cocycle has arity $3$ and degree $2$, it is root-free (of depth at least 2) and then by item 1 of Lemma \ref{lemmefondamental2} it is a coboundary. Thus there exists a degree 1 element $Q^{(3)}$ of arity 3 and depth greater than or equal to 3 such that:
\begin{equation}
[Q^{(1)},Q^{(2)}]=-[Q^{(0)},Q^{(3)}]
\end{equation}
which is Equation \eqref{jacques} for $n=3$.

Now assume that we have built all $Q^{(i)}$ satisfying Equations (\ref{jacques}) up to some order $n\geq3$, and let $Q_{n}=\underset{0\leq i\leq n}{\sum}Q^{(i)}$. Then we obtain that:
\begin{equation}
[Q_{n},Q_{n}]=\underset{i+j=n+1}{\sum_{1\leq i, j\leq n}} [Q^{(i)},Q^{(j)}]\ +\ \ldots
\end{equation}
where the suspension points stand for elements which have arity strictly higher than $n+1$. On the other hand the visible sum has arity $n+1$, and we choose to denote it by $D_{n+1}$. By construction $Q^{(n)}$ has depth at least $n\geq2$, thus the commutator $[Q^{(n)},X]$ is vertical. Since every other commutator in the sum is vertical, so is $D_{n+1}$. Moreover the bracket $[Q^{(i)},Q^{(n+1-i)}]$ has height $n+2$ and degree 2 then $D_{n+1}$ is a root-free element of the bicomplex $\mathfrak{U}^{(n+1)}$ of depth at least $n$. It is also a cocycle:
\begin{align}
\frac{1}{2}[Q^{(0)},D_{n+1}]\label{cerbere}
&=\underset{i+j=n+1}{\sum_{1\leq i,j\leq n}}\big[[Q^{(0)},Q^{(i)}],Q^{(j)}\big]\\
&=\big[\underbrace{[Q^{(0)},Q^{(1)}]}_{=\,0},Q^{(n)}\big]+\big[[Q^{(0)},Q^{(n)}],Q^{(1)}\big]+\underset{i+j=n+1}{\sum_{2\leq i,j\leq n-1}}\big[[Q^{(0)},Q^{(i)}],Q^{(j)}\big]\nonumber\\
&=-\underset{k+l=n}{\sum_{1\leq k,l\leq n-1}}\big[[Q^{(k)},Q^{(l)}],Q^{(1)}\big]-\underset{i+j=n+1}{\sum_{2\leq i,j\leq n-1}}\quad\underset{k+l=i}{\sum_{1\leq k,l\leq i-1}}\big[[Q^{(k)},Q^{(l)}],Q^{(j)}\big]\nonumber\\
&=-\underset{i+j=n+1}{\sum_{2\leq i,j+1\leq n}}\quad\underset{k+l=i}{\sum_{1\leq k,l\leq i-1}}\big[[Q^{(k)},Q^{(l)}],Q^{(j)}\big]\nonumber\\
&=-\underset{2\leq i\leq n}{\sum}\quad\underset{i+j=n+1}{\sum_{1\leq j\leq n-1}}\quad\underset{k+l=i}{\sum_{1\leq k,l\leq n-1}}\big[[Q^{(k)},Q^{(l)}],Q^{(j)}\big]\nonumber\\
&=-\underset{j+k+l=n+1}{\sum_{1\leq j,k,l\leq n-1}}\big[[Q^{(k)},Q^{(l)}],Q^{(j)}\big]\nonumber
\end{align}
Using the Jacobi identity for the bracket, we obtain:
\begin{align}
\frac{1}{2}[Q^{(0)},D_{n+1}]&=-\underset{j+k+l=n+1}{\sum_{1\leq j,k,l\leq n-1}}\big[[Q^{(k)},Q^{(l)}],Q^{(j)}\big]\\
&=\underset{j+k+l=n+1}{\sum_{1\leq j,k,l\leq n-1}}\big[[Q^{(k)},Q^{(j)}],Q^{(l)}\big]-\underset{j+k+l=n+1}{\sum_{1\leq j,k,l\leq n-1}}\big[Q^{(k)},[Q^{(l)},Q^{(j)}]\big]\nonumber\\
&=\underset{j+k+l=n+1}{\sum_{1\leq j,k,l\leq n-1}}2\big[[Q^{(k)},Q^{(l)}],Q^{(j)}\big]\nonumber\\
&=-[Q^{(0)},D_{n+1}]\nonumber
\end{align}
so that $[Q^{(0)},D_{n+1}]=0$. By item 1 of Lemma \ref{lemmefondamental2}, the root-free cocycle $D_{n+1}$ is automatically a coboundary: there exists an element $Q^{(n+1)}$ of arity $n+1$ and depth $n+1$ such that:
\begin{equation}\label{dn}
\frac{1}{2}\underset{i+j=n+1}{\sum_{1\leq i, j\leq n}} [Q^{(i)},Q^{(j)}]=-[Q^{(0)},Q^{(n+1)}]
\end{equation}
that is precisely Equation \eqref{jacques} for $n+1$.

\bigskip

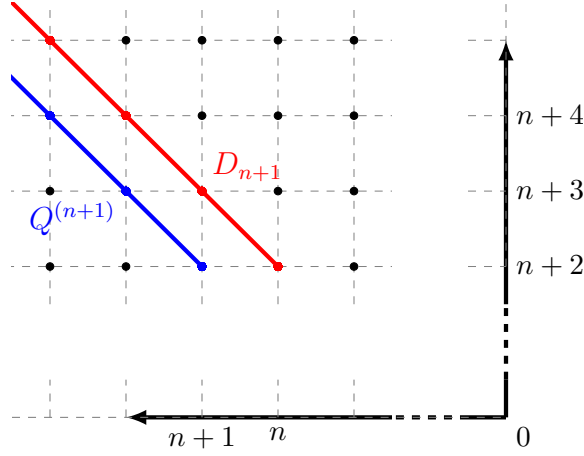
\begin{figure}[ht]
  \centering
  \begin{tikzpicture}[scale=0.50]
    \coordinate (Origin)   at (0,0);
    \coordinate (XAxisMin) at (10,-2);
    \coordinate (XAxisMax) at (0,-2);
    \coordinate (AAxisMin) at (9,-2);
    \coordinate (AAxisMax) at (7,-2);
    \coordinate (YAxisMin) at (10,1);
    \coordinate (YAxisMax) at (10,8);
    \coordinate (ZAxisMin) at (10,-2);
    \coordinate (ZAxisMax) at (10,-1);
    \coordinate (WAxisMin) at (10,-1);
    \coordinate (WAxisMax) at (10,1);
    \draw [ultra thick, black,-latex] (AAxisMax) -- (XAxisMax) node [left] {};
    \draw [ultra thick, black,-latex] (YAxisMin) -- (YAxisMax) node [above] {};
    \draw [ultra thick, black] (ZAxisMin) -- (ZAxisMax);
    \draw [ultra thick, black,dashed] (WAxisMin) -- (WAxisMax);
    \draw [ultra thick, black] (XAxisMin) -- (AAxisMin);
    \draw [ultra thick, black,dashed] (AAxisMin) -- (AAxisMax);

    \clip (-3,-3) rectangle (12cm,9cm); 
    \coordinate (Bone) at (4,2);
    \coordinate (Btwo) at (-6,12);
    \coordinate (B1) at (-7,11);
    \coordinate (B2) at (2,2);
    \coordinate (B3) at (0,4);
    \coordinate (B4) at (-4,8);
    \draw[style=help lines,dashed] (-4,1) grid[step=2cm] (7,9);
    \draw[style=help lines,dashed] (9,2) grid[step=2cm] (10,9);
    \draw (-3,-2) -- (10,-2)[dashed, gray];
    \draw (6,-2) -- (6,-1)[dashed, gray];
    \draw (4,-2) -- (4,-1)[dashed, gray];
    \draw (2,-2) -- (2,-1)[dashed, gray];
    \draw (0,-2) -- (0,-1)[dashed, gray];
    \draw (-2,-2) -- (-2,-1)[dashed, gray];
    \foreach \x in {-4,-3,...,3}{
      \foreach \y in {1,2,...,4}{
        \node[draw,circle,inner sep=1pt,fill] at (2*\x,2*\y) {};
        \node[draw,circle,inner sep=1pt,fill,red] at (Bone){};
        \node[draw,circle,inner sep=1pt,fill,red] at (0,6){};
        \node[draw,circle,inner sep=1pt,fill,red] at (2,4){};
        \node[draw,circle,inner sep=1pt,fill,red] at (-2,8){};
        \node[draw,circle,inner sep=1pt,fill,red] at (-4,10){};
        \node[draw,circle,inner sep=1pt,fill,blue] at (-6,10){};
        \node[draw,circle,inner sep=1pt,fill,blue] at (B2){};
        \node[draw,circle,inner sep=1pt,fill,blue] at (B3){};
        \node[draw,circle,inner sep=1pt,fill,blue] at (B4){};
        \node[draw,circle,inner sep=1pt,fill,blue] at (-2,6){};
      }
    }
    \node [below] at (4,-2)  {$n$};
    \node [below] at (2,-2)  {$n+1$};
    \node [right] at (10,2)  {$n+2$};
    \node [right] at (10,4)  {$n+3$};
    \node [right] at (10,6)  {$n+4$};
    \node [below right] at (10,-2)  {$0$};
    
    \draw [ultra thick,red] (Btwo)
        -- (Bone);
    \node [above right,red] at (2,4) {\large $D_{n+1}$};
    \draw [ultra thick,blue] (B1)
        -- (B2);
    \node [below left,blue] at (0,4) {\large $Q^{(n+1)}$};
  \end{tikzpicture}
  \caption{\footnotesize $D_{n+1}$ is a root-free cocycle in the space of vertical vector fields, and as such it is necessarily a coboundary: there exists $Q^{(n+1)}\in\mathfrak{\mathfrak U}^{(n+1)}$ such that $\frac{1}{2}D_{n+1}=-[Q^{(0)},Q^{(n+1)}]$.}
  \label{figure4}
\end{figure}

\bigskip

Consequently, by construction the vector field $Q_{n+1}=\underset{0\leq i\leq n+1}{\sum}Q^{(i)}$ now satisfies all equations in \eqref{jacques} up to order $n+1$:
\begin{equation}
[Q_{n+1},Q_{n+1}]=\underset{i+j=n+2}{\sum_{1\leq i, j\leq n+1}} [Q^{(i)},Q^{(j)}]+\ldots
\end{equation}
The first term is of arity $n+2$ and the suspension points stand for elements which have arity stricly higher than $n+2$. The induction can then work as well at level $n+2$. Step by step we can construct a vector field as a formal sum $Q=\sum_{i=0}^{\infty}Q^{(i)}$ such that $[Q,Q]=0$ as desired.

\begin{remarque}
At every step, we have chosen an element $Q^{(i)}$ of arity $i$, but there is some latitude in this choice since it is defined up to a $Q^{(0)}$ closed (and hence, exact) term. In the end, this leads to some freedom in the choice of the Lie $\infty$-algebroid structure. Any two such structures happen to be isomorphic up to homotopy, as shown in the next section.
\end{remarque}

Here are two examples of Lie $\infty$-algebroid structures on a resolution of a Hermann foliation.

\begin{example}
Let ${\mathcal D}$ be the Hermann foliation on $M:={\mathbb R}^2$, equipped with coordinates $x,y$, defined by all vector fields of the form $f(x,y)\frac{\partial}{\partial x} $
with $f(x,y)$ a function that vanishes at order $2$ at the origin (that is, such that $f(0,0)= \frac{\partial f}{\partial x}(0,0)=\frac{\partial f}{\partial y}(0,0)=0$).
The Hermann foliation ${\mathcal D}$ is generated by the vector fields:
\begin{equation}{\mathcal V}_{y^2}:= y^2 \frac{\partial}{\partial x},\quad{\mathcal V}_{xy}:=xy \frac{\partial}{\partial x} , \quad {\mathcal V}_{x^2}=x^2 \frac{\partial}{\partial x}\end{equation}
A resolution of this foliation is given by choosing $E_{-1}$ (resp. $E_{-2}$) to be a trivial bundle of rank $3$ (resp. $2$),
with canonical generators denoted respectively by $ (e_{x^2},e_{xy}, e_{y^2})$ (resp. $(f_{x^2,xy},f_{xy,y^2})$) together with the anchor  
\begin{equation}  \rho({e}_{y^2}) ={\mathcal V}_{y^2},\quad  \rho({e}_{xy}) ={\mathcal V}_{xy} \quad\hbox{ and } \quad  \rho({e}_{x^2}) ={\mathcal V}_{x^2}  \end{equation}
and the differential
\begin{equation} d^{(2)}( f_{xy,y^2}) = x {e}_{y^2} - y {e}_{xy}\hspace{1cm} \hbox{ and }\hspace{1cm}
 d^{(2)}( f_{x^2,xy}) = x {e}_{xy} -  y {e}_{x^2}. \end{equation}
 We then require $E_{-i}=0$ for all $i \geq 3$. Now, an easy computation gives:
 \begin{equation} [{\mathcal V}_{y^2},{\mathcal V}_{x^2}] = 2 y {\mathcal V}_{xy}
 , \quad[{\mathcal V}_{y^2},{\mathcal V}_{xy}] = y {\mathcal V}_{y^2}
 , \quad[{\mathcal V}_{x^2},{\mathcal V}_{xy}] = - y {\mathcal V}_{x^2} \end{equation}
which makes it natural to impose similar relations on the generators
$e_{y^2},e_{xy},e_{x^2}$ of $\Gamma(E_{-1})$:
 \begin{equation}\label{eq;choix}
   \{{e}_{y^2},{e}_{x^2}\} = 2y {e}_{xy}
 , \quad\{{e}_{y^2},{e}_{xy}\} = y {e}_{y^2}
 , \quad\{{e}_{x^2},{e}_{xy}\} =  -y  e_{x^2} \end{equation}
This bracket satisfies the Jacobi identity. The foliation therefore comes from a Lie algebroid action.

Consider now ${\mathcal D}'$ to be the Hermann foliation of all vector fields vanishing at order $2$ at the origin.
${\mathcal D}'$ is obtained by adding to the previous generators of ${\mathcal D}$  the following family of generators:
\begin{equation}{\mathcal W}_{y^2}:= y^2 \frac{\partial}{\partial y},\quad{\mathcal W}_{xy}:=xy \frac{\partial}{\partial y} ,\quad {\mathcal W}_{x^2}=x^2 \frac{\partial}{\partial y}\end{equation}
Since, as a module over functions on $M$, the foliation ${\mathcal D}'$ is isomorphic to two copies of ${\mathcal D}$ ,  it
still admits resolutions of length $2$, which in degrees $-1$ and $-2$ are trivial vector bundles of rank $6$ and $4$ respectively.
In addition to the generators of the previous Hermann foliation, this new resolution
admits $5$ additional generators denoted respectively by $ (f_{x^2},f_{xy}, f_{y^2})$ (resp. $(g_{x^2,xy},g_{xy,y^2})$) in degree $-1$ (resp. $-2$). We define the anchor by: 
\begin{equation}  \rho({f}_{y^2}) ={\mathcal W}_{y^2},\quad   \rho({f}_{xy}) ={\mathcal W}_{xy} \quad\hbox{ and }\quad \rho({f}_{x^2}) ={\mathcal W}_{x^2} . \end{equation}
We impose relations on the brackets similar to those of Equation (\ref{eq;choix}):
 \begin{equation}
   \{{f}_{y^2},{f}_{x^2}\} = 2x {f}_{xy}
 ,\quad \{{f}_{y^2},{f}_{xy}\} = x {f}_{y^2}
 ,\quad \{{f}_{x^2},{f}_{xy}\} =  -x  f_{x^2}. \end{equation}
We still have to define the Lie brackets of the type ${e}_{y^2},{f}_{xy}$. At this point, there is no natural choice, as long as an almost-Lie algebroid structure is obtained , and there does not seem to exist a manner to construct them that would allow the $3$-ary bracket $\Gamma(\wedge^3 E_{-1}) \to \Gamma(E_{-2})$ to be equal to zero.
\end{example}


\subsection{Universality of the Lie \texorpdfstring{$\infty$}{infinity}-algebroid resolving a foliation }\label{morphismfoliations}


In this section we prove Theorem \ref{theo:onlyOne}. The second item is a simple corollary of the first one. Again, we prove the theorem in the smooth case. The real analytic and holomorphic cases are similar, if we restrict to a neighborhood of a point.

Assume that we are given a Hermann foliation $\C{D}$ that admits a resolution $(E,\dd,\rho)$.
By Theorem \ref{theo:existe}, the resolution $(E,\dd,\rho)$ 
can be endowed with a Lie $\infty$-algebroid structure
with linear part $(E,\dd,\rho)$. 
Let $(F,Q_{F})$ be a Lie $\infty$-algebroid whose induced Hermann foliation $\mathcal{D}'$ is a sub-foliation of $ {\mathcal D}$, that is: $$\rho'\big(\Gamma(F_{-1})\big) \subset \rho\big(\Gamma(E_{-1})\big)= \mathcal{D}$$ and let $(F,\dd',\rho')$ be its linear part. It is not necessarily a resolution of the sub-foliation.

Since the image of $\rho'$ is a subset of the image of $\rho$, for any section $\sigma\in\Gamma(F_{-1})$,we have $\rho'(\sigma)\in\mathrm{Im}(\rho)$ and then there exists a section $\tau\in\Gamma(E_{-1})$ which is a preimage of $\rho'(\sigma)$ by $\rho$. The choice of $\tau$ depends smoothly on $\sigma$, therefore there exists a bundle map $\phi_{1}:F_{-1}\to E_{-1}$ covering the identity of $M$ such that the following diagram is commutative:
\begin{center}
\begin{tikzcd}[column sep=0.7cm,row sep=0.4cm]
F_{-1}\ar[dd,"\phi_{1}"]\ar[dr,"\rho'"]&\\
&T M\\
E_{-1}\ar[ur,"\rho"]&\\
\end{tikzcd}
\end{center}
Since $\rho'\circ\mathrm{d'^{(2)}}=0$ and since $\mathrm{Im}(\rho')\subset\mathrm{Im}(\rho)$, we have $\phi_{1}\big(\mathrm{Im}(\dd'^{(2)})\big)\subset\mathrm{Ker}(\rho)$. Since $(E,d,\rho)$ is a resolution of the Hermann foliation $\mathcal{D}$, this implies $\phi_{1}\big(\mathrm{Im}(\dd'^{(2)})\big)\subset\mathrm{Im}(\dd^{(2)})$. Hence there is a bundle map $\phi_{2}:F_{-2}\to E_{-2}$ such that the following diagram commutes:
\begin{center}
\begin{tikzcd}[column sep=0.9cm,row sep=0.6cm]
F_{-2}\ar[dd,"\phi_{2}"]\ar[r,"\dd'^{(2)}"]&F_{-1}\ar[dd,"\phi_{1}"]\\
&\\
E_{-2}\ar[r,"\dd^{(2)}"]&E_{-1}\\
\end{tikzcd}
\end{center}
By the above diagram, we know that $\dd^{(2)}\circ\phi_{2}\circ\dd'^{(3)}=\phi_{1}\circ\dd'^{(2)}\circ\dd'^{(3)}=0$. This implies $\phi_{2}\big(\mathrm{Im}(\dd'^{(3)})\big)\subset\mathrm{Ker}(\dd^{(2)})$, which is equivalent to $\phi_{2}\big(\mathrm{Im}(\dd'^{(3)})\big)\subset\mathrm{Im}(\dd^{(3)})$ because $(E,\dd)$ is a resolution. Therefore there exists a bundle map $\phi_{3}:F_{-3}\to E_{-3}$ such that the following diagram is commutative:
\begin{center}
\begin{tikzcd}[column sep=0.9cm,row sep=0.6cm]
F_{-3}\ar[dd,"\phi_{3}"]\ar[r,"\dd'^{(3)}"]&F_{-2}\ar[dd,"\phi_{2}"]\\
&\\
E_{-3}\ar[r,"\dd^{(3)}"]&E_{-2}\\
\end{tikzcd}
\end{center}
By recursion, we can therefore construct a family of bundle maps which intertwine $\mathrm{d}$ and $\mathrm{d}'$. Summing all the $\phi_{i}$ together, we obtain a bundle map in the category of graded manifolds $\phi:F\to E$ preserving the degree, which acts as a chain map:
\begin{equation}\label{mapresolution}
\phi \circ\mathrm{d}'=\mathrm{d}\circ\phi
\end{equation}
We summarize the situation by the following commutative diagram:
\begin{center}
\begin{tikzcd}[column sep=0.9cm,row sep=0.6cm]
\dots  \ar[r,"\dd'^{(4)}"] &F_{-3} \ar[dd,"\phi_{3}"] \ar[r,"\dd'^{(3)}"]&F_{-2}\ar[dd,"\phi_{2}"] \ar[r,"\dd'^{(2)}"] &F_{-1}\ar[dd,"\phi_{1}"] \ar[r,"\rho' "] &TM\ar[dd,"id"]\\
 & & & &\\
\dots  \ar[r,"\dd^{(4)}"]  &E_{-3} \ar[r,"\dd^{(3)}"]&E_{-2} \ar[r,"\dd^{(2)}"]  &E_{-1} \ar[r,"\rho"] &TM \\
\end{tikzcd}
\end{center}
\begin{remarque}
Note that these are so far general facts for $\mathscr{O}$-modules.
\end{remarque}

We denote by $\widehat{\Phi}^{(0)}:E^{\ast}\to F^{\ast}$ the dual map of $\phi$, which is obviously of arity 0. We extend it to a morphism of algebras $\Phi^{(0)}:\funct\to\functt$ using Equation \eqref{baboum}. The chain map condition then translates as:
\begin{equation}
Q_{F}^{(0)}\Phi^{(0)}=\Phi^{(0)}Q_{E}^{(0)}
\end{equation}
However it may not be a Lie $\infty$-morphism, i.e. it may not intertwine the homological vector fields $Q_{F}$ and $Q_{E}$. By construction, $\Phi^{(0)}$ also satisfies the following identity at level $-1$:
\begin{equation}\label{commutativity}
\rho'^{\ast}=\Phi^{(0)}\circ\rho^{\ast}
\end{equation}
which can be seen as the commutativity of the following diagram:
\begin{center}
\begin{tikzcd}[column sep=0.9cm,row sep=0.6cm]
F_{-1}\ar[dd,"\phi_{1}"]\ar[r,"\rho'"]&T M\ar[dd,"\mathrm{id}"]\\
&\\
E_{-1}\ar[r,"\rho"]&T M\\
\end{tikzcd}
\end{center}

%
%



We will mostly follow the same method as in the proof of the last section. We expect that $\Phi$ formally involves linear maps of various arities: $\Phi=\sum_{i\geq0}\Phi^{(i)}$ where each component $\Phi^{(i)}$ is defined by its restriction to  the sections of $E^{\ast}$, that is $\widehat{\Phi}^{(i)}:\Gamma(E^{\ast})\to \Gamma(S^{i+1}(F^{\ast}))$, and extended to $\funct$ using Equations \eqref{baboum}-\eqref{biboum} and all other components of lower arities. Then $\Phi^{(i)}$ can be formally seen as en element of $\Gamma(S^{i+1}(F^{\ast})\otimes E)$ of degree 0. We start at arity zero with $\widehat{\Phi}^{(0)}=\phi^{\ast}$, the dual map of the chain map, and extend it to all of $\funct$ by requiring that it be a morphism of algebras. We can build each of the $\Phi^{(i)}$ step by step using the same logic as in the proof of the main theorem. Indeed, for every $n\geq2$, $\mathscr{O}$-linear maps from $\funct$ to $\functt$ of arity $n-1$ can be represented as elements of the direct sum
\begin{equation*}
\mathfrak{V}^{(n-1)}=\bigoplus_{k=-\infty}^{+\infty}\underset{i,j\geq1}{\bigoplus_{i-j=k}}\Gamma\Big(S^{n}(F^{\ast})_{i}\otimes E_{-j}\Big)
\end{equation*}
Again, we say that an element in $S^{n}(F^*)_l \otimes E_{-k} $ is of \emph{depth $k$
and height $l$}, and that the root $rt$ of an element of arity $n-1$ is its component of depth 1.

The proof of the existence part of the first item of Theorem \ref{theo:onlyOne} relies on a variation of Lemma \ref{lemmefondamental2}. For all $n \geq 1$, there is a natural bicomplex structure on $\mathfrak{V}^{(n-1)}$:

\bigskip

\begin{center}
\begin{tikzpicture}
\draw[xstep=3.5cm,ystep=1,color=gray] (1,0) grid (14,3.5);
    \node at (2,0.5)  {$\cdots$};
    \node at (5.3,0.5)  {$S^n(F^*)_{n} \otimes E_{-3}$};
    \node at (8.8,0.5)  {$S^n(F^*)_{n} \otimes E_{-2}$};
    \node at (12.3,0.5)  {$S^n(F^*)_{n} \otimes E_{-1}$};
    \node at (2,1.5)  {$\cdots$};
    \node at (5.3,1.5)  {$S^n(F^*)_{n+1} \otimes E_{-3}$};
    \node at (8.8,1.5)  {$S^n(F^*)_{n+1} \otimes E_{-2}$};
    \node at (12.3,1.5)  {$S^n(F^*)_{n+1} \otimes E_{-1}$};
    \node at (2,2.5)  {$\cdots$};
    \node at (5.3,2.5)  {$S^n(F^*)_{n+2} \otimes E_{-3}$};
    \node at (8.8,2.5)  {$S^n(F^*)_{n+2} \otimes E_{-2}$};
    \node at (12.3,2.5)  {$S^n(F^*)_{n+2} \otimes E_{-1}$};
    \node at (5.3,3.5)  {$\cdots$};
    \node at (8.8,3.5)  {$\cdots$};
    \node at (12.3,3.5)  {$\cdots$};
    \node at (14.6,0.5)  {$0$};
    \node at (14.6,1.5)  {$0$};
    \node at (14.6,2.5)  {$0$};
    \node at (5.3,-0.5)  {$0$};
    \node at (8.8,-0.5)  {$0$};
    \node at (12.3,-0.5)  {$0$};
    
    \draw [-latex] (3,0.5) -- (3.9,0.5);
    \draw [-latex] (6.7,0.5) -- (7.4,0.5);
    \draw [-latex] (10.2,0.5) -- (10.9,0.5);
    \draw [-latex] (13.7,0.5) -- (14.4,0.5);
    \draw [-latex] (3,1.5) -- (3.8,1.5);
    \draw [-latex] (6.8,1.5) -- (7.3,1.5);
    \draw [-latex] (10.3,1.5) -- (10.8,1.5);
    \draw [-latex] (13.8,1.5) -- (14.4,1.5);
    \draw [-latex] (3,2.5) -- (3.8,2.5);
    \draw [-latex] (6.8,2.5) -- (7.3,2.5);
    \draw [-latex] (10.3,2.5) -- (10.8,2.5);
    \draw [-latex] (13.8,2.5) -- (14.4,2.5);
    \draw [-latex] (5.3,-0.2) -- (5.3,0.2);
    \draw [-latex] (5.3,0.8) -- (5.3,1.2);
    \draw [-latex] (5.3,1.8) -- (5.3,2.2);
    \draw [-latex] (5.3,2.8) -- (5.3,3.2);
    \draw [-latex] (8.8,-0.2) -- (8.8,0.2);
    \draw [-latex] (8.8,0.8) -- (8.8,1.2);
    \draw [-latex] (8.8,1.8) -- (8.8,2.2);
    \draw [-latex] (8.8,2.8) -- (8.8,3.2);
    \draw [-latex] (12.3,-0.2) -- (12.3,0.2);
    \draw [-latex] (12.3,0.8) -- (12.3,1.2);
    \draw [-latex] (12.3,1.8) -- (12.3,2.2);
    \draw [-latex] (12.3,2.8) -- (12.3,3.2);
\end{tikzpicture}
\label{bicomplexe2}
\end{center}

\bigskip
\noindent where the horizontal lines correspond to the action of $\mathrm{id}\otimes \dd$ and the vertical lines to the action of $Q_{F}^{(0)}\otimes\mathrm{id}$. Following the same arguments as in the proof of Lemma \ref{lemmedifferentiel}, the total differential $Q_{F}^{(0)}\otimes\mathrm{id}+\mathrm{id}\otimes \dd$ on the sections of the bicomplex can be naturally identified with the following operator:
\begin{equation}\label{differential}
\forall\ \alpha\in\mathfrak{V}^{(n-1)}\hspace{1cm}\partial(\alpha)=Q_{F}^{(0)}\circ\alpha-(-1)^{|\alpha|}\alpha\circ Q_{E}^{(0)}
\end{equation}
where for clarity elements of $\mathfrak{V}^{(n-1)}$ have been identified with maps from $\funct$ to $\functt$. This turns $\mathfrak{V}^{(n-1)}$ into a differential complex, whose elements of homogeneous degree are represented by antidiagonals in the bicomplex. The cohomology is governed by results similar to those in Lemma \ref{lemmefondamental2}:

\begin{lemme}\label{lemmefondamental3}
Let  $n \geq 1$ be an integer, and consider the bicomplex $(\F{V}^{(n)},\partial) $. 
\begin{enumerate}
\item A root-free cocycle is a coboundary.
\item A  cocycle whose root is in the kernel of $ \R{id} \otimes \rho$ is a coboundary.
\end{enumerate}
\end{lemme}

\begin{proof}
In the bicomplex above, lines are exact except maybe in degree $-1$, where coboundaries are given by the kernel of  $ \R{id} \otimes \rho$.
This comes from the fact that $(E,\dd,\rho)$ is exact, and $\Gamma(S^n(F^*))$ is a projective $ {\mathscr O}$-module, so that
tensoring over $ {\mathscr O}$ preserves exactness.
Now, root-free cocycles have no component in sections of $S^{n}(F^*) \otimes E_{-1}$, then such a cocycle takes values only in a sub-bi-complex of  ${\mathfrak V}^{(n-1)} $ where all lines are exact, i.e. the bicomplex of elements of depth greater than or equal to $2$, 
hence it is a coboundary by simple diagram chasing.
This proves the first item. The second item comes from the simple observation that a cocycle in $\Gamma( S^{n}(F^*) \otimes E )$  of depth $1$ has a component in the sections of $S^{n}(F^*) \otimes E_{-1}$.
If this term lies in the kernel of $ \R{id} \otimes \rho$, it is in the image of  $\R{id} \otimes \dd^{(2)}$
and the result then follows by diagram chasing.
\end{proof}

Let us now explain the meaning of this bicomplex. Let $\Phi:{\mathscr E} \to {\mathscr F}$ be any graded ${\mathscr O}$-linear algebra morphism
(not a priori a Lie $\infty$-algebroid morphism). Consider \emph{$\Phi$-derivations} of degree $k$, i.e. maps 
$\delta:{\mathscr E} \to {\mathscr F}$ such that:

\begin{equation}
\delta^{(n)}(fg) = \sum_{i=0}^{n}\,\delta^{(i)}(f) \Phi^{(n-i)}(g) + (-1)^{k|f|} \Phi^{(n-i)}(f) \delta^{(i)}(g)
\end{equation}
for all functions $f,g \in {\mathscr E}$, and all arities $n\geq 0$, or equivalently that:
\begin{equation}
\delta(fg)=\delta(f)\Phi(g)+(-1)^{k|f|}\Phi(f)\delta(g).
\end{equation}
When  ${\mathscr E} = {\mathscr F} $, we recover the previous definition of $\Phi$-derivations, see Equation \eqref{derivation}. By ${\mathscr O}$-linearity,  $\Phi$-derivations are determined by their restrictions to sections of $E^* $, and, as such,
can be identified with sections of $S(F^*) \otimes E $ and those of arity $n$
are sections of $S^{n+1}(F^*) \otimes E $, i.e. elements of ${\mathfrak V}^{(n)}$.
It is clear that if $\delta$
is a $\Phi$-derivation of degree $k$ and arity $n$, then 
\begin{equation}
\label{eq:diffPsider}
\partial (\delta) = Q_{F}^{(0)}\circ\delta  -(-1)^{k} \delta\circ Q_{E}^{(0)} \end{equation}
is a $\Phi$-derivation of degree $k+1$ and arity $n$,
and that the previous operation makes the space of $\Phi$-derivations of arity $n$
a complex. 

\begin{lemme}
\label{lem:meaning}
 Let $\Phi:{\mathscr E} \to {\mathscr F}$ be any graded ${\mathscr O}$-linear algebra morphism.
 Then the space of $\Phi$-derivations of arity $n$, equiped with the differential (\ref{differential})
coincides, as a complex, with the bicomplex ${\mathfrak V}^{(n)}$, equipped with the total differential. 
\end{lemme}

The proof of Theorem \ref{theo:onlyOne} will proceed in two steps. We first have to show the existence of a Lie $\infty$-morphism between $(F,Q_{F})$ and $(E,Q_{E})$, and then prove that two such Lie $\infty$-morphisms are homotopic. The existence part of the first item of Theorem \ref{theo:onlyOne} comes from the following proposition:

\begin{proposition}\label{prop:existencepart}
Any chain map $\phi:F\to E$ is the linear part of a Lie $\infty$-morphism from $ (F,Q_F)$ to $ (E,Q_{E})$ over the identity of $M$.
\end{proposition}

\begin{proof}
We have to show that there exists a Lie ${\infty}$-algebroid morphism from $(F,Q_F)$ to $(E,Q_E)$ over $M$,
that is a degree $0$ morphism of algebras, $\Phi: {\mathscr E} \to {\mathscr F}$, commuting with the homological vector fields:
\begin{equation}\label{intertwine}
Q_F\circ\Phi=\Phi\circ Q_E.
\end{equation}
As usual we expect that $\Phi$ formally involves linear maps of various arities: $\Phi=\sum_{i\geq 0}\Phi^{(i)}$ where each 
component $\Phi^{(i)}$ restricts to a map $\widehat{\Phi}^{(i)}:\Gamma(E^{\ast})\to \Gamma(S^{i+1}(F^{\ast}))$, and satisfies Equations \eqref{baboum}-\eqref{biboum} on $\funct$. By ${\mathscr O}$-linearity, $\widehat{\Phi}^{(i)}$ can be seen as en element of $\Gamma(S^{i+1}F^*)\otimes E$ of degree $0$,
i.e. an element in the bicomplex ${\mathfrak V}^{i}$. Notice that $\Phi^{(i)}$ depends on $\widehat{\Phi}^{(0)},\ldots,\widehat{\Phi}^{(i)}$ only.

Thus, the map $\Phi^{(n-1)}:\Gamma(E^{\ast})\to \Gamma(S^{n}(F^{\ast}))$ will be represented by an antidiagonal of depth $n$ and height $n$ (so that $height-depth=0$). The necessary conditions that the $\Phi^{(i)}$ have to satisfy are obtained by isolating each side of \eqref{intertwine} by arity. Then they are equivalent to the set of following equations:
\begin{equation}
\forall\ k\geq0\hspace{1.5cm}\sum_{i+j=k}Q_{F}^{(i)}\Phi^{(j)}=\sum_{i+j=k}\Phi^{(j)}Q_{E}^{(i)}
\end{equation}
They can be rewritten more symbolically when putting the terms involving the component of $\Phi$ with highest arity on the left-hand side, and using Equation \eqref{differential}:
\begin{align}
\partial(\Phi^{(0)})&=0\\
\forall\ k\geq1\hspace{1.5cm}\partial(\Phi^{(k)})&=\underset{i+j=k}{\sum_{1\leq i,j+1\leq k}}\Phi^{(j)}Q_{E}^{(i)}-Q_{F}^{(i)}\Phi^{(j)}\label{raindrop}
\end{align}
where in each case, $\Phi^{(k)}$ satisfies Equations \eqref{baboum}-\eqref{biboum}. We will construct $\Phi^{(k)}$ by induction, and even if $\Phi^{(k)}$ is not a $\Phi$-derivation, it will be constructed so that it satisfies Equations \eqref{baboum}-\eqref{biboum} and that \eqref{raindrop} makes sense in $\mathrm{Hom}_{\mathscr{O}}(\funct,\functt)$.

The first equation is automatically satisfied by construction of the given chain map $\phi$. One can notice that for $k\geq1$, Equation \eqref{intertwine} means that the right-hand side of Equation \eqref{raindrop} is a $\partial$-coboundary. For $k=1$, we get the  following necessary condition that has to be satisfied by the map $\Phi^{(1)}$:
\begin{equation}\label{huisclos}
\partial(\Phi^{(1)})		= \Phi^{(0)}Q_{E}^{(1)}	-Q_{F}^{(1)}\Phi^{(0)}
\end{equation}
The right-hand side $-$ that we chose to denote by $C_{1}$ $-$ is a linear map from $\funct$ to $\functt$, of arity 1 and of degree 1. Its action is first defined on $\Gamma(E^{\ast})$, and then it is extended to all of $\funct$ as a $\Phi^{(0)}$-derivation. Equation \eqref{commutativity} and the fact that $\Phi^{(0)}$ intertwines the anchors of $E$ and $F$ imply that $C_{1}$ is vertical (i.e. $\mathscr{O}$-linear). We then have to show that it is a $\partial$-coboundary in $\mathfrak{V}^{(1)}=\Gamma\big(S^{2}(E^{\ast})\otimes F\big)$. It is a $\partial$-cocycle:
\begin{align}
\partial(C_{1})
&=Q_{F}^{(0)}\circ(C_{1})+(C_{1})\circ Q_{E}^{(0)}\\
&=(Q_{F}^{(0)}\Phi^{(0)}-\Phi^{(0)}Q_{E}^{(0)}	)Q_{E}^{(1)}+Q_{F}^{(1)}(Q_{F}^{(0)}\Phi^{(0)}-\Phi^{(0)}Q_{E}^{(0)}	)\nonumber\\
&=0\nonumber
\end{align}
where we used the fact that $[Q_{F}^{(0)},Q_{F}^{(1)}]=0$ and $[Q_{E}^{(0)},Q_{E}^{(1)}]=0$. Moreover, we already know from the last section that $[Q_{F}^{(1)},Q_{F}^{(1)}]$ and $[Q_{E}^{(1)},Q_{E}^{(1)}]$ are vertical. Then we get:
\begin{equation}
\forall\ f\in\Funct\hspace{1.3cm}\left(\Phi^{(0)}\circ [Q_{E}^{(1)},Q_{E}^{(1)}]-[Q_{F}^{(1)},Q_{F}^{(1)}]\circ\Phi^{(0)}\right)(f)=0
\end{equation}
and a short calculation shows that this equation is equivalent to:
\begin{equation}\label{fantasia}
\forall\ f\in\Funct\hspace{1.3cm}C_{1}\circ Q_{E}^{(1)}(f)-Q_{F}^{(1)}\circ C_{1}(f)=0
\end{equation}
The $\Phi^{(0)}$-derivation $C_{1}$ being vertical, the last term of the left-hand side of \eqref{fantasia} vanishes and we are left with the first term. The henceforth obtained identity can be translated as:
\begin{equation}\label{bonjouuuur}
(\mathrm{id}\otimes\rho)\circ rt(C_{1})=0
\end{equation}
Being a cocycle whose root lies in the kernel of $\mathrm{id}\otimes\rho$, Lemma \ref{lemmefondamental3} tells us that $C_{1}$ is a $\partial$-coboundary, i.e. that there exists an element $\widehat{\Phi}^{(1)}\in\mathfrak{V}^{(1)}$ such that:
\begin{equation}\label{2homomorphism}
\partial(\widehat{\Phi}^{(1)})=\Phi^{(0)}Q_{E}^{(1)}	-Q_{F}^{(1)}\Phi^{(0)}
\end{equation}
that is exactly Equation \eqref{raindrop} for $k=1$. By using Equation \eqref{boum}, we can extend $\widehat{\Phi}^{(1)}$ to a $\Phi^{(0)}$-derivation $\Phi^{(1)}:\funct\to\functt$.

\begin{figure}[ht]
  \centering
  \begin{tikzpicture}[scale=0.50]
    \coordinate (Origin)   at (0,0);
    \coordinate (XAxisMin) at (10,0);
    \coordinate (XAxisMax) at (0,0);
    \coordinate (YAxisMin) at (10,0);
    \coordinate (YAxisMax) at (10,8);
    \coordinate (ZAxisMin) at (10,-2);
    \coordinate (ZAxisMax) at (10,-1);
    \coordinate (WAxisMin) at (10,-1);
    \coordinate (WAxisMax) at (10,1);
    \draw [ultra thick, black,-latex] (XAxisMin) -- (XAxisMax) node [left] {};
    \draw [ultra thick, black,-latex] (YAxisMin) -- (YAxisMax) node [above] {};

    \clip (-1,-1) rectangle (12cm,9cm); 
    \coordinate (Bone) at (8,4);
    \coordinate (Btwo) at (1,11);
    \coordinate (B1) at (-1,11);
    \coordinate (B2) at (6,4);
    \draw[style=help lines,dashed] (-4,0) grid[step=2cm] (10,11);
    \foreach \x in {-4,-3,...,4}{
      \foreach \y in {2,...,5}{
        \node[draw,circle,inner sep=1pt,fill] at (2*\x,2*\y) {};
        \node[draw,circle,inner sep=1pt,fill,red] at (Bone){};
        \node[draw,circle,inner sep=1pt,fill,red] at (4,8){};
        \node[draw,circle,inner sep=1pt,fill,red] at (2,10){};
        \node[draw,circle,inner sep=1pt,fill,red] at (6,6){};
        \node[draw,circle,inner sep=1pt,fill,blue] at (B2){};
        \node[draw,circle,inner sep=1pt,fill,blue] at (4,6){};
        \node[draw,circle,inner sep=1pt,fill,blue] at (0,10){};
        \node[draw,circle,inner sep=1pt,fill,blue] at (2,8){};
      }
    }
    \node [below] at (8,0)  {$1$};
    \node [below] at (6,0)  {$2$};
    \node [below] at (4,0)  {$3$};
    \node [right] at (10,2)  {$1$};
    \node [right] at (10,4)  {$2$};
    \node [right] at (10,6)  {$3$};
    \node [below right] at (10,0)  {$0$};
    
    \draw [ultra thick,red] (Btwo)
        -- (Bone);
    \node [above right,red] at (6,6) {\large $C_{1}$};
    \draw [ultra thick,blue] (B1)
        -- (B2);
    \node [below left,blue] at (4,6) {\large $\Phi^{(1)}$};
  \end{tikzpicture}
  \caption{\footnotesize $C_{1}$ is a $\partial$-cocycle whose root lies in the kernel of the anchor map $\rho$, and a such it is a coboundary: $C_{1}=\partial(\Phi^{(1)})$.}
  \label{figure5}
\end{figure}
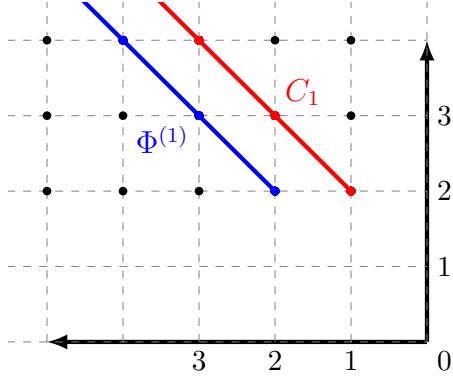

\bigskip
\begin{remarque}
The fact that $rt(C_{1})$ is in the kernel of $\mathrm{id}\otimes\rho$ (see Equation \eqref{bonjouuuur}) is equivalent to the following condition:
\begin{equation}
\forall\ x,y\in \Gamma(F_{-1})\hspace{1.3cm}\rho\Big(\big\{\phi_{1}(x),\phi_{1}(y)\big\}_{E_{-1}}-\phi_{1}\big(\{x,y\}_{F_{-1}}\big)\Big)=0
\end{equation}
which holds true due to the morphism condition \eqref{algebroid2} on $E_{-1}$ and to the identity $\rho'=\rho\circ\phi_{1}$. That is to say, the bundle morphism $\phi_{1}:F_{-1}\to E_{-1}$ is not necessarily compatible with the brackets, but it is rather compatible up to homotopy:
\begin{equation}\label{2homomorphism}
\forall\ x,y\in \Gamma(F_{-1})\hspace{1.3cm}\big\{\phi_{1}(x),\phi_{1}(y)\big\}_{E_{-1}}-\phi_{1}\big(\{x,y\}_{F_{-1}}\big)=\dd\theta_{1,1}(x,y)
\end{equation}
for $\theta_{1,1}:F_{-1}\odot F_{-1}\to E_{-2}$ a degree 0 symmetric map. More generally, Equation \eqref{2homomorphism} defines a chain homotopy $\theta_{i,j}:F_{-i}\vee F_{-j}\to E_{-(i+j)}$ of degree 0 between $\phi\circ\{\ .\ ,\,.\ \}_{F}$ and $\{\ .\ ,\,.\ \}_{E}\circ\phi$, that is to say, a \emph{2-homomorphism} in the $L_{\infty}$ setting.
\end{remarque}

\bigskip
Let $n\geq1$ and assume that $\widehat{\Phi}^{(i)}$ is already defined for $0\leq i\leq n$. For every $0\leq i\leq n$, we denote by $\Phi^{(i)}$ the unique map from $\funct$ to $\functt$, that obeys Equations \eqref{baboum}-\eqref{biboum} and restricts to $\widehat{\Phi}^{(i)}$ on $\Gamma(E^{\ast})$. We want to define $\Phi^{(n+1)}$, so that Equation \eqref{raindrop} for $k=n+1$ be satisfied:
\begin{equation}\label{whitewidow}
\partial(\Phi^{(n+1)})	=\underset{i+j=n+1}{\sum_{0\leq i-1,j\leq n}}\Phi^{(j)}Q_{E}^{(i)}-Q_{F}^{(i)}\Phi^{(j)}
\end{equation}
The right-hand side $-$ denoted by $C_{n+1}$ $-$ is well defined since every term appearing in the expression of $C_{n+1}$ has been defined before. It is first defined on $\Gamma(E^{\ast})$ and then we extend it to all of $\funct$ by the derivation properties of $Q_{E}, Q_{F}$ and  the derivation conditions \eqref{baboum}-\eqref{biboum} satisfied by $\Phi^{(i)}$, for $0\leq i\leq n+1$.
$C_{n+1}$ is of arity $n+1$ and of degree 1, then for degree reasons, it is vertical, has depth $n+1$, and belongs to $\mathfrak{V}^{(n+1)}=\Gamma(S^{n+2}(F^{\ast})\otimes E)$. Moreover, it is a $\partial$-cocycle:
\begin{equation}\label{cocycle}
\partial(C_{n+1})=\underset{i+j=n+1}{\sum_{1\leq i,j+1\leq n+1}}Q_{F}^{(0)}\Phi^{(j)}Q_{E}^{(i)}-Q_{F}^{(0)}Q_{F}^{(i)}\Phi^{(j)}+\Phi^{(j)}Q_{E}^{(i)}Q_{E}^{(0)}-Q_{F}^{(i)}\Phi^{(j)}Q_{E}^{(0)}
\end{equation}
The first term gives:
\begin{align}
\underset{i+j=n+1}{\sum_{1\leq i,j+1\leq n+1}}Q_{F}^{(0)}\Phi^{(j)}Q_{E}^{(i)}&=
Q_{F}^{(0)}\Phi^{(0)}Q_{E}^{(n+1)}+\underset{i+j=n+1}{\sum_{1\leq i,j\leq n}}Q_{F}^{(0)}\Phi^{(j)}Q_{E}^{(i)}\\
&=\Phi^{(0)}Q_{E}^{(0)}Q_{E}^{(n+1)}+\underset{i+j=n+1}{\sum_{1\leq i,j\leq n}}\Phi^{(j)}Q_{E}^{(0)}Q_{E}^{(i)}\nonumber\\
&\quad\quad+\underset{i+j=n+1}{\sum_{1\leq i,j\leq n}}\quad\underset{k+l=j}{\sum_{1\leq k,l+1\leq j}}\big(\Phi^{(l)}Q_{E}^{(k)}-Q_{F}^{(k)}\Phi^{(l)}\big)Q_{E}^{(i)}\nonumber
\end{align}
where we used the recursion hypothesis for all $j\leq n$ on the right-hand side. The last term in \eqref{cocycle} is:
\begin{align}
\underset{i+j=n+1}{\sum_{1\leq i,j+1\leq n+1}}Q_{F}^{(i)}\Phi^{(j)}Q_{E}^{(0)}
&=Q_{F}^{(n+1)}\Phi^{(0)}Q_{E}^{(0)}+\underset{i+j=n+1}{\sum_{1\leq i,j\leq n}}Q_{F}^{(i)}\Phi^{(j)}Q_{E}^{(0)}\\
&=Q_{F}^{(n+1)}Q_{F}^{(0)}\Phi^{(0)}+\underset{i+j=n+1}{\sum_{1\leq i,j\leq n}}Q_{F}^{(i)}Q_{F}^{(0)}\Phi^{(j)}\nonumber\\
&\quad\quad-\underset{i+j=n+1}{\sum_{1\leq i,j\leq n}}\quad\underset{k+l=j}{\sum_{1\leq k,l+1\leq j}}Q_{F}^{(i)}\big(\Phi^{(l)}Q_{E}^{(k)}-Q_{F}^{(k)}\Phi^{(l)}\big)\nonumber
\end{align}
We can then write Equation \eqref{cocycle} as:
\begin{align}\label{cocycle2}
\partial(C_{n+1})
&=\Phi^{(0)}[Q_{E}^{(0)},Q_{E}^{(n+1)}]+\underset{i+j=n+1}{\sum_{1\leq i,j\leq n}}\Phi^{(j)}[Q_{E}^{(0)},Q_{E}^{(i)}]+\underset{i+j=n+1}{\sum_{1\leq i,j\leq n}}\quad\underset{k+l=j}{\sum_{1\leq k,l+1\leq j}}\Phi^{(l)}Q_{E}^{(k)}Q_{E}^{(i)}\nonumber\\
&\quad-[Q_{F}^{(0)},Q^{(n+1)}]\Phi^{(0)}-\underset{i+j=n+1}{\sum_{1\leq i,j\leq n}}[Q_{F}^{(0)},Q_{F}^{(i)}]\Phi^{(j)}-\underset{i+j=n+1}{\sum_{1\leq i,j\leq n}}\quad\underset{k+l=j}{\sum_{1\leq k,l+1\leq j}}Q_{F}^{(k)}Q_{F}^{(i)}\Phi^{(l)}\nonumber\\
&\hspace{1.5cm}-\underset{i+j=n+1}{\sum_{1\leq i,j\leq n}}\quad\underset{k+l=j}{\sum_{1\leq k,l+1\leq j}}Q_{F}^{(k)}\Phi^{(l)}Q_{E}^{(i)}-Q_{F}^{(i)}\Phi^{(l)}Q_{E}^{(k)}
\end{align}
Since $j\in\{1,\ldots,n\}$ then the range of the index $l$ goes from $0$ to $n-1$, and for every $l$, we have $i+k=n+1-l$, in particular $k$ is less than or equal to $n$, and the last line in Equation \eqref{cocycle2} is equal to:
\begin{equation*}
-\sum_{0\leq l\leq n-1}\ \underset{i+k=n+1-l}{\sum_{1\leq i,k\leq n}}Q_{F}^{(k)}\Phi^{(l)}Q_{E}^{(i)}-Q_{F}^{(i)}\Phi^{(l)}Q_{E}^{(k)}
\end{equation*}
which vanishes because the labelling of the second sum is symmetric in $i$ and $k$, whereas the summand is antisymmetric. Moreover Equation \eqref{cocycle2} gives:
\begin{align}\label{cocycle3}
\partial(C_{n+1})
&=\Phi^{(0)}[Q_{E}^{(0)},Q_{E}^{(n+1)}]+\underset{i+j=n+1}{\sum_{1\leq i,j\leq n}}\Phi^{(j)}[Q_{E}^{(0)},Q_{E}^{(i)}]+\sum_{0\leq l\leq n-1}\ \underset{i+k=n+1-l}{\sum_{1\leq i,k\leq n}}\Phi^{(l)}Q_{E}^{(k)}Q_{E}^{(i)}\\
&\quad-[Q_{F}^{(0)},Q_{F}^{(n+1)}]\Phi^{(0)}-\underset{i+j=n+1}{\sum_{1\leq i,j\leq n}}[Q_{F}^{(0)},Q_{F}^{(i)}]\Phi^{(j)}-\sum_{0\leq l\leq n-1}\ \underset{i+k=n+1-l}{\sum_{1\leq i,k\leq n}}Q_{F}^{(k)}Q_{F}^{(i)}\Phi^{(l)}\nonumber
\end{align}

Here we have to split the discussion into two cases. For $n=1$, \eqref{cocycle3} becomes:
\begin{equation}\label{cocycle3bis}
\partial(C_{2})=\Phi^{(1)}\circ[Q_{E}^{(0)},Q_{E}^{(1)}]-[Q_{F}^{(0)},Q_{F}^{(1)}]\circ\Phi^{(1)}=0
\end{equation}
then $C_{2}$ is a $\partial$ cocycle. Since it has depth 2, then by Lemma \ref{lemmefondamental3}, it is a coboundary, that is there exists an element $\widehat{\Phi}^{(2)}\in\mathfrak{V}^{(2)}$ of depth 3 such that:
\begin{equation}
C_{2}=\partial(\widehat{\Phi}^{(2)})
\end{equation}
which is Equation \eqref{raindrop} for $k=2$. We extend it to a map $\Phi^{(2)}:\funct\to\functt$ using Equation \eqref{biboum}.

First let us assume that $n\geq2$. For $l=0$, the sum $\underset{i+k=n+1}{\sum_{1\leq i,k\leq n}}\Phi^{(0)}Q_{E}^{(k)}Q_{E}^{(i)}$ is combined with the first term of Equation \eqref{cocycle3} to give:
\begin{equation}
\Phi^{(0)}\Big([Q_{E}^{(0)},Q_{E}^{(n+1)}]+\underset{i+k=n+1}{\sum_{1\leq i,k\leq n}}Q_{E}^{(k)}Q_{E}^{(i)}\Big)=\Phi^{(0)}[Q_{E},Q_{E}]^{(n+1)}=0
\end{equation}
The same argument applies to the fourth term, and then Equation \eqref{cocycle3} for $n\geq2$ becomes:
\begin{align}\label{cocycle4}
\partial(C_{n+1})
&=\underset{i+j=n+1}{\sum_{1\leq i,j\leq n}}\Phi^{(j)}[Q_{E}^{(0)},Q_{E}^{(i)}]+\sum_{1\leq l\leq n-1}\ \underset{i+k=n+1-l}{\sum_{1\leq i,k\leq n}}\Phi^{(l)}Q_{E}^{(k)}Q_{E}^{(i)}\\
&\hspace{1cm}-\underset{i+j=n+1}{\sum_{1\leq i,j\leq n}}[Q_{F}^{(0)},Q_{F}^{(i)}]\Phi^{(j)}-\sum_{1\leq l\leq n-1}\ \underset{i+k=n+1-l}{\sum_{1\leq i,k\leq n}}Q_{F}^{(k)}Q_{F}^{(i)}\Phi^{(l)}\nonumber\\
&=\Phi^{(n)}\underbrace{[Q_{E}^{(0)},Q_{E}^{(1)}]}_{=\ 0}+\underset{i+j=n+1}{\sum_{2\leq i,j+1\leq n}}\Phi^{(j)}[Q_{E}^{(0)},Q_{E}^{(i)}]+\sum_{1\leq l\leq n-1}\ \underset{i+k=n+1-l}{\sum_{1\leq i,k\leq n}}\Phi^{(l)}Q_{E}^{(k)}Q_{E}^{(i)}\nonumber\\
&\hspace{1cm}-\underbrace{[Q_{F}^{(0)},Q_{F}^{(1)}]}_{=\ 0}\Phi^{(n)}-\underset{i+j=n+1}{\sum_{2\leq i,j+1\leq n}}[Q_{F}^{(0)},Q_{F}^{(i)}]\Phi^{(j)}-\sum_{1\leq l\leq n-1}\ \underset{i+k=n+1-l}{\sum_{1\leq i,k\leq n}}Q_{F}^{(k)}Q_{F}^{(i)}\Phi^{(l)}\nonumber\\
&=\sum_{1\leq l\leq n-1}\quad\underset{i+l=n+1}{\sum_{2\leq i\leq n}}\Phi^{(l)}[Q_{E}^{(0)},Q_{E}^{(i)}]+\underset{i+k=n+1-l}{\sum_{1\leq i,k\leq n-1}}\Phi^{(l)}Q_{E}^{(k)}Q_{E}^{(i)}\nonumber\\
&\hspace{1cm}-\sum_{1\leq l\leq n-1}\quad\underset{i+l=n+1}{\sum_{2\leq i\leq n}}[Q_{F}^{(0)},Q_{F}^{(i)}]\Phi^{(l)}+\underset{i+k=n+1-l}{\sum_{1\leq i,k\leq n-1}}Q_{F}^{(k)}Q_{F}^{(i)}\Phi^{(l)}\nonumber\\
&=\sum_{1\leq l\leq n-1}\quad\underset{i+l=n+1}{\sum_{2\leq i\leq n}}\Phi^{(l)}\Big([Q_{E}^{(0)},Q_{E}^{(i)}]+\underset{j+k=i}{\sum_{1\leq j,k\leq i-1}}Q_{E}^{(k)}Q_{E}^{(j)}\Big)\nonumber\\
&\hspace{1cm}-\sum_{1\leq l\leq n-1}\quad\underset{i+l=n+1}{\sum_{2\leq i\leq n}}\Big([Q_{F}^{(0)},Q_{F}^{(i)}]+\underset{j+k=i}{\sum_{1\leq j,k\leq i-1}}Q_{F}^{(k)}Q_{F}^{(i)}\Big)\Phi^{(l)}\nonumber\\
&=\underset{i+l=n+1}{\sum_{2\leq i,l+1\leq n}}\quad\Phi^{(l)}[Q_{E},Q_{E}]^{(i)}-[Q_{F},Q_{F}]^{(j)}\Phi^{(l)}\nonumber
\end{align}
which vanishes identically (for any $l$ and $i$) because $[Q_{F},Q_{F}]=0$ and $[Q_{E},Q_{E}]=0$. Thus, we indeed have that $C_{n+1}$ is a $\partial$-cocycle:
\begin{equation}
\partial(C_{n+1})=0
\end{equation}
Since it has depth $n+1\geq3$, it is necessarily root-free, and then by Lemma \ref{lemmefondamental3} it is a coboundary: there exists an element $\widehat{\Phi}^{(n+1)}$ of degree 0 such that:
\begin{equation}\label{explication}
C_{n+1}=\partial(\widehat{\Phi}^{(n+1)})
\end{equation}
which is Equation \eqref{raindrop} for $k=n+1$. We extend $\widehat{\Phi}^{(n+1)}$ to a map $\Phi^{(n+1)}:\funct\to\functt$ by using Equation \eqref{biboum}, i.e. the $\Phi$-derivation property. Then, we can apply the same process to the component of Equation \eqref{raindrop} of arity $n+2$.
\begin{figure}[ht]
  \centering
  \begin{tikzpicture}[scale=0.50]
    \coordinate (Origin)   at (0,0);
    \coordinate (XAxisMin) at (10,-2);
    \coordinate (XAxisMax) at (0,-2);
    \coordinate (AAxisMin) at (9,-2);
    \coordinate (AAxisMax) at (7,-2);
    \coordinate (YAxisMin) at (10,1);
    \coordinate (YAxisMax) at (10,8);
    \coordinate (ZAxisMin) at (10,-2);
    \coordinate (ZAxisMax) at (10,-1);
    \coordinate (WAxisMin) at (10,-1);
    \coordinate (WAxisMax) at (10,1);
    \draw [ultra thick, black,-latex] (AAxisMax) -- (XAxisMax) node [left] {};
    \draw [ultra thick, black,-latex] (YAxisMin) -- (YAxisMax) node [above] {};
    \draw [ultra thick, black] (ZAxisMin) -- (ZAxisMax);
    \draw [ultra thick, black,dashed] (WAxisMin) -- (WAxisMax);
    \draw [ultra thick, black] (XAxisMin) -- (AAxisMin);
    \draw [ultra thick, black,dashed] (AAxisMin) -- (AAxisMax);

    \clip (-3,-3) rectangle (12cm,9cm); 
    \coordinate (Bone) at (4,2);
    \coordinate (Btwo) at (-6,12);
    \coordinate (B1) at (-7,11);
    \coordinate (B2) at (2,2);
    \coordinate (B3) at (0,4);
    \coordinate (B4) at (-4,8);
    \draw[style=help lines,dashed] (-4,1) grid[step=2cm] (7,9);
    \draw[style=help lines,dashed] (9,2) grid[step=2cm] (10,9);
    \draw (-3,-2) -- (10,-2)[dashed, gray];
    \draw (6,-2) -- (6,-1)[dashed, gray];
    \draw (4,-2) -- (4,-1)[dashed, gray];
    \draw (2,-2) -- (2,-1)[dashed, gray];
    \draw (0,-2) -- (0,-1)[dashed, gray];
    \draw (-2,-2) -- (-2,-1)[dashed, gray];
    \foreach \x in {-4,-3,...,3}{
      \foreach \y in {1,2,...,4}{
        \node[draw,circle,inner sep=1pt,fill] at (2*\x,2*\y) {};
        \node[draw,circle,inner sep=1pt,fill,red] at (Bone){};
        \node[draw,circle,inner sep=1pt,fill,red] at (0,6){};
        \node[draw,circle,inner sep=1pt,fill,red] at (2,4){};
        \node[draw,circle,inner sep=1pt,fill,red] at (-2,8){};
        \node[draw,circle,inner sep=1pt,fill,red] at (-4,10){};
        \node[draw,circle,inner sep=1pt,fill,blue] at (-4,8){};
        \node[draw,circle,inner sep=1pt,fill,blue] at (-2,6){};
        \node[draw,circle,inner sep=1pt,fill,blue] at (2,2){};
        \node[draw,circle,inner sep=1pt,fill,blue] at (0,4){};
      }
    }
    \node [below] at (4,-2)  {$n+1$};
    \node [below] at (2,-2)  {$n+2$};
    \node [right] at (10,2)  {$n+2$};
    \node [right] at (10,4)  {$n+3$};
    \node [below right] at (10,-2)  {$0$};
    
    \draw [ultra thick,red] (Btwo)
        -- (Bone);
    \node [above right,red] at (2,4) {\large $C_{n+1}$};
    \draw [ultra thick,blue] (2,2)
        -- (-3,7);
    \node [above right,blue] at (-2,2) {\large $\Phi^{(n+1)}$};
  \end{tikzpicture}
  \caption{\footnotesize $C_{n+1}$ is a root-free cocycle, then it is a coboundary: $C_{n+1}=\partial(\Phi^{(n+1)})$}
  \label{figure6}
\end{figure}
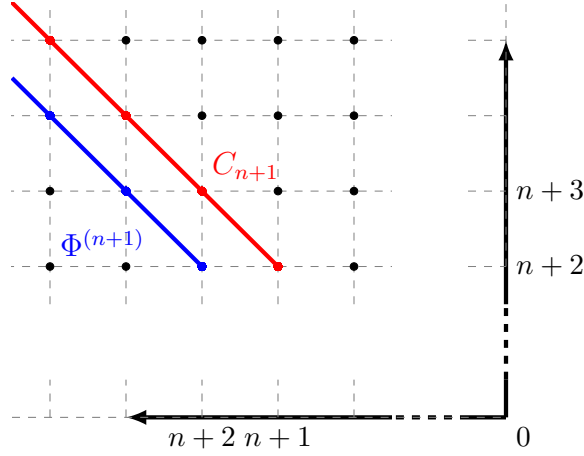

The unique graded algebra morphism $\Phi:\funct\to \functt$ whose components are the $\Phi^{(i)}$'s intertwines $Q_{F}$ and $Q_{E}$ by construction. Thus it is a Lie ${\infty}$-morphism compatible with the Lie $\infty$-algebroid structures on $E$ and $F$.
\end{proof}

We then have shown the existence part of the first item of Theorem \ref{theo:onlyOne}. However, the proof has involved several choices in the definition of such a morphism. Indeed, at each step of the recursion we could have chosen another element of arity $n+1$ which satisfies Equation \eqref{explication}, which would have given another Lie $\infty$-morphism. We will now prove that they are in fact homotopic, in the sense of Definition \ref{def:homotopy}.

In fact we will show a stronger result: that any two Lie $\infty$-morphisms from $(F,Q_{F})$ to $(E,Q_{E})$ are homotopic, that is, the second part of the first item of Theorem \ref{theo:onlyOne}. We first show in Proposition \ref{bonresult} that two Lie $\infty$-morphisms that have the same linear part are homotopic. This result is based on the crucial Lemma \ref{ultralemme}. Then we show in Proposition \ref{lemiam} that two homotopic chain maps from the chain complex $(F,\dd',\rho')$ to $(E,\dd,\rho)$ induce homotopic Lie $\infty$-morphisms, which concludes the proof.

Recall that we can define an operator $[Q,\,.\ ]$ on the space of maps $\alpha:\funct\to\functt$ by:
\begin{equation}
[Q,\alpha]=Q_{F}\circ\alpha-(-1)^{|\alpha|}\alpha\circ Q_{E}
\end{equation}
The following lemma casts light on what happens when we decide to compare two different Lie $\infty$-morphisms built from the map $\phi$ defined at the beginning of this section:

\begin{lemme}\label{ultralemme}
Let $\Phi,\Psi:\funct\to\functt$ be two Lie $\infty$-morphisms from $(F,Q_{F})$ to $(E,Q_{E})$ such that $\Phi^{(i)}=\Psi^{(i)}$ for every $0\leq i\leq n$ for some $n$. Then there exists a Lie $\infty$-morphism $\chi$ which is homotopic to $\Phi$ and which satisfies $\chi=\Psi$ up to arity $n+1$.
\end{lemme}

\begin{proof}
Let $\Psi-\Phi:\funct\to\functt$ be the map obtained by applying $\Psi$ and $\Phi$ to a function of $\funct$, and taking the difference in $\functt$. We can decompose $\Psi-\Phi$ into components of homogeneous arities, and in particular we are interested in the component $(\Psi-\Phi)^{(n+1)}$ of arity $n+1$. We understand it as the difference $\Psi^{(n+1)}-\Phi^{(n+1)}$. Since $\Phi$ and $\Psi$ are graded algebra morphisms, they satisfy Equation \eqref{biboum}, which gives here:
\begin{align}
(\Psi-\Phi)^{(n+1)}(fg)
&=\sum_{k=0}^{n+1}\Psi^{(k)}(f)\Psi^{(n+1-k)}(g)-\Phi^{(k)}(f)\Phi^{(n+1-k)}(g)\\
&=(\Psi-\Phi)^{(n+1)}(f)\Phi^{(0)}(g)+\Phi^{(0)}(f)(\Psi-\Phi)^{(n+1)}(g)\nonumber
\end{align}
because $\Phi^{(k)}=\Psi^{(k)}$ for every $0\leq k\leq n$. Then, $(\Psi-\Phi)^{(n+1)}$ is a $\Phi^{(0)}$-derivation. On the other hand, $\Phi$ and $\Psi$ being Lie $\infty$-morphisms, we have:
\begin{equation}
Q_{F}\circ(\Psi-\Phi)=(\Psi-\Phi)\circ Q_{E}
\end{equation}
Isolating the linear part of $Q_{E},Q_{F}$ on the left-hand side, we obtain the following equation at arity $n+1$:
\begin{equation}
\partial\left((\Psi-\Phi)^{(n+1)}\right)=\sum_{k=0}^{n}\ Q_{F}^{(n-k)}\circ(\Psi-\Phi)^{(k)}-(\Psi-\Phi)^{(k)}\circ Q_{E}^{(n-k)}
\end{equation}
Since $(\Psi-\Phi)^{(k)}=0$ for every $0\leq k\leq n$, the right-hand side vanishes. Thus $(\Psi-\Phi)^{(n+1)}$ is a $\partial$-cocycle in the bicomplex $\mathfrak{V}^{(n+1)}$, and given that it is root free, it is in fact a coboundary. That is to say there is an element $\widehat{H}\in \mathfrak{V}^{(n+1)}$ such that:
\begin{equation}
(\Psi-\Phi)^{(n+1)}=\partial\big(\widehat{H}\big)
\end{equation}
It can then be seen as a map $\widehat{H}:\Gamma(E)^{\ast}\to \Gamma(S^{(n+2)}(F^{\ast}))$. We extend $\widehat{H}$ to a  $\Phi^{(0)}$-derivation of degree $-1$ on $\funct$ that is denoted by $H^{(n+1)}$:
\begin{equation}
\forall\ f,g\in\funct\hspace{1cm}H^{(n+1)}(fg)=H^{(n+1)}(f)\Phi^{(0)}(g)+(-1)^{|f|}\Phi^{(0)}(f)H^{(n+1)}(g)
\end{equation}
We can define a family of functions $(H_{2n+1|t})_{t\in[0,1]}$ up to arity $2n+1$ as:
\begin{align}
\forall\ 0\leq k\leq n\hspace{2.1cm}H_{2n+1|t}^{(k)}(f)&=0\\
\forall\ n+1\leq k\leq 2n+1\hspace{1cm}H_{2n+1|t}^{(k)}(fg)&=H^{(n+1)}(f)\Phi^{(k-n-1)}(g)\\
&\hspace{1cm}+(-1)^{|f|}\Phi^{(k-n-1)}(f)H^{(n+1)}(g)\nonumber
\end{align}
for every $f,g\in\funct$ and every $t\in[0,1]$. That is, $H_{2n+1|t}$ does not actually depend on $t$ for any arity. Let $\Phi_{2n+1|t}$ be the map from $\funct$ to $\functt$ defined up to arity $2n+1$ and for every $t\in[0,1]$ by:
\begin{equation}\label{abovequation}
\forall\ 0\leq i\leq 2n+1\hspace{2cm}\Phi_{2n+1|t}^{(i)}=\Phi^{(i)}+t\,[Q,H_{2n+1|t}]^{(i)}
\end{equation}
Then we observe that $\Phi_{2n+1|t}$ coincides with $\Phi$ and $\Psi$ up to arity $n$, and at arity $n+1$ it satisfies the boundary conditions:
\begin{equation}
\Phi_{2n+1|t=0}^{(n+1)}=\Phi^{(n+1)}\hspace{1cm}\text{and}\hspace{1cm}\Phi_{2n+1|t=1}^{(n+1)}=\Psi^{(n+1)}
\end{equation}
whereas for higher arities, $\Phi_{2n+1|t=1}^{(i)}$ does not necessarily coincide with $\Psi^{(i)}$. Equation \eqref{abovequation} implies that $\Phi_{2n+1|t}$ satisfies Equations \eqref{baboum}-\eqref{biboum}, and that it is differentiable.
At arity $2n+2$, we extend $H_{2n+1|t}$ to a map $H_{2n+2|t}:\funct\to\functt$ for every $t\in[0,1]$:
\begin{align}
H_{2n+2|t}^{(k)}(f)&=H_{2n+1|t}^{(k)}(f)\hspace{2.1cm}\text{for all}\hspace{0.3cm}0\leq k\leq 2n+1\\
H_{2n+2|t}^{(2n+2)}(fg)&=\sum_{k=0}^{2n+2}H_{2n+1|t}^{(k)}(f)\Phi_{2n+1|t}^{(2n+2-k)}(g)+(-1)^{|f|}\Phi_{2n+1|t}^{(2n+2-k)}(f)H_{2n+1|t}^{(k)}(g)\label{pfff}
\end{align}
for every $f,g\in\funct$.

This enables us to define $\Phi_{t}^{(2n+2)}$ as the map satisfying the following boundary value problem:
\begin{equation}
\left\{\begin{array}{ll}
\Phi_{t=0}^{(2n+2)}&=\Phi^{(2n+2)}\\
\frac{\dd\Phi_{t}^{(2n+2)}}{\dd t}&=[Q,H_{2n+2|t}]^{(2n+2)}
\end{array}\right.
\end{equation}
for every $t\in\ ]0,1[$. In particular this implies that $\Phi^{(2n+2)}_{t}$ satisfies Equation \eqref{biboum}, since:
\begin{align}
&\frac{\dd}{\dd t}\left(\Phi_{t}^{(2n+2)}(fg)-\sum_{k=0}^{2n+2}\Phi_{t}^{(k)}(f)\Phi_{t}^{(2n+2-k)}(g)\right)\\
&\hspace{1cm}=[Q,H_{2n+2|t}]^{(2n+2)}(fg)-\sum_{k=0}^{2n+2}[Q,H_{2n+2|t}]^{(k)}(f)\Phi_{t}^{(2n+2-k)}(g)-\Phi_{t}^{(k)}(f)[Q,H_{2n+2|t}]^{(2n+2-k)}(g)\nonumber
\end{align}
which is vanishing because $[Q,H_{2n+2|t}]$ satisfies Equation \eqref{pfff}. We define $\Phi_{2n+2|t}:\funct\to\functt$ by:
\begin{align}
\Phi_{2n+2|t}^{(k)}(f)&=\Phi_{2n+1|t}^{(k)}(f)\hspace{1.1cm}\text{for all}\hspace{0.3cm}0\leq k\leq 2n+1\\
\Phi^{(2n+2)}_{2n+2|t}&=\Phi^{(2n+2)}_{t}
\end{align}
for every $t\in]0,1[$ and $f\in\funct$. By construction, it satisfies equations \eqref{baboum}-\eqref{biboum}.

By recursion we define $H_{i|t}$ and then $\Phi_{i|t}$ for $i\geq 2n+3$. For every $t\in[0,1]$, there is a unique algebra morphism $\Phi_{t}$ restricting to $\Phi_{i|t}$ (for $2n+1\leq i$) when one restricts its action to components of arity $i$ and lower. Moreover, for every $t\in[0,1]$, there is a unique $\Phi_{t}$-derivation $H_{t}$ restricting to $H_{i|t}$ when one considers only its components of arity $i$ or lower. By construction, the two maps are related by the following identity:
\begin{equation}\label{groudon}
\frac{\dd \Phi_{t}^{(i)}}{\dd t}=[Q,H_{t}]^{(i)}
\end{equation}
for every $i\geq0$ and every $t\in]0,1[$ (in particular, there is no subdivision of the interval). In particular, the map $\Phi_{t}^{(i)}$ is differentiable for every $t\in]0,1[$.

By construction the morphism $\Phi_{t=0}$ coincides with $\Phi$, and we define $\chi$ to be $\Phi_{t=1}$. For every arity $i\geq 0$, the commutator $[Q,\Phi_{t}]$ is differentiable on $]0,1[$ and its derivative vanishes, because $\frac{\dd{\Phi}_{t}}{\dd t}$ is $[Q,\,.\ ]$-exact. Since $\Phi^{(i)}_{t}$ is continuous, it means that the commutator $[Q,\Phi_{t}]$ is constant along this closed interval $[0,1]$. And since for $t=0$, $\Phi_{0}=\Phi$ is a Lie $\infty$-morphism, i.e. $[Q,\Phi]=0$, it implies that the commutator is in fact null for every $t$, \emph{i.e.} that $\Phi_{t}$ is a Lie $\infty$-morphism for every $t\in[0,1]$, and in particular $\chi$ is a Lie $\infty$-morphism.

This turns $(\Phi_{t},H_{t})$ into a homotopy  between $\Phi$ and $\chi$, that is: they are homotopic. Since $\chi=\Phi_{t=1}$ and the component of arity $n+1$ of $\Phi_{t=1}$ and $\Psi$ coincide by construction, the Lemma is proven.
\end{proof}


The following proposition is then easily proven:
\begin{proposition}\label{bonresult}
Two Lie $\infty$-morphisms $\Phi_{0},\Phi_{\infty}:\funct\to\functt$ over $M$ admitting the same linear part $\phi=(\phi_{i})_{i\geq1}$ are homotopic.
\begin{center}
\begin{tikzcd}[column sep=0.9cm,row sep=0.6cm]
\dots  \ar[r,"\dd'^{(4)}"] &F_{-3} \ar[dd,"\phi_{3}"] \ar[r,"\dd'^{(3)}"]&F_{-2}\ar[dd,"\phi_{2}"] \ar[r,"\dd'^{(2)}"] &F_{-1}\ar[dd,"\phi_{1}"] \ar[r,"\rho' "] &TM\ar[dd,"id"]\\
 & & & &\\
\dots  \ar[r,"\dd^{(4)}"]  &E_{-3} \ar[r,"\dd^{(3)}"]&E_{-2} \ar[r,"\dd^{(2)}"]  &E_{-1} \ar[r,"\rho"] &TM \\
\end{tikzcd}
\end{center}
\end{proposition}
\begin{proof}
Let us consider two such Lie $\infty$-morphisms $\Phi_{0}$ and $\Phi_{\infty}$. By construction $\Phi_{0}$ and $\Phi_{\infty}$ coincide at arity 0. By Lemma \ref{ultralemme}, there is a Lie $\infty$-morphism $\Phi_{1}$ which is homotopic to $\Phi_{0}$ via a family of Lie ${\infty}$-morphisms $(\Phi_{1,t})_{t\in[0,1]}$  and such that $\Phi_{1}^{(0)}=\Phi_{\infty}^{(0)}$ and $\Phi_{1}^{(1)}=\Phi_{\infty}^{(1)}$. Applying Lemma \ref{ultralemme} once more, we know that there is a Lie $\infty$-morphism $\Phi_{2}$ which is homotopic to $\Phi_{1}$ and such that $\Phi_{2}^{(i)}=\Phi_{\infty}^{(i)}$ for $i=0,1,2$. But $\Phi_{2}$ is actually homotopic to $\Phi_{0}$ since we can create an homotopy $(\Psi_{2,t})_{t\in[0,2]}$ which links these two morphisms:
\begin{equation}
\Psi_{2,t}=\begin{cases}\Phi_{1,t} \quad \text{for}\quad 0\leq t\leq 1\\
\Phi_{2,t} \quad \text{for}\quad 1\leq t\leq 2
\end{cases}
\end{equation}
It is differentiable except possibly in $t=0,1,2$ because the homotopy data provided in the proof of Lemma \ref{ultralemme} are constructed so that the homotopies are differentiable on $]0,1[$. Thus $t\to\Psi_{2,t}$ is a piecewise-$C^{1}$ map taking values in Lie $\infty$-morphisms which satisfies the derivative condition \eqref{Hortense}, and as such it is an honest homotopy between $\Phi_{0}$ and $\Phi_{2}$. By recursion, we can built a sequence of Lie $\infty$-morphism $(\Phi_{k})_{k\geq0}$ such that
\begin{equation}
\Phi_{k}^{(i)}=\Phi_{\infty}^{(i)}
\end{equation}
for all $0\leq i\leq k$, and such that there is a homotopy $\Phi_{k,t}$ which links $\Phi_{k-1}$ to $\Phi_{k}$. Then we can create a family of Lie $\infty$-morphisms $(\Psi_{k,t})_{t\in[0,k]}$ between $\Phi_{0}$ and $\Phi_{k}$ by:
\begin{equation}
\Psi_{k,t}=\begin{cases}\Phi_{1,t} \quad \text{for}\quad 0\leq t\leq 1\\
\Phi_{2,t} \quad \text{for}\quad 1\leq t\leq 2\\
\ldots\\
\Phi_{k,t} \quad \text{for}\quad k-1\leq t\leq k\\
\end{cases}
\end{equation}
 Notice that for $t\geq i$, the component of arity $i$ of $\Psi_{k,t}$ is invariant, equal to $\Phi_{i,t=i}^{(i)}$.  For any arity $0\leq i\leq k$, the path $t\mapsto\widehat{\Psi}^{(i)}_{t}$ is piecewise-$C^1$ in the usual sense, except possibly when $t$ takes integer values.  By construction (see Lemma \ref{ultralemme}), the derivative condition \eqref{Hortense} on $\Psi_{k,t}$ is satisfied wherever it is defined. Then $\Psi_{k,t}$ is a homotopy between $\Phi_{0}$ and $\Phi_{k}$. By recursion, after an infinite number of steps, and applying the bijective property of the function $arctan:\mathbb{R}_{+}\to [0,\frac{\pi}{2}]$, we end up with a homotopy $(\Psi_{\infty, t})_{t\in[0,1]}$ which links $\Phi_{0}$ to $\Phi_{\infty}$.
\end{proof}

Now recall that the choice of the chain map $\phi:F\to E$ was not unique, since there is some latitude in choosing $\phi_{n}$ for each $n\geq 1$. If we are given two chain maps $\phi$ and $\psi$,
\begin{center}
\begin{tikzcd}[column sep=0.7cm,row sep=0.4cm]
\ldots\ar[r,"\dd'"]&\ar[dd,shift left =0.5ex,"\phi"]\ar[dd,shift right =0.5ex,"\psi" left]F_{-2}\ar[r,"\dd'"]&\ar[dd,shift left =0.5ex,"\phi"]\ar[dd,shift right =0.5ex,"\psi" left]F_{-1}\ar[dr,"\rho'"]&\\
&&&T M\\
\ldots\ar[r,"\dd"]&E_{-2}\ar[r,"\dd"]&E_{-1}\ar[ur,"\rho"]&\\
\end{tikzcd}
\end{center}
then they are homotopic in the traditional sense: there exists a vector bundle morphism $h:F\to E$ of degree $-1$ such that:
\begin{equation}
\psi-\phi=h\circ\dd'+\dd\circ h
\end{equation}
\noindent By Proposition \ref{prop:existencepart}, the chain maps $\phi$ and $\psi$ induce Lie $\infty$-morphisms $\Phi,\Psi:\funct\to\functt$. The following proposition gives the relation between these two morphisms:
\begin{proposition}\label{lemiam}
Let $\phi,\psi:F\to E$ be two homotopic chain maps between a Lie $\infty$-algebroid $(F,Q_{F})$ covering $\mathcal{D}'\subset\mathcal{D}$ and a universal Lie $\infty$-algebroid $(E,Q_{E})$ of $\mathcal{D}$. Then any Lie $\infty$-morphisms $\Phi$ and $\Psi$ whose respective linear parts are $\phi$ and $\psi$ are homotopic.
\end{proposition}
%

\begin{proof}
The interpolating map $\phi_{t}=\phi(1-t)+t\,\psi$ is a chain map. Assume that to each $t$ we can associate a Lie $\infty$-morphism $\Phi_{t}:\funct\to\functt$ (see for example Proposition \ref{prop:existencepart}) such that $\Phi_{t=0}=\Phi$. There is no reason that $\Phi_{t=1}=\Psi$, however we necessarily have $\Phi_{t=1}^{(0)}=\Psi^{(0)}$ since $\phi_{t=1}$ and $\psi$ coincide. Although defining the Lie $\infty$-algebroid morphisms $\Phi_{t}$ arity by arity at each $t$, we can nonetheless require that the path $t\mapsto\Phi_{t}$ is $C^{1}$, that is: the map $t\mapsto\Phi^{(i)}_{t}$ is $C^{1}$ for every arity $i\geq0$. This is made possible because all equations in the proof of Proposition \ref{prop:existencepart} can be differentiated with respect to the dependence on $t$.

More precisely, being the dual map of $\phi_{t}$, the morphism $\Phi_{t}^{(0)}$ is $C^{1}$ for every $t\in]0,1[$. Then for every $t\in]0,1[$, the right-hand side of Equation \eqref{huisclos} is $C^{1}$, hence we can choose a coboundary which is $C^{1}$ as well. This defines a $C^{1}$ path  $t\mapsto\Phi^{(1)}_{t}$ in the space of maps from $\funct$ to $\functt$. For the same reason, the right-hand side of Equation \eqref{whitewidow} for $n=1$ is $C^{1}$ on $]0,1[$ because it only involves $\Phi^{(0)}_{t}$ and $\Phi^{(1)}_{t}$. Hence we can choose a $C^{1}$ path $t\mapsto\Phi^{(2)}_{t}$ that solves Equation \eqref{raindrop} for $k=2$, for every $t\in[0,1]$. By recursion, we show that the quantity $C_{n+1}$ defined on the right-hand side of Equation \eqref{whitewidow} is $C^{1}$ for every $t\in]0,1[$, hence we can choose a $C^{1}$ map $t\mapsto\Phi^{(n+1)}_{t}$ which solves Equation \eqref{raindrop}:
\begin{equation}
\partial(\Phi^{(n+1)}_{t})=\underset{i+j=n+1}{\sum_{1\leq i,j+1\leq n+1}}\Phi_{t}^{(j)}\circ Q_{E}^{(i)}-Q_{F}^{(i)}\circ\Phi_{t}^{(j)}
\end{equation}
The proof then provides us with a  path $t\mapsto \Phi_{t}$ taking values in Lie $\infty$-morphisms, such that its components of homogeneous arity are $C^{1}$. Since $\phi_{0}=\phi$ and $\phi_{1}=\psi$, we obtain that $\Phi_{t=0}^{(0)}=\Phi^{(0)}$ and $\Phi_{t=1}^{(0)}=\Psi^{(0)}$. The components of higher arity of $\Phi_{t=1}$ are not necessarily coinciding with those of $\Psi$. 

We have to show that there exists a degree $-1$ map $(H_{t})_{t\in]0,1[}$, which is a $\Phi_{t}$-derivation such that:
\begin{equation}\label{eq:recurrence}
\frac{\dd\Phi^{(k)}_{t}}{\dd t}=[Q,H_{t}]^{(k)}
\end{equation}
for all arities $k\geq0$ wherever the equation makes sense. By construction $\dot{\Phi}^{(0)}_{t}=[Q^{(0)},H^{(0)}_{t}]$ where $H^{(0)}_{t}$ is the (constant) dual map of $h$ seen as a $\Phi_{t}$-derivation. Then we observe that wherever it is defined:
\begin{align}
\partial\left(\frac{\dd\Phi^{(1)}_{t}}{\dd t}-[Q^{(1)},H^{(0)}_{t}]\right)&=\dot{\Phi}^{(0)}_{t}Q_{E}^{(1)}-Q_{F}^{(1)}\dot{\Phi}^{(0)}_{t}-\partial([Q^{(1)},H^{(0)}_{t}])\\
&=\partial(H^{(0)}_{t})Q_{E}^{(1)}-Q_{F}^{(1)}\partial(H^{(0)}_{t})-[\partial(Q^{(1)}),H^{(0)}_{t}]+[Q^{(1)},\partial(H^{(0)}_{t})]\nonumber\\
&=0\nonumber
\end{align}
where $[\partial(Q^{(1)}),H_{t}^{(0)}]=[Q_{F}^{(0)},Q_{F}^{(1)}]\circ H^{(0)}_{t}-H^{(0)}_{t}\circ [Q_{E}^{(0)},Q_{E}^{(1)}]$ vanishes because of Equation \eqref{jacques0}. Then for every $t$ where the above equation is defined, there exists a map $\widehat{H}^{(1)}_{t}$ of degree $-1$ and arity 1 such that:
\begin{equation}
\frac{\dd\Phi^{(1)}_{t}}{\dd t}-\big[Q^{(1)},H^{(0)}_{t}\big]=\big[Q^{(0)},\widehat{H}^{(1)}_{t}\big]
\end{equation}
Since the left-hand side is continuous, we can choose $t\mapsto \widehat{H}_{t}^{(1)}$ to be continuous. Moreover, we extend $\widehat{H}_{t}^{(1)}$ as a $\Phi^{(0)}_{t}$-derivation $H^{(1)}_{t}:\funct\to\functt$. Letting $H_{1|t}=H^{(0)}_{t}+H^{(1)}_{t}$, we obtain the derivative condition \eqref{eq:recurrence} at arity one:
\begin{equation}
\frac{\dd\Phi^{(1)}_{t}}{\dd t}=[Q,H_{1|t}]^{(1)}
\end{equation}

Now assume that for some $n\geq1$ we have a continuous family of degree $-1$ $\Phi_{t}$-derivations $(H_{n|t})_{t\in]0,1[}$ that  satisfy Equation \eqref{eq:recurrence} for all $0\leq i\leq n$. Then we observe that for every $t\in]0,1[$:
\begin{align}
\partial\left(\frac{\dd\Phi^{(n+1)}_{t}}{\dd t}-\big([Q,\,.\ ]-\partial\big)(H_{n|t})^{(n+1)}\right)&=\underset{i+j=n+1}{\sum_{0\leq i-1,j\leq n}}\dot{\Phi}_{t}^{(j)}Q_{E}^{(i)}-Q_{F}^{(i)}\dot{\Phi}_{t}^{(j)}-\partial\circ[Q,H_{n|t}]^{(n+1)}\\
&=\underset{i+j=n+1}{\sum_{0\leq i-1,j\leq n}}[Q,H_{n|t}]^{(j)}Q_{E}^{(i)}-Q_{F}^{(i)}[Q,H_{t}]^{(j)}-\partial\circ[Q,H_{n|t}]^{(n+1)}\nonumber\\
&=-([Q,\,.\ ]-\partial)\circ[Q,H_{n|t}]^{(n+1)}-\partial\circ[Q,H_{n|t}]^{(n+1)}\nonumber\\
&=0\nonumber
\end{align}
Then there exist an element $\widehat{H}_{t}^{(n+1)}$ which satisfies:
\begin{equation}
\frac{\dd\Phi^{(n+1)}_{t}}{\dd t}-\big([Q,\,.\ ]-\partial\big)(H_{n|t})^{(n+1)}=\big[Q^{(0)},\widehat{H}_{t}^{(n+1)}\big]
\end{equation}
for every $t\in]0,1[$. Since the left-hand side is continuous, we can choose $\widehat{H}^{(n+1)}_{t}$ to be continuous. We extend it to all of $\funct$ by the following conditions:
\begin{equation}
\begin{cases}H_{t}^{(n+1)}(fg)=\sum_{i=0}^{n+1}\,H_{t}^{(i)}\Phi_{t}^{(n+1-i)}(g)+(-1)^{|f|}\Phi_{t}^{(n+1-i)}(f)H^{(i)}_{t}(g)\\
H_{t}^{(2n+1)}(\alpha)=\widehat{H}_{t}^{(2n+1)}(\alpha)
\end{cases}
\end{equation}
for every $f,g\in\funct$ and $\alpha\in\Gamma(E^{\ast})$. We define a new map $H_{n+1|t}=H_{n|t}+H^{(n+1)}_{t}$, which is a $\Phi_{t}$-derivation by construction. Then this map satisfies Equation \eqref{eq:recurrence} at arities up to $n+1$.

By recursion, we obtain a continuous family of functions $(H_{\infty|t})_{t\in]0,1[}$ satisfying Equation \eqref{eq:recurrence} for every $t\in]0,1[$. This provides us with a homotopy $(\Phi_{t},H_{t})$ between $\Phi_{t=0}=\Phi$ and $\Phi_{t=1}$. Since $\Phi_{t=1}^{(0)}=\Psi^{(0)}$, using Lemma \ref{ultralemme} we show that $\Phi_{t=1}$ and $\Psi$ are homotopic, hence proving that $\Phi$ and $\Psi$ are homotopic, because homotopy is preserved by composition.\end{proof}

Since any two Lie $\infty$-morphisms $\Phi,\Psi$ from $(F,Q_{F})$ to $(E,Q_{E})$ induce linear parts $\phi,\psi$ which are homotopic in the usual sense (as chain maps), Proposition \ref{lemiam} implies that $\Phi$ and $\Psi$ are homotopic. Hence, the second part of the first item of Theorem \ref{theo:onlyOne} is proven.

\bigskip
Now assume that $(F,Q_{F})$ is a universal Lie $\infty$-algebroid for the Hermann foliation $\mathcal{D}$ as well. Then we could have defined a bundle chain map $\psi:E\to F$ which would satisfy
\begin{equation*}
\rho=\rho'\circ\psi\hspace{1.3cm}\text{and}\hspace{1.3cm}\dd'\circ\psi=\psi\circ\dd
\end{equation*}
as can be represented on the diagram:
\begin{center}
\begin{tikzcd}[column sep=0.7cm,row sep=0.4cm]
\ldots\ar[r,"\dd'"]&\ar[dd,shift left =0.5ex,"\phi"]F_{-2}\ar[r,"\dd'"]&\ar[dd,shift left =0.5ex,"\phi"]F_{-1}\ar[dr,"\rho'"]&\\
&&&T M\\
\ldots\ar[r,"\dd"]&\ar[uu,shift left =0.5ex,"\psi"]E_{-2}\ar[r,"\dd"]&\ar[uu,shift left =0.5ex,"\psi"]E_{-1}\ar[ur,"\rho"]&\\
\end{tikzcd}
\end{center}
A careful analysis shows that $\phi\circ\psi$ and $\psi\circ\phi$ are homotopic to the identity, i.e. there exist families of vector bundle morphisms $h=(h_{i})_{i\geq1}$ and $h'=(h'_{i})_{i\geq1}$ of degree $-1$ such that
\begin{equation*}
h_{i}:E_{-i}\longrightarrow E_{-i-1}\hspace{1cm}\text{and}\hspace{1cm}h'_{i}:F_{-i}\longrightarrow F_{-i-1}
\end{equation*}
and such that the homotopy relations are satisfied:
$$\phi \circ \psi - \mathrm{id} = {\mathrm d} \circ h + h \circ {\mathrm d}\hspace{1cm} \hbox{ and } \hspace{1cm}
\psi \circ \phi - \mathrm{id} = {\mathrm d}' \circ h' + h' \circ {\mathrm d}',$$
This provides us with the result, well-known in abelian categories, that two projective resolution are isomorphic up to homotopies. Since on the one hand the choice of a Lie $\infty$-algebroid structure on a resolution of $\mathcal{E}$ is unique up to homotopy, and on the other hand, two different resolutions induce isomorphic Lie $\infty$-algebroid structures, this finally justifies the name of \emph{universal Lie $\infty$-algebroid of a Hermann foliation $\mathcal{D}$}. Hence we naturally come to the second item of Theorem \ref{theo:onlyOne}, as a corollary of the first item:

\begin{corollaire}
Two  universal Lie $\infty$-algebroids  resolving the Hermann foliation $ {\mathcal D}$ are isomorphic up to homotopy and two such isomorphisms are homotopic.
\end{corollaire}

\section{The geometry of a Hermann foliation through the universal Lie \texorpdfstring{$\infty$}{infinity}-algebroid resolving it}\label{applications}

The purpose of this section is to make use of the universal Lie $\infty$-algebroid resolving a Hermann foliation in order to understand its geometry. 
For this purpose we must associate objects to the universal Lie $\infty$-algebroid which do not depend on the many choices made in the construction. 
In the next sections, we shall first study  the global invariants, and then turn to local ones, attached to the leaves.

\subsection{Universal foliated cohomology}\label{univsection}

Let ${\mathcal D}$ be a Hermann foliation, and let $ (E,Q)$ be a universal Lie $\infty$-algebroid resolving it, with sheaf of functions ${\mathscr E}$. 
Let us denote by $H_{\mathscr{U}}(\mathcal D)$ the cohomology of the complex given by the homological vector field $Q$ acting on $ {\mathscr E}$ and 
call it the  \emph{universal foliated cohomology of $ {\mathcal D}$}. This definition makes sense, i.e. does not depend on the choice of a universal Lie $\infty$-algebroid resolving
$ {\mathcal D}$, in view of the following corollary of Theorem \ref{theo:onlyOne}: 

\begin{corollaire}
\label{cor:cancohom}
Let ${\mathcal D}$ be a Hermann foliation on $M$. Let $ (E,Q_{E})$ and $(F,Q_{F})$ with sheaves of functions ${\mathscr E}$ and ${\mathscr F}$ be universal
Lie $\infty$-algebroids resolving ${\mathcal D}$. The cohomologies of $ ({\mathscr E}, Q_{E}) $ and $ ({\mathscr F}, Q_{F}) $ are canonically isomorphic as graded commutative algebras.  
\end{corollaire}
\begin{proof}
By  Theorem \ref{theo:onlyOne}, 
there exist two Lie $\infty$-algebroid morphisms $\Phi:\mathscr{F}\to\mathscr{E}$ and $\Psi:\mathscr{E}\to\mathscr{F}$
whose compositions are homotopic to the identity maps of $\mathscr{E}$ and $\mathscr{F}$ respectively. 
Any two such morphisms are moreover homotopic.

Hence, both $\Phi$ and $\Psi$ naturally pass to the cohomology 
to give graded commutative algebra morphisms $\widetilde{\Phi}$ and $\widetilde{\Psi}$ relating the cohomologies of $ ({\mathscr E}, Q_{E}) $ and $ ({\mathscr F}, Q_{F})$.
Proposition \ref{prop:HomotopyMeansHomotopy} implies that $\widetilde{\Phi}$ and $\widetilde{\Psi}$ are inverse to one another at the level of cohomology, and that 
any two choices for $\Phi$ and $\Psi$ would yield the same morphisms at the level of cohomology.
\end{proof}

 \begin{remarque}
 It is clear that the first quotient space $H^0_{\mathscr{U}}({\mathcal D})$ consists 
 in smooth functions on $M$ vanishing under the action of $\rho^{\ast}\circ\dd_{\mathrm{dR}}$, that is, the functions
 which are constant along the leaves of $ {\mathcal D}$. Higher cohomologies are more difficult to interpret. 
\end{remarque}
 
Call \emph{forms on ${\mathcal D}$}, and denote the space of forms on $\mathcal{D}$ by $ \Omega({\mathcal D}) $, ${\mathscr O}$-multi-linear skew-symmetric assignments from
$ {\mathcal D}$ to $ {\mathscr O}$:
 $$ \Omega({\mathcal D}) := \mathrm{Hom}_{\mathscr O}\big(\wedge^{\bullet}_{\mathscr O}{\mathcal D},{\mathscr O}\big)= \bigoplus_{k \geq 0} \mathrm{Hom}\big(\wedge_{\mathscr O}^k {\mathcal D},{\mathscr O}\big).$$
 Note that $0$-forms on $\mathcal{D}$ are just functions on $M$. Also, a $k$-form $\alpha$ on $ {\mathcal D}$ induces an honest $k$-form $\alpha_L$ on each regular leaf $L$,
 but maybe not on singular ones. Conversely, any such family $L \mapsto \alpha_L$ of 
$k$-forms, defined on each leaf of ${\mathcal D}$, gives a $k$-form on ${\mathcal D}$ provided that for all $ X_1, \dots,X_k \in {\mathcal D}$, 
the function on $M$ defined at all regular points  $ x$ by $\alpha_{L_x}(X_1, \dots,X_k)$ (with $L_x$ the leaf through $ x$) extends to a function on $M$ which is smooth, real
analytic or holomorphic depending on the context, i.e. extends to a function in ${\mathscr O}$. 
A \emph{foliated de Rham operator} $\dd_{\mathrm{dR}}$ on $ \Omega({\mathcal D})$ is defined by the usual formula:
\begin{align}
{\dd}_{\mathrm{dR}} (\alpha)\big(X_0, \dots , X_k\big) &=\sum_{i=0}^k (-1)^{i} X_i\big[\alpha (X_{0},\ldots,\widehat{X_i},\ldots,X_k)\big]\\&\hspace{1cm}+ \sum_{0\leq i < j\leq k}(-1)^{i+j}\alpha\Big([X_i,X_j],X_0,\ldots, \widehat{X_{i}},\ldots,\widehat{X_{j}},\ldots,X_k\Big) ,\nonumber
\end{align}
 with the understanding that $\widehat{X_i}$ means that $X_i$ is omitted.
 When  forms on $\mathcal{D}$ are interpreted as families of forms on each leaf of ${\mathcal D}$, the de Rham operator consists in taking the usual de Rham operator on each leaf. 
 We call the cohomology of this operator the \emph{foliated  de Rham cohomology of $ {\mathcal D}$} and denote it by $H_{\mathrm{dR}}({\mathcal D})$.

 Let ${\mathcal D}$ be a Hermann foliation on $M$ and let $ (E,Q)$, with sheaf of functions ${\mathscr E}$, 
 be a  Lie $\infty$-algebroid resolving ${\mathcal D}$.
 There is a natural map $\rho^*$ from $\Omega({\mathcal D})$ to $ {\mathscr E}$ 
 given  by associating to each $ \alpha \in \Omega^k({\mathcal D})$
 the element $\rho^* {\alpha} \in  \Gamma(\wedge^k E_{-1}^*) \subset {\mathscr E}_k $ such that for all $x_1 , \dots,x_k \in \Gamma(E_1)$ by:
  \begin{equation}\label{eq:pullbackByAnchor} 
 \rho^* \alpha (x_1, \dots,x_k) = \alpha\big(\rho(x_1), \dots,\rho(x_k)\big) .
  \end{equation}
  It is routine to check that $\alpha \mapsto \rho^* {\alpha}$ is a chain map and a graded commutative algebra morphism, inducing therefore an algebra morphism, still denoted by $\rho^*$,
  from $H_{\mathrm{dR}}({\mathcal D})$ (the foliated de Rham cohomology of ${\mathcal D}$) to 
$ H_{\mathscr{U}}({\mathcal D})$ (the univeral foliated cohomology of ${\mathcal D}$).

\begin{proposition}
Let ${\mathcal D}$ be a Hermann foliation on $M$ that admits a universal
Lie $\infty$-algebroid resolving it.
The algebra morphism $\rho^*$ from the foliated de Rham cohomology of ${\mathcal D}$ to 
the univeral foliated cohomology of ${\mathcal D}$ given by Equation (\ref{eq:pullbackByAnchor}) is canonical 
(i.e. two universal Lie $\infty$-algebroids resolving ${\mathcal D}$ will induce the same morphism). It is an isomorphism in degree $0$ and $1$,
and it is injective in degrees $2$ and $3$.
\end{proposition} 
  \begin{proof}
   Let $ (E,Q)$ and $(E',Q')$, with sheaves of functions ${\mathscr E}$ and ${\mathscr E}'$,
   be two universal 
   Lie $\infty$-algebroids resolving ${\mathcal D}$, and $\Phi:{\mathscr E}' \to {\mathscr E}$ a
   Lie $\infty$-morphism from the first one to the second one as in Corollary \ref{cor:cancohom}. 
   Let $\phi_1: E_{-1} \to E_{-1}'$ be the vector bundle morphism induced by $\Phi$, then $\rho' \circ \phi_1 = \rho $ with $ \rho, \rho'$
   the anchors of the Lie $\infty$-algebroids $E$ and $E'$, which proves
  the first claim by Proposition \ref{cor:cancohom}.
  
  Recall that elements of the cohomology of ${\mathscr E}$ and ${\mathscr E}'$ in degree 0
  are functions on $M$ constant along the leaves of ${\mathcal D}$. 
  A function of ${\mathscr E}$  of degree 1 is a section $\beta$ of ${E}_{-1}^*$.  If it is $Q$-closed, in particular, 
  it has to satisfy $\beta(\dd^{(2)}y)=0$ for all $y \in \Gamma (E_{-2})$. 
  This means that it vanishes on the image of $\dd^{(2)}$, which is the kernel of $\rho: \Gamma(E_{-1} )\to {\mathfrak X}(M)$. 
  It therefore pushes to the quotient $\Gamma(E_{-1})/\mathrm{Ker}(\rho) $ to define 
  an element $\alpha$ in $ \Omega^{1}(\mathcal{D})=\mathrm{Hom}_{\mathscr O}({\mathcal D},{\mathscr O})$ satisfying $ \rho^* \alpha =\beta $. 
  This implies $\rho^*( {\dd }_{\mathrm{dR}}\alpha) = Q[\beta]=0 $.
   Since $\rho^*$ is injective, ${\dd }_{\mathrm{dR}}\alpha =0 $. Also, it is clear that $\alpha= {\dd }_{\mathrm{dR}}f$  for some $f \in  {\mathscr O}$ if and only if 
   $\beta = Q[f] $. This proves that $\rho^*$ induces an isomorphism between the foliated de Rham cohomology and the univeral foliated cohomology.
   
   Injectivity in degrees $2$ and $3$ comes from the fact that if $\rho^* \alpha$ is a coboundary, with $\alpha \in  \Omega ({\mathcal D}) $ of degree $2$ or $3$,
   it has to be of the form $\rho^* \alpha = [Q,\beta]$
   with $\beta$ in $\Gamma(E_{-1}^*) $ or $\beta \in  \Gamma(\wedge^2 E_{-1}^* \oplus E_{-2}^*)$ respectively. In the first case, $\beta$ has to be in the kernel of 
   $(\dd^{(2)})^* : \Gamma(E_{-1}^*) \to \Gamma(E_{-2}^*)$, which implies, since $(E,\dd)$ is a resolution, that it is in the image of $ \rho^*$, i.e. $\beta =\rho^* \gamma$.
   But then $\alpha = {\dd}_{dR} \gamma$ is a coboundary. In the second case, the relation $\rho^* \alpha = [Q,\beta]$ imposes that the component in
   $\Gamma(E_{-2}^*)$ be in the kernel of  $(\dd^{(3)})^* : \Gamma(E_{-2}^*) \to \Gamma(E_{-3}^*)$, so that it is in the image under
   $(\dd^{(2)})^* : \Gamma(E_{-1}^*) \to \Gamma(E_{-2}^*)$ of some $\gamma$. Then $\beta':=\beta + [Q,\gamma]$ is an element that satisfies
   $ \rho^* \alpha = [Q,\beta'] $ and lies in $ \Gamma(\wedge^2 E_{-1}^*)$. It is routine to check that these two conditions imply that $\beta' $
   has to be in the image of $\rho^*$ and that $\alpha$ is exact, hence the claim.
  \end{proof}

Recall that the first item of Theorem \ref{theo:onlyOne} claims that for any Lie $\infty$-algebroid $(F,Q_{F})$ whose anchor map takes values in a sub-foliation $\mathcal{D}'\subset\mathcal{D}$, there exists a Lie ${\infty}$-morphism $\Phi:\funct\to\functt$. Since this morphism commutes with the respective homological vector fields $Q_{F}$ and $Q_{E}$, it passes to the quotient and it sends the universal foliated cohomology of $E$ into the $Q_{F}$-cohomology $H_{Q_{F}}(\functt)$ of the sheaf of functions $\mathscr{F}$ over $F$. Then composing with the anchor map we have a morphism from $H_{\mathrm{dR}}(\mathcal{D})$ to $H_{Q_{F}}(\functt)$:
\begin{center}
\begin{tikzcd}[column sep=0.9cm,row sep=0.6cm]
& H_{\mathrm{dR}}(\mathcal D) \ar[r,"\rho^{\ast}"]& H_{\mathscr U}(\mathcal D)\ar[r,"\Phi"] &  H_{Q_{F}}(\functt)
\end{tikzcd}
\end{center}
In particular if $A$ is a Lie algebroid over $M$, and $\mathcal D$ is the Hermann foliation induced by the anchor map, there is a morphism of cohomologies, provided that the foliation admits a resolution:
\begin{center}
\begin{tikzcd}[column sep=0.9cm,row sep=0.6cm]
& H_{\mathrm{dR}}(\mathcal D) \ar[r]&  H_A
\end{tikzcd}
\end{center}
where $H_{A}$ is the \emph{Lie algebroid cohomology} of $A$. This may be useful to relate the modular class of the Lie algebroid $A$ to the modular class of the foliation induced by its anchor map. The precise meaning of the spaces of the universal foliated cohomology is still to be understood: the modular class, we have been told, is being studied by R. Caseiro.

\subsection{The holonomy graded Lie algebra at a point}\label{pervenche}

In \cites{AndrouSkandal, AndrouZamb}, the \emph{holonomy of a Hermann foliation} at a point $x$ is defined, and a Lie algebra structure is defined on it. In this section, we show that this Lie algebra is the first component (in degree $-1$) of a graded Lie algebra, which is canonically associated to the Hermann foliation.

Let ${\mathcal D}$ be a Hermann foliation on a manifold $M$ and $(E,Q)$ be a universal
Lie $\infty$-algebroid resolving it.
Choose  an arbitrary point $x \in M$. Denote by ${\mathfrak i}_x^* E_{-i}$ the fiber of $E_{-i}$ at $x$. It is clear that the following
sequence is still a complex.
\begin{center}
\begin{tikzcd}[column sep=0.7cm,row sep=0.4cm]
\ldots \ar[r,"\dd^{(4)}"] &{\mathfrak i}_x^* E_{-3}  \ar[r,"\dd^{(3)}"] &{\mathfrak i}_x^* E_{-2}\ar[r,"\dd^{(2)}"] & {\mathrm{Ker}}(\rho_x)
\end{tikzcd}
\end{center}
where $\rho_x$ stands for the anchor map $ \rho_x : {\mathfrak i}_x^* E_{-1}  \to T_x M$ at the point $x$.
Note that the previous sequence may have cohomology: the exactness of the complex  in Definition \ref{def:resolution} at the level of section does not imply that it is exact at all points - 
see Example \ref{sl2}.

The linearity properties of the brackets
defining the Lie $\infty$-algebroid $(E,Q)$ 
in Definition \ref{def:Linftyoids} 
imply that the Lie $\infty$-algebroid structure $(E,Q)$ restricts to yield a $L_\infty$ algebra structure on the graded vector space
 $$ {\mathrm{Ker}} (\rho_{x})\oplus\bigoplus_{i \geq 2} {\mathfrak i}_x^* E_{-i}  $$ 
 whose $ Q$-vector we denote by  $Q_x $. Seen as a $Q$-manifold, its space of functions is the quotient of ${\mathscr E}$ by the $Q$-ideal ${\mathcal I}_x$ of functions vanishing 
 at the point  $x \in M $, and its linear part is precisely the complex above. We call this $L_\infty$ algebra the 
 \emph{holonomy $L_\infty$-algebra at $x \in M$ of the universal Lie $\infty$-algebroid $(E,Q)$ resolving ${\mathcal D}$}.

Taking the cohomology of the holonomy $L_\infty $-algebra at a point $x \in M$, we obtain a graded Lie algebra
(in the symmetric language - i.e. the bracket is symmetric and satisfies the graded Jacobi identity). 
We denote by $H_{\mathcal D}(x)=\oplus_{i \geq 1} H_{\mathcal D}^{-i}(x)$ this cohomology and call it the \emph{holonomy graded Lie algebra of ${\mathcal D}$ at the point $x \in M$}. 
We now see that different choices of universal Lie $\infty$-algebroids resolving ${\mathcal D}$ lead to the same holonomy graded Lie algebra:
 
\begin{proposition}
Let $ {\mathcal D}$ be a Hermann foliation on $M$ that admits two  universal Lie $\infty$-algebroids $(E,Q_{E})$ and $(F,Q_{F})$ resolving it. 
The holonomy graded Lie algebras at $x \in M$, computed with respect to these structures, are canonically isomorphic.
\end{proposition}
\begin{proof}
Two resolutions of the Hermann foliation $\mathcal D$ are isomorphic up to homotopy.
Then they are isomorphic as vector spaces at the level of cohomology. But then Equation \eqref{2homomorphism} tells us that they are in fact isomorphic as graded Lie algebras.
\end{proof}

Since $\sum_{i\geq 1} H^{-i}_{\mathcal D}(x)$ is a graded Lie algebra, $ H^{-1}_{\mathcal D}(x)$ is a Lie algebra.
We show that it is isomorphic to the holonomy Lie algebra constructed by I. Androulidakis and G. Skandalis in \cite{AndrouSkandal}, defined to be the quotient of the Lie 
algebra ${\mathcal D}_x$ of local sections in ${\mathcal D}$ vanishing at $x \in M$ by the Lie ideal $I_x {\mathcal D}$ 
(here $I_x$ stands for the ideal of local fonctions vanishing at $x$). It is equipped with a quotient Lie algebra structure. 

\begin{proposition}
\label{prop:hol=hol}
Let $\mathcal{D}$ be a Hermann foliation that admits a universal Lie $\infty$-algebroid resolving it.
For every $x \in M$, the holonomy Lie algebra of the Hermann foliation at $x $ as defined by Androulidakis and Skandalis  is isomorphic  to the Lie algebra $H^{-1}_{\mathcal D}(x)$. 
\end{proposition}
\begin{proof}
An isomorphism $\tau$ is defined as follows. For all $e \in {\mathfrak i}_x^*  E_{-1}$ in the kernel of $\rho$, 
let $\tilde{e}$ be a local section through $e$. Then $\rho(\tilde{e})$ is a local section of ${\mathcal D}$ that vanishes at $x$. 
Its class modulo the Lie ideal  $I_x {\mathcal D}$ is well-defined, since another choice for $\tilde{e}$ would differ from the first
one by a section in ${\mathcal I}_x \Gamma(E_{-1})$. If $e = \dd^{(2)} f $ for some $f \in {\mathfrak i}_x^* E_{-2}$, $\tilde{e}$ can be chosen 
to be $\dd^{(2)}(\tilde{f})$ with $\tilde{f}$ any section through $f$,
so that $\rho(\tilde{e})= \rho  \circ \dd^{(2)} \tilde{f}=0$ at $x$. 
This yields a well defined map $\tau$ from the holonomy of the Hermann foliation at $x \in M$ to $H^{-1}_{\mathcal D}(x)$. 
It is clear that $\tau$ is surjective, since any local section of ${\mathcal D}$ vanishing at $x \in M$ is of the
form $\rho(\tilde{e})$ with $\tilde{e}$ a local section of $E_{-1}$ whose value at $x$ is in the kernel of $\rho$.
Now, let us prove injectivity. Let $e \in {\mathfrak i}_x^* E_{-1}$ with $\tau(e)=0$. Then for any section $\tilde{e} $ of $E_{-1}$ through $e$, 
we have that $\rho(\tilde{e})$ is in the ideal  $I_x {\mathcal D}$, i.e.  it is a finite sum of the form $\sum_{i=1}^r f_i X_i$,  
where for all $ i=1,  \dots,r$, $X_i \in {\mathcal D}$ and $f_i \in {\mathcal I}_x$. This implies $ \rho( \tilde{e} - \sum_i f_i \tilde{e_i} )=0$, 
where $\tilde{e_i}$ is, for all $i=1, \dots,n$, a local section of $E_{-1}$ mapped to $X_i$ through $\rho$. By definition of a resolution, 
there exists a section $\tilde{h} \in \Gamma(E_{-2})$ such that:
 \begin{equation}  \tilde{e} - \sum_i f_i \tilde{e_i} = \dd^{(2)} \tilde{h}\end{equation}  Evaluating this last relation at $ x\in M$ gives that $\tilde{e}(x)$ is in
 the image of $\dd^{(2)}$. This proves the injectivity of $\tau$. It is clear that $\tau$ is a Lie algebra morphism. This completes the proof.
\end{proof}

\begin{example}\label{symetriesofS}
Let $S$ be a homogeneous function on $\mathbb{R}^{n}$ with isolated singularities. Then by homogeneity, the function $S$ admits only one zero, at the origin, for otherwise the line through 0 and this singularity would be in the zero-locus of $S$. Indeed, if $x\neq0$ is a zero of $S$ then by homogeneity, the function vanishes on  the whole line generated by $x$. If the weight of $S$ is different from 1, then Euler's theorem on homogeneous functions ensures that $\dd S$ admits only one zero, at the origin as well, because the function $S$ stays invariant under the action of the Euler vector field. The space of multivector fields forms a chain complex when equipped with the differential $\iota_{\dd S}$:
\begin{center}
\begin{tikzcd}[column sep=0.7cm,row sep=0.4cm]
\ldots \ar[r,"\iota_{\dd S}"] &\Gamma\big(\bigwedge^{3}(T\mathbb{R}^{n})\big)  \ar[r,"\iota_{\dd S}"] &\Gamma\big(\bigwedge^{2}(T\mathbb{R}^{n})\big)\ar[r,"\iota_{\dd S}"] & \mathfrak{X}(\mathbb{R}^{n})
\end{tikzcd}
\end{center}
The image of the last arrow is the space of vector fields that vanish at $0$. It is easily seen to define a Hermann foliation $\mathcal{F}$ whose leaves are contained in $\mathbb{R}^{n}\backslash \{0\}$. Over $\mathbb{R}^{n}\backslash \{0\}$, the inner derivation $\iota_{\dd S}$ is non degenerate as the differential $\dd S$ is not vanishing. However at the origin, the inner derivation $\iota_{\dd S}$ acts as the trivial differential in the spaces of multivector fields, hence the cohomology reduces to:
\begin{center}
\begin{tikzcd}[column sep=0.7cm,row sep=0.4cm]
\ldots \ar[r,"0"] &\bigwedge^{3}(\mathbb{R}^{n})  \ar[r,"0"] &\bigwedge^{2}(\mathbb{R}^{n})\ar[r,"0"] & \mathbb{R}^{n}
\end{tikzcd}
\end{center}
Hence the $i$-th cohomology space is $H_{\mathcal{D}}^{-i}(0)=\bigwedge^{i+1}(\mathbb{R}^{n})$. We would like to find the graded Lie algebra structure on this cohomology (finding the Lie $\infty$-algebroid structure on the total space over $\mathbb{R}^{n}$ seems to be much more difficult). We expect the bracket $\{\ .\ ,\,.\ \}_{S}$ of two homogeneous elements $P,Q$ of $H_{\mathcal{F}}(0)=\bigwedge^{\bullet}(\mathbb{R}^{n})$ to be a derived bracket with respect to the function $S$, evaluated at the origin:
\begin{equation}
\{P,Q\}_{S}=\big[[P,S],Q\big]\big|_{0}
\end{equation}
where $[\ .\ ,\,.\ ]$ is the  Schouten-Nijenhuis bracket on multivector fields. This bracket is indeed skew-symmetric (or more precisely, graded symmetric for degree one elements) since the last term $\big[[P,Q],S\big]$ in the graded Jacobi identity involving $P,Q$ and the function $S$ vanishes at the origin, for $\dd S=0$ wherever $S=0$. Hence we obtain a structure of graded Lie algebra over the holonomy cohomology at the point zero.
\end{example}

\begin{example}
In Example \ref{sl2}, a resolution of the action of $\mathfrak{sl}_{2}$ on $\mathbb{R}^{2}$ was given, and we obtained a short exact sequence:
\begin{center}
\begin{tikzcd}[column sep=0.7cm,row sep=0.4cm]
0\ar[r] &\Gamma(E_{-2})  \ar[r,"\dd"] &\Gamma(E_{-1})\ar[r,"\rho"]  & \mathcal{D}\ar[r] &0
\end{tikzcd}
\end{center}
where $E_{-1}$ and $E_{-2}$ are trivial vector bundles with respective fiber $\mathbb{R}^{3}\simeq\mathfrak{sl}_{2}$ and $\mathbb{R}$. We want to define the universal Lie $\infty$-algebroid structure on this resolution.

We define the bracket between two constant sections $a,b$ of $E_{-1}$ as their bracket in $\mathfrak{sl}_{2}$. Then we extend it to every sections of $E_{-1}$ by the Leibniz identity \eqref{algebroid}. To define the bracket between sections of $E_{-1}$ and $E_{-2}$, we compute the three brackets:
\begin{equation*}
\{\widetilde{e},\dd r\},\quad \{\widetilde{f},\dd r\}\quad\text{and}\quad \{\widetilde{h},\dd r\}
\end{equation*}
The first one gives:
\begin{equation}
\{\widetilde{e},\dd r\}=xy\{\widetilde{e},\widetilde{h}\}+\rho(\widetilde{e})(xy)\widetilde{h}+\rho(\widetilde{e})(y^{2})\widetilde{e}-x^{2}\{\widetilde{e},\widetilde{f}\}=0
\end{equation}
The other brackets vanish as well. Since $\{\widetilde{e},\dd r\}=\dd\{\widetilde{e}, r\}$, and since $\dd$ is injective, then $\{\widetilde{e}, r\}=0$. The same arguments apply for the brackets with $\widetilde{f}$ and $\widetilde{h}$. We extend these brackets to a bracket between sections of  $E_{-1}$ and $E_{-2}$ by the Leibniz property \eqref{robinson}. Hence we have found the Lie $\infty$-algebroid structure associated to this resolution, since there is no 3-bracket in that example.

Over any point $x\in\mathbb{R}^{3}\backslash \{0\}$, the resolution is exact, thus there is no cohomology. On the contrary, over $0$ both the image of the anchor map and the image of $\dd$ vanish, hence we have some cohomology: $H_{\mathcal{D}}^{-1}(0)\simeq\mathbb{R}^{3}$ and $H_{\mathcal{D}}^{-2}(0)\simeq\mathbb{R}$. The Lie $\infty$-algebroid structure projects to the cohomology and we obtain a graded Lie algebra structure  on  $H_{\mathcal{D}}(0)=\mathbb{R}\oplus\mathfrak{sl}_{2}$: the usual bracket of $\mathfrak{sl}_{2}$ equips $H_{\mathcal{D}}^{-1}(0)$ and all other brackets vanish.
\end{example}

We conclude this section with a characterization of the Hermann foliations described by C. Debord \cite{Debord}.
We call \emph{Debord foliation} a Hermann foliation ${\mathcal D}$ which is a projective $ {\mathscr O}$-module, i.e. which is covered
by a vector bundle $(A,\rho)$ such that $\rho: \Gamma(A) \to {\mathcal D}$ is an isomorphism of ${\mathscr O}$-modules, 
with $\rho$ injective on a dense open subset of $M$. 

\begin{proposition}
Let ${\mathcal D}$ be a Hermann foliation. For every $x \in M$
the following are equivalent:
\begin{enumerate}
\item[(i)] There is a neighborhood of $x \in M$ on which $ {\mathcal D}$  is a Debord foliation. 
\item[(ii)] There is a neighborhood of $x \in M$ on which $ {\mathcal D}$ admits resolutions and $H^{-i}_{\mathcal D}(y)=0$ for all $ i\geq 2$
and all $y$ in this neighborhood.
\item[(iii)] There is a neighborhood of $x \in M$ on which $ {\mathcal D}$ admits resolutions and $H^{-2}_{\mathcal D}(x)=0$. 
\end{enumerate}
\end{proposition}
\begin{proof}
Every Debord foliation is given by a Lie algebroid $A$ whose anchor is injective on an open subset. A resolution is therefore given by
$ E_{-1}=A$ and $E_{-i}:=0$ for all $i \geq 2$. Since a resolution of length $1$ exists in a neighborhood of $x$, the universal 
cohomologies $H^{-i}_{\mathcal D}(y)$ are all trivial for all $i \geq 2$. Hence (i) implies (ii). 
It is obvious  that (ii) implies (iii). Let us assume  that  (iii) holds. Let $E$ be a resolution of ${\mathcal D}$ with anchor $\rho$. 
Since the dimension of the image of $ \dd^{(3)}: E_{-3} \to  E_{-2}$ around $x$ is greater than or equal to its dimension at $x \in M$, while the dimension of the kernel of $ \dd^{(2)}: E_{-2} \to  E_{-1}$ is, around $x$, lower than or equal to its dimension at the point $x$, we indeed have $H^{-2}_{\mathcal D}(y)=0$ 
in a neighborhood of $x \in M$. Moreover, the dimension of the kernel of  $ \dd^{(2)}: E_{-2} \to  E_{-1}$ is constant in a neighborhood of $x$.
This implies that $E_{-1}':= E_{-1}/\dd(E_{-2})$ is a vector bundle. The anchor passes to the quotient to define a  morphism of ${\mathscr O}$-modules
$\rho: \Gamma(E_{-1}') \to {\mathcal D}$ that we still denote by $\rho$, and which is, by construction, an isomorphism of $\mathscr{O}$-modules. This completes the proof.  
\end{proof}

\subsection{Restriction to singular leaves }\label{moineau}

This section is devoted to the proof of Proposition \ref{prop:mainConsequences}. We let $\C{D}$ be a Hermann foliation on a smooth, real analytic or holomorphic manifold $M$ 
with sheaf of functions ${\mathscr O}$. Assume that $\C{D}$ comes equipped with a universal Lie $\infty$-algebroid $(E,Q)$ resolving 
it and let $L$ be a leaf of ${\C{D}}$. We start with a proposition:

\begin{proposition}
\label{prop:parallel}
Let ${\mathcal D}$ be a Hermann foliation on a manifold $M$ and $(E,Q)$ be a universal Lie $\infty$-algebroid resolving it.
The holonomy Lie $\infty$-algebras associated to two points $x$ and $y$ in the same leaf are isomorphic.
\end{proposition}
%
\begin{proof}
The points $x$ and $y$ being in the same leaf, there exist time-dependent sections $e_t $ of $E_{-1}$ such that the flow of the vector field $\rho(e_t)$, is well-defined
and maps $x$ to $y$ at $t=1$. Let $ {\partial}_{e_t} $ stand for the degree $-1$ vertical derivation of the algebra ${\mathscr E}$ of functions on $(E,Q)$
associated with the dual action:
\begin{equation}
\forall\ \xi\in\Gamma(S(E^*))\hspace{2cm}{\partial}_{e_t}(\xi)=\langle\xi,e_{t}\rangle
\end{equation}
The vector field $ V_{t} =[Q,{\partial}_{e_t} ]$ is for all $t \in I$ a derivation of degree $0$ of ${\mathscr E}$
that satisfies, for every function $F$ on $M$:
 \begin{equation} V_{t}[F] ={\partial}_{e_t}\circ\rho^{\ast}(\dd_{\mathrm{dR}}F)= \rho(e_t) [F]\end{equation}
The $1$-parameter family $(\Phi_t)_{t \in I}$ of algebra endomorphims of ${\mathscr E}$ obtained for all $F \in {\mathscr E}$ by solving  the differential equation
 \begin{equation} \frac{ \dd \Phi_t (F)}{\dd t} =  V_{t}  \circ \Phi_t (F),\end{equation}
is therefore a $1$-parameter family of Lie $\infty$-algebroid morphisms of $ (E,Q)$ to itself, which is defined, by construction, over the identity on $M$. Since $V_t$ is of arity 0, and since $\Phi_0=\mathrm{id}$ then $\Phi_1$ is of arity 0.
In particular $\Phi_1$ is a strict isomorphism of the Lie $\infty$-algebroid $(E,Q)$ mapping $x$ to $y$. This completes the proof.  
\end{proof}

We obtain the immediate consequence:

\begin{corollaire}\label{prop:doesNotDependsPoint}
 For any two points $x$ and $y$ in the same leaf of the foliation, $H_{\mathcal D}(x)$ and $H_{\mathcal D}(y)$ are isomorphic as graded Lie algebras.
 \end{corollaire}

Proposition \ref{prop:parallel} in particular implies that for all $i \geq 2$, the vector bundle morphism  $\dd^{(i)} : E_{-i} \to E_{-i+1}$ is of constant rank
at all points of a given leaf  $L$. 
This allows to truncate the Lie $\infty$-algebroid at a certain order $i$, to get a Lie $\infty$-algebroid structure on the graded vector bundle:
 $$  \bigslant{ {\mathfrak i}_L^* E_{-i}}{{\dd}^{(i+1)} ( {\mathfrak i}_L^* E_{-i-1})} \longrightarrow  {\mathfrak i}_L^*E_{-i+1} \longrightarrow \dots \longrightarrow {\mathfrak i}_L^* E_{-1} \longrightarrow TL $$
 Above, ${\mathfrak i}_L^*$ stands for the restriction to the leaf $L$ of a vector bundle over $M$.
 Since grading is bounded below, this Lie $\infty$-algebroid is a Lie $i$-algebroid, that we call the \emph{ $i$-th truncation of $E$}.
 For $i=1$, we get a Lie algebroid that we call the \emph{holonomy Lie algebroid of the leaf $L$}.
 The  name is justified by the following:
 
\begin{proposition}
\label{prop:holoLieAlg}
Let $L$ be a leaf of a Hermann foliation that admits a universal Lie $\infty$-algebroid resolving it.
The $1$-truncation of the Lie $\infty$-algebroid over the leaf $L$ coincides with the holonomy Lie algebroid of $L$ defined by Androulidakis and Skandalis in \cite{AndrouSkandal}. 
\end{proposition}
\begin{proof}
In \cite{AndrouSkandal}, the holonomy Lie algebroid is defined by the vector bundle whose fiber over $x \in L$ is the germ at $x$ of ${\mathcal D}/I_x{\mathcal D}$,
with $I_x$ the ideal of functions vanishing at $x$. The anchor map is defined by the evaluation at $x$ of an element in ${\mathcal D}$
and the bracket is induced from the Lie bracket of vector fields.
 Notice that the kernel of the anchor map is the holonomy Lie algebra at $x$ by construction.
  The proof is essentially a paraphrase of the proof of Proposition \ref{prop:hol=hol}.
\end{proof}

To any Lie $\infty$-algebroid $(E,Q)$ over a manifold $M$, one associates a topological groupoid as follows.
Let $I=[0,1]$.  Morphisms of Lie $\infty$-algebroids from the tangent Lie algebroid $TI$ to 
$(E,Q)$ are in one-to-one correspondence with paths $a:I \to E_{-1}$ over a path $\gamma:I \to M$ such that:
 \begin{equation} \frac{{\mathrm d} \gamma(t)}{{\mathrm  d}t} = \rho \circ a \, (t)\end{equation}
 It is said to be \emph{trivial} when $\gamma(t)$ is a constant path equal to some $m \in M$ and $a(t)=0_m $ for all $t \in I$. 
 A \emph{homotopy between two morphisms of Lie $\infty$-algebroids $a_0,a_1$ from the tangent Lie algebroid $TI$ to $(E,Q)$} is a
 Lie $\infty$-algebroid morphism from the tangent Lie algebroid $TI^2$ to $(E,Q)$ whose restrictions to $\{0\} \times I$ and
 $\{1\} \times I$ in $I^2$ are $a_0$ and $a_1$ respectively, while the restriction to $ I \times \{0\}$ and $ I \times \{1\}$ is trivial. 
 The groupoid product is given by concatenation of paths, which makes sense if we assume them to be trivial in neighborhoods of $t=0$ and $t=1$. 
 To obtain a topology on this quotient, we restrict ourself to $C^1$-paths and equip it with a Banach manifold topology, as in \cite{CrainicFernandes}.
 We call this groupoid the \emph{$1$-truncated groupoid} of  $(E,Q)$.

\begin{proposition}
Let $(E,Q)$ be a universal Lie $\infty$-algebroid resolving a singular foliation ${\mathcal D}$.
The $1$-truncated groupoid of $(E,Q)$ is a cover of the connected component of the manifold of units of the holonomy groupoid
described by Androulidakis and Skandalis in \cite{AndrouSkandal}.
\end{proposition}

\begin{proof}
The holonomy groupoid described in \cite{AndrouSkandal} admits ${\mathcal D}$ for induced foliation on $M$.
Moreover, for any leaf $L$ of ${\mathcal D}$, its restriction to $L$ is a smooth groupoid
integrating the holonomy Lie algebroid, which is shown in Proposition \ref{prop:holoLieAlg}
to coincide with the $1$-truncation  ${\mathfrak i}_L^* E_{-1}/{\dd^{(2)}} {\mathfrak i}_L^*E_{-2}$.

Let us check that the $1$-truncated groupoid of  $(E,Q)$ satisfies the same property.
It admits ${\mathcal D}$ for induced foliation on $M$. Let us show that its restriction to any leaf $L$
coincides with the universal cover of the Lie algebroid $A:={\mathfrak i}_L^* E_{-1}/{\dd^{(2)}} {\mathfrak i}_L^*E_{-2}$.

It is clear that any $E$-path induces an $A$-path in the usual sense of Crainic-Fernandes \cite{CrainicFernandes},
with $A$ being the Lie algebroid ${\mathfrak i}_L^* E_{-1}/{\dd^{(2)}} {\mathfrak i}_L^*E_{-2} $ over $L$.
It is also obvious that if two such paths are homotopic as $E$-paths, their induced $A$-paths are homotopic as $A$-paths.
Hence, the $1$-truncated groupoid of  $(E,Q)$ maps to the source $1$-connected Lie groupoid integrating $A$.

Let us check that this map is bijective. Surjectivity is obvious:  any $A$-path comes from a $E$-path
called its \emph{lift} because the quotient map ${\mathfrak i}_L^*  E_{-1} \to A$ is surjective.
Now, let us check that homotopic $A$-paths arise from homotopic $E$-paths, i.e. that
any Lie $\infty$-algebroid morphism from $TI^2$ to ${\mathfrak i}_L^* E_{-1}/{\dd^{(2)}} {\mathfrak i}_L^*E_{-2} $
lifts to a Lie $\infty$-algebroid morphism to $(E,Q)$ whose boudary values are arbitrary lifts of the initial $A$-paths.
Let $ \alpha$ be a Lie algebroid morphism from $TI^2 \to  {\mathfrak i}_L^* E_{-1}/{\dd^{(2)}} {\mathfrak i}_L^*E_{-2}$ whose restriction
to the boundaries satisfy the usual requirements of homotopies relating two $A$-paths $a_1$ and $a_2$.
The vector bundle morphism $\alpha : TI \to {\mathfrak i}_L^* E_{-1}/{\dd^{(2)}} {\mathfrak i}_L^*E_{-2}$ can be lifted to
a vector bundle morphism $\Phi_{\alpha} $ valued in $ {\mathfrak i}_L^* E_{-1}$ that still satisfies the requirements of homotopies of $E$-paths when restricted to boundaries,
and that relates two arbitrary lifts of $a_1$ and $a_2$. It is not a Lie $\infty$-algebroid morphism a priori, i.e
$ \Psi : = \Phi_\alpha \circ Q - {\diff}_{\mathrm{dR}} \Phi_\alpha $ is not zero. For degree reasons, $\Psi$ is a $\Phi_\alpha$ derivation which may be non-zero only in degree $1$.
In other words, its only term which may be vanishing is a map from $\Gamma({\mathfrak i}_L^*  E_{-1}^*)$ to $\Omega^2(I^2)$.

Since $\Phi_\alpha$ induces a Lie  $\infty$-morphism (in fact, a Lie algebroid morphism) when taking the quotient, $\Psi$
is zero on the image of  $  \Gamma ( A^*) \to \Gamma({\mathfrak i}_L^*  E_{-1}^*)$, i.e. the conormal of the image of $\dd^{(2)}: E_{-2} \to E_{-1}$.
This allows us to modify $\Phi_\alpha$ by adding a map $\Gamma({\mathfrak i}_L^* E_{-2}) $ to $ \Omega^2(I^2)$
so that the relation $ \Phi_\alpha \circ Q = {\diff}_{\mathrm{dR}} \Phi_\alpha $ holds on $\Gamma(E_{-1}^*) \subset {\mathcal E}_1$.
This modified $ \Phi_\alpha$ is a homotopy of $E$-paths by construction.
\end{proof}

\subsection{The Leibniz algebroid of a foliation}
\label{sec:leibnizoid}

As already mentionned in the introduction, it is a long standing problem to decide whether  any Hermann foliation is, locally,
the image of a Lie algebroid under the anchor map. It is known not to be the case globally, see \cite{AndrouZambis}.
We are not able to bring a positive or negative answer to decide this question, but we show that a Leibniz algebroid defining the foliation always exists, 
at least when a finite resolution of the Hermann foliation exists.
Conversely, for every Leibniz algebroid $(L,[\ .\ ,\,. \ ]_L,\rho)$ the \emph{morphism condition}
$ \rho([X,Y]_L)= [\rho(X),\rho(Y)] $ is satisfied and $\rho(\Gamma(L))$ is a Hermann foliation.

\begin{proposition}
Let ${\mathcal D}$ be a Hermann foliation that admits a universal Lie $\infty$-algebroid structure $(E,Q)$ with anchor $\rho$. Assume that its associated resolution is of finite length. Then  the vector bundle of finite rank over $M$ given by $L=\big(S(E^*) \otimes E\big)_{-1} $  comes equipped 
with a Leibniz algebroid structure, when equipped with the Leibniz bracket defined by: \begin{equation} [X,Y]_L := \big[[Q,X],Y\big]  \end{equation} for all $X,Y \in \Gamma(L)$ (identified with vertical 
vector fields $\partial_{X}$ and $\partial_{Y}$ of degree $-1$ on the graded manifold $E$). The anchor of $L$ is given by the composition
\begin{center}
\begin{tikzcd}
 &E_{-1}\oplus\bigoplus_{k \geq 1} S^k(E^*) \otimes E\big|_{-1} \ar[r]& E_{-1 } \ar[r,"\rho"]& TM 
 \end{tikzcd}
 \end{center}
\end{proposition} 
\begin{proof}
For every graded Lie algebra, ${\mathfrak g}:=\sum_{ i\in {\mathbb Z}}{\mathfrak g}_i $ and any homological element $Q$ of degree $+1$,
$ {\mathfrak g}_{-1}$ is a graded Leibniz algebra when equipped with the bracket $(X,Y) \mapsto \big[[Q,X],Y\big]$, see \cite{YKS}. Applied to a Lie $\infty$-algebroid $(E,Q)$ with sheaf 
of functions ${\mathscr E}$, this implies that vector fields of degree $-1$ form a Leibniz algebra with respect to a bracket given as above.
 This Leibniz algebra bracket is easily seen to restrict to vertical vector fields on $E$ of degree $-1$ (i.e. ${\mathscr O}$-linear derivations of $ {\mathscr E}$ of degree $-1$). 
 Since sections of $S^{k}(E^*) \otimes E_{-n} $ are always of positive degree for $k \geq n$, the vector bundle  $L=\big(S(E^*) \otimes E\big)_{-1}$ is a vector bundle of finite rank over $M$. Its sections are precisely vertical vector fields of  degree $-1$ on the graded manifold $E$.  Let us define a degree $+1$ vector bundle morphism $\eta$ by the composition:
\begin{center}
\begin{tikzcd}
 \eta:E_{-1}\oplus\bigoplus_{k \geq 1} S^k(E^*) \otimes E\big|_{-1} \ar[r]& E_{-1 } \ar[r,"\rho"]& TM 
 \end{tikzcd}
 \end{center}
 The proof that $\eta$ is an anchor map follows easily from the identification between $X\in\Gamma(L)$ and $\partial_X$ its associated (vertical) vector field (seen as a derivation of the ring of functions as in Equation \eqref{innerder}):
 \begin{equation} \eta (X)[f] = \big<\rho^{\ast}\circ\dd_{\text{dR}}(f),X\big>=\partial_{X}\big[\rho^{\ast}\circ\dd_{\text{dR}}(f)\big]=\big[\partial_{X},Q\big][f] \end{equation}
which implies,  for all $X ,Y\in \Gamma(L)$ and $f  \in {\mathscr O} $:
 \begin{eqnarray} 
 [X,fY]_L &=& f \, \big[[Q,X],Y\big] + \big([Q,X] \cdot f \big) \, Y \\
 &=&  f [X,Y]_L + \big(\rho (X) [f] \big)\, Y\nonumber
\end{eqnarray}
 Since $\eta(\Gamma(L))= \rho(\Gamma(E_{-1}))$, it is clear that both $(L, [\ .\ ,\,. \ ]_L, \rho)$ and $(E,Q)$ induce the initial Hermann foliation on
 $M$. This proves the  proposition.\end{proof}


\subsection{Homotopy \texorpdfstring{$NQ$}{NQ}-manifolds}


In the course of this thesis, we constructed objects which have a meaning both at the global level and at the local level. But for existence of resolutions, we only have local results. We would like to define some sort of generalized $NQ$-manifold as a collection of $NQ$-manifolds over open sets on a base manifold $M$. Hence each such open set $U$ is the base manifold of an $NQ$-manifold, the only constraint being that $NQ$-manifolds defined over overlapping open sets should be isomorphic up to homotopy on the intersection. This is a clue that homotopy and the induced equivalence relationship will be a key factor in the definition:

\begin{definition}
A \emph{homotopy $NQ$-manifold structure} on a smooth/analytic/complex manifold $M$ is the data of an open cover $\mathscr{A}=\bigcup_{i\in I}U_{i}$ of $M$ together with:
\begin{itemize}
\item a family of $NQ$-manifolds $(E^{i}\to U_{i},Q_{{i}})_{U_{i}\in\mathscr{A}}$, whose respective sheaves of functions are denoted by $\mathscr{E}^{i}$,
\item a Lie $\infty$-morphism $\gamma_{ij}:\mathscr{E}^{j}\to\mathscr{E}^{i}$ covering the identity of $U_{i}\cap U_{j}$ for any intersecting open sets $U_{i}$ and $U_{j}$,
\end{itemize}
such that two consistency conditions are satisfied:

\begin{itemize}
\item for any $U_{i}\in\mathscr{A}$ the Lie $\infty$-morphism $\gamma_{ii}$ is homotopic to the identity morphism on $\mathscr{E}^{i}$:
\begin{equation}\label{sheaf}
\gamma_{ii}\sim\mathrm{id}_{\mathscr{E}^{i}}
\end{equation}
\item  the \emph{cocycle rule}: given three open sets $U_{i}, U_{j}, U_{k}\in\mathscr{A}$ such that $U_{i}\cap U_{j}\cap U_{k}\neq\varnothing$, the Lie $\infty$-morphisms : 
\begin{center}
\begin{tikzcd}[column sep=1.3cm,row sep=0.5cm]
\mathscr{E}^{i}\ar[dr,"\gamma_{ji}"]&\\
&\ar[dl,"\gamma_{kj}"]\mathscr{E}^{j}\\
\mathscr{E}^{k}\ar[uu,"\gamma_{ik}"]&
\end{tikzcd}
\end{center}
satisfy the following homotopy equivalence:
\begin{equation}\label{cocycleNQ}
\gamma_{ik}\circ\gamma_{kj}\circ\gamma_{ji}\sim\mathrm{id}_{\mathscr{E}^{i}_{U_{i}\cap U_{j}\cap U_{k}}}
\end{equation}
\end{itemize}
\end{definition}

One has to note that the collection of $NQ$-manifolds $(E^{i},Q_{i})_{U_{i}\in\mathscr{A}}$ does not necessarily form a graded manifold since $\mathscr{A}$ is not a topology on $M$ and, moreover, the dimension of the fibers can vary from one open set to another. If one takes $k=i$, then since $\gamma_{ik}=\gamma_{ii}$ is homotopic to $\mathrm{id}_{\mathscr{E}^{i}}$ the cocycle condition becomes:
\begin{equation}
\gamma_{ij}\circ\gamma_{ji}\sim\mathrm{id}_{\mathscr{E}^{i}_{U_{i}\cap U_{j}}}
\end{equation}
and since the reverse composition is homotopic to the identity as well (for the same reasons), it means that the Lie $\infty$-algebroid structures over $U_{i}$ and $U_{j}$ are isomorphic up to homotopy, as required. Moreover, if one had taken another map $\gamma'_{ji}:\mathscr{E}^{i}\to\mathscr{E}^{j}$ homotopic to $\gamma_{ji}$ over $U_{i}\cap U_{j}$, then none of the above results would have been modified. Hence the object that matters most is not the map $\gamma_{ji}$ but it is rather its homotopy class, of which $\gamma_{ji}$ should merely be seen as a \emph{representative}.

Now it may happen that $M$ is equipped with another homotopy $NQ$-manifold structure, and we would like to relate them. If $\mathscr{A}$ and $\mathscr{B}$ are two open covers of $M$ and $\mathfrak{E}=(E^{i},Q_{i})_{U_{i}\in\mathscr{A}}$, $\mathfrak{E}'=(E'^{\alpha},Q'_{\alpha})_{V_{\alpha}\in\mathscr{B}}$ are two homotopy $NQ$-structures over $M$, we say that they are \emph{equivalent homopoty $NQ$-manifolds} if for any $U_{i},U_{j}\in\mathscr{A}$ and $V_{\alpha},V_{\beta}\in\mathscr{B}$ such that $U_{i}\cap U_{j}\cap V_{\alpha}\cap V_{\beta}\neq\varnothing$ there exist Lie $\infty$-morphisms $\eta_{\alpha i}:\mathscr{E}^{i}\to\mathscr{E'}^{\alpha}$ and $\eta_{i\alpha}:\mathscr{E'}^{\alpha}\to\mathscr{E}^{i}$ covering the identity on $U_{i}\cap V_{\alpha}$ and $\eta_{\beta j}:\mathscr{E}^{j}\to\mathscr{E'}^{\beta}$ and $\eta_{j\beta}:\mathscr{E'}^{\beta}\to\mathscr{E}^{j}$ covering the identity on $U_{j}\cap V_{\beta}$, such that their compositions are homotopic:
\begin{align}
\gamma'_{\beta\alpha}\circ\eta_{\alpha i}\sim\eta_{\beta j}\circ\gamma_{ji}\hspace{1cm}&\text{and}\hspace{1cm}\gamma_{ji}\circ\eta_{i\alpha}\sim\eta_{j\beta}\circ\gamma'_{\beta\alpha}\label{mariokart1}\\
\eta_{\alpha i}\circ\eta_{i\alpha}\sim\mathrm{id}_{\mathscr{E'}^{\alpha}}\hspace{1cm}&\text{and}\hspace{1cm}\eta_{i\alpha}\circ\eta_{\alpha i}\sim\mathrm{id}_{\mathscr{E}^{i}}\label{mariokart2}\\
\label{mariokart3}\eta_{\beta j}\circ\eta_{j\beta}\sim\mathrm{id}_{\mathscr{E'}^{\beta}}\hspace{1cm}&\text{and}\hspace{1cm}\eta_{j\beta}\circ\eta_{\beta j}\sim\mathrm{id}_{\mathscr{E}^{j}}
\end{align}
on $U_{i}\cap U_{j}\cap V_{\alpha}\cap V_{\beta}$, which can be summarized by the commutativity of the following diagram:
\begin{center}
\begin{tikzcd}[column sep=2cm,row sep=0.8cm]
&\mathscr{E}^{i}\ar[dd,shift left,"\eta_{\alpha i}"]\ar[r,"\gamma_{ji}"]&\mathscr{E}^{j}\ar[dd,shift left,"\eta_{\beta j}"]\\
&&\\
&\ar[uu,shift left,"\eta_{i\alpha}"]\mathscr{E'}^{\alpha}\ar[r,"\gamma'_{\beta\alpha}"]&\ar[uu,shift left,"\eta_{j\beta}"]\mathscr{E'}^{\beta}
\end{tikzcd}
\end{center}
This is an equivalence relation. Also, a careful analysis shows that the union of two equivalent homotopy $NQ$-structures over $M$ is another homotopy $NQ$-structure:

\begin{proposition}\label{equivhomNQ}
Let $\mathfrak{E}=(E^{i},Q_{i})_{U_{i}\in\mathscr{A}}$ and $\mathfrak{E}'=(E'^{\alpha},Q'_{\alpha})_{V_{\alpha}\in\mathscr{B}}$ be two equivalent homotopy $NQ$-manifold structures associated to $M$. Then the union of  the two structures $(F^{\mu},Q_{\mu})_{W_{\mu}\in \mathscr{A}\cup\mathscr{B}}$ is a homotopy $NQ$-manifold structure for $M$, equivalent to both $\mathfrak{E}$ and $\mathfrak{E'}$.
\end{proposition}

\begin{proof}
Existence of the morphisms between $NQ$-structures over open sets comes by equivalence between $\mathfrak{E}$ and $\mathfrak{E}'$, and equations from \eqref{mariokart1} to \eqref{mariokart3} imply the homotopy equivalence \eqref{cocycleNQ}.
\end{proof}

Given two intersecting open sets $U_{i},U_{j}\in\mathscr{A}$, the $NQ$-manifold structures on $E^{i}$ and $E^{j}$ are isomorphic up to homotopy. It implies in particular that the $Q$-cohomologies induced by $Q_{i}$ and $Q_{j}$ in the sheaves of functions $\funct_{i}|_{U_{i}\cap U_{j}}$ and $\funct_{j}|_{U_{i}\cap U_{j}}$ are isomorphic. This would allow us to patch the local cohomologies into a more intricate structure, even if there is no globally defined homological vector field. However, for this to make sense, the isomorphism between the universal cohomologies associated to $E_{i}$ and $E_{j}$ should not depend on the choice of the morphisms $\gamma_{ij}$. In other words  we have to show that two homotopic maps $\gamma_{ji}:\mathscr{E}^{i}\to\mathscr{E}^{j}$ and $\gamma'_{ji}:\mathscr{E}^{i}\to\mathscr{E}^{j}$ induce the same map at the level of cohomology, that is: the image of cocycles under $\gamma_{ji}$ and $\gamma'_{ji}$ differ by a coboundary term. Indeed for any $\alpha\in\mathscr{E}^{i}$ which is $Q_{i}$-closed we have by Proposition \ref{prop:HomotopyMeansHomotopy}:
\begin{align}
\gamma_{ji}(\alpha)-\gamma'_{ji}(\alpha)=Q_{j}\circ H(\alpha)+H\circ Q_{i}(\alpha)
\end{align}
for some degree $-1$ map $H:\mathscr{E}^{i}_{U_{i}\cap U_{j}}\to\mathscr{E}^{j}_{U_{i}\cap U_{j}}$. Since $Q_{i}(\alpha)=0$, it implies that the left-hand side is $Q_{j}$-exact, and thus that the two homotopic maps $\gamma_{ji}$ and $\gamma'_{ji}$ induce the same isomorphism in cohomology, noted $\widetilde{\gamma}_{ji}$.

We denote by $H_{Q_{i}}$ the cohomology defined by $Q_{i}$ in the sheaf of functions $\funct_{i}$. Two homotopic maps being related by a $[Q,\,.\ ]$-exact term, Equations \eqref{sheaf} and \eqref{cocycleNQ} become strict at the level of cohomology. Hence we are left with the following data:
\begin{enumerate}
\item a sheaf $H_{Q_{i}}$ over each open set $U_{i}$ of the covering $\mathscr{A}$;
\item for any intersecting open sets $U_i,U_j\in\mathscr{A}$, an isomorphism of sheaves:
\begin{equation}
\widetilde{\gamma}_{ji}:H_{Q_{j}}|_{U_{i}\cap U_{j}}\to H_{Q_{i}}|_{U_{i}\cap U_{j}};
\end{equation}
\item for any overlapping open sets $U_i,U_j,U_k\in\mathscr{A}$, a cocycle condition:
\begin{equation}
\widetilde{\gamma}_{ik}\circ\widetilde{\gamma}_{kj}\circ\widetilde{\gamma}_{ji}=\mathrm{id}_{H_{Q_{i}}|_{U_{i}\cap U_{j}\cap U_{k}}}
\end{equation}
\end{enumerate}
We call this collection of data $(H_{Q_{i}},\widetilde{\gamma}_{ij})$ gluing data for sheaves with respect to the covering $\mathscr{A}$. We have the following result \cite{Felder}:

\begin{lemme}\label{isosheaf}
Let $M$ be a manifold, and $\mathscr{A}=\bigcup_{i\in I}U_{i}$ be an open covering of $M$. Let $(F_{i},\psi_{ij})$ be  gluing data for sheaves with respect to the covering $\mathscr{A}$. Then there exists a sheaf $\mathcal{F}$ on $M$, together with isomorphisms of sheaves:
\begin{equation*}
\phi_{i}:\mathcal{F}(U_{i})\longrightarrow F_{i}(U_{i})
\end{equation*}
such that the following diagram is commutative:
\begin{center}
\begin{tikzcd}[column sep=2cm,row sep=0.8cm]
&\mathcal{F}(U_{i}\cap U_{j})\ar[dd,"\mathrm{id}"]\ar[r,"\phi_{i}"]&F_{i}(U_{i}\cap U_{j})\ar[dd,"\psi_{ji}"]\\
&&\\
&\mathcal{F}(U_{i}\cap U_{j})\ar[r,"\phi_{j}"]&F_{j}(U_{i}\cap U_{j})
\end{tikzcd}
\end{center}
\end{lemme}

\begin{proof}
The sheaf $\mathcal{F}$ is defined over any open set $U\subset M$ by its set of sections:
\begin{equation}
\mathcal{F}(U)=\Big\{ (s_{i})_{i\in I}\,|\,s_{i}\in F_{i}(U\cap U_{i}), \,\psi_{ji}(s_{i}|_{U\cap U_{i}\cap U_{j}})=s_{j}|_{U\cap U_{i}\cap U_{j}}\Big\}
\end{equation}
The condition on the right-hand side is an equivalence relation. Since the restriction maps of $\mathcal{F}$ are induced by the ones for $F_{i}$, $\mathcal{F}$ is a presheaf. Now let us show that $\mathcal{F}$ is a sheaf, i.e. that is satisfies the locality and the gluing conditions.

The locality condition is automatically satisfied. Let $U$ be an open set admitting an open cover $U=\bigcup_{\alpha}U_{\alpha}$. Let $s,t\in\mathcal{F}(U)$ be two sections which satisfy $s|_{U_{\alpha}}=t|_{U_{\alpha}}$. Let $U_{i}\in\mathscr{A}$ which intersects $U$, then $s_{i}|_{U_{\alpha}\cap U_{i}}=t_{i}|_{U_{\alpha}\cap U_{i}}$. By construction the set $\{U_{\alpha}\cap U_{i}\}_{\alpha}$ covers $U\cap U_{i}$, and since $F_{i}$ is a sheaf, it implies that $s_{i}|_{U\cap U_{i}}=t_{i}|_{U\cap U_{i}}$. Since it is true for every open set $U_{i}$, it proves that $s=t$ by definition of $\mathcal{F}$.

The proof of the gluing condition follows the same lines. Assume that on the open cover $U=\bigcup_{\alpha}U_{\alpha}$, there are sections $s^{\alpha}\in\mathcal{F}(U_{\alpha})$ which coincide when two open sets overlap: $s^{\alpha}=s^{\beta}$ on $U_{\alpha}\cap U_{\beta}$. Let $U_{i}\in\mathscr{A}$ whose intersection with $U$ is not empty, then the equality $s^{\alpha}=s^{\beta}$ is still valid on $U_{\alpha}\cap U_{\beta}\cap U_{i}$. The family $\{U_{\alpha}\cap U_{i}\}_{\alpha}$ forms an open cover of $U\cap U_{i}$, and since $F_{i}$ is a sheaf, we deduce that the sections $\{s^{\alpha}\}_{\alpha}$ can be glued to induce a section $s_{i}$ on $U\cap U_{i}$. Let $U_{j}$ be another open set which intersects $U\cap U_{i}$ and let $s_{j}$ be the section of $F_{j}(U\cap U_{j})$ constructed from the $s^{\alpha}$'s. Then $\psi_{ji}(s_{i})=s_{j}$ because this is true on each restrictions $U_{\alpha}\cap U_{i}\cap U_{j}$. Then $s=(s_{i})_{i\in I}$ defines a section $s$ of $\mathcal{F}(U)$, which restricts to $s^{\alpha}$ on $U_{\alpha}$. From the two conditions, we conclude that $\mathcal{F}$ is a sheaf.

Now we need to show that the projection:
\begin{align}
\phi_{i}:\hspace{0.2cm}\mathcal{F}(U_{i})&\xrightarrow{\hspace*{2cm}} F_{i}(U_{i})\\
(s_{k})_{k\in I}&\xmapsto{\hspace*{2cm}} s_{i}\nonumber
\end{align}
induces an isomorphism. Commutativity of the diagram in Lemma \ref{isosheaf} is guaranteed by definition of $\mathcal{F}$. To define the reverse map, take some section $s_{i}\in F_{i}(U_{i})$, and define $s_{j}$ on the overlap $U_{i}\cap U_{j}$ by $s_{j}=\psi_{ij}^{-1}(s_{i})=\psi_{ji}(s_{i})$ for every open set $U_{j}\in\mathscr{A}$ intersecting $U_{i}$. We then have a collection of maps $(s_{k})_{k\in I}$ which defines a section of $\mathcal{F}(U_{i})$. The compatibility condition $\psi_{ji}(s_{i})=s_{j}$ is obtained by construction. However one needs the additional condition that $\psi_{kj}(s_{j})=s_{k}$ for every $j,k\in I$. This is equivalent to the cocycle condition: $\psi_{kj}\circ\psi_{ji}\circ\psi_{ik}=\mathrm{id}$. Hence the map $\phi_{i}:\mathcal{F}(U_{i})\to F_{i}(U_{i})$ is an isomorphism.
\end{proof}

This lemma enables us to glue the cohomologies $H_{Q_{i}}$ over each open set $U_{i}$ together to construct a sheaf $\mathscr{H}_{\mathscr{A}}$ on $M$, called the \emph{homotopy cohomology of $M$ with respect to the covering $\mathscr{A}$}. Note however that there is no homological vector field globally defined. The following result is crucial:

\begin{proposition}\label{haricotvert}
Two equivalent homotopy $NQ$-structures on $M$ induce isomorphic homotopy cohomologies.
\end{proposition}

\begin{proof}
Let $\mathfrak{E}=(E^{i},Q_{i})_{U_{i}\in\mathscr{A}}$ and $\mathfrak{E}'=(E'^{\alpha},Q'_{\alpha})_{V_{\alpha}\in\mathscr{B}}$ be two equivalent homotopy $NQ$-manifold structures on $M$. We show that there is an isomorphism of sheaves between their respective homotopy cohomologies $\mathscr{H}_{\mathscr{A}}$ and $\mathscr{H}_{\mathscr{B}}$. A morphism of sheaves  $\Xi:\mathscr{H}_{\mathscr{A}}\to \mathscr{H}_{\mathscr{B}}$ is a collection of algebra morphisms $\Xi(U):\mathscr{H}_{\mathscr{A}}(U)\to \mathscr{H}_{\mathscr{B}}(U)$ that are compatible with restrictions, where $U$ ranges over all open sets of $M$. An isomorphism of sheaves of algebras is an isomorphism of algebras on each open set $U$.

For any $U_{i}\in\mathscr{A}$ and $V_{\alpha}\in\mathscr{B}$, define the open set $W_{i\alpha}=U_{i}\cap V_{\alpha}$. Since $\mathfrak{E}$ and $\mathfrak{E}'$ are equivalent homotopy $NQ$-structures, there exist maps $\eta_{\alpha i}:\funct^{i}\to\funct'^{\alpha}$ and $\eta_{i\alpha}:\funct'^{\alpha}\to\funct$ covering the identity on $W_{i\alpha}$ which are inverse up to homotopy to one another. Hence at the level of cohomologies, the quotient map $\widetilde{\eta}_{\alpha i}$ becomes an isomorphism between $H_{Q_{i}}(W_{i\alpha})$ and $H_{Q_{\alpha}'}(W_{i\alpha})$, with inverse map $\widetilde{\eta}_{i \alpha}$. The morphism $\widetilde{\eta}_{\alpha i}$ can be lifted to a morphism $\Xi_{\alpha i}$ between $\mathscr{H}_{\mathscr{A}}(W_{i\alpha})$ and $\mathscr{H}_{\mathscr{B}}(W_{i\alpha})$, using the bijectivity of the morphisms described in Lemma \ref{isosheaf}:
\begin{equation}
\Xi_{\alpha i}=\phi_{\alpha}^{-1}\circ\widetilde{\eta}_{\alpha i}\circ\phi_{i}:\mathscr{H}_{\mathscr{A}}(W_{i\alpha})\longrightarrow\mathscr{H}_{\mathscr{B}}(W_{i\alpha})
\end{equation}
We write $\Xi_{i\alpha}$ for its inverse. We want to define a morphism of sheaves $\Xi:\mathscr{H}_{\mathscr{A}}\to\mathscr{H}_{\mathscr{B}}$ which restricts on each open set $W_{i\alpha}$ to $\Xi_{i\alpha}$.

Let us show that if we have two open sets $W_{i\alpha}$ and $W_{j\beta}$ such that $W_{i\alpha}\cap W_{j\beta}=V_{\alpha}\cap V_{\beta}\cap U_{i}\cap U_{j}\neq\varnothing$, the associated isomorphisms $\Xi_{\alpha i}$ and $\Xi_{\beta j}$ coincide on $\mathscr{H}_{\mathscr{A}}(W_{i\alpha}\cap W_{j\beta})$. Showing that they coincide is equivalent to showing that the image of a section $s=(s_{k})_{k\in I}$ by both isomorphisms gives the same section of $\mathscr{H}_{\mathscr{B}}(W_{i\alpha}\cap W_{j\beta})$. By definition of the sections of the sheaves $\mathscr{H}_{\mathscr{A}}$ and $\mathscr{H}_{\mathscr{B}}$, this is equivalent to showing that the following diagram commutes:
\begin{center}
\begin{tikzcd}[column sep=2cm,row sep=0.8cm]
&H_{Q_{i}}(W_{i\alpha}\cap W_{j\beta})\ar[dd,"\widetilde{\eta}_{\alpha i}"]\ar[r,"\widetilde{\gamma}_{ji}"]&H_{Q_{j}}(W_{i\alpha}\cap W_{j\beta})\ar[dd,"\widetilde{\eta}_{\beta j}"]\\
&&\\
&H_{Q_{\alpha}'}(W_{i\alpha}\cap W_{j\beta})\ar[r,"\widetilde{\gamma}'_{\beta\alpha}"]&H_{Q_{\beta}'}(W_{i\alpha}\cap W_{j\beta})
\end{tikzcd}
\end{center}
This is precisely the case since Equation \eqref{mariokart1} becomes exact at the level of cohomologies:
\begin{equation}\label{hautdegamme}
\widetilde{\gamma}'_{\beta\alpha}\circ\widetilde{\eta}_{\alpha i}=\widetilde{\eta}_{\beta j}\circ\widetilde{\gamma}_{ji}
\end{equation}
This shows that $\Xi_{\alpha i}$ and $\Xi_{\beta j}$ coincide on $\mathscr{H}_{\mathscr{A}}(W_{i\alpha}\cap W_{j\beta})$. Now given any open set $U$, and any section $s\in\mathscr{H}_{\mathscr{A}}(U)$, we show that there exists a section $t\in\mathscr{H}_{\mathscr{B}}(U)$, and that it is uniquely defined. On $U\cap U_{i}$ the section $s$ reduces to some $s_{i}\in H_{Q_{i}}(U\cap U_{i})$. Let $V_{\alpha}\in\mathscr{B}$, then the section $s_{i}$ can be restricted to $U\cap U_{i}\cap V_{\alpha}$, and the restriction is denoted by $s_{i}|_{\alpha}$. It is mapped to a unique section $t_{\alpha}|_{i}\in H_{Q_{\alpha}'}(U\cap U_{i}\cap V_{\alpha})$ via the isomorphism $\widetilde{\eta}_{\alpha i}$. By Equation \eqref{hautdegamme}, by the fact that $\mathscr{A}=\bigcup_{\alpha\in I}U_{i}$ is an open cover of $M$, and by the fact that $H_{Q_{\alpha}'}(U\cap U_{\alpha})$ is a sheaf, we deduce that the different sections $t_{\alpha}|_{i}$ defined on various $U_{i}\cap U_{\alpha}$ can be glued. On the open set $U_{\alpha}$, we obtain a section $t_{\alpha}\in H_{Q_{\alpha}'}(U\cap V_{\alpha})$. On some open set $V_{\beta}$ intersecting $U\cap V_{\alpha}$, we obtain another section $t_{\beta}\in H_{Q_{\beta}'}(U\cap V_{\beta})$, but Equation \eqref{hautdegamme} gives us the compatibility condition $\widetilde{\gamma}'_{\beta\alpha}(t_{\alpha})=t_{\beta}$ on the overlap. Hence, by definition, the data $t=(t_{\alpha})_{\alpha\in A}$ forms a section of $\mathscr{H}_{\mathscr{B}}(U)$. Thus, the map $\Xi$ is indeed a morphism of sheaves restricting to $\Xi_{i\alpha}$ on each open set $W_{i\alpha}$.

Now if we had chosen another image $t'=(t'_{\alpha})_{\alpha\in A}$ of $s$, then on any intersection $U\cap U_{i}\cap V_{\alpha}$, the restrictions coincide:
\begin{equation}
t_{\alpha}|_{i}=t'_{\alpha}|_{i}
\end{equation}
because $\widetilde{\eta}_{\alpha i}$ is an isomorphism. However since $H_{Q_{\alpha}'}$ is a sheaf, and since $\mathscr{A}$ is an open cover of $M$, it implies that $t'_{\alpha}=t_{\alpha}$. To conclude, the image of $s$ by $\Xi$ is unique. Following the same lines, we construct the reverse morphism $\Xi^{-1}$ from the local morphisms $\Xi_{i\alpha}:\mathscr{H}_{\mathscr{B}}(W_{i\alpha})\to\mathscr{H}_{\mathscr{A}}(W_{i\alpha})$. Since $\Xi$ is an isomorphism on each open set $U$, we deduce that it is an isomorphism of sheaves, giving the desired result.
\end{proof}

%
%

\begin{remarque}
By Proposition \ref{equivhomNQ}, we know that the union $\mathfrak{F}$ of two equivalent homotopy $NQ$-structures $\mathfrak{E}=(E^{i},Q_{i})_{U_{i}\in\mathscr{A}}$ and $\mathfrak{E}'=(E'^{\alpha},Q'_{\alpha})_{V_{\alpha}\in\mathscr{B}}$ is itself a homotopy $NQ$-structure for $M$, equivalent to both $\mathfrak{E}$ and $\mathfrak{E}'$. 
Hence it defines a homotopy cohomology which is isomorphic to the homotopy cohomologies $\mathscr{H}_{\mathscr{A}}$ and $\mathscr{H}_{\mathscr{B}}$.\end{remarque}

Now we want to relate these notions to universal Lie $\infty$-algebroids and their universal foliated cohomologies, as defined in Section \ref{univsection}. Let $M$ be a real analytic (resp. complex) manifold, and let $\mathcal{D}$ be a real analytic (resp. holomorphic) Hermann foliation on $M$. We know by the first item of Theorem \ref{theo:existe}, that there exists a universal Lie $\infty$-algebroid resolving $\mathcal{D}$ in the neighborhood  of every point of $M$. Then let us take three open sets $U,V,W$ small enough so that Theorem \ref{theo:existe} applies: we are provided with a universal Lie $\infty$-algebroid structure over each set. Then Theorem \ref{theo:onlyOne} does the following:  we know by the second item that over $U\cap V\cap W$ there exist Lie $\infty$-morphisms (invertible up to homotopy) between the universal Lie $\infty$-algebroids associated respectively to $U,V$ and $W$. And by the first item, we deduce that the composition of any three of such morphisms is homotopic to the identity. Since the collection of neighborhoods around every point of $M$ forms an open cover $\mathscr{A}=\bigcup_{i\in I} U_{i}$, this provides $M$ with a homotopy $NQ$-manifold structure $\mathfrak{E}=(E^{i},Q_{i})_{U_{i}\in\mathscr{A}}$, that we call the \emph{universal homotopy $NQ$-manifold structure associated to the foliation $\mathcal{D}$}. The name is justified because any other homotopy $NQ$-manifold structure consisting of universal Lie $\infty$-algebroids is equivalent to $\mathfrak{E}$, by Theorem \ref{theo:onlyOne}. In the smooth case, this construction is not possible in general, but in some cases such structure exists, for example:
\begin{proposition}
Let $M$ be a smooth manifold and let $\mathcal{D}$ be a locally real analytic Hermann foliation on $M$. Then $M$ admits a universal homotopy $NQ$-manifold structure associated to $\mathcal{D}$.
\end{proposition}
Then, we deduce that $\mathfrak{E}$ admits a homotopy cohomology $\mathscr{H}_{\mathfrak{U}}(\mathcal{D})$, that we call the \emph{universal homotopy cohomology of the foliation $\mathcal{D}$}. The name is justified by the fact that any other choices of equivalent homotopy $NQ$-manifold structures give isomorphic homotopy cohomologies by Proposition \ref{haricotvert}. The universal homotopy cohomology and the universal foliated cohomology defined in Section \ref{univsection} are related:

\begin{proposition}
Let $M$ be a manifold and let $\mathcal{D}$ be a Hermann foliation resolved by a universal Lie $\infty$-algebroid $(E,Q)$. Assume that $M$ admits a universal homotopy $NQ$-structure $\mathfrak{E}'=(E'_{\alpha},Q_{\alpha}')_{V_{\alpha}\in\mathscr{B}}$ associated to $\mathcal{D}$. Then the universal homotopy cohomology $\mathscr{H}_{\mathfrak{U}}(\mathcal{D})$ and the universal foliated cohomology $H_{\mathfrak{U}}(\mathcal{D})$ are isomorphic as sheaves.
\end{proposition}

\begin{proof}
Let $\mathscr{A}=\bigcup_{i\in I}U_{i}$ be an open cover of $M$. The universal Lie $\infty$-algebroid structure can be restricted to any open set $U_{i}$, providing it with a $NQ$-manifold structure $(E^{i},Q_{i})$. On the overlap $U_{i}\cap U_{j}$ the transition map $\gamma_{ij}:E_{j}\to E_{i}$ is the identity map because the resolution is by trivial vector bundles. Hence the universal Lie $\infty$-algebroid structure provides $M$ with a homotopy $NQ$-manifold structure $\mathfrak{E}=(E_{i},Q_{i})_{U_{i}\in\mathscr{A}}$ adapted to the foliation $\mathcal{D}$, by construction. Then by Theorem \ref{theo:onlyOne}, following the same lines as in the discussion above, the two homotopy $NQ$-manifold structures $\mathfrak{E}$ and $\mathfrak{E'}$ are equivalent. Thus by Proposition \ref{haricotvert} they define isomorphic homotopy cohomologies. The homotopy cohomology of $\mathfrak{E}$ is the universal foliated cohomology $H_{\mathfrak{U}}(\mathcal{D})$, and the homotopy cohomology associated to $\mathfrak{E}'$ is the universal homotopy cohomology $\mathscr{H}_{\mathfrak{U}}(\mathcal{D})$. Hence the result.
\end{proof}


\cleardoublepage
\phantomsection
\addcontentsline{toc}{chapter}{Bibliography}
\bibliographystyle{naturemag}
\bibliography{mabibliographie}

\end{document}